\documentclass{article}
\usepackage[english]{babel}
\usepackage{amssymb}
\usepackage{amsmath}
\usepackage{hyperref}
\usepackage{enumerate}
\usepackage{amsthm}
\usepackage{amsfonts}
\usepackage[dvipsnames]{xcolor}
\usepackage[curve]{xypic}
\usepackage{diagbox}
\usepackage{multirow}

\usepackage{MnSymbol}

\DeclareMathSymbol{\mlq}{\mathord}{operators}{``}
\DeclareMathSymbol{\mrq}{\mathord}{operators}{`'}

\usepackage{tikz}
\newcommand*\mathcircled[1]{\tikz[baseline=(char.base)]{
            \node[shape=circle,draw,inner sep=2pt] (char) {#1};}}

\usepackage[numbers,sort&compress]{natbib}

\newcommand*{\tildeotimes}{\mathbin{\tilde{\otimes}}}
\newcommand*{\utildeotimes}{\mathbin{\tilde{\underline{\otimes}}}}
\newcommand*{\tildeboxtimes}{\mathbin{\tilde{\boxtimes}}}
\newcommand*{\utildetimes}{\mathbin{\tilde{\underline{\times}}}}

\newtheorem{SATZ}{Theorem}[section]

\newtheorem{KEYLEMMA}[SATZ]{Key Lemma}
\newtheorem{LEMMA}[SATZ]{Lemma}
\newtheorem{DEF}[SATZ]{Definition}
\newtheorem{PROP}[SATZ]{Proposition}

\newtheorem{DEFLEMMA}[SATZ]{Definition/Lemma}
\newtheorem{BEISPIEL}[SATZ]{Example}

\newtheorem{EX}[SATZ]{Exercise}

\newtheorem{KOR}[SATZ]{Corollary}

\newtheorem{BEM}[SATZ]{Remark}

\newtheorem{WARNING}[SATZ]{Warning}

\newtheoremstyle{bare}        
  {}            
  {}            
  {\normalfont}                 
  {}                            
  {\mdseries\scshape}                   
  {}                            
  {.0em}                           
  {\thmnumber{#2}#1. \thmnote{\normalfont\textsc{(#3)}} } 

\theoremstyle{bare}
\newtheorem{PAR}[SATZ]{}

\makeatletter
\NewDocumentCommand{\intrev}{e{_^}}{%
  \mathop{\mathpalette\intrev@{{#1}{#2}}}%
}
\NewDocumentCommand{\intrev@}{mm}{%
  \intrev@@#1#2%
}
\NewDocumentCommand{\intrev@@}{mmm}{%
  \begingroup
  \sbox\z@{$\m@th#1\int$}%
  \reflectbox{\usebox\z@}%
  \IfValueT{#2}{
    _{#2}%
  }%
  \IfValueT{#3}{
    ^{\kern-\ifx#1\displaystyle0.5\else0.4\fi\wd\z@#3}%
  }%
  \endgroup
}
\makeatother

\newcommand{\iso}{\stackrel{\sim}{\longrightarrow}}

\newcommand{\Mat}[1]{ \left(\begin{matrix} #1 \end{matrix} \right) }

\newcommand{\R}{ \mathbb{R} }

\newcommand{\Z}{ \mathbb{Z} }

\newcommand{\N}{ \mathbb{N} }

\newcommand{\OOO}{\text{\footnotesize$\mathcal{O}$}}
\newcommand{\OO}{ {\cal O} }

\DeclareMathOperator{\sgn}{sgn}

\DeclareMathOperator{\eval}{eval}
\DeclareMathOperator{\PF}{PF}
\DeclareMathOperator{\Sh}{Sh}
\DeclareMathOperator{\Sz}{Sz}
\DeclareMathOperator{\AW}{AW}
\DeclareMathOperator{\EZ}{EZ}
\DeclareMathOperator{\Aw}{Aw}
\DeclareMathOperator{\Ez}{Ez}
\DeclareMathOperator{\Alg}{Alg}
\DeclareMathOperator{\Def}{Def}
\newcommand{\uHom}{ \underline{\Hom} }

\newcommand{\uAlg}{ \underline{\Alg} }
\newcommand{\uDef}{ \underline{\Def} }

\newcommand{\Awfrak}{ \mathfrak{Aw} }
\newcommand{\Ezfrak}{ \mathfrak{Ez} }
\newcommand{\AWfrak}{ \mathfrak{AW} }
\newcommand{\EZfrak}{ \mathfrak{EZ} }
\newcommand{\Cfrak}{ \mathfrak{C} }

\DeclareMathSymbol{\mlq}{\mathord}{operators}{``}
\DeclareMathSymbol{\mrq}{\mathord}{operators}{`'}

\DeclareMathOperator{\switch}{switch}

\DeclareMathOperator{\fg}{fg}
\DeclareMathOperator{\nd}{nd}
\DeclareMathOperator{\Kar}{Kar}

\DeclareMathOperator{\Cor}{Cor}
\DeclareMathOperator{\Nat}{Nat}
\newcommand{\uNat}{ \underline{\Nat} }
\DeclareMathOperator{\Coh}{Coh}
\newcommand{\uCoh}{ \underline{\Coh} }
\DeclareMathOperator{\Dinat}{Dinat}

\DeclareMathOperator*{\colim}{colim}

\DeclareMathOperator*{\laxlim}{laxlim}
\DeclareMathOperator*{\laxcolim}{laxcolim}
\DeclareMathOperator*{\oplaxlim}{oplaxlim}
\DeclareMathOperator*{\oplaxcolim}{oplaxcolim}
\DeclareMathOperator{\Ind}{Ind}

\DeclareMathOperator{\id}{id}
\DeclareMathOperator{\barconst}{bar}
\DeclareMathOperator{\cobarconst}{cobar}
\DeclareMathOperator{\barlurie}{Bar}
\DeclareMathOperator{\Coalg}{Coalg}
\DeclareMathOperator{\cobarlurie}{Cobar}
\DeclareMathOperator{\dd}{d}
\DeclareMathOperator{\cart}{cart}
\DeclareMathOperator{\cocart}{cocart}
\DeclareMathOperator{\lax}{lax}
\DeclareMathOperator{\oplax}{oplax}
\DeclareMathOperator{\pseudo}{pseudo}
\DeclareMathOperator{\inert}{inert}
\DeclareMathOperator{\act}{act}

\DeclareMathOperator{\coker}{coker}
\DeclareMathOperator{\Hom}{Hom}
\DeclareMathOperator{\Fun}{Fun}

\DeclareMathOperator{\End}{End}

\DeclareMathOperator{\can}{can}

\DeclareMathOperator{\Pro}{Pro}

\DeclareMathOperator{\FinSet}{FinSet}

\DeclareMathOperator{\inj}{inj}

\DeclareMathOperator{\pr}{pr}

\DeclareMathOperator{\Cat}{Cat}
\DeclareMathOperator{\SCat}{SCat}
\DeclareMathOperator{\AbCat}{AbCat}
\DeclareMathOperator{\Ab}{Ab}
\DeclareMathOperator{\Op}{Op}
\DeclareMathOperator{\Coop}{Coop}
\DeclareMathOperator{\Gpd}{Gpd}
\DeclareMathOperator{\Set}{Set}
\DeclareMathOperator{\Dia}{Dia}

\DeclareMathOperator{\im}{im}
\DeclareMathOperator{\gr}{gr}
\DeclareMathOperator{\op}{op}

\newcommand{\ddd}{{}^{\downarrow \downarrow}}
\newcommand{\tw}{{}^{\uparrow \downarrow}}
\newcommand{\twc}{{}^{\uparrow \downarrow \uparrow}}
\newcommand{\twcop}{{}^{\downarrow \uparrow \downarrow}}
\newcommand{\twop}{{}^{\downarrow \uparrow}}
\newcommand{\twtw}{{}^{\uparrow \downarrow\uparrow \downarrow}}

\DeclareMathOperator{\Ch}{Ch}

\DeclareMathOperator{\dec}{dec}
\DeclareMathOperator{\tot}{tot}

\DeclareFontFamily{U}{wncy}{}
    \DeclareFontShape{U}{wncy}{m}{n}{<->wncyr10}{}
    \DeclareSymbolFont{mcy}{U}{wncy}{m}{n}
    \DeclareMathSymbol{\Sha}{\mathord}{mcy}{"58} 

\begin{document}

\title{Lectures on bar and cobar}
\date{20.7.2025}
\author{Fritz H\"ormann\footnote{These lectures are funded by donations. If the notes helped you, please consider a contribution: \url{https://donorbox.org/fritz-hormann-independent-researcher-in-mathematics}.}}

\maketitle

\begin{abstract}

We discuss Lurie's (derived) bar and cobar constructions, the classical ones 
for simplicial groups and sets (due to Eilenberg-MacLane and Kan), and the classical ones for differential graded 
(co)algebras (due to Eilenberg-MacLane and Adams) and their relations, putting them 
 into an abstract framework which makes sense much more generally for any cofibration of $\infty$-operads. 
Along these lines we give new and rather conceptual existence proofs of Lurie's adjunction (where bar is left adjoint) 
and the classical adjunction (where bar is right adjoint). We also recover various classical comparison maps, 
e.g.\@ the Szczarba and Hess-Tonks maps comparing Adams cobar with Kan's loop group.
\end{abstract}

\tableofcontents

\section{Introduction}

 A small category $I$ gives rise to a diagram
\begin{equation} \label{eqintro0} \vcenter{ \xymatrix{ & \twop I \ar[ld]_{\pi_1} \ar[rd]^{\pi_2} \\
I & & I^{\op}
} }
\end{equation}
where $\twop I$ is the (dual) twisted arrow category\footnote{see Section~\ref{SECTTW} for a discussion and legitimization of the unusual notation. }.
This yields for a complete and cocomplete $\infty$-category $\mathcal{C}$ an adjunction
\begin{equation}\label{eqintrocobar} \xymatrix{\mathcal{C}^{I} \ar@<3pt>[rrr]^{\barlurie_{\mathcal{C}}\, :=\, \pi_{2,!} \pi_1^*} & & &  \ar@<3pt>[lll]^{\cobarlurie_{\mathcal{C}}\, :=\, \pi_{1,*} \pi_2^*}  \mathcal{C}^{I^{\op}}  }\end{equation}
with $\pi_{1,*} \pi_2^*$ {\em right} adjoint, 
where the $\pi_{1,!}$ and $\pi_{2,*}$ are the left and right Kan extensions along $\pi_1$, and $\pi_2$, respectively. 
This is already (a simple special case) of our {\em derived} bar and cobar adjunction. There is also an adjunction
\[ \xymatrix{ \mathcal{C}^{\twop I}  \ar@<3pt>[rr]^{\cobarconst_{\mathcal{C}}\, :=\, \pi_{1,!}} & &  \ar@<3pt>[ll]^{\barconst_{\mathcal{C}}\, :=\, \pi_1^*}  \mathcal{C}^{I}  }\]
with $\pi_1^*$ fully-faithful and $\pi_{1,!}$ {\em left} adjoint. This is already (a simple special case) of our {\em classical} bar and cobar adjunction. Here we imagine a 1-category $\mathcal{C}$, but this is not necessary. 

It might be surprising that these simple constructions can be useful. However, they have a straight-forward generalization to (co)fibrations of $\infty$-operads\footnote{In these notes, so far, I decided to work with planar (co)operads as opposed to symmetric ones, but nothing special about planar operads is used in the
general definitions.}, such that for a monoidal $\infty$-category $\mathcal{C} \to  \OOO$ and $I= \OOO$ (planar associative operad) we have
\[ \boxed{ \mathrm{(Co)}\barlurie_{\mathcal{C} \to \OOO} = \text{ Lurie's (co)bar construction. } }   \]
In the presence of augmentations\footnote{i.e.\@ passing to a situation where the unit is final} and for $I= \OOO$, this yields a formula (where $\rho^*$ is a certain equivalence of $\infty$-categories, see below)
\[ \barlurie_{\mathcal{C} \to \OOO}\, :=\, \pi_{2,!} \circ  \widetilde{\pi_1^*} \cong \colim_{\Delta^{\op}} \circ (\rho^*)^{-1} \circ \widetilde{\pi_1^*} .  \]
Here $\widetilde{\pi_1^*}$ --- the ``{\em classical} bar construction'' $\barconst_{\mathcal{C} \to \OOO}$ in this case --- is not a simple pull-back as above anymore but a slightly twisted variant that transfers an algebra into a coalgebra and also $\tw I$, for an operad $I$, is a bit different from the usual notion. 
This implies the existence of Lurie's bar construction, as soon as geometric realizations exist, and hence a simple alternative to the existence proof presented in \cite{Lur11}.

For the {\em classical} (co)bar construction, we have (justifying the name)
\[ \boxed{ \dec_* \circ (\rho^*)^{-1} \circ \barconst_{(\Set^{\Delta^{\op}}, \times) \to \OOO} = \overline{W} }  \]
with the functor $\overline{W}$ (Eilenberg-MacLane classifying space)  and 
\[ \boxed{  \cobarconst_{(\Set^{\Delta^{\op}}, \times) \to \OOO} \circ \rho^* \circ \dec^*  = M^K   } \]
where $M^K$ is the geometric cobar construction, a simplicial monoid whose associated simplicial group is Kan's loop group, 

Also, we have for an Abelian tensor category $(\mathcal{D}, \otimes)$:
\[ \boxed{ \dec_* \circ (\rho^*)^{-1} \circ \barconst_{(\Ch_{\ge 0}(\mathcal{D})_{/1}, \tildeotimes) \to \OOO} \cong \barconst^{\mathrm{EM}}  }\]
where $\barconst^{\mathrm{EM}}$ is the Eilenberg-MacLane bar construction and 
\[ \boxed{  \cobarconst_{(\Ch_{\ge 0}(\mathcal{D}), \tildeotimes) \to \OOO} \circ \rho^* \circ \dec^* \cong \cobarconst^{\mathrm{Adams}} \circ P } \]
where $\cobarconst^{\mathrm{Adams}}$ is the Adams cobar construction and $P$ is the functor ``connected cover'' with its canonical coaugmentation. 

The reader familiar with Lurie's definition will have noticed that this is not the definition in terms of a certain pairing of categories that Lurie gives. 
However, the (co)bar adjunction in the simple case above (i.e.\@ over a point) can also be described (not completely obviously) by saying that the functor $\rho_1 \times \rho_2$ in the diagram
\begin{equation} \label{eqintro1} \vcenter{ \xymatrix{ & (\twop \mathcal{C})^I \ar[ld]_{\rho_1} \ar[rd]^{\rho_2} \\
\mathcal{C}^I & & (\mathcal{C}^{I^{\op}})^{\op} } }
\end{equation}
 is a fibration represented by the the bifunctor
\[  X, Y \mapsto  \Hom_{\mathcal{C}^{I^{\op}}}(\barlurie_{\mathcal{C}} X, Y) \cong \Hom_{\mathcal{C}^{I}}(X, \cobarlurie_{\mathcal{C}} Y).  \]
with $\barlurie$ and $\cobarlurie$ defined in (\ref{eqintrocobar}).
This also generalizes straight-forwardly to the setting of cofibrations of operads, and gives precisely Lurie's definition. 

The primary aim of these lectures was to explain the facts mentioned so far in detail.  
Eventually, however, they grew into lectures about many related things:
\begin{enumerate}
\item An alternative approach to Eilenberg-Zilber theorems using the notion of symmetry (Section~\ref{SECTSYMM}).  In particular, we give definitions that do not need any combinatorics, of the Alexander-Whitney and Eilenberg-Zilber morphisms,
which work in the non-Abelian and Abelian setting alike, yet give back the classical morphisms in the Abelian case. Similar ideas are already in \cite[\S 13]{May75}.
\item An alternative description of
the canonical simplicial enrichment on simplicial objects (Section~\ref{SECTALT}, cf.\@ also \cite{OR20}):
 \[ \boxed{ \uHom(X, Y)_{[n]} \cong \Hom_{\mathcal{C}^{(\Delta^{\op})^{n+1}}}(\dec_{n+1}^* X, \dec_{n+1}^*Y). } \]
 This is used to construct, for instance, the homotopy (Shih operator) $\Ez \Aw \Rightarrow \id$ using abstract principles (as opposed to specifying a formula, or referring to the method of ``acyclic models'').
\item A self-contained discussion of the Dold-Kan theorem with the least possible amount of calculation (Section~\ref{SECTDOLDKAN}). 
\item A comparison between different classical cobar constructions (Section~\ref{SECTFUNCT}). In particular, we recover from abstract principles the Szczarba map \cite{Szc61} from the Adams cobar construction to the ``singular chains'' on the geometric cobar constructions or Kan's loop group, and its (homotopy) inverse, the Hess-Tonks map \cite{HT10}.
This uses the functoriality of (a completion) of cobar in the Abelian case in morphisms of $A_{\infty}$-algebras. To this end, we discuss a connection between coherent transformations in the sense
of Cordier-Porter \cite{CP97} and $A_{\infty}$-morphisms (Section~\ref{SECTCOHERENTAINFTY}). 
\item Using (dual) classical cobar to construct cofree coalgebras (Section~\ref{SECTCOFREECOALG}).
\end{enumerate}

I would have liked to base all {\em constructions} in these lectures on abstract principles (as opposed to writing down a formula). This could be largely achieved, but not in every case. 
The most notable exception is the definition of the map from coherent transformations to $A_{\infty}$-morphisms. Although the definition is {\em very} simple it is by means of specifying a formula with signs...  
Another excepetion is the  functoriality of the Abelian cobar in $A_{\infty}$-morphisms. This works --- in general --- only after completion of the latter, thus one cannot expect a too simple ``abstract principle'' behind. 
I hope to improve on these points in the future.

Chapters 4--6 and the appendices deal almost entirely with 1-categories and very classical constructions concerning them. 
The reader not familiar with $\infty$-categories can thus concentrate on this part. Some facts and definitions from  Chapters 1--3 are used, which are
stated there for $\infty$-categories. However, the relevant definitions and statements are not very different from the ones which one would make restricting to 1-categories. 
Let me emphasize, however: Although the definition of the {\em derived } (co)bar (as opposed to classical (co)bar) could also be literally stated for 1-categories, it would be almost useless. Its non-triviality is an entirely higher categorical phenomenon. 
Concerning pro-functors, also a little care is needed\footnote{The ``embedding'' $\Cat^{\PF} \to \Cat_{\infty}^{\PF}$ is not compatible with composition.}, therefore in the corresponding sections, an explicit distinction of 1-categories and $\infty$-categories is carried along anyway.

\subsection{(Co)bar}

Bar and cobar constructions are ubiquitous in mathematics, especially in algebraic topology and homological algebra, and there seems to be no general definition encompassing 
all their appearances. Several attempts have been made to generalize and unify them in certain contexts. The first most general definition
of a (one-sided, two-sided) bar construction in the context of monads has been given by Godement \cite{God58}. 
In the homological context the first bar construction is due to Eilenberg-MacLane \cite{EM53} and the first cobar construction due to Adams \cite{Ada56}, cf.\@ also \cite{Bau80}. 
Boardman and Vogt defined a very general bar construction for ``theories'' in \cite{BV73}.
Meyer has suggested a unification of different kinds of (co)bar constructions in \cite{Mey84, Mey86}. 
There are very general definitions of (co)bar constructions in 
the context of differential graded (co)algebras and (co)operads, cf.\@ \cite{MSS02, LV12}. 
In the $\infty$-categorical context, Lurie has given a quite general definition
of a (co)bar duality  in \cite{Lur11}.

\subsection{General classical and derived (co)bar duality}

Lurie calls a diagram such as (\ref{eqintro1}) a left and right representable pairing of categories (cf.\@ Definition~\ref{DEFPAIRING}), if  $\rho_1 \times \rho_2$ is a fibration with groupoid fibers  in such a way that
allows to extract the functors $\barlurie$ and $\cobarlurie$ which are then automatically adjoint.

Diagram (\ref{eqintro1}) has a straightforward generalization to $\infty$-(co)operads and more generally to cofibrations of $\infty$-operads $\mathcal{C} \to \mathcal{S}$, and a small $\infty$-operad $I$:
Consider the  diagram\footnote{In which $\mathcal{C}^{\vee} \to \mathcal{S}^{\op}$ is the fibration of cooperads with the same fibers obtained from $\mathcal{C} \to \mathcal{S}$ and $\mathcal{C} \twop \to \mathcal{S}$ is the cofibration of operads whose fibers are the twisted arrow categories of the fibers of $\mathcal{C} \to \mathcal{S}$.} 
\[ \xymatrix{ & ( \mathcal{C} \twop )^I \ar[ld]_{\rho_1} \ar[rd]^{\rho_2} \\
\mathcal{C}^I  \ar[rd] & & ((\mathcal{C}^{\vee})^{I^{\op}})^{\op} \ar[ld] \\
& \mathcal{S}^I  
}\]
which in the simplest case $\mathcal{S} = I = \OOO$, when $\mathcal{C}$ is just a  monoidal $\infty$-category, gives
\[ \xymatrix{ & \Alg( \twop \mathcal{C} ) \ar[ld]_{\rho_1} \ar[rd]^{\rho_2} \\
\Alg(\mathcal{C}) & & \Coalg(\mathcal{C})^{\op}
}\]
the diagram Lurie considers to define his (co)bar functors. 

We show that {\em also the diagram (\ref{eqintro0})} has a straight-forward generalization to the relative setting of cofibrations of $\infty$-operads 
permitting thus to reduce the existence of Lurie's (co)bar functors to the existence of a certain relative (a.k.a.\@ operadic) Kan extensions. 
In fact (\ref{eqintro0}) has also an switched form involving $\tw I$ and thus {\em two} generalizations: 

\begin{equation}\label{eqbarintro} 
\vcenter{ \xymatrix@C=1pc{
& (\mathcal{C}^{\vee})^{\twop I}_{\pi_2^* S^{\op}} \ar@{<-}[ld]_{\widetilde{\pi_1^*}} \ar@{<-}[rd]^{\pi_2^*}\\
\mathcal{C}^I_{S}  & & (\mathcal{C}^\vee)^{I^{\op}}_{S^{\op}}  \\
}} \quad \vcenter{\xymatrix@C=1pc{
& \mathcal{C}^{\tw I}_{\Pi_2^* S} \ar@{<-}[ld]_{\Pi_2^*} \ar@{<-}[rd]^{\widetilde{\Pi_1^*}} \\
\mathcal{C}^I_{S}  & & (\mathcal{C}^\vee)^{I^{\op}}_{S^{\op}}   \\
}}
\end{equation}
and we have (cf.\@ Proposition~\ref{PROPCOBAR}) similarly that $\rho_1 \times \rho_2$ is the fibration represented fiber-wise over $S \in \mathcal{S}^{I}$ by equivalently
\[ X, Y \mapsto \Hom(\widetilde{\pi_1^*} X,  \pi_2^*Y) \cong \Hom(\Pi_2^* X, \widetilde{\Pi_1^*} Y).  \]
Thus (the generalization of) Lurie's (co)bar functors --- which are, by definition, functors such that  $\rho_1 \times \rho_2$ is the fibration represented fiber-wise by
\[ X, Y \mapsto \Hom(\barlurie X,  Y) \cong \Hom( X, \cobarlurie Y)  \]
--- exist, if suitable relative Kan extensions along $\pi_2$ and $\Pi_2$ exist. 
The appearing {\em operad} $\tw I$ (resp.\@ cooperad $\twop I$)  is a generalization of the notion of ``twisted arrow category'' to operads. It is an operad whose ``category of operators''  is not quite the twisted arrow category
of the category of operators of $I$ but merely a localization of it. Its objects can be identified with {\em active} morphisms in $I$.

The classical (co)bar adjunction generalizes to
\[ \xymatrix{  (\mathcal{C}^{\vee})^{\twop I}_{\pi_2^* S^{\op}}
 \ar@<3pt>[rr]^-{\cobarconst_{\mathcal{C} }}  & &  \ar@<3pt>[ll]^-{\barconst_{\mathcal{C}}\, :=\, \widetilde{\pi_1^*}} 
\mathcal{C}^I_{S}  
}. \]
Here $\barconst_{\mathcal{C}}$ always exists and is fully-faithful and $\cobarconst_{\mathcal{C}}$ is its {\em left} adjoint.

In the simplest case $\mathcal{S} = I = \OOO$ this gives
\begin{gather}\label{eqbarintro2} 
\vcenter{ \xymatrix@C=1pc{
& (\mathcal{C}^{\vee})^{(\Delta_{\act}, \ast')^{\op}} \ar@{<-}[ld]_{\widetilde{\pi_1^*}} \ar@{<-}[rd]^{\pi_2^*}\\
\Alg(\mathcal{C})  & & \Coalg(\mathcal{C}) \\
} }
\end{gather}
where $(\Delta_{\act}, \ast')$ is the monoidal category $\Delta_{\act}$, the simplex category with endpoint-preserving morphisms, equipped with the monoidal product $\ast'$
\footnote{In fact, we have an isomorphism $(\Delta_{\act}, \ast') \cong (\Delta^{\op}_{\emptyset}, \ast)$ by the ``duality of ordered sets and intervals'' (Lemma~\ref{LEMMADUAL}) and you may choose your favorite model among both.}.
It is quite unlikely that $\pi_2^*$ will have a left adjoint that commutes with the forgetful functor (for  $\mathcal{S} = I = \OOO$ it certainly does not). 
However, if we replace $\mathcal{C}$ by the category of augmented objects in $\mathcal{C}$, which has the effect that the unit
will become final
\begin{equation} \label{eqintrorho} \rho^*:  (\mathcal{C}^{\vee})^{(\Delta, \ast)^{\op}} \to (\mathcal{C}^{\vee})^{(\Delta_{\act}, \ast')^{\op}}   \end{equation}
is an equivalence. 
Here $(\Delta, \ast)$ is not monoidal anymore, because it lacks a unit, but still pro-monoidal or, as we say, an exponential (even $\infty$-exponential) fibration of operads over $\OOO$. (This suffices also to construct a Day convolution.)
Now $\pi_2^*$ has a left adjoint that commutes with the forgetful functor under no assumptions at all\footnote{The technical reason is that $\ast'$ is not cofinal, whereas $\dec=\ast$ is cofinal.} (apart from existence of the colimit). In particular, on underlying objects is just given by $\colim_{\Delta^{\op}}$ (geometric realization). 
What happens here is very simple. $\widetilde{\pi_1^*}$ maps an algebra $A$ to the diagram of shape $\Delta^{\op}_{\act}$ (all but one degeneracy not depicted):
\begin{equation}\label{eqintro2} \xymatrix{ \cdots  \ar@<4pt>[r]  \ar[r] \ar@<-4pt>[r] &A \otimes A\otimes A \ar@<2pt>[r]  \ar@<-2pt>[r] & A \otimes A \ar[r] & A  & \ar@{-->}[l] 1 } \end{equation}
whose morphisms are the structure morphisms of $A$ with a canonical coalgebra structure 
which extends thus, if the unit is final, canonically to a diagram of shape $\Delta^{\op}$  (all but one degeneracy not depicted):
\[ \xymatrix{ \cdots  \ar@<8pt>[r]  \ar[r] \ar@<-8pt>[r] \ar@<4pt>[r]  \ar[r] \ar@<-4pt>[r] &A \otimes A\otimes A \ar@<2pt>[r]  \ar@<-2pt>[r]\ar@<6pt>[r]  \ar@<-6pt>[r] & A \otimes A \ar[r]  \ar@<4pt>[r] \ar@<-4pt>[r]  & A \ar@<4pt>[r] \ar@<-4pt>[r] & \ar@{-->}[l] 1 } \]
with coalgebra structure now w.r.t.\@  the monoidal product $\dec_* - \boxtimes -$.
The bar construction $\barlurie(A)$ is then just the colimit of this diagram. The cobar construction $\cobarlurie$ is precisely the dual construction. 

The {\em classical} cobar construction $\cobarconst$ is a left adjoint to the (fully-faithful) association  which maps $A$ to the coalgebra (\ref{eqintro2}).
We show (Theorem~\ref{THEOREMEXISTENCECOBAR}) that it can --- under very general assumptions --- be given as a certain colimit over $\tw \Delta^{\op}_{\act}$ (which is essentially the category of {\em necklaces} \cite{DS11, Riv22, BS23}).
{\em In the 1-categorical context} it is thus just given by mapping
\[ \xymatrix{ \cdots  \ar@<4pt>[r]  \ar[r] \ar@<-4pt>[r] & A_{[3]} \ar@<2pt>[r]  \ar@<-2pt>[r] & A_{[2]} \ar[r] & A_{[1]}  & \ar@{-->}[l] A_{[0]} } \]
(with its coalgebra structure) to the coequalizer of
\[ \coprod A_{[1]} \otimes \cdots \times A_{[0]} \otimes \cdots \otimes A_{[1]} \amalg  \coprod A_{[1]} \otimes \cdots \times A_{[2]} \otimes \cdots \otimes A_{[1]} \rightrightarrows \coprod_{n=0}^{\infty} A_{[1]}^{\otimes n}  \]
with maps, for instance, given by $A_{[2]} \to A_{[1]}$ (diagram map) and $A_{[2]} \to A_{[1]} \otimes A_{[1]}$ (part of the coalgebra structure),
i.e.\@ a very simple quotient of the free algebra in $A_{[1]}$. You will immediately recognize the construction of the fundamental group (or better: monoid) of a connected simplicial set as a special case.
In the $\infty$-categorical context the same construction works, only that we have to take the full diagram of shape $\tw \Delta^{\op}_{\act}$ into account here, including the other $A_{[n]}$ as well.

\subsection{Concrete instances of classical bar and cobar}

In Chapter~\ref{CHAPTERBASICEXAMPLES} we show, as mentioned above, that the quite simple classical bar and cobar constructions (in case $I = \mathcal{S} = \OOO$) agree with many constructions in the literature, as for example the
construction of Kan's loop group and the Adams cobar and Eilenberg-Maclane bar construction. However, these constructions are not quite {\em the same} as the ``classical bar''  and ``classical cobar'' but have the form of the (still adjoint) functors
\[ \dec_* \circ (\rho^*)^{-1} \circ \barconst \]
and
\[  \cobarconst \circ \rho^* \circ \dec^*. \]
(a way of thinking certainly inspired by Stevenson's article~\cite{Ste12}). Here $\rho^*$ is the functor (\ref{eqintrorho}) above.
To make this precise, we must turn $\dec_*$ and $\dec^*$ into functors of (co)operads and then explicitly calculate the composition. This will be done for $\mathcal{C} = (\Set^{\Delta^{\op}}, \times)$ in \ref{SECTCOBARSSET}, 
and for $\mathcal{C} = (\Ch_{\ge 0}(\mathcal{D}), \tildeotimes)$ with $(\mathcal{D}, \otimes)$ Abelian tensor category in \ref{SECTCOBARDG}. 

The derived (Lurie's) bar construction $\barlurie$ can in these cases (under mild assumptions on $\mathcal{D}$ in the Abelian case) also be computed as
\[ \dec_*\circ (\rho^*)^{-1}\circ \barconst \]
because $\dec_*$ represents the (homotopy) colimit over $\Delta^{\op}$. This will be explained in detail in Section~\ref{SECTCOMPLURIE}.

However, the left adjoint $\cobarconst \circ \rho^* \circ \dec^*$ is, a priori, not directly related to the derived (Lurie's) cobar construction $\cobarlurie$. Note that the first is a left adjoint before localization and the latter a right adjoint after localization.
Often times, however, $\cobarconst \circ \rho^* \circ \dec^*$ also preserves weak equivalences {\em at least when restricted to a large subcategory}, and the classical adjunction gives a derived {\em equivalence}. Then, accordingly, also the (restrictions of the) two cobar constructions agree. 

The chosen notation ``classical (co)bar'' and ``derived (co)bar'' is thus probably a bit unfortunate. My motivation has been to stay as close to the 
existing nomenclature as possible. Keep in mind, however, that while the classical bar construction is really a component of the derived bar construction, the constructions denoted ``cobar'' 
denote something a bit different, yet related, in the classical and derived case. Rather the following is true: The derived cobar is the dual of the derived bar construction and thus the dual classical bar construction $\barconst^{\vee}$ is a component of it.

\subsection{Comparisons of Abelian and non-Abelian classical cobar constructions}

One motivation of this work has been to understand the relation between the geometric (Kan) cobar for simplicial sets and Adams cobar for complexes of Abelian groups.
Since they are formally given by exactly the same construction ``$\cobarconst \circ \rho^* \circ \dec^*$'' one would expect a simple formal comparison.
This is indeed true, if one chooses on (non-negatively graded) complexes of Abelian groups the monoidal product $\otimes$ (i.e.\@ under Dold-Kan the point-wise tensor product on simplicial Abelian groups). In this case, the construction (quite obviously) commutes with the free Abelian group functor $\Z[-]$, i.e.\@ ``singular chains''.
It remains thus to compare the two purely Abelian constructions in complexes of Abelian groups w.r.t.\@ $\otimes$ and $\tildeotimes$. This turns out to be fairly intricate:

\begin{itemize}
\item The difficulty in the construction of a map  ``$\cobarconst_{\tildeotimes} \to \cobarconst_{\otimes}$'' is the following: 
 The morphism given by abstract functoriality w.r.t.\@ the Eilenberg-Zilber  map does not land in the cobar construction of the dg-algebra $\Z[X]$ with its usual (diagonal) coalgebra structure w.r.t.\@ $\otimes$ but with its composition $\Z[X] \to \Z[X] \otimes \Z[X] \to \Z[X] \otimes \Z[X] $ with $\EZ \circ \AW$!
 This problem has been dealt with by either giving explicit constructions (Adams \cite{Ada56}, Szczarba \cite{Szc61}, etc.) or using homotopy deformation theory (Shih \cite{Shi62}, etc.) We propose
 a more conceptual approach which is completely explicit. Roughly it is as follows (hiding here some details about the transport along $\dec^*$): The Shih operator $\Xi: \id \Rightarrow \EZ \circ \AW$ will be constructed from abstract principles, and not by specifying a formula. For a coalgebra $C$ it gives rise first to a coherent transformation $\exp(\Xi)(C)$ in the sense of Cordier-Porter \cite{CP97}, and then --- via a very general comparison map from coherent transformations to morphisms of $A_{\infty}$-coalgebras --- rise to a morphism of $A_{\infty}$-coalgebras $\exp(\Xi)(C)_{\infty}$. It turns out that the components of this $A_{\infty}$-morphism are essentially\footnote{Up to different indexing conventions} the {\bf Szczarba-maps} \cite{Szc61}. It gives thus a map between the cobar constructions (a priori, only a completion of cobar is functorial in $A_{\infty}$-morphisms).
 
There are many other attempts to understand the Szczarba maps in the literature, see for example \cite{MRZ23, Fra21, Fra24}.

\item The difficulty in the construction of a map  ``$\cobarconst_{\otimes} \to \cobarconst_{\tildeotimes}$'' is the following: We can make $\dec^*$ naturally (lax) monoidal w.r.t.\@ the respective tensor products, but the resulting extensions $\dec^*_{\otimes}$ and $\dec^*_{\tildeotimes}$ 
are not compatible with the $\AW$-morphism! It turns out, however, that one can plug in the inverse of the $A_{\infty}$-morphism $\exp(\Xi)_{\infty}$ constructed before to get a morphism  in the other direction. 
It is however, in general, only defined after a completion (or, a posteriori, localization) of cobar. It is very likely that this is (up to the different indexing issue) the {\bf Hess-Tonks map} \cite{HT10}, although this remains to be checked in detail. 
\end{itemize}

\subsection{Plans}

These lecture notes are far from complete and also far from how I would have imagined them. The reason is mainly that they got quite long already and I wanted to 
make them available before (hopefully) being able to elaborate on the following points: 

\begin{enumerate}
\item A  {\em motivation} for the bar and cobar constructions, in particular, for the classical cases discussed in detail. 
This omission can hopefully be excused for the moment because many of the sources discuss this thoroughly. 
\item A discussion of the {\em concrete properties} of the bar and cobar adjunctions (classical and derived) in the main cases simplicial sets and complexes and $\Gpd_{\infty}$ (spaces). 
\item A discussion of the vast generalizations in the dg-setting: (co)bar for dg-{\em operads}, Sweedler theory \cite{AJ13}, and so on. Many of these
features should generalize (cf.\@ also \cite{Chi12, Per22, HM25}).
\item A discussion of the classical and derived (co)bar for $\mathrm{LMod}$, the planar operad encoding (left) modules over algebras over $\OOO$. 
For example, the classical (co)bar (for simplicial sets) should recover the theory of prinicipal twisted Cartesian products. 
The derived (co)bar should give a generalization of Lurie's duality to modules and comodules (cf.\@ also \cite{BP23}). 
\item A concrete  discussion of the derived (Lurie) (co)bar for other operads than $\OOO$ and for non-planar (i.e.\@ symmetric) operads.
Whereas the abstract part carries over to symmetric operads without modification, I would have liked to include an extistence proof in the same spirit for $\mathbb{E}_k$-(co)algebras at least.

\end{enumerate}

\section{Categorical prerequisites}

This chapter discusses several categorical concepts that will be used in the sequel of these lectures. It is intended for referential purpose and to fix notation.
Proofs are only occasionally sketched. The reader is advised to skim over it on a first reading. 

\subsection{Pro-functors}

\begin{PAR}
In this lecture Kan extensions which are the left adjoints $\alpha_!$ or right adjoints $\alpha_*$ of pre-composition $\alpha^*: \mathcal{C}^{J} \to \mathcal{C}^{I}$ with a functor $\alpha: I \to J$, are ubiquitous. 
They comprise in particular all limits and colimits. 
The collection of all $\alpha^*$ and $\alpha_!$ (say), where $\alpha$ runs through all functors $\alpha: I \to J$ between small categories (or small $\infty$-categories) fulfill a rich algebra. This algebra is encoded
in the 2-category $\Cat^{\PF}$ (resp.\@ $(\infty,2)$-category $\Cat_{\infty}^{\PF}$) of small categories and pro-functors. The reader who is not completely at ease with pro-functors should keep in mind that they are all of the form $\beta_! \alpha^*$ for suitable $\alpha$ and $\beta$
and the 2-morphisms between compositions are precisely those, that these functors acquire universally for all categories $\mathcal{C}$, or which amounts to the same, for $\mathcal{C} = \Set$ (resp.\@ for $\mathcal{C} =  \Gpd_{\infty}$ when working with $\infty$-categories). However, as definition, a more concrete approach is convenient:  
\end{PAR}

\begin{PAR} 
{\bf Pro-functors} $\gamma: I \to J$ are functors
\[ J^{\op} \times I   \to \Set \]
(resp.\@ {\bf $\infty$-pro-functors} are functors $J^{\op} \times I \to \Gpd_{\infty}$, where $I$ and $J$ can be $\infty$-categories themselves) 
with composition given by
\[ \beta \circ \alpha := \int^{j}  \beta(-, j) \times \alpha(j, -).  \]
While this is a proper definition in the 1-categorical context, it is of course more involved in the $\infty$-categorical setting (cf.\@ \cite{AF20} for a precise construction).
There are canonical functors 
\[ \iota: \Cat^{} \to \Cat^{\PF} \qquad (\text{resp.} \Cat_{\infty}^{} \to \Cat_{\infty}^{\PF})   \]
mapping $\alpha: I \to J$ to $j, i \mapsto \Hom(j, \alpha(i))$, a pro-functor $I \to J$, 
as well as
\[  {}^t \iota:  \Cat^{1-\op, 2-\op} \to \Cat^{\PF} \qquad (\text{resp.}  \Cat_{\infty}^{1-\op, 2-\op} \to \Cat_{\infty}^{\PF}) \]
mapping $\alpha: I \to J$ to ${}^t \alpha: i, j \mapsto \Hom(\alpha(i), j)$, a pro-functor $J \to I$.
The pro-functor ${}^t \alpha$ is, in fact, right adjoint to $\alpha$.
 \end{PAR}

\begin{PAR}\label{PARPFLR}
A {\em cocomplete} category $\mathcal{C}$ gives rise to a functor
\begin{eqnarray*} L_{\mathcal{C}}: \Cat^{\PF, 1-\op} &\to& \Cat \qquad (\text{resp.}  \Cat_{\infty}^{\PF, 1-\op} \to \Cat_{\infty})  \\
 I &\mapsto& \mathcal{C}^{I} 
\end{eqnarray*}
 mapping $\alpha$ to $\alpha^*$ and its adjoint ${}^t \alpha$ to $\alpha_!$ (left Kan extension), and more generally
\begin{eqnarray*} L(\gamma): \mathcal{C}^{J} &\to& \mathcal{C}^{I} \\
  X  &\mapsto& \int^j \gamma(j,-) \times X(j). 
  \end{eqnarray*} 

A {\em complete} category $\mathcal{C}$ gives rise to  a functor
\begin{eqnarray*} 
R_{\mathcal{C}}: \Cat^{\PF, 2-\op} &\to& \Cat \qquad (\text{resp.}  \Cat_{\infty}^{\PF, 2-\op} \to \Cat_{\infty})  \\
 I &\mapsto& \mathcal{C}^{I} 
\end{eqnarray*}
mapping ${}^t \alpha$ to $\alpha^*$ and $\alpha$ to $\alpha_*$ (right Kan extension), and more generally
\begin{eqnarray*} R(\gamma): \mathcal{C}^{I} &\to& \mathcal{C}^{J} \\
  X  &\mapsto& \int_i \Hom(\gamma(-,i), X(i)).
  \end{eqnarray*} 
 \end{PAR}

\begin{PAR}\label{PAROPPF}
There is an operation 
\[ \op: \Cat^{\PF, 1-\op} \to \Cat^{\PF}  \qquad (\text{resp.}  \Cat_{\infty}^{\PF, 1-\op} \to \Cat_{\infty}^{\PF}) \]
\begin{eqnarray*} I &\mapsto& I^{\op} \\
 \gamma: I \to J &\mapsto&  \gamma: J^{\op} \to I^{\op} 
 \end{eqnarray*}
making the following commutative (and similarly in the $\infty$-categorical context)
\begin{equation}\label{eqdiapf}
 \vcenter{ \xymatrix{ \Cat^{1-\op} \ar[r]^-{\iota} \ar[d]^{\op} & \Cat^{\PF, 1-\op} \ar[r]^-{L_{\mathcal{C}}} \ar[d]^{\op} & \Cat \ar[d]^{\op} \ar[d] \\
 \Cat^{1-\op, 2-\op} \ar[r]^-{{}^t\! \iota} & \Cat^{\PF} \ar[r]^-{R_{\mathcal{C}^{\op}}} &   \Cat^{2-\op} } } 
 \end{equation}
 Furthermore, we have $L_{\mathcal{C}}\,\iota = R_{\mathcal{C}}\, {}^{t}\! \iota$.
 
 We discussed pro-functors separately for usual categories and $\infty$-categories for a good reason: 
 Although they behave completely analogously, there does not exist an embedding $\Cat^{\PF} \hookrightarrow \Cat_{\infty}^{\PF}$ because the inclusion
 $\Set \hookrightarrow \Gpd_{\infty}$ does not commute with colimits and thus the compositions of pro-functors are different.  
 This reflects the fact that there much less relations between (co)limits and Kan extensions in the $\infty$-categorical context than in the
 classical context.
 \end{PAR}
 
 \begin{PAR}
A 2-morphism
\[ \gamma \Rightarrow \gamma' \]
is an isomorphism precisely, if $L_{\mathcal{C}}(\gamma) \Rightarrow L_{\mathcal{C}}(\gamma')$ is an isomorphism for all cocomplete categories (resp.\@ $\infty$-categories) $\mathcal{C}$ or precisely, if 
$L_{\Set}(\gamma) \Rightarrow L_{\Set}(\gamma')$ (resp.\@ $L_{\Gpd_{\infty}}(\gamma) \Rightarrow L_{\Gpd_{\infty}}(\gamma')$) is an isomorphism. 
In fact, even more is true: The functor on morphism categories: 
\begin{gather*} 
L_{\Set}: \Hom_{\Cat^{\PF}}(I, J) \rightarrow \Hom(\Set^J, \Set^I)  \\ 
(\text{resp.\@} L_{\Gpd_{\infty}}: \Hom_{\Cat^{\PF}_{\infty}}(I, J) \rightarrow \Hom(\Gpd_{\infty}^J, \Gpd_{\infty}^I))   
\end{gather*}
is full. The essential image is {\em precisely the category of colimit preserving functors}. 
 \end{PAR}
 
  \begin{PAR}
 There is a similar picture for additive categories and additive pro-functors (only in the 1-categorical setting discussed and needed)
 \[ J^{\op} \times I \to \mathcal{AB} \]
 forming a 2-category $\AbCat^{\PF}$. 
 There is an embedding
 \[ \Cat^{\PF} \hookrightarrow \AbCat^{\PF} \]
 applying $\Z[-]$ (free Abelian group) to the pro-functors. This is compatible with composition because $\Z[-]$ commutes with coproducts and maps $\times$ into $\otimes$. 
 We have then 
 \[ L_{\mathcal{C}}: \AbCat^{\PF, 1-\op} \to \AbCat 
\quad R_{\mathcal{C}}: \AbCat^{\PF, 2-\op} \to \AbCat \]
into additive categories when $\mathcal{C}$ is additive and (co)complete and (co)tensored over Abelian groups. Actually this will be applied to Abelian categories only. These are automatically (co)tensored over $\Ab$ when they admit infinite (co)products. 
\end{PAR}

 \begin{PAR}
Consider an arbitrary pro-functor $\gamma: J^{\op} \times I \to \Set$ (resp.\@ $\gamma: J^{\op} \times I \to \Gpd_{\infty}$). It gives rise to a category
\[ \xymatrix{& \ar[ld]_{l_{\gamma}} \nabla \int \gamma \ar[rd]^{r_{\gamma}} \\
I & &  J } \]
(applying the Grothendieck construction, and its dual, respectively) equipped with a cofibration $l_{\gamma}$, and a fibration $r_{\gamma}$, respectively. 
We have then an isomorphism
\[ \gamma \cong r_\gamma \, {}^t\! l_\gamma \]
i.e.\@ all pro-functors can be expressed in this form. 
For $\alpha: I \to K$ and $\beta: J \to K$, letting $\gamma = {}^t \alpha\, \beta$, i.e.\@ $\gamma = \Hom(\alpha(-), \beta(-))$, we get 
the diagram 
\[ \xymatrix{& \ar[ld]_{l} I \times_{/K} J \ar[rd]^{r} \\
I & &  J } \]
(comma category) and hence an isomorphism ${}^t\! \alpha\, \beta \cong r\, {}^t\! l$. This is commonly known as {\bf Kan's formula} (For $\beta: \cdot \to K$ and applying $L_{\mathcal{C}}$ it gives a point-wise formula for the Kan extension because, being  a cofibration, $L(\,{}^t\!l_\gamma) = l_{\gamma, !}$ is computed fiber-wise).
 Here, when $I, J$ and $K$ are 1-categories, for {\em both} compositions, it accidentally does not matter whether they are considered in $\Set$ or $\Gpd_{\infty}$. 
More generally, we have
 \end{PAR}

\begin{DEFLEMMA}\label{PROPEXACT}
A diagram of small categories (resp.\@ small $\infty$-categories) of the form
\begin{equation}\label{EQEXACT} \vcenter{ \xymatrix{ I \ar[r]^{\alpha} \ar[d]_{\beta} \ar@{}[rd]|\Downarrow & J \ar[d]^{\gamma} \\
K \ar[r]_{\delta} & L  } }
\end{equation}
is called  {\bf exact} (resp.\@ {\bf $\infty$-exact}) if the following, equivalent statements hold true: 
\begin{enumerate}
\item The mate 
\[  \alpha\, {}^t\! \beta  \to  {}^t\! \gamma\, \delta  \]
is an isomorphism in $\Cat^{\PF}$ (resp.\@ $\Cat^{\PF}_{\infty}$), i.e.\@ for all objects $(j,k) \in J \times K$
the canonical morphism
\[ \int^{i} \Hom_J(j, \alpha(i)) \times \Hom_K(\beta(i), k)  \to \Hom_{L}(\gamma(j), \delta(k)) \]
is an isomorphism.
\item 
$ \beta_!\, \alpha^* \to \delta^*\,\gamma_!$ is an isomorphism for diagrams in any cocomplete 1-category (resp.\@ $\infty$-category). 
\item
 $\gamma^*\, \delta_* \to \alpha_*\,\beta^* $ is an isomorphism for diagrams in any complete 1-category (resp.\@ $\infty$-category). 
\item For all objects $(j,k) \in J \times K$, the canonical functor
\[ j \times_{/ J } I \times_{/K} k \to j \times_{/L} k \quad  ( = \Hom_{L}(\gamma(j), \delta(k)) ) \]
has connected fibers (resp.\@ becomes an isomorphism\footnote{or if $I, J, K, L$ are 1-categories, expressed more classically by saying that the nerve applied to this functor is a weak equivalence of simplicial sets} in $\Gpd_{\infty}$).
\end{enumerate}
In particular, those conditions are satisfied (for 1-categories and $\infty$-categories alike) if the 2-morphism is an isomorphism, the diagram is Cartesian, and one of the following holds true:
\begin{enumerate}
\item $\gamma$ (hence $\beta$) is a cofibration,
\item $\delta$ (hence $\alpha$) is a fibration.
\end{enumerate}
\end{DEFLEMMA}

Observe that for squares of 1-categories, by criterion 4.\@ for instance, $\infty$-exact implies exact but not vice versa.

\begin{PAR}
A special case are diagrams of the form
\[ \xymatrix{ I \ar[d] \ar[r]^{\alpha} &  J \ar[d] \\
\cdot \ar@{=}[r] & \cdot
}  \]
In this case, we say that $\alpha$ is {\bf cofinal} (resp.\@ {\bf $\infty$-cofinal}), if this diagram is exact (resp.\@ $\infty$-exact) or --- in other words --- if the pull-back $\alpha^*$ does not change colimits (resp.\@ $\infty$-colimits).  

Similarly, if 
\[ \xymatrix{ I \ar[d]_{\alpha} \ar[r] &  \cdot \ar@{=}[d] \\
J \ar[r] & \cdot
}  \]
is exact (resp.\@ $\infty$-exact), we say that $\alpha$ is {\bf final} (resp.\@ {\bf $\infty$-final}) or --- in other words --- if the pull-back $\alpha^*$ does not change limits (resp.\@ $\infty$-limits).  

{\em Warning: } This notation varies according to source. 
\end{PAR}

\begin{BEISPIEL}\label{PARFINAL}
Let $I$ be a small 1-category. 
Consider the nerve of $I$, i.e.\@ the simplicial set $N(I)$ with $n$-simplices being sequences of morphisms $i_0 \to \cdots \to i_n$ (i.e.\@ functors $[n] \to I$). This is, in fact, the usual nerve construction associated 
with the cosimplicial object in 1-categories $[n] \mapsto [n]$. As functor $N(I): \Delta^{\op} \to \Set$, it has an associated cofibration (unstraightening) which is equipped with a functor
\[ \alpha: \int N(I) \to I \]
mapping a pair $[n], \{i_0 \to \cdots \to i_n\}$ to $i_n$. One can show that it is $\infty$-cofinal\footnote{Criterion 4.\@ of Proposition~\ref{PROPEXACT} boils down to ``$N \int N(I \times_{/I} i)$ contractible''. However, for any 1-category  $J$ with final object $N(J)$ is contractible, and there is a weak equivalence $X \to N \int X$ for any simplicial set $X$ \cite{Cis19}.}. Since $p: \int N(I) \to \Delta^{\op}$ is a cofibration, $p_!$ is computed fiber-wise by a coproduct (because $p$ is discrete) and the equation
\[ \colim_{\Delta^{\op}} \, p_!\, \alpha^* \cong \pi_{I,!} \]
shows that any $\infty$-colimit over $I$ can be computed by a fiber-wise coproduct ($p_!$) followed by a $\infty$-colimit over $\Delta^{\op}$ (called ``geometric realization'' by Lurie \cite{Lur09}).
Also, it is well-known that the inclusion
\[ \Delta^{\op}_{\ge 1} \hookrightarrow \Delta^{\op}\]
is cofinal (but, of course, not $\infty$-cofinal) hence every 1-colimit can be computed by a fiber-wise coproduct followed by a colimit over $\Delta^{\op}_{\ge 1}$ (i.e.\@ a split coequalizer). 
\end{BEISPIEL}

\begin{BEISPIEL}[Variant] \label{PARFINALCOEND}
Also the morphism
\[ \alpha': \int N(I) \to \tw I  \]
mapping a pair $[n], i_0 \to \cdots \to i_n$ to $i_0 \to i_n$ is $\infty$-cofinal\footnote{Criterion 4.\@ of Proposition~\ref{PROPEXACT} boils down to ``$N \int N((i \times_{/I} I \times_{/I} j)_{\alpha})$ contractible'' for any $\alpha: i \to j$. However also $(i \times_{/I} I \times_{/I} j)_{\alpha}$ has a final object.}. This shows that a 1-coend (resp.\@ $\infty$-coend) can be computed 
directly by the formula
\[ \colim_{\Delta_{\op}}\, p_!\, {\alpha'}^*\, \iota^* \cong \colim_{\tw I}\, \iota^* = \int^I \]
(where $\iota$ is the fibration $\tw I \to I^{\op} \times I$) as a coequalizer of the last two maps (resp.\@ ``geometric realization'') of the diagram
\[ \xymatrix{ \cdots  \coprod_{i_0 \to i_1 \to i_2} A(i_0,i_2)  \ar@<6pt>[r] \ar@<0pt>[r] \ar@<-6pt>[r] &  \coprod_{i_0 \to i_1} A(i_0,i_1) \ar@<3pt>[r] \ar@<-3pt>[r]  & \coprod_{i} A(i,i). }  \]
In fact, more generally, for a functor $A: \tw I \to \mathcal{C}$:
\[ \colim_{\tw I} A = \colim \left( \xymatrix{ \cdots  \coprod_{i_0 \to i_1 \to i_2} A(i_0 \to i_2)  \ar@<6pt>[r] \ar@<0pt>[r] \ar@<-6pt>[r] &  \coprod_{i_0 \to i_1} A(i_0 \to i_1) \ar@<3pt>[r] \ar@<-3pt>[r]  & \coprod_{i} A(\id_i) }  \right). \]

In the 1-categorical case one can even restrict the $\coprod_{i_0 \to i_1}$ to a generating set of morphisms. 
\end{BEISPIEL}

\subsection{Commutative diagrams and correspondences}

We are often in need to show commutativity of a diagram of 1-morphisms and 2-morphisms in a (2,1)-category\footnote{The section has an analogue for $(\infty, 2)$-categories, that we will not mention because it is not needed in these lectures.} $\mathcal{C}$ (in the examples: $\Cat$ or $\Cat^{\PF}$)
in which the 1-morphisms  are compositions of a 1-morphism and the adjoint of another. 
\begin{DEF}Let $X$ and $Y$ be objects. We define the 2-category
\[ \Cor(X, Y) \]
with objects diagrams
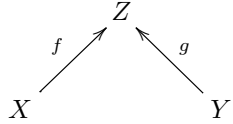
\[ \xymatrix{ & Z \ar@{<-}[ld]_f \ar@{<-}[rd]^g \\
X & & Y }  \]
in which $f$ has a right adjoint ${}^t\!f$. 1-Morphisms $h: (f,g) \to (f',g')$ and 2-Morphisms $h \Rightarrow h'$ are 2-commutative diagrams
\[ \xymatrix{ & Z \ar@{<-}[ld]_f \ar[dd]^{h} \ar@{<-}[rd]^g \\
X \ar@{}[r]|{\Rightarrow} & & Y \ar@{}[l]|{\Rightarrow} \\
& Z' \ar@{<-}[lu]^{f'}  \ar@{<-}[ru]_{g'} 
} \qquad \xymatrix{ & Z \ar@{<-}[ld]_f \ar@/_7pt/[dd]_{h}^{\Rightarrow} \ar@/^7pt/[dd]^{h'} \ar@{<-}[rd]^g \\
X \ar@{}[r]|{\Rightarrow} & & Y \ar@{}[l]|{\Rightarrow} \\
& Z' \ar@{<-}[lu]^{f'}  \ar@{<-}[ru]_{g'} 
} \]
\end{DEF}

\begin{LEMMA}
There is a 2-functor
\begin{align*}
\can: \Cor(X, Y) &\to \Hom_{\mathcal{C}}(X, Y), \\
 (f, g) &\mapsto  {}^t\!f \circ g,  \\
 h & \mapsto  \left( {}^t\!f \circ  g  \to  {}^t\!f \circ {}^t\!h \circ h  \circ   g  \to {}^t\!f' \circ g'\right) \qquad \text{(unit)}.
 \end{align*}
 In particular, since the target category is a 1-category, any two 1-morphisms that are connected by a chain of 2-morphisms are mapped to the same morphism\footnote{In the analogous ($\infty$,2)-categorical construction, 
 2-morphisms become isomorphisms.}. 
\end{LEMMA}

\begin{PAR}\label{PARSTANDARD}
A diagram in the image of ``$\can$'' is called of {\bf standard form} and to 
check its commutativity is thus a matter of checking the commutativity of two diagrams involving only the left adjoints. 

Of course, there is a dual construction, where $f$ is assumed to have a left adjoint, which we leave to the reader to state. If the $f$ and $g$ are of the form $\phi^*, \gamma^*$ and hence ${}^t f = \phi_*$ (or their corresponding 1-morphisms in $\Cat^{\PF}$) 
a composition $\Cor(X, Y) \times \Cor(Y, Z) \to \Cor(X, Z)$ can be defined, such that ``$\can$'' is compatible. And there is an operadic version of this, see e.g.\@~\cite{Hor18}.  
\end{PAR}

\subsection{Preliminaries on pre-sheaves}

This section contains some basic facts about pre-sheaves. It should be consulted only when needed, except for the following definition:

\begin{DEFLEMMA}\label{PARENRICHMENTPRESHEAVES}
Let $\mathcal{C}$ be a 1-category and $I$ a 1-category.
The category $\mathcal{C}^{I^{\op}}$ is enriched in $(\Set^{I^{\op}}, \times)$ by the formula
\[  \uHom(C, D)_k := \int_{i \in I} \Hom(\Hom_I(i,k),\Hom(C(i),D(i)).  \]
If $\mathcal{C}$ has coproducts then it is left-tensored via
\[ (C \times D)(k) = C(k) \times D(k). \]
This is called the {\bf canonical enrichment} in pre-sheaves. 
\end{DEFLEMMA}

\begin{PAR}\label{PAR2VAR}
Consider an adjunction in 2 variables
\[ F: \mathcal{C} \times \mathcal{D} \to \mathcal{E} \qquad   G: \mathcal{C}^{\op} \times \mathcal{E} \to \mathcal{D} \qquad H: \mathcal{E} \times \mathcal{D}^{\op} \to \mathcal{C}.  \]
It yields functors
\[  \mathcal{C}^I \times \mathcal{D}^J \to \mathcal{E}^{I \times J} \qquad  (\mathcal{C}^{I})^{\op} \times \mathcal{E}^J \to \mathcal{D}^{I^{\op} \times J}  \qquad  \mathcal{E}^{I} \times (\mathcal{D}^J)^{\op} \to \mathcal{C}^{I \times J^{\op}} \ \]
by applying $F, G, H$ point-wise. 
Assuming that suitable Kan extensions exist, we also have again an adjunction in 2 variables
\[ \mathcal{C}^I \times \mathcal{D}^J \to \mathcal{E}^{I \times J} \qquad  (\mathcal{C}^{I})^{\op} \times \mathcal{E}^{I\times J} \to \mathcal{D}^{J} \qquad \mathcal{E}^{I \times J} \times (\mathcal{D}^{J})^{\op} \to \mathcal{C}^{I}   \]
where, for instance, the second functor is the composition
\[  (\mathcal{C}^{I})^{\op} \times \mathcal{E}^{I \times J} \to \mathcal{D}^{I^{\op} \times I \times J}  \to   \mathcal{D}^{J}  \]
in which the second functor is $R_{\mathcal{C}}$ applied to the product of $\id_J$ with the  pro-functor $I^{\op} \times I \to \cdot$ given by $\Hom_I(-, -)$ which is equivalently $r\,{}^t\!l$ for 
\[ \xymatrix{& \ar[ld]_{l} \tw I \ar[rd]^{r}   \\
 I^{\op} \times I & & \cdot }\]
 In fact, $R_{\mathcal{C}}(r\,{}^t\!l) = \int_I$ (the end). (Pre-)composing with the diagonal and its adjoint, we get in particular an adjunction in two variables:
\[ F: \mathcal{C}^I \times \mathcal{D}^I \to \mathcal{E}^{I} \quad   G: (\mathcal{C}^{I})^{\op} \times \mathcal{E}^I \to \mathcal{D}^{I} \quad H: \mathcal{E}^{I} \times (\mathcal{D}^{I})^{\op} \to \mathcal{C}^{I}.   \]

All this is nicely explained by the fact that the functor $\op: I \to I^{\op}$ (\ref{PAROPPF}) can be seen as the internal Hom $\Hom_{\Cat^{\PF}}(-, \cdot)$ (resp.\@ $\Hom_{\Cat_{\infty}^{\PF}}(-, \cdot)$) in pro-functors, the 
``evaluation morphism'' being $r\,{}^t\!l$, but we will not discuss this connection here. 
\end{PAR}

\begin{LEMMA}\label{LEMMAHOM}
\begin{enumerate}
\item Let $\alpha: I \to J$ be functor and $F, G, H$ an adjunction as above.  We have
\begin{equation} \label{eqpullback} \alpha^* F(-,-) \cong F(\alpha^* -, \alpha^* -)  \end{equation}
or equivalently 
\[ G(-,\alpha_*) \cong \alpha_* G(\alpha^*-, -) \qquad H(\alpha_* -, -) \cong \alpha_* H(-,\alpha^*-)    \]
\item If $\alpha$ is fully-faithful, we also have: 
\[ \alpha^* G(\alpha_!,\alpha_*) \cong  G(-,-) \qquad \alpha^* H(\alpha_* -, \alpha_! -) \cong H(-,-)    \]
\item For {\em cofibrations} $\alpha: I \to J$ the mates of (\ref{eqpullback})
\[  \alpha_! F(-,\alpha^* -) = F(\alpha_!, -) \quad \alpha_! F(\alpha^* -, -) = F(-, \alpha_!-) \]
are isomorphisms as well, or equivalently:
\[ G(\alpha_!,-) \cong \alpha_* G(-,\alpha^* -) \qquad \alpha^* G(-,-) \cong G(\alpha^* -,\alpha^*)      \]
\[ H(-,\alpha_!) \cong \alpha_* H(\alpha^*-, -) \qquad \alpha^* H(-,-) \cong H(\alpha^* -,\alpha^*)      \]
\end{enumerate}
\end{LEMMA}
\begin{proof}
1.\@ is clear, 2.\@ follows by observing that the assumption implies that $\alpha_*$ and $\alpha_!$ are fully-faithful, and 3.\@ follows because the left Kan extension along cofibrations is computed fiber-wise and $F$, being a left adjoint when regarded as a functor of any of the two variables, commutes with colimits. 
The other morphisms are the formal adjoints of the morphisms involving $F$.
\end{proof}

\begin{PAR}For example:
For a complete and cocomplete 1-category $\mathcal{C}$, we have the (co)tensoring adjunction
\[  \times: \Set \times \mathcal{C} \to \mathcal{C} \qquad  \Hom_l: \Set^{\op} \times \mathcal{C} \to \mathcal{C} \qquad \Hom_{r}= \Hom: \mathcal{C}^{\op} \times \mathcal{C} \to \Set  \]
which induces as in \ref{PAR2VAR} an adjunction on pre-sheaves
\[  \times: \Set^{I^{\op}} \times \mathcal{C}^{I^{\op}} \to \mathcal{C}^{I^{\op}}  \quad \Hom_l: (\Set^{I^{\op}})^{\op} \times \mathcal{C}^{I^{\op}} \to \mathcal{C}^{I^{\op}} \quad \Hom_{r}:  (\mathcal{C}^{I^{\op}})^{\op} \times \mathcal{C}^{I^{\op}} \to  \Set^{I^{\op}}   \]
with formula (which can be extracted from the discussion in \ref{PAR2VAR}):
\begin{eqnarray*}
 (C \times D)(k) &=& C(k) \times D(k) \\
 \Hom_r(C, D)(k) &=& \int_{i \in I} \Hom(\Hom_I(i,k),\Hom(C(i),D(i)) 
\end{eqnarray*}
The formula for $\Hom_r$ is valid even without any assumption on $\mathcal{C}$ and turns  $\mathcal{C}^{I^{\op}}$ into a $(\Set^{I^{\op}}, \times)$-enriched category. This gives \ref{PARENRICHMENTPRESHEAVES}.
\end{PAR}

\subsection{(Co)operads}\label{SECTCOOP}

\begin{PAR}\label{PAROP}
In these notes we adopt a very flexible notion of $\infty$-operad. However, the purpose is clearly to discuss either $\infty$-(co)operads {\em or} planar $\infty$-(co)operads in the sense of Lurie. 
Planar (co)operads (Definition~\ref{DEFPLANAR}) are combinatorially very pleasant to discuss the classical bar and cobar constructions, whereas more general $\infty$-operads will be important, for instance, for generalizations to (co)algebras over the $\mathbb{E}_k$ (little discs) operads. 
\end{PAR}

 {\em For simplicity}, we understand in the sequel for the moment
\begin{center}
\boxed{ \text{ ($\infty$-)(co)operad = planar ($\infty$-)(co)operad. } }
\end{center}
This may change in a future version of these notes.

\begin{PAR}\label{PARSIMPLEX}
Let $\Delta$ be the simplex category. We adopt the convention that its objects are precisely the  ordinals $[n] = \{0, \dots, n\}$. (One can include all finite non-empty ordinals, obtaining an equivalent category, of course). 
Define $\Delta_{\emptyset}$ be the
category of the $[n]$, including $[-1]:=\emptyset$. 
Call a morphism in $\Delta$ of the form $i < i+1 < \dots < i+k \hookrightarrow 0 < \dots < n$ {\bf inert} and a morphism $a: [n] \to [m]$ {\bf active}, if
$a(0) = 0$ and $a(n)=m$. Every morphism factors uniquely as an active morphism followed by an inert one. A morphism also factors  
uniquely into a surjective map (called a {\bf degeneracy}) followed by an injective map (called {\bf face} map). 
Call the subcategory of active morphism $\Delta_{\act}$. This is sometimes called the category of {\bf finite intervals}. Both $\Delta_{\emptyset}$ and $\Delta_{\act}$ are symmetric monoidal
with products $\ast$ and $\ast'$, given by
\[ [n] \ast [m] := [n+m+1] \]
concatenation and
 \[ [n] \ast' [m] := [n+m] \]
concatenation with identification of the extremal points. The units being $\emptyset$, and $[0]$, respectively. 
The restriction of $\ast$ to $\Delta_{\act}$ comes equipped with 
a natural transformation
\begin{equation}\label{eqcan} s_{\can}: \ast \Rightarrow \ast' \end{equation}
the {\em canonical} degeneracy, identifying the extremal points.
\end{PAR}

Recall that
\begin{DEF}\label{DEFPLANAR}
A {\bf planar ($\infty$-)operad} is an ($\infty$-)category $\mathcal{C}$ equipped with a functor $p: \mathcal{C} \to \Delta^{\op}$ such that
\begin{enumerate}
\item coCartesian morphisms over inert morphisms exist,
\item for the standard family $\alpha_i: [n] \hookleftarrow [1]$ of inert morphisms \footnote{given by the inclusions $\alpha_i: \{i < i+1\} \hookrightarrow [n]$}, the morphism (choosing push-forward functors along these inert ones)
\begin{eqnarray*} \mathcal{C}_{[n]} &\rightarrow& \prod \mathcal{C}_{[1]}  \\
 X &\mapsto& (X_1, \dots, X_n) 
\end{eqnarray*}
is an isomorphism such that it induces (via composition with the corresponding coCartesian morphisms $X \to X_i$) an isomorphism
\[ \Hom_f(Z, X) \cong \Hom_{f_1}(Z, X_1) \times \cdots \times \Hom_{f_n}(Z, X_n)  \]
for all objects $X, Z \in \mathcal{C}$ and $f: p(Z) \to p(X)$.

We denote by $\mathrm{(co)Op}$ (resp.\@ $\mathrm{(co)Op}_{\infty}$) the (1,2)-category (resp.\@ $(\infty,2)$-category) of planar (co)operads.
We denote by $\OOO$ the associative planar operad, i.e.\@ the category $\Delta^{\op}$ equipped with the identity, considered as operad. 
\end{enumerate}

A planar ($\infty$-)cooperad is an ($\infty$-)category equipped with a functor $p: \mathcal{C} \to \Delta$ such that $\mathcal{C}^{\op}, p^{\op}$ is a planar ($\infty$-)operad.
\end{DEF}

\begin{PAR}
The morphisms over the unique active morphism $[n] \to [1]$ may, via the identification  $\mathcal{C}_{[n]} \cong \prod \mathcal{C}_{[1]}$, be seen as
as sets (resp.\@ $\infty$-groupoids) of multi-morphims
\[ \Hom(X_1, \dots, X_n; Y) \]
for all $X_1, \dots, X_n, Y \in \mathcal{C}_{[1]}$ that are composed as in a multi-category. In fact a planar 1-operad is the same as a non-symmetric multi-category and a planar $\infty$-operad is a non-symmetric multi-category
which is ``weakly enriched in $\infty$-groupoids'' in the same way that an $\infty$-category is a usual category ``weakly enriched in $\infty$-groupoids''.
\end{PAR}

\begin{PAR}
Already, {\em in general statements}, only the following facts about operads will be used, which are true for $\infty$-operads and planar $\infty$-operads alike, and even in generalizations of these concepts:
\begin{enumerate}
\item 
There is an $\infty$-category $\OOO$ (in both cases we are interested in, in fact, a 1-category) with a unique factorization system in the sense of \cite[5.2.8.8]{Lur09} into {\bf inert} and {\bf active} morphisms.
\item 
An operad $X$ is an $\infty$-category with a functor  $X \to \OOO$ satisfying certain conditions (as in Definition~\ref{DEFPLANAR}, for instance) including the existence of coCartesian morphisms over inert morphisms. 
Such morphisms are again called inert in $X$ and arbitrary morphisms lying over active ones in $\OOO$ are called active. They form thus again a unique factorization system in $X$  \cite[Proposition~2.1.2.5]{Lur11}.
\item  A functor of operads $X \to Y$ is a functor of $\infty$-categories over $\OOO$ which maps inert morphisms to inert morphisms. 
It is called a cofibration of operads if it is, in addition, a cofibration of $\infty$-categories.
\item There is an operad $\Cat_{\infty}^\times \to \OOO$ with functor $\times: \Cat_{\infty}^{\times} \to \Cat_{\infty}$ of $\infty$-categories such that a cofibration $X \to I$, where $X$ is an arbitrary $\infty$-category and $I$ an operad, is a cofibration of operads if and only if
the straighening $I \to \Cat_{\infty}$ classifying $X \to I$ factorizes  into
\[ I \to \Cat_{\infty}^{\times} \to \Cat_{\infty}. \]
and the functor $I \to \Cat_{\infty}^{\times}$ is a functor of operads. 
In that case, it factors in an essentially unique way. In other words, $\Cat_{\infty}^\times$ (as operad) classifies cofibrations of operads. 
\item 
For two morphisms between cofibrations of operads classified by $F, G: I \to \Cat_{\infty}^{\times}$ with compositions with $\times$ denoted $F^{\times}$ and $G^{\times}$ 
we have\footnote{where the first equivalence is induced by the usual one for cofibrations of $\infty$-categories, cf.\@ Proposition~\ref{PROPTRANSLATE} below.}
\[ \Hom_{\Op_{\infty}/M}(\int F^{\times}, \int G^{\times} ) \cong \Hom^{\lax,\inert-\pseudo}_{\Cat^I}(F^{\times}, G^{\times}) \cong \Hom^{\lax,\inert-\pseudo}_{(\Cat_{\infty}^{\times})^I}(F, G) \]
for $\Cat_{\infty}^{\times}$ is considered as $(\infty, 2)$-category in the obvious way.
\item For an operad $I$ the functor
\[ i \mapsto I_{\act} \times_{/I_{\act}} i \]
(comma category) which, by the unique factorization, extends to a functor $I \to \Cat_{\infty}$ factors through a morphism of operads $I \to \Cat_{\infty}^\times$.
\end{enumerate}
\end{PAR}

\subsection{(Co)fibrations, exponential fibrations, and Day convolutions}

\begin{PAR}\label{PAREXTRACT}
Recall the definition of fibration and cofibration for a functor $F: J \to I$. The definition can be split up into two conditions of which one is shared by both of them. 
Functors satisfying only the shared condition are a very convenient class. First observe that {\em any} functor $F: J \to I$ defines for every morphism $\alpha: i \to j$ a pro-functor $\alpha^{\bullet}: J_j \to J_i$
between the fibers given by
\[  \alpha^{\bullet}: J_i^{\op}  \times J_j \to \Set \qquad (\text{resp.\@}\ \alpha^{\bullet}: J_i^{\op}  \times J_j \to \Gpd_{\infty})   \]
\[ x \times y \mapsto \Hom_{J}(x, y)  \]
where $J_i$ is the fiber of $F$ over $i$. 
In fact, there is an equivalence of categories between pro-functors from $X$ to $Y$ and the category of functors $\mathcal{C} \to [1]$ with fibers $Y$ and $X$ (this is a 1-category (resp.\@ $(\infty, 1)$-category) because the fibers are fixed).

For composable morphisms $\alpha$ and $\beta$, composition in $J$ yields a morphism:
\begin{equation}\label{eqexp} \alpha^{\bullet}\ \beta^{\bullet} \Rightarrow (\beta \alpha)^{\bullet}.   \end{equation}
(One could say that $F$ yields a {\em lax functor} $I^{\op} \to \Cat^{\PF}$ (resp.\@ $I^{\op} \to \Cat_{\infty}^{\PF}$) but we will not use this unless it is an actual functor). 
We consider the following conditions:
\end{PAR}
\begin{DEF}
\begin{enumerate}
\item $F$ is called {\bf locally coCartesian} (resp.\@ {\bf locally Cartesian}), if for all $\alpha$, the pro-functor $\alpha^{\bullet}$ is of the form ${}^t\!\beta$ (resp.\@ $\beta$) for a functor $\beta: J_i \to J_j$ (resp.\@ $\beta: J_j \to J_i$).
\item $F$ is called an {\bf exponential fibration} (resp.\@ {\bf an $\infty$-exponential fibration}), if (\ref{eqexp}) is always an isomorphism.
\item $F$ is called a {\bf cofibration} (resp.\@ a {\bf fibration}) if it satisfies condition 2.\@ and the corresponding version of 1.
\end{enumerate}
\end{DEF}
One does not have to distinguish between fibrations and $\infty$-fibrations for a functor $I \to J$ of 1-categories because here fibration already implies $\infty$-fibration. Similarly for cofibrations. 

Exponential fibrations (resp.\@ $\infty$-exponential fibrations) without necessarily satisfying a version of condition 1.\@ are thus a natural generalization of both cofibrations and fibrations and they are classified by functors $I \to \Cat^{\PF}$ (resp.\@ $I \to \Cat_{\infty}^{\PF}$).
Many important constructions that work for cofibrations and fibrations alike have generalizations to exponential fibrations. The first is the internal Hom in categories over $I$. Its existence is in fact equivalent to being
``exponential fibration'' (cf.\@ \cite{AF20}):

\begin{PROP}
$F$ is an exponential fibration (resp.\@ $\infty$-exponential fibration) if and only if the functor
\[ \mathcal{C} \mapsto J \times_{I} \mathcal{C} \]
has a right 2-adjoint in $\Cat_{/I}$ (resp.\@ $\Cat_\infty{/I}$) 
\[ \mathcal{C} \mapsto D_I(J, \mathcal{C}). \]
In particular, sections of $D_I(J, \mathcal{C}) \to I$ are equal to $\Hom_{I}(J, \mathcal{C})$.
\end{PROP}

The objects and morphisms in the category $D_I(J, \mathcal{C})$ can be explicitly described: Objects are pairs of an object $i \in I$ and a functor 
$X \in (\mathcal{C}_i)^{J_i}$ and morphisms $(X, i) \to (Y, i')$ are morphisms $\alpha: i \to i'$ in $I$ together with an element in
\begin{eqnarray*} 
\Hom_{\Cat^{\PF}}(\alpha^{\bullet}_I, (X,Y)^*\alpha^\bullet_{\mathcal{C}}) & = & \int_{j,j'} \Hom(\alpha^{\bullet}_I(j, j'), \Hom(X(j), Y(j'))) 
\end{eqnarray*}
where $\alpha^{\bullet}_{\mathcal{C}}: \mathcal{C}_{i}^{\op}  \times \mathcal{C}_{i'}$ is the pro-functor extracted from $\mathcal{C}$ as in \ref{PAREXTRACT} ($\mathcal{C}$ does not have to be small to extract it).
One sees immediately that those can be composed, if the $\alpha^{\bullet}_I$ are functorial (and not only laxly functorial). The $\alpha^\bullet_{\mathcal{C}}$ can be laxly functorial though. Taking this as a definition in the 1-categorical context, one can show that it satisfies the universal property above. For the $\infty$-categorical generalization, see \cite{AF20}.

For (co)operads (resp.\@ $\infty$-(co)operads) the same is true mutatis mutandis.
See section~\ref{SECTCOOP} for a discussion of the definition, and of generalizations and an axiomatization. 
\begin{DEF}
Let $J \to I$ be a morphism of (co)operads.  
\begin{enumerate}
\item $F$ is called {\bf locally Cartesian} (resp.\@ {\bf locally coCartesian}), if $\alpha^{\bullet}$ is of the form $\beta$ (resp.\@ ${}^t\, \beta$) for a functor $\beta: J_{i'} \to J_i$ (resp.\@ $\beta: J_i \to J_{i'}$) \underline{for all active morphisms} $\alpha: i \to i'$.
\item $F$ is called an {\bf exponential fibration} (resp.\@ {\bf $\infty$-exponential fibration}), if (\ref{eqexp}) is always an isomorphism \underline{for active $\alpha$ and $\beta$}.
\item $F$ is called a {\bf cofibration} (resp.\@ a {\bf fibration}) if it satisfies conditions 1.\@ and 2.
\end{enumerate}
\end{DEF}

\begin{BEM}
Usually, we state the definition of coCartesian for operads (over $I=\OOO$ this is equivalent to being a monoidal category considered as operad) and the definition of Cartesian for cooperads (over $I=\OOO$ it is also equivalent to monoidal category, but considered as cooperad). In these cases, it is equivalent to claim that the whole functor of associated categories (of operators) is a (co)fibration in the usual sense, because the locally (co)Cartesianity for the inert morphisms holds by definition and
the exponential fibration condition follows for compositions involving inert morphisms from the other axioms. 
A fibration of operads will also occasionally be considered. Be aware that, in the monoidal case, say, this is {\em different} from the existence of a right adjoint w.r.t.\@ to one of the arguments, i.e.\@ the existence of internal Homs. 
\end{BEM}
\begin{WARNING} For a general exponential fibration of operads according to the above definition, (\ref{eqexp}), in general, does not have to be an isomorphism for compositions involving inert morphisms. This is 
related to the fact that $\times$ is not a product in $\Cat^{\PF}$ and $\cdot$ is not a terminal object.
\end{WARNING}

The proposition holds also true for (co)operads:
\begin{PROP}\label{PROPDAY}
$F: J \to I$ is an exponential fibration (resp.\@ $\infty$-exponential fibration) if and only if the functor
\[ \mathcal{C} \mapsto J \times_{I} \mathcal{C} \]
has a right 2-adjoint in $\mathrm{(co)Op}_{/I}$ (resp.\@ $\mathrm{(co)Op}_{\infty/I}$)
\[ \mathcal{C} \mapsto D_I(J, \mathcal{C}). \]

In particular, sections of $D_I(J, \mathcal{C}) \to I$ (e.g.\@ $I$-(co)algebras) are equal to $\Hom_{I}(J, \mathcal{C})$.
\end{PROP} See e.g.\@ \cite[2.2.6]{Lur11}.
$D_I(J, \mathcal{C})$ is called the {\bf Day convolution} of the (co)operads. 

The objects and morphisms in the (co)operad $D_I(J, \mathcal{C})$ can be explicitly described in the same way as before. 
E.g.\@ in the cooperad case, objects (over $[1]$) are pairs of an object $i \in I_{[1]}$ and a functor 
$X \in (\mathcal{C}_i)^{J_i}$ and morphism $(X, i) \to (Y_1, i_1'), \dots, (Y_n, i_n')$ are morphisms $\alpha: i \to i_1', \dots, i_n'$ in $I$ together with an element in
\begin{align*} 
& \Hom_{\Cat^{\PF}}(\alpha^{\bullet}_I, (X;Y_1, \dots, Y_n)^*\alpha^\bullet_{\mathcal{C}}) \\
= & \int_{j, j_1', \dots, j_n'} \Hom(\alpha^{\bullet}_I(j, j'_1, \dots, j'_n), \Hom(X(j); Y(j'_1), \dots, Y(j'_n))) 
\end{align*}
where $\alpha^{\bullet}_{\mathcal{C}}: \mathcal{C}_i^{\op} \times \mathcal{C}_{i_1'} \times \cdots \times \mathcal{C}_{i_n'}$ is the pro-functor extracted from $\mathcal{C}$ as in \ref{PAREXTRACT} ($\mathcal{C}$ does not have to be small to extract it).

In the following case, we have more control over the morphism spaces in this (co)operad: 
\begin{DEF}\label{DEFLR}
An arbitrary functor $\mathcal{C} \to I$ of $\infty$-categories or $\infty$-cooperads is
\begin{enumerate}
\item {\bf $L$-admissible}, if the fibers are cocomplete and the canonical morphism
\[ \Hom(\colim_J X; Y_1, \dots, Y_n) \to \lim_{J^{\op}} \Hom(X; Y_1, \dots, Y_n) \]
via composition with the canonical element in $\lim_{J^{\op}} \Hom(X, \colim_J X)$
is an isomorphism for all $i \in I, J, X \in \mathcal{C}_{i}^J, Y_1, \dots, Y_n$. 
\item {\bf $R$-admissible}, if the fibers are complete and for any small category $J$
\[   \Hom(X; Y_1, \dots, \lim_J Y_k, \dots, Y_n) \to \lim_{J} \Hom(X; Y_1, \dots, Y_n)   \]
via composition with the canonical element in $\lim_{J} \Hom(\lim_J Y_k, Y_k)$ 
is an isomorphism for all $X, Y_1, \dots, Y_n, i \in I, k, J, Y_k \in \mathcal{C}_{i}^J$.
\end{enumerate}
Similarly for $\infty$-operads. 
\end{DEF}
The following is clear from the definition:
\begin{LEMMA}
If $\mathcal{C} \to I$ is a functor between $\infty$-{\em categories} (not $\infty$-(co)operads) then the $L$-admissibility (resp.\@ $R$-admissibility) amounts to 
$\mathcal{C} \to I$ having (co)complete fibers such that the inclusion of the fibers preserves (co)limits. 
\end{LEMMA}

\begin{PAR}\label{PARLR}
Definition~\ref{DEFLR} may be stated (for cooperads) as follows: Let $\alpha: i \to i_1', \dots, i_n'$ be a morphism in $I$ and $X \in \mathcal{C}_i^{K}$, $Y_k \in \mathcal{C}_{i_k'}^{K_{k}}$.
We have for each $m \in \{1, \dots, n\}$ and a pro-functor $\beta: K_m \to K_m'$ a morphisms of pro-functors
\[ (X,Y_1, \dots, R(\beta)Y_m, \dots, Y_n)^*\alpha_{\mathcal{C}}^{\bullet} \to R(\beta)((X;Y_1, \dots, Y_n)^*\alpha_{\mathcal{C}}^{\bullet}) \]
and for each pro-functor $\beta: K' \to K$ 
\[ (L(\beta)X;Y_1, \dots, Y_n)^*\alpha_{\mathcal{C}}^{\bullet} \to R(\beta^{\op})((X;Y_1, \dots, Y_n)^*\alpha_{\mathcal{C}}^{\bullet}) \]
which are isomorphisms if and only if $\mathcal{C}$ is $R$, resp.\@ $L$-admissible. 
\end{PAR}

\begin{PAR}\label{PARLR2}
We also have a canonical morphism
\begin{equation} \label{eqr} (X;Y_1, \dots, Y_n)^*\alpha_{\mathcal{C}} \to R(\beta^{\op})(R(\beta)X;Y_1, \dots, Y_n)^*\alpha_{\mathcal{C}} \end{equation}
which is the composition with the counit $L(\beta)R(\beta) \to \id$, if $L(\beta)$ exists. However, the morphism always exists. 

We also have canonical morphisms
\begin{equation} \label{eql} (X;Y_1, \dots, Y_n)^*\alpha_{\mathcal{C}} \to R(\beta)(X;Y_1, \dots, L(\beta)Y_k, \dots, Y_n)^*\alpha_{\mathcal{C}} \end{equation}
which is the composition with the unit $\id \to R(\beta)L(\beta)$ if $R(\beta)$ exists. However, the morphism always exists. 
\end{PAR}

\begin{PAR}\label{PARPROJECTION}By (slight) abuse of notation we continue to write $\times$ for the usual product of two categories also as objects in $\Cat^{\PF}_{\infty}$. 
Notice, however, that $(\Cat^{\PF}_{\infty}, \times)$ is not Cartesian (as monoidal $(\infty,2)$-category, i.e.\@ $\times$ is not the product), yet there is a natural morphism 
\[ \Hom_{\Cat^{\PF}_{\infty}}(I_1, \dots, I_n; J) \to \Hom_{\Cat^{\PF}_{\infty}}(I_1; J) \times \cdots \times  \Hom_{\Cat^{\PF}_{\infty}}(I_n; J)  \]
given by composition with the pro-functor $\cdot \to I_i$ given by the final object in $\Hom(I_i^{\op}, \Set)$.
Similarly for $(\Cat^{\PF}_{\infty}, \times)^{\vee}$, there is a natural morphism 
\[ \Hom_{\Cat^{\PF}_{\infty}}(J; I_1, \dots, I_n) \to \Hom_{\Cat^{\PF}_{\infty}}(J; I_1) \times \cdots \times  \Hom_{\Cat^{\PF}_{\infty}}(J; I_n)  \]
given by composition with the pro-functor $I_i \to \cdot$ given by the final object in $\Hom(I_i, \Set)$.
\end{PAR}

\begin{LEMMA}\label{LEMMALR}
\begin{enumerate}
\item 
Let $\mathcal{C} \to I$ be a cofibration of $\infty$-operads. Then it is automatically $R$-admissible if it has complete fibers and $L$-admissible if it has cocomplete fibers and the push-forward functors are cocontinous in the sense that for all $\alpha: i \to i_1', \dots, i_n'$, and all entries $\in \{1, \dots, n\}$ the diagram
\[ \xymatrix{ \mathcal{C} \times  \cdots \times \mathcal{C}^J \times  \cdots \times \mathcal{C} \ar[r]^-{\alpha_{\bullet, \mathcal{C}}} \ar[d]_-{\id \times \cdots \times  \colim \times \cdots \times \id} & \mathcal{C}^{J}  \ar[d]^{ \colim}  \\
\mathcal{C} \times \cdots \times \mathcal{C} \ar[r]_-{\alpha_{\bullet, \mathcal{C}}} &  \mathcal{C}
 }  \]
is commutative (and no condition on the 0-ary morphisms, e.g.\@ units). 

\item 
Let $\mathcal{C} \to I$ be a fibration of $\infty$-operads. Then it is automatically $L$-admissible if it has cocomplete fibers and $R$-admissible if it has complete fibers and the pull-back functors are cocontinous in the sense that 
 for all $\alpha: i \to i_1', \dots, i_n'$, the diagram
\[ \xymatrix{  \mathcal{C}^J \ar[r]^-{\alpha^{\bullet}_{\mathcal{C}}} \ar[d]_{\lim} & \mathcal{C}^J \times \cdots \times \mathcal{C}^J \ar[d]^{\lim \times \cdots \times  \lim}  \\
\mathcal{C} \ar[r]^-{\alpha^{\bullet}_{\mathcal{C}}} & \mathcal{C} \times \cdots \times \mathcal{C}
 } \]
 are commutative (and no condition on the 0-ary morphisms).
\end{enumerate}
Similarly for cooperads. 
\end{LEMMA}

In particular, for cofibrations of $\infty$-operads, the notion of $L$-admissible is the same as a ``cofibration compatible with $K$-indexed colimits'' (for all $K$) in the sense of  \cite[Definition~3.1.1.18]{Lur11}.
For example, a monoidal category $(\mathcal{C}, \otimes) \to \OOO$ is $L$-admissible, if and only if it is cocomplete and $\otimes$ commutes with colimits in each variable separately.

\begin{PROP}\label{PROPDAYCOFIB}
Let $J \to I$ be an ($\infty$-)exponential fibration of ($\infty$-)(co)operads, and $\alpha: i_1, \dots, i_n \to i$ (resp.\@ $\alpha: i \to i_1', \dots, i_n'$) be an active morphism in $I$ 
with associated pro-functor $\alpha^{\bullet}_J: J_{i_1'}^{\op}\times \cdots \times J_{i_1'}^{\op} \times  J_{i}   \to \Gpd_{\infty}$
(resp.\@ $\alpha^{\bullet}_J:   J_{i'}^{\op} \times J_{i_1} \times \cdots \times J_{i_n}   \to \Gpd_{\infty}$).
\begin{enumerate}
\item 
If $\mathcal{C} \to I$ is an $L$-admissible cofibration of ($\infty$-)(co)operads. Then also 
\[ D_I(J, \mathcal{C}) \to I \]
is a cofibration with push-forward along $\alpha$ given by 
\[ L(\alpha^{\bullet}_J)  \circ \alpha_{\bullet, \mathcal{C}} (-, \dots, -)  
\text{ resp.\@ } 
(L(\alpha^{\bullet}_{J,1}), \dots, L(\alpha^{\bullet}_{J,n}))   \circ \alpha_{\bullet, \mathcal{C}} (-)   
\]

\item
If $\mathcal{C} \to I$ is an $R$-admissible fibration of ($\infty$-)(co)operads. Then also 
\[ D_I(J, \mathcal{C}) \to I \]
is a fibration with pull-back along $\alpha$ given by 
\[ (R(\alpha^{\bullet}_{J,1}), \dots, R(\alpha^{\bullet}_{J,n}))  \circ \alpha^{\bullet}_{\mathcal{C}} (-)
\text{ resp.\@ } 
 R(\alpha^{\bullet}_J) \circ \alpha^{\bullet}_{\mathcal{C}} (-, \dots, -)    . 
\]
\end{enumerate}
Here the $\alpha^{\bullet}_{J,k}$ are the projections of the pro-functor $J_{i_1'} \times \cdots \times J_{i_n'} \to J_{i}$ (resp.\@  $J_{i} \to J_{i_1'} \times \cdots \times J_{i_n'}$) (\ref{PARPROJECTION}).
\end{PROP}

\begin{proof}
We prove the statements for cooperads, the others are dual. 
If $\mathcal{C} \to I$ is an $L$-admissible cofibration then
\begin{align*} 
& \int_{j, j'_1, \dots, j_n'} \Hom(\alpha^{\bullet}_I(j; j'_1, \dots, j_n'), \Hom(X(j); Y_1(j_1'), \dots, Y_n(j_n'))) \\
\cong & \int_{j, j'_1, \dots, j_n'}  \Hom(\alpha^{\bullet}_I(j; j'_1, \dots, j_n'), \Hom((\alpha_{\bullet, \mathcal{C},1}X(j)); Y_1(j_1')) \times \cdots  \\
& \times \Hom((\alpha_{\bullet, \mathcal{C},n}X(j)), Y_n(j_n'))) \\
\cong & \int_{j'_1, \dots, j_n'}  \Hom(\int^{j} \alpha^{\bullet}_I(j; j'_1, \dots, j_n') \times (\alpha_{\bullet, \mathcal{C},1}X(j)); Y_1(j_1')) \times \cdots  \\
&  \times\Hom(\int^{j} \alpha^{\bullet}_I(j; j'_1, \dots, j_n') \times (\alpha_{\bullet, \mathcal{C},n}X(j)), Y_n(j_n'))) \\
\cong & \int_{j_1'}  \Hom((L(\alpha^{\bullet}_{I,1})(\alpha_{\bullet, \mathcal{C},1}X(j'_1))); Y_1(j_1')) \times \cdots  \\
& \times \int_{j_n'} \Hom(L(\alpha^{\bullet}_{I,n})(\alpha_{\bullet, \mathcal{C},n}X(j_n')), Y_n(j_n'))).
\end{align*} 
(notice that the end is just a limit, once the argument depends on the variable only contra- or covariantly).
If $\mathcal{C} \to I$ is an $R$-admissible fibration then
\begin{align*} 
& \int_{j, j'_1, \dots, j_n'}  \Hom(\alpha^{\bullet}_I(j; j'_1, \dots, j_n'), \Hom(X(j); Y_1(j_1'), \dots, Y_n(j_n'))) \\
\cong & \int_{j, j'_1, \dots, j_n'}  \Hom(\alpha^{\bullet}_I(j; j'_1, \dots, j_n'), \Hom(X(j); \alpha^{\bullet}_{\mathcal{C}}(Y_1(j_1'), \dots, Y_n(j_n')))) \\
\cong & \int_{j} \Hom(X(j); \int_{j'_1, \dots, j_n'}\Hom(\alpha^{\bullet}_I(j; j'_1, \dots, j_n'),  \alpha^{\bullet}_{\mathcal{C}}(Y_1(j_1'), \dots, Y_n(j_n')))) \\
\cong & \int_{j} \Hom(X(j); R(\alpha^{\bullet}_I) (\alpha^{\bullet}_{\mathcal{C}}(Y_1, \dots, Y_n))(j)).
\end{align*} 
In each case (\ref{eqexp}) is an isomorphism by a formal calculation using the $L-$, resp.\@ $R$-admissibility.
\end{proof}

Of course, instead of the $L$- or $R$-admissibility one needs only to assume the existence of {\em certain} (co)limits and compatibility for them. We will not put this into an explicit statement.

\begin{BEISPIEL}\label{EXDAY}
Let $J \to \OOO$ be an exponential fibration of operads (sometimes $J$ is then called a pro-monoidal category). For
 $(\mathcal{C}, \otimes)$ a monoidal category with $\mathcal{C}$ having the relevant colimits\footnote{more precisely, the Kan extensions $L(\gamma)$, where $\gamma$ is one of the pro-functors encoding $J$ must exist} such that $\otimes$ commutes with them entrywise,  $D(J, (\mathcal{C}, \otimes))$ is again monoidal, with product
 \[ - \otimes -  := L(m^{\bullet}) - \boxtimes -  \]
 where 
 \[ \boxtimes: \mathcal{C}^{J_{[1]}} \times \mathcal{C}^{J_{[1]}} \to \underbrace{\mathcal{C}^{J_{[1]} \times J_{[1]}}}_{\cong \mathcal{C}^{J_{[2]}}} \] 
 is the point-wise application of $\otimes$.
 If $J$ is itself monoidal (i.e.\@ $J \to \OOO$ a cofibration) with product $\dec: J_{[1]} \times J_{[1]} \to J_{[1]}$ then $- \otimes -  := \dec_! - \boxtimes - $.
 If $\mathcal{C}$ has the relevant colimits but $\otimes$ does not commute with them then $D(J, (\mathcal{C}, \otimes)) \to \OOO$ is locally coCartesian but not a cofibration. 

 Similarly, if $J \to \OOO$ is an exponential fibration of cooperads. For
 $(\mathcal{C}, \otimes)^{\vee}$ a monoidal category (considered as cooperad) with $\mathcal{C}$ having the relevant limits such that  $\otimes$ commutes with them, $D(J, (\mathcal{C}, \otimes)^\vee)$ is the monoidal category (considered as cooperad) with product
 \[ - \otimes -  := R(m^{\bullet}) - \boxtimes - \]
  If $J$ is itself monoidal (i.e.\@ $J \to \OOO^{\op}$ a fibration) with product $\dec: J_{[1]} \times J_{[1]} \to J_{[1]}$ then $- \otimes -  := \dec_* - \boxtimes - $.  
  
This example is central for the understanding of the differences between the two natural tensor products on non-negatively graded complexes in the Abelian case (cf.\@ \ref{EXDAY2}).
\end{BEISPIEL}

\begin{BEISPIEL}\label{EXDAY2}
If $J \to I$ is a fibration and $\mathcal{C} \to I$ is a cofibration of operads or vice versa then $D_I(J, \mathcal{C})$ is a cofibration (or fibration, respectively) and we have
\[ (D_I(J, \mathcal{C}))^{\vee} \cong D_I(J^{\vee}, \mathcal{C}^{\vee}). \]
Moreover, if they are classified by functors $\Xi_{J}: I^{\op} \to (\Cat_{\infty}, \times)$ and $\Xi_{\mathcal{C}}: I \to (\Cat_{\infty}, \times)$ then $D_I(J, \mathcal{C})$ is classified by the functor
\begin{align*} I &\to (\Cat_{\infty}, \times) \\
 i & \mapsto \Hom(\Xi_{J}(i),\Xi_{\mathcal{C}}(i)) \qquad \text{for $i \in I_{[1]}$}.
\end{align*}
Dually, the same holds true, of course, for cooperads. 
\end{BEISPIEL}

\begin{PAR}
A functor $F: \mathcal{C} \to \mathcal{D}$ over $I$ between ($\infty$-)exponential fibrations of ($\infty$-)(co)operads gives rise to an oplax transformation 
\[ F: \Hom^{\oplax, \inert-\pseudo}_{(\Cat^{\PF,\times})^{I^{\op}}}(\Xi_{\mathcal{C}}, \Xi_{\mathcal{D}})  \qquad (\text{resp.\@ } F: \Hom^{\oplax, \inert-\pseudo}_{(\Cat^{\PF,\times}_{\infty})^{I^{\op}}}(\Xi_{\mathcal{C}}, \Xi_{\mathcal{D}}))   \]
where $\Xi_{\mathcal{C}}, \Xi_{\mathcal{D}}: I^{\op} \to (\Cat^{\PF}, \times)$ (resp.\@ $(\Cat^{\PF}_{\infty}, \times)$) are the classifying functors
and $\Cat^{\PF,\times}$ (resp.\@ $\Cat^{\PF,\times}_{\infty}$)  is the underlying $(1,2)$- (resp.\@ $(\infty, 2)$-) category (of operators) of $(\Cat^{\PF}, \times)$ (resp.\@ $(\Cat^{\PF}_{\infty}, \times)$) 
 with components for active $\alpha: i \to i'$:
\begin{equation}\label{eqcart} F_i\, \alpha_{\mathcal{C}}^{\bullet} \to \alpha_{\mathcal{D}}^{\bullet}\, F_{i'} \end{equation}
or equivalently, gives rise to a lax transformation (its mate):
\[ F: \Hom^{\lax, \inert-\pseudo}_{(\Cat^{\PF,\times})^{I^{\op}}}(\Xi_{\mathcal{D}}, \Xi_{\mathcal{C}})  \qquad (\text{resp.\@ } F: \Hom^{\lax, \inert-\pseudo}_{(\Cat^{\PF,\times}_{\infty})^{I^{\op}}}( \Xi_{\mathcal{D}}, \Xi_{\mathcal{C}}))   \]
with components: 
\begin{equation}\label{eqcocart}  \alpha_{\mathcal{C}}^{\bullet}\, {}^t\! F_{i'} \to {}^t\! F_i\, \alpha_{\mathcal{D}}^{\bullet}.  \end{equation}
Notice that the functors have values in $\Cat^{\PF,\times}_{(\infty)}$ (not $\Cat^{\PF}_{(\infty)}$). Therefore the formation of mate lands in the full  subcategory of $\inert-\pseudo$ transformation again. Notice also $(\Cat^{\PF}_{(\infty)}, \times)$ is to be considered as cooperad, if $I$ is an operad and vice versa. $\Cat^{\PF, \times}_{(\infty)}$ depends on this!
\end{PAR}

\begin{DEF}\label{DEFCART}
We call $F$ {\bf Cartesian} (resp.\@ {\bf $\infty$-Cartesian}), if (\ref{eqcart}) is an isomorphism for active morphisms and {\bf coCartesian} (resp.\@ $\infty$-coCartesian), if (\ref{eqcocart}) is an isomorphism for active morphisms. 
\end{DEF}

Note that a functor is a morphism of fibrations in the usual sense (i.e.\@ maps Cartesian morphisms to Cartesian ones) if and only if it is Cartesian according to the Definition above. Similarly, 
a functor is a morphism of cofibrations (i.e.\@ maps coCartesian morphisms to coCartesian ones) if and only if it is coCartesian according to the Definition above.

\begin{PROP}\label{PROPDAYLAX}
If $\mathcal{C} \to I$ is an $L$-admissible functor of cooperads then the functor Day convolution extends to a functor of 2-categories:
\[ \xymatrix{ \Hom^{1-\oplax, 1-\inert-\pseudo}(I^{\op}, (\Cat^{\PF}, \times))^{1-\op} \ar[r]^-{L_{\mathcal{C}}}  & \mathrm{coOp}_{/I} } \]
If $\mathcal{C} \to I$ is an $R$-admissible functor of cooperads then the functor Day convolution extends to a functor of 2-categories:
\[ \xymatrix{ \Hom^{1-\lax, 1-\inert-\pseudo}(I^{\op}, (\Cat^{\PF}, \times)) \ar[r]^-{R_{\mathcal{C}}}  & (\mathrm{coOp}_{/{I}})^{2-\op} } \]
The same is true with the decoration $\infty$ everywhere and for operads. 
\end{PROP}

To be clear: The 2-category on the left hand side has as objects functors of operads $\Xi_{\mathcal{C}}: I^{\op} \to (\Cat^{\PF}, \times)$ and morphism categories are given by
\[ \Hom^{\mathrm{oplax}, \inert-\pseudo}_{(\Cat^{\PF,\times})^{I^{\op}}}(\Xi_{\mathcal{D}}, \Xi_{\mathcal{C}})  \]
and
\[ \Hom^{\mathrm{lax}, \inert-\pseudo}_{(\Cat^{\PF,\times})^{I^{\op}}}(\Xi_{\mathcal{C}}, \Xi_{\mathcal{D}})  \]
where $\Cat^{\PF,\times}$ is as before. 

\begin{proof}We will not rigorously prove this and be content with specifying the map on morphism spaces.
Let $J \to I, K \to I$ be exponential fibrations of cooperads classified by functors $\Xi_J, \Xi_K: I^{\op} \to (\Cat^{\PF}, \times)$ and assume given a 
lax transformation $\mu: \Xi_J \Rightarrow \Xi_K$. 
For a morphism $\alpha: i \to i'_1, \dots i_n'$ in $I$ this yields a diagram
\[ \xymatrix{ J_{i'_1} \times \cdots \times J_{i'_n} \ar[r]^-{\alpha^{\bullet}_J} \ar[d]_{\beta_{i'_1}, \dots, \beta_{i'_n}} \ar@{}[rd]|{\Nearrow} & J_i \ar[d]^{\beta_i} \\
K_{i'_1} \times \cdots \times K_{i'_n} \ar[r]_-{\alpha^{\bullet}_K} & K_i  } \] 
Let $\alpha: i \to i_1', \dots, i_n'$ be a morphism in $I$. Then the map is given by:
\begin{align*}
& \  [\alpha^{\bullet}_J, (X; Y_{1}, \dots, Y_{n})^* \alpha_{\mathcal{C}}^{\bullet}]  \\
\to & \  
[\alpha^{\bullet}_J, R(\beta_i^{\op}) \left( (R(\beta_i)X; Y_{1}, \dots, Y_{n})^* \alpha_{\mathcal{C}}^{\bullet} \right)]  & \text{canonical map (\ref{eqr})} \\
= & \ 
[L(\beta_i^{\op}) \alpha^{\bullet}_J,  (R(\beta_i)X; Y_{1}, \dots, Y_{n})^* \alpha_{\mathcal{C}}^{\bullet} ]  & \text{$L$/$R$-adjunction} \\
\to & \ 
[L(\beta_{i'_1}, \dots, \beta_{i'_n}) \alpha^{\bullet}_K,  (R(\beta_i)X; Y_{1}, \dots, Y_{n})^* \alpha_{\mathcal{C}}^{\bullet} ] & \text{laxness constraint}   \\
= & \ 
[ \alpha^{\bullet}_K, R(\beta_{i'_1}, \dots, \beta_{i'_n}) \left( (R(\beta_i)X; Y_{1}, \dots, Y_{n})^* \alpha_{\mathcal{C}}^{\bullet} \right)] & \text{$L$/$R$-adjunction}     \\
\cong & \ 
[ \alpha^{\bullet}_K, \left( (R(\beta_i)X; R(\beta_{i'_1})Y_1, \dots, R(\beta_{i'_n})Y_n)^* \alpha_{\mathcal{C}}^{\bullet} \right)]  & \text{$R$-admissibility}   
\end{align*}
Similarly for the $L$-functoriality: 
\begin{align*}
& \  [\alpha^{\bullet}_J, (X; Y_{1}, \dots, Y_{n})^* \alpha_{\mathcal{C}}^{\bullet}]  \\
\to & \  
[\alpha^{\bullet}_J, (R(\beta_{i_1'}), \dots,  R(\beta_{i_n'})) \left( (X;  L(\beta_{i_1'})  Y_{1}, \dots, L(\beta_{i_n'}) Y_{n})^* \alpha_{\mathcal{C}}^{\bullet} \right)]  & \text{canonical map (\ref{eql})} \\
= & \ 
[(L(\beta_{i_1'}, \dots, \beta_{i_1'}) \alpha^{\bullet}_J,  (X; L(\beta_{i_1'})Y_{1}, \dots, L(\beta_{i_n'})Y_{n})^* \alpha_{\mathcal{C}}^{\bullet} ]  & \text{$L$/$R$-adjunction} \\
\to & \ 
[L(\beta_{i}^{\op}) \alpha^{\bullet}_K,  (X; L(\beta_{i_1'})Y_{1}, \dots, L(\beta_{i_n'})Y_{n})^* \alpha_{\mathcal{C}}^{\bullet} ] & \text{oplaxness constraint}   \\
= & \ 
[ \alpha^{\bullet}_K, R(\beta_{i}^{\op}) \left( (X; L(\beta_{i_1'})Y_{1}, \dots, L(\beta_{i_n'})Y_{n})^* \alpha_{\mathcal{C}}^{\bullet} \right)] & \text{$L$/$R$-adjunction}     \\
\cong & \ 
[ \alpha^{\bullet}_K, \left( (L(\beta_i)X; L(\beta_{i'_1})Y_1, \dots, L(\beta_{i'_n})Y_n)^* \alpha_{\mathcal{C}}^{\bullet} \right)]  & \text{$L$-admissibility}   
\end{align*}
The operad case is dual and everything works the same with $\infty$-cooperads. 
\end{proof}

\begin{PAR}\label{PARCONVOLUTION}
If $\mathcal{C} \to I$ and $\mathcal{D} \to I$ are exponential fibrations, an oplax transformation (and inert-pseudo) of the associated functors $\Xi_\mathcal{C} \Rightarrow \Xi_\mathcal{D}$  is called an {\bf oplax pro-functor of (co)operads} $\mathcal{C} \to \mathcal{D}$,
and a lax transformation (and inert-pseudo) $F_\mathcal{D} \Rightarrow F_\mathcal{C}$ is called a {\bf  lax pro-functor of (co)operads} $\mathcal{C} \to \mathcal{D}$. We call an oplax pro-functor {\bf Cartesian}, if it a natural transformation (not only oplax) and a lax pro-functor {\bf coCartesian} if it a natural transformation (not only lax).

As just seen, a usual functor $F: \mathcal{C} \to \mathcal{D}$ of (co)operads over $I$ can be seen as a either of these morphisms. This mirrors the situation over a point discussed in \ref{PAROPPF}. In fact, if $\mathcal{C} \to I$ is $L$-admissible, the whole diagram (\ref{eqdiapf}) 
extends to fibrations of (co)operads over $I$: 
\[ \footnotesize \xymatrix{ \mathrm{(co)Op}_{/I}^{\exp,1-\op} \ar[r]^-{\iota} \ar[d]^{\op} & \Hom^{1-\oplax, 1-\inert-\pseudo}(I^{\op}, (\Cat^{\PF}, \times))^{1-\op} \ar[r]^-{L_{\mathcal{C}}} \ar[d]^{\op} & \mathrm{(co)Op}_{/I} \ar[d]^{\op} \ar[d] \\
 \mathrm{(co)Op}_{/I^{\op}}^{\exp,1-\op, 2-\op} \ar[r]^-{{}^t \iota} & \Hom^{1-\lax, 1-\inert-\pseudo}(I, (\Cat^{\PF}, \times)) \ar[r]^-{R_{\mathcal{C}^{\op}}} &   \mathrm{(co)Op}^{2-\op}_{/I^{\op}} }\]
 where Op and coOp are interchanged in the two rows. 
 Here $L_{\mathcal{C}}$ and also $R_{\mathcal{C}}$, for a category or (co)operad $\mathcal{C} \to I$, are given by the Day convolution on objects, and on morphisms by Proposition~\ref{PROPDAYLAX}.
 
Also $\iota$ and ${}^t \iota$ agree on objects. More generally, if $F,G: I^{\op} \to (\Cat^{\PF}, \times)$ are functors and
\[ \rho \in \Hom^{\oplax, \inert-\pseudo}_{\Cat^{\PF, \times}}(F, G)   \]
has a point-wise right adjoint, then they assemble to a mate
\[ {}^t\!\rho \in \Hom^{\lax, \inert-\pseudo}_{\Cat^{\PF, \times}}(G, F)  \]
(and vice versa) and for $\mathcal{C} \to I$ we have 
\[ L_{\mathcal{C}}(\rho) = R_{\mathcal{C}}({}^t\!\rho). \]
In particular, $L_{\mathcal{C}}\, \iota = R_{\mathcal{C}}\, {}^{t}\! \iota$ also on morphisms. 
The same is true in the $\infty$-categorical context. 
\end{PAR}

\begin{PAR}\label{PAREXPSUB}
Let $\mathcal{C} \to I$ be an ($\infty$-)exponential fibration of ($\infty$-)(co)operads with associated pro-functors $\alpha^{\bullet}: \mathcal{C}_{i'} \to \mathcal{C}_i$.
Let $\iota: \mathcal{C}' \hookrightarrow \mathcal{C}$ be a full embedding. Then the pro-functors associated with  $\mathcal{C}' \to I$ are obviously given by
$(\alpha')^\bullet = {}^t\! \iota\, \alpha^\bullet \iota$ and the transition morphism
\begin{eqnarray} \label{eqtrans2}  \,{}^t\! \iota\, g^\bullet \iota \,{}^t\! \iota f^\bullet \iota  \Rightarrow  \,{}^t\! \iota  (fg)^\bullet \iota    \end{eqnarray}
is induced by the counit  $\iota \,{}^t\!\iota \to \id$. Thus $\mathcal{C}'$ is an ($\infty$-)exponential fibration again, if and only if (\ref{eqtrans2}) is an isomorphism.
This is obviously the case 
if $\iota$ is ($\infty$\nobreakdash-)Cartesian (i.e.\@ if $\iota\,{}^t\!\iota\,f^{\bullet}\, \iota \to f^{\bullet}\, \iota$ is an isomorphism) or ($\infty$-)coCartesian (i.e.\@ if ${}^t\!\iota\, f^{\bullet}\,  \iota\, {}^t\!\iota  \to {}^t\!\iota\,f^{\bullet}$ is an isomorphism). We have proven: 
\end{PAR}
\begin{LEMMA}\label{LEMMAFLAT}
Let $\mathcal{C} \to I$ be an ($\infty$-)exponential fibration of ($\infty$-)(co)operads and $\iota: \mathcal{C}' \hookrightarrow \mathcal{C}$ be a full embedding.
If $\iota$ is ($\infty$-)Cartesian or ($\infty$-)coCartesian\footnote{cf.\@ Definition~\ref{DEFCART}, which is not literally applicable as stated, thus interpreted as in~\ref{PAREXPSUB}} then $\mathcal{C}' \to I$ is again an ($\infty$-)exponential fibration. 
 \end{LEMMA}

\subsection{Twisted arrow categories for operads}\label{SECTTW}

\begin{PAR}
Let $\Xi \in \{\downarrow, \uparrow\}^n$ be an ordered sequence of directions. For an $\infty$-category, or more generally, an $\infty$-operad $I$, we will define 
$\infty$-categories (resp.\@ $\infty$-operads) ${}^{\Xi} I$ such that for $\infty$-categories:
\[ {}^{\downarrow} I = I,\ {}^{\uparrow} I = I^{\op},\ \twop I = \mathrm{tw}(I), \ \tw I = \mathrm{tw}^{\op}(I)  \]
where $\mathrm{tw}(I)$ is the twisted arrow category \cite[5.2.1]{Lur11}
and such that, more generally, the {\em active} morphism spaces in ${}^{\Xi} I$ are equivalent to the space of commutative diagrams of the form
\begin{equation} \label{EQMORTW}
\vcenter{ \xymatrix{ x_1 \ar[r] \ar@{<->}[d] & x_2  \ar[r] \ar@{<->}[d] & \cdots \ar[r] & x_n  \ar@{<->}[d] \\
x_1' \ar[r] & x_2' \ar[r] & \cdots \ar[r] & x_n' 
 } }  \end{equation}
in which all morphisms are active, and where the arrow directions are dictated by $\Xi$. For operads, the equation $\twop I = \mathrm{tw}(I)$ does not hold true 
when one understands by $\mathrm{tw}(I)$ the usual twisted arrow category associated with the underlying $\infty$-category (of operators). The two are related, however, see Lemma~\ref{LEMMAOPTWIST} below. 
\end{PAR}

\begin{DEF}\label{DEFTW}
Let  $I$ be an $\infty$-operad or $\infty$-category. Define inductively:
\[ {}^{\Xi \downarrow} I := \int {}^{\Xi}(I_{\act} \times_{/I_{\act}} i)   \]
the unstraightening construction associated with the functor of $\infty$-categories (of operators):
\begin{eqnarray*} I &\to& \Cat_{\infty} \\
 i & \mapsto & I_{\act} \times_{/I_{\act}} i
 \end{eqnarray*}
the functoriality being obtained by unique factorization into inert and active.
Finally set:
\[ {}^{\Xi \uparrow} I := ({}^{\Xi' \downarrow} I)^{\op}  \]
where in $\Xi'$ all arrow directions are reversed. 
\end{DEF} 
This defines recursively ${}^{\Xi} I$ for all operads $I$. Notice that the $\infty$-operad structure only plays a role in the first step; the comma category $I _{\act} \times_{/I_{\act}} i$ is just to be considered as a plain $\infty$-category and does not
carry any operad structure. 

\begin{LEMMA}
 ${}^{\Xi} I$ is an $\infty$-operad again, if the last arrow of $\Xi$ is $\downarrow$, and otherwise an $\infty$-cooperad. 
 The last projection $\pi_n: {}^{\Xi} I \to I$ is a cofibration of operads, resp.\@ a fibration of cooperads $\pi_n: {}^{\Xi} I \to I^{\op}$.
\end{LEMMA}
\begin{proof}
By basic facts on $\infty$-operads \ref{PAROP}, 6.\@
the association
 \[  i  \mapsto  I_{\act} \times_{/I_{\act}} i  \]
factors
\[ I \to \Cat_{\infty}^\times \to \Cat_{\infty} \]
where the first is a functor of $\infty$-operads, and thus the corresponding cofibration is a cofibration of $\infty$-operads  (cf.\@ \ref{PAROP}, 4.\@).
\end{proof}

\begin{DEF}\label{DEFTYPEI}
An {\em active} morphism which corresponds to a diagram (\ref{EQMORTW}) in which all vertical morphisms, except for the $i$-th one, are isomorphisms, is called a {\bf type-$i$ morphism.} With respect to the recursive definition, those are precisely the morphisms that lie in the fiber for $n-i$ steps and are coCartesian for the next unstraighening (resp.\@ for $i=1$ are just any morphism in the fibers for the last step). 
In particular, all type-$i$ morphisms for $i<n$ lie over an isomorphism in $\OOO$.
\end{DEF}

\begin{LEMMA}\label{LEMMATWLOC}
Let $I$ be an $\infty$-category.
Let $\Xi' \subset \Xi$ be a non-empty subsequence of arrow directions. Then the projection 
\[ \pi: {}^{\Xi}I \rightarrow {}^{\Xi'}I \]
is a localization, i.e.\@ it exhibits the target as the localization of ${}^{\Xi}I$ at those morphisms that are mapped to isomorphisms under $\pi$. 
\end{LEMMA}
\begin{proof}
By induction, it suffices to show the following.
Let $\Xi = L\!\! \downarrow \!\! R$ be a sequence of arrow directions and let $\Xi' = L\,R$.  One of $R$ or $L$ may be empty, but not both. Then the projection 
\[ {}^{\Xi}I \rightarrow {}^{\Xi'}I \]
is a localization. 
We first discuss the case in which the right-most arrow of $L$ is $\downarrow$ or the leftmost arrow of $R$ is $\downarrow$.
Then there is even a section 
\[ {}^{\Xi'}I \rightarrow {}^{\Xi}I \]
which turns ${}^{\Xi'}I$ a (co)reflective subcategory. 
Now assume that the condition is not satisfied. If $L$ or $R$ is empty then this projection is actually a (co)fibration with fibers of the form
$i \times_{/I} I$, resp.\@ $I \times_{/I} i$, hence contractible. The morphism from a (co)fibration with contractible fibers to its base is a localization.
If $R$ and $L$ are not empty consider the projection

\[  \xymatrix{ {}^{L \downarrow R} I  \ar[dr] \ar[rr] & & {}^{L R} I \ar[dl] \\
& {}^{R} I  } \]
This is a morphism of (co)fibrations (i.e.\@ mapping  (co)Cartesian morphisms to (co)Cartesian ones) and the morphism between the fibers is of the form:
\[ {}^{L \downarrow } I \times_{/I} i  \to  {}^{L }  I \times_{/I} i  \]
hence of a form discussed already. A morphism of (co)fibrations which is fiber-wise a localization is a localization itself \cite[4.1.11]{AF20}.
\end{proof}

\begin{LEMMA}\label{LEMMAOPTWIST}
Let $I$ be an $\infty$-operad. There is a  functor
\[ \Pi: \mathrm{tw}^{\op}(I) \to \tw I  \]
which is a localization. Here $\mathrm{tw}^{\op}(I)$ denotes the usual (dual) twisted arrow category of $I$ considered as $\infty$-category (of operators). 
\end{LEMMA}

\begin{proof}
The functor is the unstraightening (in the sense of $\infty$-categories) of the fiber-wise (over $I$) functor
\[ I \times_{/I} i \to I_{\act} \times_{/I_{\act}} i \]
given by unique factorization. 
This is even a {\em reflective} localization. 
One checks that for a morphism $\alpha: i \to i'$ in $I$ one has a commutative diagram 
\[ \xymatrix{ I \times_{/I} i \ar[r]  \ar[d] & I_{\act} \times_{/I_{\act}} i \ar[d] \\
 I \times_{/I} i' \ar[r] & I_{\act} \times_{/I_{\act}} i' } \]
where the left vertical functor is composition and the right hand side is the functoriality used to define $\tw I$ which also stems from the unique factorization. 
They assemble to a natural transformation of functors $I \to \Cat_{\infty}$.
\end{proof}

\subsection{(Co)fibrations of (co)operads and (op)lax limits}

This section contains a brief discussion of the relation between (co)fibrations of (co)operads and the corresponding notions of 
(op)lax morphisms of their classifying functors (more background, explicit constructions in a particular model of $(\infty,2)$-categories, and proofs can be found in \cite{
HHLN23,HHLN23a, Abe22, Ber24, AGH24, GHL24, GHL20, AF20}).

\begin{PAR}
Let $\mathcal{C}$ be an $(\infty,2)$-category and $I$ a small $\infty$-category. We denote by
\[ \Hom^{\lax}_{\mathcal{C}^I}(F, G)  \]
the $(\infty,1)$-category of lax natural transformations (cf.\@ e.g.\@ \cite{Ber24, Abe22, HHLN23a}). We adopt the convention that for a lax transformation  and
for each morphism $i \to i'$ in $I$ there is a diagram in $C$ of shape
\[ \xymatrix{ F(i) \ar[r] \ar[d] \ar@{}[rd]|\Uparrow  & F(i') \ar[d] \\
 G(i) \ar[r] & G(i')
} \]
(as opposed to the other direction of the 2-morphism). 
For $F \in C^{I^{\op}}$ and $G \in C^{I}$ we also denote by  
\[ \Dinat^{\lax}(F, G)  \]
a similar construction in which for each morphism $i \to i'$ in $I$ there is a diagram in $C$
\[ \xymatrix{ F(i) \ar@{<-}[r] \ar[d] \ar@{}[rd]|\Rightarrow  & F(i') \ar[d] \\
 G(i) \ar[r] & G(i')
} \]
\end{PAR}

We have the notions of (op)lax limit and colimit which are the 2-adjoint to the constant functor. They exist with decoration for a set $C$ of morphisms in $I$, i.e.
\[ \Hom_{\Cat_{\infty}^I}^{\mathrm{(op)lax}, C-\pseudo}(F, \alpha^*G) \cong \Hom_{\Cat_{\infty}}(\mathrm{(op)laxcolim}^{C-\pseudo} F, G)   \]
where $\Hom_{\Cat_{\infty}^I}^{\mathrm{(op)lax}, C-\pseudo} \subseteq \Hom_{\Cat_{\infty}^I}^{\lax}$ means the full subcategory of lax transformation in which the constraint is an isomorphism for all morphisms in $C$ and
\[ \Hom_{\Cat_{\infty}^I}^{\mathrm{(op)lax}, C-\pseudo}(\alpha^*F, G) \cong \Hom_{\Cat_{\infty}}(F, \mathrm{(op)laxlim}^{C-\pseudo} G).  \]
They are  called  {\bf partially (op)lax limits} in \cite{Ber24, AGH24}.

This allows to write also 
\[ \mathrm{(op)laxlim}^{C-\pseudo} G = \Hom_{\Cat_{\infty}^I}^{\mathrm{(op)lax}, C-\pseudo}(\cdot, G). \]
Furthermore, we have
\[ \laxcolim F = \int F \qquad  \oplaxcolim F = \intrev F. \]
the unstraightening construction, and its dual, respectively. These are a cofibration, resp.\@ a fibration over $I$. 
Then 
\begin{gather*}
 \laxlim_I^{C-\pseudo} G = \Hom_{\Cat_{\infty}/I}(I, \laxcolim_I G)^{C-\cocart} \\
   \oplaxlim_I G = \Hom_{\Cat_{\infty}/I}(I, \oplaxcolim_I G)^{C-\cart}  \end{gather*}
Here $C-\mathrm{(co)cart}$ means the full subcategory of functors that map morphisms in $C$ to (co)Cartesian morphisms. 

We have also 
\[ \laxcolim_I^{C-\pseudo} YG =  \laxcolim_I G[\{C-\cart\}^{-1}] \qquad  \oplaxcolim_I^{C-\pseudo} G =  \oplaxcolim_I G [\{C-\cocart\}^{-1}]  \]
(localization at the Cartesian morphism over morphisms in $C$) which will no be needed in the sequel, though.

\begin{LEMMA}\label{LEMMAOPTWIST2}
If $F, G: I \to (\Cat_{\infty}, \times)$ are functors of $\infty$-operads then the association\footnote{Here $\Cat_{\infty}^{\times}$ denotes the underlying $(\infty,2)$-category (of operators) of the $(\infty,2)$-operad $(\Cat_{\infty}, \times)$, i.e.\@ with $(\Cat_{\infty}^{\times})_{[n]} = \Cat_{\infty}^n$. }
\begin{eqnarray*}
\mathrm{tw}^{\op}(I) &\to& \Cat_{\infty} \\
 \alpha: i \to i' &\mapsto& \Hom^{p(\alpha)}_{\Cat_{\infty}^{\times}}(F(i), G(i'))  
 \end{eqnarray*}
where $p: I \to \OOO$ is the structural morphism, 
factors through $\tw I$ and induces an isomorphism of $(\infty, 1)$-categories: 
\begin{gather*}
\Hom_{(\Cat^\times_\infty)^I}^{\lax, \inert-\pseudo}(F, G) \cong
\laxlim_{\mu: i \to i' \in \tw I }^{1-\pseudo, \inert-\pseudo} \Hom^{p(\mu)}_{\Cat_{\infty}^{\times}}(F(i), G(i')).     
\end{gather*}
\end{LEMMA}
\begin{proof}
Recall from Lemma~\ref{LEMMAOPTWIST} that
\[ \Pi: \mathrm{tw}^{\op}(I) \to \tw I  \]
is a localization. Thus, we have to show that the functor maps morphisms $\beta$, which are mapped to isomorphisms under $\Pi$, to isomorphisms. 
Such a $\beta$ induces a diagram
\[ \xymatrix{  i_1 \ar[r]^{\iota_1} \ar[d]_{\beta} &  i' \ar[r]^{\alpha} &  i  \\
i_2 \ar[ru]_{\iota_2}  }\]
where $\iota_1$ and $\iota_2$ are inert and $\alpha$ is active, and which in turn induces a commutative diagram
\[ \xymatrix{  \Hom^{p(\alpha \iota_2)}_{\Cat_{\infty}^{\times}}(F(i_2), G(i)) \ar@{<-}[d]^{\sim} \ar[rr]^{F(\beta) \circ -}  & &  \Hom^{p(\alpha \iota_1)}_{\Cat_{\infty}^{\times}}(F(i_1), G(i)) \ar@{<-}[d]^{\sim} \\
 \Hom^{p(\alpha)}_{\Cat_{\infty}^{\times}}(\iota_{2,\bullet}F(i_2), G(i)) \ar[rr] \ar@{<-}[rd] &  & \Hom^{p(\alpha)}_{\Cat_{\infty}^{\times}}(\iota_{1,\bullet} F(i_1), G(i)) \ar@{<-}[ld] \\
&  \Hom^{p(\alpha)}_{\Cat_{\infty}^{\times}}(F(i'), G(i))} 
\]
Here $\iota_{1, \bullet}$ and $\iota_{2,\bullet}$ are the (2-categorical) push-forward along the inert morphisms $\iota_1$, $\iota_2$.
Since $F$ is a functor of $(\infty, 2)$-operads the diagonal morphisms are isomorphisms as well. 

The last assertion boils down to that statement that this following morphism induces by the localization functor
$\Pi: \mathrm{tw}^{\op} I \to \tw I$ is an equivalence: 
\[ \laxlim_{\alpha: i \to i' \in \mathrm{tw}^{\op}(I) }^{1-\pseudo, \inert-\pseudo} \Hom^{p(\alpha)}_{\Cat_{\infty}^{\times}}(F(i), G(i'))   \cong
\laxlim_{\alpha: i \to i' \in \tw I }^{1-\pseudo, \inert-\pseudo} \Hom^{p(\alpha)}_{\Cat_{\infty}^{\times}}(F(i), G(i'))    \]
From the first part follows that 
\begin{align*} & \laxlim_{\alpha: i \to i' \in \mathrm{tw}^{\op}(I) }^{1-\pseudo, \inert-\pseudo} \Hom^{p(\alpha)}_{\Cat_{\infty}^{\times}}(F(i), G(i'))   \\
  = & \Hom_{\mathrm{tw}^{\op}(I)}(\mathrm{tw}^{\op}(I),\Pi^* \laxcolim_{\tw I} (\Hom^{p(\alpha)}_{\Cat_{\infty}^{\times}}(F(i), G(i')))^{1-\cocart, \inert-\cocart}   \\
 = & \Hom_{\tw I}(\mathrm{tw}^{\op}(I), \laxcolim_{\tw I} (\Hom^{p(\alpha)}_{\Cat_{\infty}^{\times}}(F(i), G(i')))^{1-\cocart, \inert-\cocart}  .
\end{align*}
$\Pi$ is a localization at inert type-1 morphisms. Those are mapped to isomorphisms by functors that map inert morphisms to coCartesian (here iso-) morphisms.
Hence this is the same as: 
\[ = \Hom_{\tw I}(\tw I, \laxcolim_{\tw I} (\Hom^{p(\alpha)}_{\Cat_{\infty}^{\times}}(F(i), G(i')))^{1-\cocart, \inert-\cocart}. \qedhere \]
\end{proof}

\begin{PROP}\label{PROPTRANSLATE}
Let $F: I \to \Cat_{\infty}$ be a functor. 
\begin{enumerate}
\item 
Let $I$ be an $\infty$-category and $F, G: I \to \Cat_{\infty}$ be functors. 
Denote by $\Cat_{\infty}/I$ the slice $(\infty,2)$-category (not lax slice) over $I$. 
Then we have a natural isomorphism of $(\infty,1)$-categories
\[ \Hom_{\Cat_{\infty}/I}(\int F, \int G) \cong \Hom^{\lax}_{\Cat^I_{\infty}}(F, G).  \]
\item Let $I$ be an $\infty$-category and $F: I^{\op} \to \Cat_{\infty}$ and $G: I \to \Cat_{\infty}$ be functors. Then 
\[ \Hom_{\Cat_{\infty}/I}(\intrev F, \int G) \cong \Dinat^{\lax}(F, G).  \]
\item Let $I$ be an $\infty$-category and $F, G: I^{\op} \to \Cat_{\infty}$ be functors.
Then  
\[ \Hom_{\Cat_{\infty}/I}(\intrev F, \intrev G) \cong \Hom^{\oplax}_{\Cat_{\infty}^{I^{\op}}}(F, G).  \]
\item Let $I$ be an $\infty$-category and $F: I \to \Cat_{\infty}$ and $G: I^{\op} \to \Cat_{\infty}$ be functors. Then 
\[ \Hom_{\Cat/I}(\int F, \intrev G) \cong \Dinat^{\oplax}(F, G).  \]
\end{enumerate}
\end{PROP}

\begin{PAR}\label{PARC}
If $C$ is a class of morphisms in $I$ then more generally 
\[ \Hom_{\Cat_{\infty}/I}^{C-\cocart}(\int F, \int G) \cong \Hom^{\lax, C-\pseudo}_{\Cat^I_{\infty}}(F, G)  \]
where the LHS is the full subcategory of 
of those functors which map  coCartesian morphisms lying over morphisms in $C$ to coCartesian morphisms 
and the RHS is the full subcategory of those lax natural transformations where the corresponding 2-morphism is an isomorphism for every morphism in $C$. The same is true dually. 

This can be used to establish analogues of the preceding equivalences for operads: 
\end{PAR}

Let $I$ be an operad and $F: I \to (\Cat_{\infty},\times)$ a functor of $\infty$-operads. We write $\Cat_{\infty}^{\times} \to \OOO$ for the category of operators of the target, i.e.\@ with $(\Cat_{\infty}^{\times})_{[n]} = \Cat_{\infty}^n$.
The composition $I \to \Cat_{\infty}^{\times} \to \Cat_{\infty}$ with $\times$ is denoted by $F^{\times}$.
Then $F$ classifies a cofibration of $\infty$-operads $\int F \to I$ whose underlying $\infty$-category (of operators) is given by $\int F^{\times} = \laxcolim_I F^{\times}$.

\begin{PROP}\label{PROPTRANSLATEOP}
\begin{enumerate}
\item 
Let $I$ be an $\infty$-operad and $F, G: I \to (\Cat_{\infty},\times)$ be functors of $\infty$-operads. 
Then 
\[ \Hom_{\Op_{\infty}/I}(\int F, \int G) \cong \Hom^{\lax,\inert-\pseudo}_{(\Cat_{\infty}^\times)^I}(F, G) \cong \Hom^{\lax,\inert-\pseudo}_{\Cat_{\infty}^I}(F^{\times}, G^{\times})  \]
\item Let $I$ be an $\infty$-cooperad and $F, G: I^{\op} \to (\Cat_{\infty}, \times)$ be functors of $\infty$-operads. 
Then 
\[ \Hom_{\Coop_{\infty}/I}(\intrev F, \intrev G) \cong \Hom^{\oplax,\inert-\pseudo}_{(\Cat_{\infty}^{\times})^{I^{\op}}}(F, G) \cong \Hom^{\oplax,\inert-\pseudo}_{\Cat_{\infty}^{I^{\op}}}(F^{\times}, G^{\times})  \]
\end{enumerate}
\end{PROP}
\begin{proof}Follows immediately from Proposition~\ref{PROPTRANSLATE}, the discussion in \ref{PARC} and basic facts on operads \ref{PAROP}.
\end{proof}
See \cite[Proposition~5.2.]{Ber24} for a similar statement.

\begin{PAR}
Let $I$ be a small $\infty$-category. The 
twisted arrow categories $\twop I$, and $\tw I$, discussed in Section~\ref{SECTTW} may be used to transfer between lax and oplax natural transformations.
This construction is at the heart of the (co)bar constructions discussed in the next chapter. 
For $\tw I$, or more generally for categories of the form ${}^\Xi I$, for $\Xi \in \{\uparrow, \downarrow\}^n$, we write $i-\pseudo$ instead of $C-\pseudo$ (\ref{PARC}) for $C$ being the class of type $i$-morphisms (\ref{DEFTYPEI}).  
\end{PAR}

\begin{figure}
\[ \Hom^{\lax}_{\Cat^I_\infty}: \vcenter{ \xymatrix{ F(i) \ar[r]^{F(\alpha)} \ar[d]_{\mu(i)} \ar@{}[rd]|\Uparrow & F(i') \ar[d]^{\mu(i')} \\
G(i) \ar[r]_{G(\alpha)} & G(i')  } }\]

\[ \Hom^{\oplax}_{\Cat^I_\infty}: \vcenter{ \xymatrix{ F(i) \ar[r]^{F(\alpha)} \ar[d]_{\mu(i)} \ar@{}[rd]|\Downarrow & F(i') \ar[d]^{\mu(i')} \\
G(i) \ar[r]_{G(\alpha)} & G(i')  } }\]

\[\Dinat^{\oplax, 1-\pseudo}:  \vcenter{ \xymatrix{ F(i) \ar[d]_{\mu(i)}  \ar@{}[rd]|{\Rightarrow} & \ar[l]_{=} \ar[d]|{\mu(i')F(\alpha)} F(i) \ar[r]^{F(\alpha)} & F(i') \ar[d]^{\mu(i')}   \\
G(i) \ar[r]_{G(\alpha)} & G(i') &  \ar[l]^= G(i') } }\]
\[\Dinat^{\oplax, 2-\pseudo}:  \vcenter{ \xymatrix{ F(i) \ar[d]_{\mu(i)}  & \ar[l]_{=} \ar[d]|{G(\alpha)\mu(i)} F(i) \ar[r]^{F(\alpha)} & F(i') \ar[d]^{\mu(i')}   \ar@{}[ld]|{\Leftarrow} \\
G(i) \ar[r]_{G(\alpha)} & G(i') &  \ar[l]^= G(i') } }\]

\[  \Hom^{\oplax, 1-\pseudo}_{\Cat^{\tw I}_\infty}:\vcenter{ \xymatrix{ \cdot \ar[r]^= \ar[d]_{\mu(i)} \ar@{}[rd]|{\Leftarrow} & \ar[d]|{\mu(i')F(\alpha)} \cdot   & \ar[l]_= \cdot \ar[d]^{\mu(i')} \\
 \Hom(F(i),G(i)) \ar[r] &  \Hom(F(i),G(i')) &  \ar[l] \Hom(F(i'),G(i'))  } }\]

\[  \Hom^{\oplax, 2-\pseudo}_{\Cat^{\tw I}_\infty}: \vcenter{ \xymatrix{ \cdot \ar[r]^= \ar[d]_{\mu(i)}& \ar[d]|{G(\alpha)\mu(i)} \cdot  & \ar[l]_= \cdot \ar[d]^{\mu(i')} \ar@{}[ld]|{\Rightarrow} \\
 \Hom(F(i),G(i)) \ar[r] &  \Hom(F(i),G(i')) &  \ar[l] \Hom(F(i'),G(i'))  } }\]
 
(squares without depicted 2-morphisms  commute, i.e.\@ have invertible 2-morphisms)
\caption{Illustration of the different categories appearing in Proposition~\ref{PROPTWIST}.}
\end{figure}

\begin{PROP}\label{PROPTWIST}
Let $F, G \in \Cat^I_{\infty}$.  There are canonical isomorphisms:
\begin{align*} 
 \Hom^{\lax}_{\Cat^I_\infty}(F, G) &\cong \Dinat^{\oplax, 2-\pseudo}(\pi_1^*F, \pi_2^*G) \cong \oplaxlim^{2-\pseudo}_{\tw I} \Hom(F(i), G(i'))  \\
 \Hom^{\oplax}_{\Cat^I_\infty}(F, G) &\cong \Dinat^{\oplax, 1-\pseudo}(\pi_1^*F, \pi_2^*G) \cong \oplaxlim^{1-\pseudo}_{\tw I} \Hom(F(i), G(i'))  
\end{align*}
The same holds true with lax and oplax interchanged. 
The terms on the right might be called {\bf (op)lax ends} of $\Hom(F, G)$ as functor on $I^{\op} \times I \to \Cat_{\infty}$.
\end{PROP}

Mutatis mutandis the Proposition holds also for $\infty$-operads: 
\begin{PROP}\label{PROPOPTWIST}
Let $I$ be an $\infty$-operad and $F,G \in (\Cat_{\infty}, \times)^I$.  
\begin{align*}
 \Hom^{\lax, \inert-\pseudo}_{(\Cat_\infty^{\times})^I}(F, G) &\cong  \oplaxlim^{\inert-\pseudo, 2-\pseudo}_{\tw I} \Hom_{\Cat_\infty^{\times}}^{p(\mu)}(F(i), G(i'))  \\
 \Hom^{\oplax,\inert-\pseudo}_{(\Cat_\infty^{\times})^I}(F, G) &\cong \oplaxlim^{\inert-\pseudo, 1-\pseudo}_{\tw I}  \Hom^{p(\mu)}_{\Cat_\infty^{\times}}(F(i), G(i'))  
\end{align*}
Notice that the expression $\Hom_{\Cat_\infty^{\times}}^{p(\mu)}(F(i), G(i'))$ factors through 
$\tw I \to \Cat_\infty$
as shown in Lemma~\ref{LEMMAOPTWIST2}. 
\end{PROP}
\begin{proof}Follows immediately from Proposition~\ref{PROPTWIST}, the discussion in \ref{PARC} and basic facts on operads \ref{PAROP}.
\end{proof}

\begin{BEM}\label{BEMOPTWISTLIM}
In the special case $F= \cdot$ we get, in particular: 
\begin{gather*}
 \oplaxlim_{I}(G)  \cong \oplaxlim^{1-\pseudo}_{\tw I}(\pi_2^*G),   \\
 \oplaxlim_{I}(G^{\op}) \cong (\laxlim_{I}(G))^{\op}  \cong (\oplaxlim^{2-\pseudo}_{\tw I}(\pi_2^*G))^{\op} . 
 \end{gather*}
Translated to the corresponding (co)fibrations this gives a way to relate cofibrations of operads with fibrations of cooperads, which will be at the heart of the (co)bar constructions, see \ref{PARTWIST}.
\end{BEM}

\begin{PAR}
We need a simple fact about compositions of fibrations that is, however, a bit tricky to state:
Recall that the composition $K \to J \to I$ of two fibrations is a fibration. If $J \to I$ is classified by $G: I^{\op} \to \Cat_{\infty}$ and $K \to J$ is classified by 
$F: J^{\op}=(\oplaxcolim_{I^{\op}} G)^{\op} \to \Cat_{\infty}$ then $K \to I$ is classified by a functor $I^{\op} \to \Cat_{\infty}$ which is a left oplax Kan extension along $J^{\op} \to I^{\op}$ of $F$. It maps $i \mapsto \oplaxcolim_{G(i)^{\op}} F|_{G(i)^{\op}}$. We get a corresponding functor between categories of sections 
\begin{equation}\label{eqhomcofib} \Hom_{\Cat_{\infty/I}}(I, K) \to \Hom_{\Cat_{\infty/I}}(I, J). \end{equation}
This has fibers over $\alpha: I \to J$ given by $\Hom_{\Cat_{\infty/I}}(I, \alpha^{*} K) = \oplaxlim_I F \circ \alpha^{\op}$
where $\alpha^{*} K$ is defined by the Cartesian diagram
\[ \xymatrix{ \alpha^* K \ar[d] \ar[r] & K \ar[d] \\
I \ar[r]^{\alpha} & J } \]
actually because the whole $\Hom$-categories commute with fiber products (and not just their groupoids of invertible morphisms). 
We claim: The construction gives a functor
\begin{align*} \Hom_{\Cat_{\infty/I}}(I, J)^{\op} &\to \Cat_{\infty}, \\
 \alpha &\mapsto \oplaxlim_I F \circ \alpha^{\op}, \end{align*}
and the functor (\ref{eqhomcofib}) is a fibration classified by this functor. A rigorous proof is omitted for the moment.  
Translated purely in terms of the classifying functors it means: 
\end{PAR}

\begin{PROP}\label{PROPHOMCOFIB}
Let $G \in \Cat^{I^{\op}}_{\infty}$ be a functor and $F \in \Cat^{(\oplaxcolim_{I^{\op}} G)^{\op}}_{\infty}$ a functor. Then there is a canonical isomorphism: 
\[ \boxed{ \oplaxlim_{i \in I^{\op}}(\oplaxcolim_{x \in G(i)^{\op}} F(i, x))  \cong \oplaxcolim_{(x(i) \in G(i))_i \in (\oplaxlim_{I^{\op}} G)^{\op}} ( \oplaxlim_{i \in I^{\op}} F(i, x(i))). } \]
\end{PROP}
A similar statement holds true with a set $C$ of morphisms in $I$: 
\[  \oplaxlim_{i \in I^{\op}}^{C-\pseudo}(\oplaxcolim_{x \in G(i)^{\op}} F(i, x))  \cong \oplaxcolim_{(x(i) \in G(i))_i \in (\oplaxlim^{C-\pseudo}_{I^{\op}} G)^{\op} }( \oplaxlim_{i \in I^{\op}}^{C-\pseudo} F(i, x(i))).  \]

\begin{BEM}\label{BEMOPLAXLIM}
One can also construct the map (almost) purely by abstract adjunction properties of lax limits and colimits. Because we have for $F: I \to \Cat_{\infty}$:
\[ \oplaxcolim F = \intrev F \]
and
\[ \oplaxlim F = \Hom_{\Cat_{\infty}/I^{\op}}(I^{\op}, \intrev F) \]
we have a canonical morphism
\[ \eval: I^{\op} \times \oplaxlim F \to  \oplaxcolim F. \]
Then consider the diagram
\[ \xymatrix{I \times \oplaxlim_{I^{\op}} G \ar[r]^{\eval} \ar[d]_{\pi_2} \ar[rd]^{\pi_1} & \oplaxcolim_{I^{\op}} G \ar[d]^\pi \\
\oplaxlim_{I^{\op}} G  \ar[rd]^p & I \ar[d]^p \\
& \cdot }\]

The non-horizontal functors are either cofibrations (the map $\pi$) or projections (in particular, fibrations and cofibrations). 
Therefore there exist fiber-wise oplax limits and colimits  (oplax Kan extensions) along these maps which we denote here by $p_!$ and $p_*$. 
They form 2-adjunctions
\[ \Cat_{\infty}^{I, \oplax} \leftrightarrow  \Cat_{\infty}^{J, \oplax}  \]
in the sense that there is an isomorphism of $\infty$-categories (not just groupoids)
\[ \Hom_{\mathcal{C}^{I}}^{\oplax}(F, \alpha^* G) \cong \Hom_{\mathcal{C}^{J}}^{\oplax}(\alpha_! F, G). \]
Therefore, we can calculate with them like with usual Kan extensions. First obviously
\[ \pi_1^* p^*  \cong \pi_2^* p^* \]
therefore there is a mate
\[ p^* p_*  \Rightarrow \pi_{2,*} \pi_1^*  \]
which is an isomorphism because $\pi_{2,*}$ is computed fiber-wise.
Hence we get a morphism
\[   p_! \pi_{2,*}  \Rightarrow  p_* \pi_{1,!}  \]
and finally: 
\[   p_! \pi_{2,*} \eval^*  \Rightarrow  p_* \pi_{1,!}  \eval^* \Rightarrow p_* \pi_{!}     \]
This is the functor in the statement. 
\end{BEM}

\begin{PROP}\label{PROPHOM}
Let $F \in \Cat^I_{\infty}$ be a functor and 
\[ X, Y \in \oplaxlim_I F \] 
then we have a canonical equivalence
\[ \Hom_{\oplaxlim_I F} (X, Y) \iso \lim_{\twop I} \Hom_{F(i)}(X(i), F(\mu)Y(i'))  = \lim_{\twop I} \Hom_{\oplaxcolim_I F}^{\mu}(X(i), Y(i'))  \]
where $\mu: i \to i'$ denotes an object in $\twop I$. 
\end{PROP}

\begin{proof}
This has nothing to do with fibrations and is a simple consequence of the end formula for morphisms in functor categories. 
 Actually, for any functor $\mathcal{C} \to \mathcal{S}$, for $X, Y \in \mathcal{C}^{I}_S$, for $S \in \mathcal{S}$, we have
\[ \Hom_{\mathcal{C}^{I}_S} (X, Y) = \lim_{\twop I} \Hom_{\mathcal{C}, S(\mu)}(X(i), Y(i')) . \]
Indeed, we have
\[ \Hom_{\mathcal{C}^{I}} (X, Y) \cong \lim_{\twop I} \Hom_{\mathcal{C}}(X(i), Y(i')) \]
and, by definition, the Cartesian diagram
\[ \xymatrix{ \Hom_{\mathcal{C}, S(\mu)}(X(i), Y(i')) \ar[r] \ar[d] & \Hom_{\mathcal{C}}(X(i), Y(i')) \ar[d] \\
\{S(\mu)\} \ar[r] & \Hom_{\mathcal{S}}(S(i), S(i')) } \]
Applying $\lim_{\twop I}$, we get a Cartesian diagram
\[ \begin{gathered}[b] \xymatrix{ \lim_{\twop I} \Hom^{S(\mu)}_{\mathcal{C}}(X(i), Y(i')) \ar[r] \ar[d] &  \Hom_{\mathcal{C}^I}(X, Y) \ar[d] \\
 \{ \id_S  \} \ar[r] &  \Hom_{\mathcal{S}^I}(S, S) }\\[-\dp\strutbox] \end{gathered} \qedhere \]
\end{proof}

\begin{KOR}\label{KORTW}
Let $I$ be a small $\infty$-category.
For a functor $G \in \Cat^{I}_{\infty}$  there is an isomorphism
\[ \oplaxlim_I \twop G \cong \twop (\oplaxlim_{\tw I} \pi_2^* G)^{{(1,2),(2,1)-\pseudo}}  \]
where the RHS is the full subcategory of objects $X \to Y \in \twop (\oplaxlim_{\tw I} \pi_2^* G)$ such that $X \in (\oplaxlim_{\tw I}^{2-\pseudo} \pi_2^* G)$ and $Y \in (\oplaxlim_{\tw I}^{1-\pseudo} \pi_2^* G)$. 
Similarly, let $I$ be a small $\infty$-operad. For a functor $G \in (\Cat_{\infty}, \times)^I$, inducing $G^{\times}: I \to \Cat_{\infty}^{\times} \to \Cat_{\infty}$ classifying the categories of operators, there is an isomorphism
\[ \oplaxlim_I^{\inert-\pseudo} \twop G^{\times} \cong \twop (\oplaxlim_{\tw I}^{\inert-\pseudo} \pi_2^* G^{\times})^{{(1,2),(2,1)-\pseudo}}  \]
where $\tw I$ means twisted arrow category in the sense of $\infty$-operads (\ref{DEFTW}), observing that in all other occurrences it is applied to just $\infty$-categories. 
\end{KOR} 
\begin{proof}[Proof (sketch):]
Consider the functor $\mlq \Hom \mrq: (\oplaxcolim_I  G \times G^{\op})^{\op} \to \Gpd_{\infty}$ given by $\Hom$ fiber-wise, where $G^{\op}$ is the {\em composition} of $G$ with $\op$. To see that this is a valid morphism, observe that 
\[ \mlq \Hom \mrq \in \Hom^{\lax}_{\Cat^I_{\infty}}(G^{\op} \times G, \pi_I^* \Gpd_{\infty}) \cong \Hom_{\Cat_{\infty}}(\laxcolim_I G^{\op} \times G, \Gpd_{\infty}).    \]
\[ \cong \Hom_{\Cat_{\infty}}((\oplaxcolim_I G \times G^{\op})^{\op}, \Gpd_{\infty}) \]
Apply Proposition~\ref{PROPHOMCOFIB} with $(I, G, F)$ being $(I^{\op}, G \times G^{\op}, \mlq \Hom \mrq)$:
\[ \oplaxlim_{I}( \underbrace{ \oplaxcolim_{G(i)^{\op} \times G(i)} \mlq \Hom \mrq}_{\twop G(i)} )  \cong \oplaxcolim_{(X,Y) \in (\oplaxlim_I G \times G^{\op})^{\op}} (\oplaxlim_I \Hom(X(i),Y(i))). \]
Using Remark~\ref{BEMOPTWISTLIM} identify
\begin{equation}\label{eqoplaxlim}  \oplaxlim_I G \times \oplaxlim_I  G^{\op}  \cong \oplaxlim_{\tw I}^{1-\pseudo} \pi_2^* G \times  (\oplaxlim_{\tw I}^{2-\pseudo} \pi_2^* G)^{\op}.  \end{equation}
We have to show that there is a commutative diagram
\[ \xymatrix{  (\oplaxlim_I G \times \oplaxlim_I  G^{\op})^{\op}  \ar[r]^-{\sim} \ar[d]^{\oplaxlim_I \Hom } & (\oplaxlim_{\tw I}^{1-\pseudo} \pi_2^* G)^{\op} \times  \oplaxlim_{\tw I}^{2-\pseudo} \pi_2^* G \ar[d]^{\Hom} \\ 
\Hom(I, \Gpd_{\infty})  \ar[r]^{\lim} &  \Gpd_{\infty} } \]
By  Proposition~\ref{PROPHOM}, and using the identification $\twop(\twop I) = \twtw I$ (denoting an object by $\xymatrix{i_1 \ar[r]^{\mu_1} & i_2 \ar[r]^{\mu_2} & i_3 \ar[r]^{\mu_3} & i_4  }$), and denoting by $\widetilde{X}, \widetilde{Y}$ the image of the pair $X, Y$ under (\ref{eqoplaxlim}): 
\[ \Hom_{\oplaxlim_{\tw I}}(\widetilde{X}, \widetilde{Y}) \cong \lim_{\twtw I}\Hom_{G(i_4)}(\widetilde{X}(i_1 \to i_4), G(\mu_3)\widetilde{Y}(i_2 \to i_3))  \]
\[ \cong  \lim_{\twtw I}\Hom_{G(i_4)}(X(i_4), G(\mu_3)G(\mu_2)Y(i_2)) \cong \lim_{\ddd I}\Hom_{G(i_4)}(X(i_4),G(\mu)Y(i_2)) \] 
(using Lemma~\ref{LEMMATWLOC} and numbering  $\ddd I$ with the corresponding subindices of $\twtw I$) and this is
\[ \cong \lim_I \Hom_{G(i)}(X(i), Y(i)).   \]
In the operad-case we get  by Proposition~\ref{PROPHOMCOFIB}
\[ \oplaxlim_I^{\inert-\pseudo}( \underbrace{ \oplaxcolim_{G(i)^{\op} \times G(i)} \mlq \Hom \mrq}_{\twop G(i)} )  \cong \oplaxcolim_{(X,Y) \in (\oplaxlim_I^{\inert-\pseudo} G \times G^{\op})^{\op}} \oplaxlim_I^{\inert-\pseudo} \Hom(X(i),Y(i)). \]
and refining the argument above: 
\[ \oplaxlim_I^{\inert-\pseudo} \tw G \cong \twop (\oplaxlim_{(\mathrm{tw}^{\op} I)^{\op}}^{\inert-\pseudo} \pi_2^* G)^{{(1,2),(2,1)-\pseudo}}.  \]
Then notice that 
\begin{gather*} \oplaxlim_{(\mathrm{tw}^{\op} I)^{\op}}^{\inert-\pseudo} \pi_2^* G = \Hom_{\mathrm{tw}^{\op} I}(\mathrm{tw}^{\op} I, \oplaxcolim \pi_2^* G)^{\inert-\cocart}    \\
 = \Hom_{I}(\mathrm{tw}^{\op} I, \oplaxcolim_I G)^{\inert-\cocart} 
\end{gather*}
and that the morphism that go to isomorphisms under $\mathrm{tw}^{\op} I \to  \tw I$ are precisely the type-1 inert morphisms which go thus to isomorphisms, hence by Lemma~\ref{LEMMAOPTWIST2} this is the same as
\begin{gather*} \cong \Hom_{I}(\twop I, \oplaxcolim_I G)^{\inert-\cocart}  \\
 \cong  \Hom_{\twop I}(\twop I, \oplaxcolim \pi_2^* G)^{\inert-\cocart} \cong  \oplaxlim_{\tw I}^{\inert-\pseudo} \pi_2^* G. \qedhere   
\end{gather*}
\end{proof}

\subsection{Relative (operadic) Kan extensions}

Let $\mathcal{C} \to \mathcal{S}$ and  $\alpha: J \to I$ be functors of $\infty$-(co)operads, the latter small. 
Let $S \in \mathcal{S}^I$. By pre-composition, it gives rise to a functor
\[ \alpha^*: \mathcal{C}^{I}_{S} \to \mathcal{C}^{J}_{\alpha^*S}.  \]
\begin{DEF}\label{DEFKAN}
If a left adjoint  $\alpha_!^{(S)}$ (resp.\@ right adjoint $\alpha_*^{(S)}$) of $\alpha^*$ exists, we call it a {\bf relative left (resp.\@ right) Kan extension}.
\end{DEF}
Note that this comprises, in particular, the functors called {\em operadic Kan extensions} by Lurie \cite[3.1.2]{Lur11}.
Various sufficient criteria for their existence will be discussed in Sections~\ref{SECTKAN1}--\ref{SECTKAN3}. 

\begin{PAR}
We call a relative left Kan extension, if it exists, {\bf fiber-wise}, if for any object $i \in I_{[1]}$, the mate
\[  \colim_{J_i} \iota_i^* \to  i^*\, \alpha_{!}^{(S)}  \]
where $\iota_i: J_i \hookrightarrow J$ is the inclusion, is an isomorphism. Similarly for right Kan extensions. 
Somewhat informally, we say that a relative left or right Kan extension is {\bf point-wise}, if there is a simple formula à la Kan, cf.\@ Proposition~\ref{PROPKAN}
for $i^*\, \alpha_{!}^{(S)}$.
\end{PAR}

 \subsection{Relative Kan extensions --- fiber-wise case}\label{SECTKAN1}

\begin{PROP}\label{PROPDAYADJOINT}
Let $\mathcal{C}, I, J, K$ be ($\infty$-)(co)operads, the latter three small. Let $\mathcal{C} \to I$ be a functor and 
\[ \xymatrix{ J \ar[rr] \ar[rd] & & K \ar[ld] \\
& I } \]
 be a morphism of ($\infty$-)exponential fibrations over $I$.
\begin{enumerate}
\item 
If $\mathcal{C} \to I$ is $L$-admissible (Definition~\ref{DEFLR}) and 
$\alpha$
 is ($\infty$-)coCartesian (Definition~\ref{DEFCART}). Then 
the functor $\alpha^*: D_I(K, \mathcal{C}) \to D_I(J, \mathcal{C})$ has a left adjoint in the $(\infty, 2)$-category $\mathrm{(co)Op}_{\infty/I}$ (resp.\@ in the  $(1, 2)$-category $\mathrm{(co)Op}_{/I}$). 

\item 
If $\mathcal{C} \to I$ is $R$-admissible (Definition~\ref{DEFLR}) and 
$\alpha$ is ($\infty$-)Cartesian (Definition~\ref{DEFCART}). Then 
the functor, $\alpha^*: D_I(K, \mathcal{C}) \to D_I(J, \mathcal{C})$ has a right adjoint in the $(\infty, 2)$-category $\mathrm{(co)Op}_{\infty/I}$ (resp.\@ in the $(1, 2)$-category $\mathrm{(co)Op}_{/I}$). 
\end{enumerate}
\end{PROP}
\begin{proof}
The functor $\alpha$ can be seen as a morphism in $\Hom^{\oplax, \inert-\pseudo}_{\Cat^{\PF, \times}}(\Xi_{K}, \Xi_{J})$ or (via passing to the mate) as $\Hom^{\oplax, \inert-\pseudo}_{\Cat^{\PF, \times}}(\Xi_{J}, \Xi_{K})$.
The functor in question is given equivalently by applying $L$ to the first or $R$ to the second (cf.\@ Proposition~\ref{PROPDAYLAX}, assuming that $\mathcal{C} \to I$ is $L$-, resp.\@ $R$-admissible).
If $\alpha$ is ($\infty$\nobreakdash-)coCartesian we can also apply $R_{\mathcal{C}}$ to the first, being a natural transformation (not only oplax) in this case. This yields a right adjoint. If $\alpha$ is ($\infty$\nobreakdash-)Cartesian, we can apply $L_{\mathcal{C}}$ to the second, being a natural transformation (not only lax) in this case. This yields a left adjoint. 
\end{proof}

\begin{KOR}\label{KOREXISTFIBERWISERELKANEXT}
\begin{enumerate}
\item
If $\mathcal{C} \to \mathcal{S}$ is $L$-admissible (Definition~\ref{DEFLR})  and $\alpha: J \to K$ is a ($\infty$-)Cartesian morphism of ($\infty$-)exponential fibrations over $I$. Then 
for each $S \in \mathcal{S}^I$, the functor
\[ \alpha^*: \mathcal{C}^K_{p_K^*S} \to \mathcal{C}^J_{p_J^* S} \]
has a left adjoint $\alpha_!^{(p_J^*S)}$, i.e.\@ a relative left Kan extension exists. It is computed fiber-wise over $I$, i.e.\@ the mate
\[  \colim_{J_i} \iota_i^* \to  i^* \alpha_{!}^{(S)}  \]
where $\iota_i: J_i \to J$ denotes the inclusion of the fiber, is an isomorphism. 

\item 
If $\mathcal{C} \to \mathcal{S}$ is $R$-admissible (Definition~\ref{DEFLR})  and $\alpha: J \to K$  is a ($\infty$\nobreakdash-)coCartesian morphism of ($\infty$-)exponential fibrations over $I$. Then 
for each $S \in \mathcal{S}^I$, the functor
\[ \alpha^*: \mathcal{C}^K_{p_K^*S} \to \mathcal{C}^J_{p_J^* S} \]
has a right adjoint $\alpha_*^{(p_J^*S)}$, i.e.\@ a relative (operadic) right Kan extension exists.  It is computed fiber-wise over $I$, i.e.\@ the mate
\[   i^* \alpha_{*}^{(S)} \to \lim_{J_i} \iota_i^*   \]
where $\iota_i: J_i \to J$ denotes the inclusion of the fiber, is an isomorphism.   
\end{enumerate}
\end{KOR}
\begin{proof}
We can pullback $\mathcal{C} \to \mathcal{S}$ along $S: I \to \mathcal{S}$ which preserves $L$/$R$-admissibility, and assume w.l.o.g. that $\mathcal{S}=I$ and $S=\id$.
Then we have $\mathcal{C}^K_{p_K} = \Hom_{\Cat_{/I}}(I, D_I(K, \mathcal{C}))$ and $\mathcal{C}^J_{p_J} = \Hom_{\Cat_{/I}}(I, D_I(J, \mathcal{C}))$ by the universal property of the Day convolution. Hence
the adjunction from Proposition~\ref{PROPDAYADJOINT} gives the required Kan extension adjunction. 
\end{proof}

\begin{PAR}
Let $p: J \to I$ be an ($\infty$-)exponential fibration. 
Then the functor is, in particular, a morphism of ($\infty$-)exponential fibrations. 
It is ($\infty$-)Cartesian (Definition~\ref{DEFCART}), if 
\[ \alpha^\bullet\, p_{J_i} \cong p_{I_i} \] 
and ($\infty$-)coCartesian (Definition~\ref{DEFCART}), if
\[ {}^t\!p_{J_i}\, \alpha^\bullet  \cong \,{}^t\!p_{I_i}. \] 
Obviously a fibration is always $\infty$-Cartesian and a cofibration is $\infty$-coCartesian. 
A fibration is ($\infty$-)coCartesian if and only if the pull-back functors are ($\infty$\nobreakdash-)cofinal and
a cofibration is ($\infty$-)Cartesian, if and only if the push-forward functors are ($\infty$-)final. 
This shows the following: 
\end{PAR}

\begin{BEM}\label{REMKANEXT}
Let $\mathcal{C}, \mathcal{S}, I, J$ be ($\infty$-)(co)operads (all of the same type), the latter two small. 

Let $\mathcal{C} \to \mathcal{S}$ and $J \to I$ be (co)fibrations of ($\infty$-)(co)operads (not necessarily of the same fibration type), we may summarize the assumptions of Corollary~\ref{KOREXISTFIBERWISERELKANEXT} as follows:

\begin{enumerate}
\item
Relative left Kan extensions (that are computed fiber-wise) exist along $J \to I$, if $\mathcal{C}$ has cocomplete fibers, and
\begin{center}
\begin{tabular}{ c | c c }
 \hbox{\diagbox{$\mathcal{C} \to \mathcal{S}$}{$J \to I$}} & fibration & cofibration \\ 
 \hline
 \multirow{2}{*}{fibration} & pull-backs for $J\to I$ & \multirow{2}{*}{always} \\
 &  are ($\infty$-)cofinal  \\  
 \multirow{2}{*}{cofibration} & \multirow{2}{*}{both properties hold} & push-forwards for $\mathcal{C} \to \mathcal{S}$  \\
 & & are ($\infty$-)cocontinuous    
\end{tabular}
\end{center}

\item
Relative right Kan extensions (that are computed fiber-wise) exist along $J \to I$, if $\mathcal{C}$ has complete fibers, and
\begin{center}
\begin{tabular}{ c | c c }
 \hbox{\diagbox{$\mathcal{C} \to \mathcal{S}$}{$J \to I$}} & fibration & cofibration \\ 
 \hline
 \multirow{2}{*}{fibration} & pull-backs for $\mathcal{C} \to \mathcal{S}$  &  \multirow{2}{*}{both properties hold} \\
 &  are ($\infty$-)continuous   &   \\  
 \multirow{2}{*}{cofibration} & \multirow{2}{*}{always} & push-forwards for $J\to I$ \\
& &  are ($\infty$-)final  
\end{tabular}
\end{center}
\end{enumerate}

Note that it does not matter whether the objects are operads, or cooperads, respectively, apart from the fact that {\em fibrations of operads} as well as {\em cofibrations of cooperads} are not
as commonly considered as the other cases. 
\end{BEM}

 \subsection{Relative Kan extensions --- point-wise case}\label{SECTKAN2}

 The calculus of left relative Kan extensions is particularly rich for $L$-admissible cofibrations of $\infty$-operads  $\mathcal{C} \to \mathcal{S}$. Recall that  $L$-admissible here means
that the fibers are cocomplete and that the push-forward functors commute with colimits (argument-wise). 
Similarly, of course, the same holds for right relative Kan extensions for $R$-admissible fibrations of  $\infty$-cooperads.
In particular, they always exist, and a point-wise formula holds true (analogous to Kan's formula for usual Kan extensions). See Proposition~\ref{PROPKAN} below.

\begin{DEF}\label{DEFGENMOR}
Let $I$ and $J$ be  $\infty$-operads. A {\bf generalized morphism}  $I \to J$ is a 2-commutative diagram
\[  \xymatrix{ I \ar[rr]^{\alpha} \ar[rd]  &\ar@{}[d]|-\Downarrow & J \ar[ld] \\
  &  \OOO }  \]
  such that $\alpha$ maps inert morphisms to inert morphisms and such that the 2-morphism consists of active morphisms. Denote the corresponding $(\infty,2)$-category by
  $\Op_{\infty}^{g}$.
\end{DEF}

\begin{DEF}\label{DEFDIAS}
Let $\mathcal{S}$ be an  $\infty$-operad.
We define the $(\infty,2)$-category $\Dia(\mathcal{S})$ as the full subcategory of the lax slice category $\Op_{\infty}^g // \mathcal{S}$ in which
the objects are honest morphisms of operads with {\em small} source.  Similarly, we define $\Dia^{\op}(\mathcal{S})$ as the oplax slice category.
\end{DEF}

A morphism in  $\Dia(\mathcal{S})$ is thus represented by a diagram
\[  \xymatrix{ M \ar[rr]^{\alpha} \ar[rd]  &\ar@{}[d]|-{\Downarrow^\mu} & N \ar[ld] \\
  &  \mathcal{S} }  \]
in which $\mu$ consists point-wise of active morphisms. If $\mu$ is an equivalence, we say that the morphism is {\bf of diagram type}. The corresponding $\alpha$ is then necessarily 
an honest morphism of operads. 

\begin{PAR}\label{PARFIBDER}
Let $\mathcal{C} \to \mathcal{S}$ be a cofibration of $\infty$-operads. It gives rise to a functor of $(\infty,2)$-categories
\[ y(\mathcal{C}): \Dia(\mathcal{S}) \to \Cat_{\infty} \]
mapping $(I, S)$ to $\mathcal{C}^{M}_S$, the fiber of $\mathcal{C}^I \to \mathcal{S}^I$ over $S$. 
A morphism of diagram type $\alpha: I \to J$ is mapped to the functor `composition with $\alpha$':
\[ \alpha^*: \mathcal{C}^J \to \mathcal{C}^I. \]
A morphism with components $\alpha = \id_I, \mu: X \to Y$ is mapped to the corresponding push-forward functor 
\[ \mu_\bullet: \mathcal{C}^I_X \to \mathcal{C}^I_Y. \]
\end{PAR}

 \begin{DEF}[(Co)operadic comma category] \label{DEFCOMMA}
 For two morphisms of $\infty$-(co)operads $J \to I$ and $K \to I$ we define a {\bf (co)operadic comma category} together with a 2-commutative diagram 
 in the 2-category $\Op_{\infty}^{g}$:
 \[ \xymatrix{ I \times_{/J} K \ar[d]_{\pi_2} \ar@{}[rd]|{\Swarrow} \ar[r]^-{\pi_1} & I \ar[d]^{\alpha} \\
 K \ar[r]_{\beta} & J }\]
 where the morphism $\pi_1: I \times_{/J} K \to I$ is only a generalized morphism of operads in the sense of Definition~\ref{DEFGENMOR}.
 
 For any $S \in \mathcal{S}^J$ this yields a diagram in $\Dia(\mathcal{S})$ of the form
 \[ \xymatrix{ (I \times_{/J} K, \pi_2^* \beta^* S) \ar@{}[rd]|{\Swarrow^{\mu}} \ar[r]^-{\widetilde{\pi_1}} \ar[d]_{\pi_2} & (I, \alpha^*S) \ar[d] \\
 (K, \beta^*S) \ar[r]^{\beta} & (J, S) }\]
 
We can define it as the limit of the following diagram of (co)operads and generalized morphisms:
\[ \xymatrix{ & & I \ar[d] \\
& {}^{\downarrow\downarrow} J \ar[r]^{\pi_1} \ar[d]_{\pi_2} & J \\
K \ar[r] & J &  }\]
where ${}^{\downarrow\downarrow} J$ has been defined in Definition~\ref{DEFTW}.
One can also mimic the definition of ${}^{\downarrow\downarrow} J$ and define a functor of operads
\begin{align*} K &\to (\Cat_{\infty}, \times) \\
 k &\mapsto I_{\act} \times_{/J_{\act}} \beta(k) \qquad \text{usual comma category} 
\end{align*}
using factorization into inert and active. $I \times_{/J} K  \to K$ is then the associated cofibration (unstraightening). 
 \end{DEF}

 \begin{PROP}[relative (operadic) Kan formula]\label{PROPKAN}
 If $\mathcal{C} \to \mathcal{S}$ is an $L$-admissible cofibration of $\infty$-operads (or just $\infty$-categories) then relative left Kan extensions exist along any functor $\alpha: I \to J$ of small $\infty$-operads whatsoever.
 Moreover, there is an isomorphism
 \begin{equation}\label{eqrelkan} j^* \alpha^{(S)}_! \cong \colim_{I \times_{/J} j} S(\mu)_{\bullet} \pi_1^*  \end{equation}
 for objects $k \in K$ and $S \in \mathcal{S}^K$.
 
 There is a similar dual statement for $R$-admissible fibrations of $\infty$-cooperads.
 \end{PROP}
 \begin{proof}
 We have an adjunction in the $(\infty,2)$-category  $\Dia(\mathcal{S})$ 
 \[ \xymatrix{  (I \times_{/J} J, \pi_2^* S)  \ar@<3pt>[rr]^-{\widetilde{\pi_1}=(\pi_1, S(\mu))} & &  \ar@<3pt>[ll]^-{(\iota, \id)}  (I, \alpha^* S) } \]
 and  an isomorphism
 \[ \alpha_!^{(S)} \cong \pi_{2,!}^{(S)}\, \iota_!^{(\pi_2^*S)}   \]
 in the strong sense that the existence of the RHS adjoints implies the existence of the LHS. 
 By the adjunction in $\Dia(\mathcal{S})$, we have
 \[ \iota_!^{(\pr_2^*S)} \cong \widetilde{\pi_1}^* = S(\mu)_{\bullet}\, \pi_1^*   \]
 giving the formula:
 \[ \alpha_!^{(S)} \cong \pi_{2,!}^{(S)}\, S(\mu)_{\bullet}\, \pi_1^*   \]
 $\pi_{2}$ is a cofibration of cooperads and thus a relative Kan extension exists and is computed fiber-wise by Proposition~\ref{KOREXISTFIBERWISERELKANEXT}.
 The formula (\ref{eqrelkan}) follows.  
 \end{proof}

\begin{BEISPIEL}[Free algebra]
Consider the morphism $\cdot \to \OOO$ where $\cdot$ is the final category. Let $(\mathcal{C}, \otimes)$ be a monoidal
$\infty$-category such that $\otimes$ commutes with colimits (in each variable separately). Then the left relative Kan extension
\[  \mathcal{C} \to (\mathcal{C}, \otimes)^{\OOO}  = \Alg(\mathcal{C}, \otimes) \]
exists. Notice that this is (by definition) a free algebra functor (left adjoint to the forgetful functor). Examining the proof, is suffices to see that $\mathcal{C}$ has countable coproducts and that 
$\otimes$ commutes with them (in each variable separately). Formula (\ref{eqrelkan}) is the usual formula for the free algebra. 
For, observe that 
\[ \cdot \ \times_{/ \OOO} \OOO \cong (\N_0,+) \] 
(as discrete operad). If $\otimes$ does not commute with countable coproducts then the construction of a left Kan extension becomes more difficult but not always impossible (see Section~\ref{SECTKAN3} and Section~\ref{SECTCOFREECOALG} for a dual example). 
\end{BEISPIEL}

 \subsection{Relative Kan extensions --- general case}\label{SECTKAN3}

The last section settled the case of left Kan extensions along arbitrary functors $I \to J$ of $\infty$-operads w.r.t.\@ an $L$-admissible cofibration  $\mathcal{C} \to \mathcal{S}$ of $\infty$-operads (say).
The $L$-admissibility of course can be weakened to the existence of {\em certain} colimits (depending on $I \to J$) and the commutation of the push-forwards in $\mathcal{C} \to \mathcal{S}$ with those. 
Often times, however,  the push-forward in $\mathcal{C}$ {\em does  not commute with these colimits}! 

For example, dually, the functor ``cofree coalgebra'' is a special case of a right relative Kan extension. Its construction by the general machinery accordingly assumes that $\otimes$  commutes with countable products. This already fails in $(\Ab, \otimes)^{\vee}$.

However, assuming (w.l.o.g. --- by using the yoga of Section~\ref{SECTKAN2}) that $I \to J$ is a cofibration, and using the equivalence of \ref{PARTWIST}, we may pass to the corresponding {\em fibration} $\mathcal{C}^{\vee} \to \mathcal{S}^{\op}$:
\[ \xymatrix{ \mathcal{C}^I_{\alpha^* S} \ar[r]^-{\sim} \ar@{<-}[d]_{\alpha^*} & (\mathcal{C}^{\vee})^{\twop J \times_{J} I, 2-\cart}_{\pi_2^* \alpha^* S^{\op}} \ar@{<-}[d]^{(\alpha')^*} \\
 \mathcal{C}^J_{S} \ar[r]^-{\sim} & (\mathcal{C}^{\vee})^{\twop J, 2-\cart}_{\pi_2^* S^{\op}} }  \]
 where $\twop J \times_{J} I \to \twop J$ is a cofibration (sic) of cooperads constructed similarly to the operadic comma category (cf.\@ Definition~\ref{DEFCOMMA}).
 
 Assume, that the fibers of $\mathcal{C} \to \mathcal{S}$ have all relevant colimits (but no commutation!). 
 Then, for fibrations, there is no further obstacle to construct left Kan extensions, hence a left adjoint $(\alpha')_!^{(\pi_2^* S^{\op})}$ exists (Corollary~\ref{KOREXISTFIBERWISERELKANEXT} and especially Remark~\ref{REMKANEXT}). 
 However, it is only a functor
 \[ (\mathcal{C}^{\vee})^{\twop J \times_{J} I, 2-\cart}_{\pi_2^* \alpha^* S^{\op}} \to (\mathcal{C}^{\vee})^{\twop J}_{\pi_2^* S^{\op}} .  \] 
 It will land in the full subcategory $(\mathcal{C}^{\vee})^{\twop L, 2-\cart}_{\pi_2^* S^{\op}}$ of 2-Cartesian objects precisely if the commutativity of the push-forward in $\mathcal{C}$ with (the relevant) 
colimits is satisfied. Hence this seems like circular reasoning. However, we gained some additional information:

\begin{PROP}
Let $\mathcal{C} \to \mathcal{S}$ be a cofibration of $\infty$-operads with cocomplete\footnote{or such that at least the relevant colimits exits...} fibers (not necessariy $L$-admissible).
Then relative left Kan extensions along the cofibration $I \to J$ exists, if the
fully faithful inclusion
\[ (\mathcal{C}^{\vee})^{\twop J, 2-\cart}_{\pi_2^* S^{\op}}  \hookrightarrow (\mathcal{C}^{\vee})^{\twop J}_{\pi_2^* S^{\op}} \]
has a (partial) left adjoint (a Cartesian projector) defined on the essential image of $(\alpha')_!^{(\pi_2^* S^{\op})}$.
\end{PROP}
As an illustration we will construct the cofree coalgebra using the (dual) classical cobar construction in section~\ref{SECTCOFREECOALG}.

\subsection{Pairings}

\begin{DEF}[cf.\@ {\cite[Definition~5.2.1.5]{Lur11}}]\label{DEFPAIRING}
A fibration of $\infty$-categories
\[ \mathcal{P} \to \mathcal{C} \times \mathcal{D}^{\op} \]
with groupoid fibers, or equivalently, a functor
\[ \mathcal{C}^{\op} \times \mathcal{D} \to \Gpd_{\infty} \]
is called a {\bf pairing} of $\infty$-categories. An object $M \in \mathcal{P}$ over $(C, D)$ is called {\bf left universal}, resp.\@ {\bf right universal}, if  it
is a final object in the fiber over $C$, resp.\@ $D$. The pairing is called {\bf left (resp.\@ right)  representable} if for each $C$ there exists a left universal object over $C$ (resp.\@ for each $D$ there exists a right universal object over $D$).
\end{DEF}
The distinguished example is the twisted arrow category $\twop \mathcal{C}  \to  \mathcal{C} \times \mathcal{C}^{\op}$. Here $\id_C$ is left universal over $C \in \mathcal{C}$ and right universal over $C \in \mathcal{C}^{\op}$. A pairing which is left (resp.\@ right) representable gives rise to a functor
\[ F: \mathcal{C} \to \mathcal{D}, \]
respectively 
\[ G: \mathcal{D} \to \mathcal{C}, \]
in such a way that  the pairing is classified by the functor
\[ C, D \mapsto \Hom_{\mathcal{D}}(FC, D)  \quad \text{resp.} \quad  C, D \mapsto \Hom_{\mathcal{C}}(C, GD).   \]
In particular, if $\mathcal{P}$ is left and right representable then $F$ is left adjoint to $G$. The distinguished example corresponds to $F$ and $G$ being both the identity
$\mathcal{C} \to \mathcal{C}$. This characterizes the  distinguished example up to equivalence: 
\begin{LEMMA}[{\cite[Corollary 5.2.1.22]{Lur11}}]T.f.a.e.\@ for a left and right representable pairing:
\begin{enumerate}
\item The functors $F$ and $G$ are equivalences;
\item The pairing is isomorphic to the distinguished one  given by $\twop \mathcal{C}  \to  \mathcal{C} \times \mathcal{C}^{\op}$; 
\item An object of $\mathcal{P}$ is left universal if and only if it is right universal.  
\end{enumerate}
\end{LEMMA}
If these conditions are satisfied  the pairing is called {\bf perfect}.

\section{Classical and Lurie's (co)bar} \label{CHAPTERABSTRACT}

This chapter, which forms the abstract heart of these lectures, discusses the definitions of the derived and classical bar and cobar constructions.
It contains an existence proof of classical cobar in full generality, and an existence proof of derived bar and cobar for monoidal $\infty$-categories. 

\subsection{The general classical and derived (co)bar}\label{SECTCOBAR}

In this section, we define the notion of {\em classical} bar and cobar constructions with cobar left adjoint (mainly for 1-(co)operads, but not necessarily) and the
{\em derived} bar and cobar constructions with cobar right adjoint. The latter are a generalization of Lurie's definition. 
We consider an arbitrary cofibration of $\infty$-operads 
 $\mathcal{C} \to \mathcal{S}$, classified by a functor of $\infty$-operads $\Xi: \mathcal{S} \to (\Cat_\infty, \times)$. 
 Denote by $\mathcal{C}^{\downarrow \uparrow}$ the cofibration given by composing $\Xi$ with the functor
$I \mapsto \twop I$ (twisted arrow category, cf.\@ Section~\ref{SECTTW}). 

Denote by $\mathcal{C}^{\vee} \to \mathcal{S}^{\op}$ the {\em fibration of cooperads} which is classified by the same functor $\Xi$. Then $(\mathcal{C}^\vee)^{\op} \to \mathcal{S}$ is the cofibration classified by composing $\Xi$ with the functor
$I \mapsto  I^{\op}$.

We thus get a diagram of cofibrations of operads over $\mathcal{S}$
\begin{equation}\label{eqcobar0} \vcenter{ \xymatrix{
& \mathcal{C}^{\downarrow \uparrow} \ar[ld]_{\rho_1} \ar[rd]^{\rho_2} \\
\mathcal{C} \ar[rd] & & (\mathcal{C}^\vee)^{\op} \ar[ld] \\
& \mathcal{S}
}} \end{equation}

Notice that $\mathcal{C}^{\downarrow \uparrow }$ is {\em not} the twisted arrow category $\twop \mathcal{C}$ of the operad $\mathcal{C}$ in the sense of Section~\ref{SECTTW} unless $\mathcal{S}$ is the terminal $\infty$-category. 

\begin{PAR}\label{PARCOBARSETTING}
For each small $\infty$-operad $I$ and $S \in \mathcal{S}^I$ this gives rise to a diagram
\begin{equation}\label{EQBARCOBAR} \vcenter{ \xymatrix{
& (\mathcal{C}^{\downarrow \uparrow })^I_S \ar[ld]_{\rho_1} \ar[rd]^{\rho_2} \\
\mathcal{C}^I_{S}  & & ((\mathcal{C}^\vee)^{I^{\op}}_{S^{\op}})^{\op}  \\
}}\end{equation}
The derived bar and cobar adjunction can be derived from the fact that that this
is a pairing of $\infty$-categories which is left and right representable. We will, however, proceed to give a more explicit definition of the
bar and cobar functors and then prove that (\ref{EQBARCOBAR}) is isomorphic to the left- and right representable pairing defined by their adjunction. 

In the special case $\mathcal{S} = I = \OOO$ and $S=\id$, the diagram (\ref{EQBARCOBAR}) is the one of Lurie \cite[Theorem 5.2.2.17]{Lur11}:
\begin{equation*}
\vcenter{ \xymatrix{
& \Alg(\twop \mathcal{C}) \ar[ld] \ar[rd] \\
\Alg(\mathcal{C})  & & \Coalg(\mathcal{C})^{\op} \\
}}\end{equation*}
because the cofibrations $\mathcal{C} \to \OOO$, $(\mathcal{C}^{\vee})^{\op} \to \OOO$, and $\mathcal{C}^{\downarrow \uparrow} \to \OOO$, are just monoidal $\infty$-categories in this case and their associated $\infty$-categories (forgetting the monoidal structure) are just $\mathcal{C}_{[1]}, (\mathcal{C}_{[1]})^{\op}$ and $\twop(\mathcal{C}_{[1]})$, respectively. The monoidal structure are, in all cases, the obvious induced ones, e.g.\@ for $\twop (\mathcal{C}_{[1]})$: $(X_1 \to Y_1) \otimes (X_2 \to Y_2) = (X_1 \otimes X_2 \to Y_1 \otimes Y_2)$. 
\end{PAR}

\begin{PAR}\label{PARCOBARSETTING2}
For $S \in \mathcal{S}^I$, there are two natural diagrams
\begin{equation}\label{EQBAR} 
\vcenter{ \xymatrix@C=1pc{
& (\mathcal{C}^{\vee})^{\twop I}_{\pi_2^* S^{\op}} \ar@{<-}[ld]_{\widetilde{\pi_1^*}} \ar@{<-}[rd]^{\pi_2^*}\\
\mathcal{C}^I_{S}  & & ((\mathcal{C}^\vee)^{I^{\op}}_{S^{\op}})^{\op}  \\
}}
 \quad 
 \vcenter{\xymatrix@C=1pc{
& \mathcal{C}^{\tw I}_{\Pi_2^* S} \ar@{<-}[ld]_{\Pi_2^*} \ar@{<-}[rd]^{\widetilde{\Pi_1^*}} \\
\mathcal{C}^I_{S}  & & ((\mathcal{C}^\vee)^{I^{\op}}_{S^{\op}})^{\op}   \\
}  }
\end{equation}
where $\Pi_2^*, \pi_2^*$ are of the form discussed in (\ref{PARFIBDER}) (precomposition) but $\widetilde{\Pi_1^*}$ and $\widetilde{\pi_1^*}$ are slightly twisted variants (see Definition~\ref{DEFPISTAR} below). 
See Section~\ref{SECTTW} for the definition of the twisted arrow (co)operads $\tw I$ and $\twop I$ for operads $I$. 
The diagrams (\ref{EQBAR}) are linked to the preceding (\ref{EQBARCOBAR}) by the following:
\end{PAR}

\begin{PROP}\label{PROPCOBAR}
The pairing (\ref{EQBARCOBAR}) is isomorphic to the pairing defined by the two (isomorphic) groupoid valued functors
\begin{align*}
(\mathcal{C}^I_{S})^{\op} \times (\mathcal{C}^\vee)^{I^{\op}}_{S^{\op}} &\to \Gpd_{\infty} \\
 (C, D) &\mapsto    \Hom_{(\mathcal{C}^\vee)^{\twop I}_{\pi_2^* S^{\op}}}(\widetilde{\pi_1^*}C, \pi_2^*D) \cong \Hom_{\mathcal{C}^{\tw I}_{\pi_2^* S}}(\Pi_2^*C, \widetilde{\Pi_1^*} D).
\end{align*}
\end{PROP}
Before giving the proof, we have to discuss the precise definition of the functors $\widetilde{\pi_1^*}$ and $\widetilde{\Pi_1^*}$.

\begin{PAR}\label{PARTWIST}
To define the functors $\widetilde{\pi_1^*}$ and $\widetilde{\Pi_1^*}$ and to prove Proposition~\ref{PROPCOBAR} and Lemma~\ref{LEMMACOBAR}, first observe the following. 
Let $\mathcal{C} \to \mathcal{S}$ be a cofibration of $\infty$-operads which is classified by a functor of $\infty$-operads $\Xi: \mathcal{S} \to (\Cat_\infty, \times)$. We can associate with it a fibration of cooperads $\mathcal{C}^\vee \to  \mathcal{S}^{\op}$ which is classified by the same functor. It may be described more directly by the following procedure: 
Let $I$ be a small operad and consider the pull-back: 
\[ \xymatrix{ 
 \mathcal{C}^{\tw I, 2-\cocart}_{\pi_2^*\mathcal{S}} \ar[d] \ar[r] & \mathcal{C}^{\tw I, 2-\cocart} \ar[d] \\
 \mathcal{S}^{I} \ar[r]^{\pi_2^*} & \mathcal{S}^{\tw I}  \\
  } \]
  
Here $2-\cocart$ denotes the full subcategory of those functors mapping type-2 morphisms (Definition~\ref{DEFTYPEI}) to coCartesian ones. Then we have for the fibers 
\[ (\mathcal{C}^\vee)^{I^{\op}}_{S^{\op}} \cong \mathcal{C}^{\tw I, 2-\cocart}_{\pi_2^* S} \] 
   the identification being given by the composition:
  \[ \xymatrix{ (\mathcal{C}^\vee)_{S^{\op}}^{ I^{\op}} \ar[d]^{\sim} &  \mathcal{C}_{\pi_2^* S}^{ \tw  I, 2-\cocart} \ar[d]^{\sim} \\
   \Hom^{\oplax, \inert-\pseudo}_{(\Cat^{\times}_{\infty})^{I}}(\cdot, \Xi) \ar[r]^-{\sim}  & \Hom^{\lax,\inert-\pseudo,2-\pseudo}_{(\Cat_{\infty}^{\times})^{\tw I}}(\cdot, \pi_2^*\Xi) 
   }  \]
where the bottom equivalence is Proposition~\ref{PROPOPTWIST}.
\end{PAR}

\begin{DEF}\label{DEFPISTAR} \label{DEFCOBARCLASS} 
Using the discussion in \ref{PARTWIST}, the functor $\barconst:=\widetilde{\pi^*_1}$
\[ \xymatrix{  \mathcal{C}_{S}^{I} \ar[r]^-{\sim} &  (\mathcal{C}^\vee)_{\pi_2^* S^{\op}}^{ \twop I, 2-\cart} \ar@{^{(}->}[r] &  (\mathcal{C}^\vee)_{\pi_2^* S^{\op}}^{ \twop  I}  }\]
is called the {\bf classical bar construction}, and a left adjoint (if it exists) will be called the {\bf classical cobar construction}.
Dually, the functor $\barconst^{\vee}:=\widetilde{\Pi^*_1}$ is the composition
\[ \xymatrix{ ( \mathcal{C}^\vee)_{S^{\op}}^{ I^{\op}} \ar[r]^-{\sim} &  \mathcal{C}_{\pi_2^* S}^{ \tw I, 2-\cocart} \ar@{^{(}->}[r] &  \mathcal{C}_{\pi_2^* S}^{ \tw I}  }\]
is called the {\bf dual classical bar construction}, and a right adjoint (if it exists) will be called the {\bf dual classical cobar construction}.
\end{DEF} 

The names are justified by the discussion in Chapter~\ref{CHAPTERBASICEXAMPLES} where we will see that these functors are closely related (up to composition with total decalage, and its right adjoint, respectively) to the classical bar and cobar constructions of 
Kan for simplicial groups/sets and Eilenberg-MacLane and Adams for dg-(co)algebras. 

It would be equally reasonable to discuss the functor $\widetilde{\Pi_2^*}$ (dual bar construction) and its (potential) {\em right} adjoint, but this was rarely done classically. As an illustation, however, section~\ref{SECTCOFREECOALG} uses this dual cobar construction to construct the cofree coalgebra. 
Any statement concerning them can abstractly be obtained by 
replacing $\mathcal{C} \to \mathcal{S}$ with $(\mathcal{C}^{\vee})^{\op} \to \mathcal{S}$. This interchanges Lurie's bar and cobar constructions but {\em not} the classicial bar and cobar constructions!

\begin{KOR}\label{KORCOBAR}
The following are equivalent: 
\begin{enumerate}
\item The pairing (\ref{EQBARCOBAR}) is left representable;
\item $\widetilde{\Pi_1^*}$ has a (partial) left adjoint defined on the essential image of $\Pi_2^*$;
\item $\pi_2^*$ has a (partial) left adjoint defined on the essential image of $\widetilde{\pi_1^*}$.
\end{enumerate}
and dually:
\begin{enumerate}
\item The pairing (\ref{EQBARCOBAR}) is right representable;
\item $\widetilde{\pi_1^*}$ has a (partial)  right adjoint defined on the essential image of $\pi_2^*$;
\item $\Pi_2^*$ has a (partial)  right adjoint  defined on the essential image of $\widetilde{\Pi_1^*}$.
\end{enumerate}
\end{KOR}

 \begin{DEF}\label{DEFCOBARLURIE}
We call the composition 
\[ \barlurie := \pi_{2,!}^{(S^{\op})}\ \widetilde{\pi_1^*}  \]
the {\bf derived (Lurie) bar construction} if $\pi_{2,!}^{(S^{\op})}$ (relative left Kan extension, cf.\@ \ref{DEFKAN}) exists on the image of $\widetilde{\pi_1^*}$.
We call the composition 

\[ \cobarlurie := \Pi_{2,*}^{(S)}\ \widetilde{\Pi_1^*}  \]
the {\bf derived (Lurie) cobar construction} if $\Pi_{2,*}^{(S)}$ (relative right Kan extension, cf.\@ \ref{DEFKAN}) exists on the image of $\widetilde{\Pi_1^*}$.
\end{DEF}
It is clear from Corollary~\ref{KORCOBAR} and \cite[Theorem 5.2.2.17]{Lur11} that for $I=\mathcal{S}=\OOO$ these are precisely the (co)bar constructions defined by Lurie in \cite[Definition 5.2.2.1]{Lur11}.

\begin{BEM}
The above discussion already makes sense, if $\mathcal{S}$ is a usual $\infty$-category --- considered as a trivial $\infty$-operad with only 1-ary morphism spaces --- or even if $\mathcal{S} = \cdot$ is the terminal category. In the latter case  $\mathcal{C}$ is just an $\infty$-category and $I$ a diagram (small $\infty$-category).  Then the discussion boils down to the following statement: The diagram
\begin{equation} \label{EQPAIRDIA} \vcenter{ \xymatrix{
& (\twop \mathcal{C})^I \ar[ld] \ar[rd] \\
\mathcal{C}^I  & & (\mathcal{C}^{\op})^I  \\
}} \end{equation}
defines a pairing classified by 
\[ (C, D) \mapsto \Hom_{\mathcal{C}^{\tw I}}(\pi_2^*C, \pi_1^* D)  \cong   \Hom_{\mathcal{C}^{\twop I}}(\pi_1^*C, \pi_2^*D).   \]
Thus this pairing is always left and right representable, i.e.\@ the derived bar and cobar constructions exist, if $\pi_2^*$ or $\pi_1^*$ has a right adjoint resp.\@ left adjoint, i.e.\@ if $\mathcal{C}$ has the relevant limits and colimits. 
\end{BEM}
\begin{EX}
For $I=\lefthalfcap$ the pairing (\ref{EQPAIRDIA}) is perfect if and only if $\mathcal{C}$ is stable. 
\end{EX}

\begin{proof}[Proof of Proposition~\ref{PROPCOBAR}] 
By the discussion in \ref{PARTWIST}, and Corollary~\ref{KORTW}, the diagram~(\ref{EQBARCOBAR}) is identified with  
 \[\xymatrix{  &  \twop ((\mathcal{C}^{\vee})_{\pi_2^* S^{\op}}^{\twop I})^{(1,2),(2,1)-\cart}  \ar[rd]^{\rho_2} \ar[ld]_{\rho_1} \\
 (\mathcal{C}^{\vee})_{\pi_2^* S^{\op}}^{\twop I, 2-\cart} & &  (\mathcal{C}_{\pi_2^* S^{\op}}^{\twop I, 1-\cart})^{\op}
 } \]
 where the functors are induced by the projections $\rho_1$ and $\rho_2$, and where 
 \[ \twop (\mathcal{C}_{\pi_2^* S^{\op}}^{\twop I})^{(1,2),(2,1)-\cart} \subseteq  \twop (\mathcal{C}_{\pi_2^* S}^{\tw I}) \] is precisely the full-subcategory of the objects that map via $\rho_1$ and $\rho_2$ to the
 full subcategories in the diagram.  One sees immediately that $\rho_1 \times \rho_2$ is, as full subcategory of the twisted arrow category, a fibration with fiber over $(\mathcal{E}, \mathcal{F})$ in $\twop (\mathcal{C}_{\pi_2^* S^{\op}}^{\twop I})$ being the groupoid 
  \[ \Hom_{\mathcal{C}_{\pi_2^* S^{\op}}^{\twop I}}(\mathcal{E},  \pi_2^* \mathcal{F}) \]
 where  $\mathcal{E}$ is considered via $\widetilde{\pi_1^*}$ as an object in $\mathcal{C}^{I}_{S} \cong \mathcal{C}_{\pi_2^* S^{\op}}^{\twop I, 2-\cart}$ and $\mathcal{F}$ an object in $(\mathcal{C}^{\vee})_{S^{\op}}^{I^{\op}}$.
 The other statement is obtained dually by replacing $\mathcal{C} \to \mathcal{S}$ by $(\mathcal{C}^{\vee})^{\op} \to \mathcal{S}$.
\end{proof}

\subsection{The derived (co)bar for (co)algebras}

In this section we discuss the derived (co)bar for plain (co)algebras i.e.\@ the case $I=\mathcal{S}=\OOO$. In particular, we establish easy criteria for the derived (Lurie's) (co)bar functors to exist. 
This gives an alternative (to \cite{Lur11}) approach to the proof. For some statements concerning basic 1-operads appearing, which are purely combinatorial, we refer to the subsequent Chapter~\ref{CHAPTERNONAB} where these fact are discussed thoroughly. 

Hence let $I = \mathcal{S} = \OOO$ be the associative planar operad and $S=\id$. In this case, a cofibration $\mathcal{C} \to \OOO$ is the same thing as a monoidal $\infty$-category $\mathcal{C}$ and  
\[ \mathcal{C}^I_S \cong \Alg(\mathcal{C})  \qquad  (\mathcal{C}^\vee)^{I^{\op}}_{S^{\op}} \cong \Coalg(\mathcal{C}) \]

The following is clear by construction (cf.\@ also \ref{PARNECKLACE}): 

\begin{LEMMA}\label{LEMMAOOOTW}
There is a canonical isomorphism of 1-cooperads (cf.\@ \ref{PARSIMPLEX})
\[ \twop \OOO = (\Delta_{\act}^{\op}, \ast')^{\vee} \]
where $\ast'$ is concatenation, identifying the extremal points. 
\end{LEMMA}

The bar diagram (\ref{EQBAR}, RHS) has the following form in this case:
\[  \xymatrix{. & (\mathcal{C}^{\vee})^{(\Delta_{\act}^{\op},*')^{\op}} \ar@{<-}[ld]_{\widetilde{\pi_1^*}} \ar@{<-}[rd]^{\pi_2^*}    \\
\Alg(\mathcal{C}) & & \Coalg(\mathcal{C}^{\vee})
}  \]

If the fibers of $\mathcal{C} \to \mathcal{S}$ have finite limits and $\otimes$ commutes with them,  we have
\[ (\mathcal{C}^{\vee})^{(\Delta_{\act}^{\op},*')^{\op}} \cong D((\Delta_{\act}, \ast')^{\op}, \mathcal{C}^{\vee})^{\OOO^{\op}} = \Coalg(\mathcal{C}^{\Delta_{\act}^{\op}}).  \]
where $D((\Delta_{\act}, \ast')^{\op}, \mathcal{C}^{\vee})$ is the Day convolution coalgebra (Proposition~\ref{PROPDAY}) which, in this case is  $\mathcal{C}^{\Delta_{\act}^{\op}}$ equipped with the new monoidal product 
\[ (A,B) \mapsto (\ast')_* A \boxtimes B \]
considered as cooperad
(cf.\@ Proposition~\ref{PROPDAYCOFIB}). 

$\widetilde{\pi_1^*}$ is  a functor which maps an algebra $A$ to a coalgebra in $\mathcal{C}^{\Delta_{\act}^{\op}}$ of the form (not depicting all but one degeneracy):
\[ \xymatrix{ \cdots  \ar@<4pt>[r]  \ar[r] \ar@<-4pt>[r] &A \otimes A\otimes A \ar@<2pt>[r]  \ar@<-2pt>[r] & A \otimes A \ar[r] & A  & \ar@{-->}[l] 1 } \]
in which the non-degenerate morphisms are given by the multiplication in $A$ and the degeneracies by inserting units. 

By applying Proposition~\ref{PROPKAN}, $\pi_2^*$ has a left adjoint, i.e.\@ a relative (operadic) left Kan existension $\pi_{2,!}^{(S^{\op})}$ as soon as $\mathcal{C}$ is $L$-admissible. 
However, we do not want to assume the commutation of $\otimes$ with any colimits, and this is also too general, because  we did not discuss augmentations so far\footnote{We ignore whether it would make sense to consider this in the absence of augmentations. It amounts probably to applying the augmented construction that we are about to discuss to a freely (co)augmented (co)algebra.}.
Notice that it is not reasonable to expect that $\pi_2^*$ has a relative left Kan existension $\pi_{2,!}^{(S^{\op})}$ that is computed fiber-wise (i.e.\@ here compatible with the
forgetful functors forgetting the coalgebra structure, i.e.\@ given as the colimit of the underlying diagram of shape $\Delta_{\act}^{\op}$). By Corollary~\ref{KOREXISTFIBERWISERELKANEXT} (cf.\@ also Remark~\ref{REMKANEXT}), a sufficient criterion would be that $\ast': \Delta_{\act}^{\op} \times \Delta_{\act}^{\op} \to \Delta_{\act}^{\op}$ is $\infty$-cofinal, which is not true. Notice also that the colimit of the underlying diagram of shape $\Delta_{\act}^{\op}$ would be just evaluation at $[1]$ because that is a final object in $\Delta_{\act}^{\op}$, and $A$ does not carry a coalgebra structure in general. 

\begin{PAR}\label{PARAUG}For the remaining part of the section we assume that $\mathcal{C}$ is a monoidal $\infty$-category admitting geometric realizations, and that {\em in $\mathcal{C}$ the unit $1$ is a final object}. Notice: This can always be achieved by considering augmented objects in $\mathcal{C}$ (cf.\@ \cite[5.2.3.9]{Lur11}). We do {\em not} assume any compatibility of $\otimes$ with geometric realizations (no $L$-admissibility). 
\end{PAR}

\begin{SATZ}\label{THEOREMLURIE}
\begin{enumerate}
\item
Under assumption \ref{PARAUG}
\[ \rho^*: (\mathcal{C}^{\vee})^{(\Delta, \ast)^{\op}} \to (\mathcal{C}^{\vee})^{(\Delta_{\act}, \ast')^{\op}} \]
(cf.\@ \ref{PARRHOSTAR}) is an isomorphism.
\item
Furthermore, a fiber-wise relative left Kan extension
\[ \colim_{\Delta^{\op}}:  (\mathcal{C}^{\vee})^{(\Delta, \ast)^{\op}} \to  \Coalg(\mathcal{C}) \]
exists, and hence Lurie's bar construction $\barlurie = \colim_{\Delta^{\op}} \circ (\rho^*)^{-1} \circ \widetilde{\pi_1^*}$ exists. 
\end{enumerate}
\end{SATZ}
\begin{proof}
1.\@ is Corollary~\ref{KORRHO}.
2.\@ follows from Corollary~\ref{KOREXISTFIBERWISERELKANEXT}, 
because a) $\mathcal{C}^{\vee} \to \mathcal{S}^{\op}$ is, as fibration of cooperads, automatically $L$-admissible (and geometric realization is the only colimit needed)
 and b) by Lemma~\ref{LEMMAIOTACART}, $(\Delta, \ast)^{\op} \to \OOO^{\op}$ is an $\infty$-exponential fibration of cooperads (it is almost monoidal w.r.t.\@ the product $\ast$, i.e.\@ a fibration of cooperads, but lacks counits) and the morphism is $\infty$-coCartesian (this amount to $\dec = \ast: \Delta^{\op} \times \Delta^{\op} \to \Delta^{\op}$ being $\infty$-cofinal). 
\end{proof}

\begin{KOR}[{Lurie~\cite[Theorem~5.2.2.17]{Lur11}}] \label{KORLURIE}
If $\mathcal{C}$ is a monoidal $\infty$-category such that 1 is final and initial and which admits geometric realizations, and such that $\mathcal{C}^{\op}$ admits geometric realizations, then Lurie's bar and cobar constructions exist and form an adjunction
\[ \xymatrix{ \Alg(\mathcal{C}) \ar@<3pt>[rr]^{\barlurie} & & \ar@<3pt>[ll]^{\cobarlurie} \Coalg(\mathcal{C}) } \]
with $\barlurie$ left adjoint. 
\end{KOR}

\begin{PAR}\label{PARCOBARALG}
{\em Under assumption~\ref{PARAUG} } 
Lurie's Bar is thus given by the composition: 
\[  \xymatrix{ \Alg(\mathcal{C}) \ar[r]^-{\barconst = \widetilde{\pi_1^*}} & (\mathcal{C}^{\vee})^{(\Delta_{\act}, \ast')^{\op}} \ar[r]^{(\rho^*)^{-1}} &(\mathcal{C}^{\vee})^{(\Delta, \ast)^{\op}}  \ar[rr]^{\colim_{\Delta^{\op}}} & & \Coalg(\mathcal{C})
}  \]

The composition $(\rho^*)^{-1} \circ \widetilde{\pi_1^*}$ maps an algebra $A$ to a coalgebra in $\mathcal{C}^{\Delta^{\op}}$ of the form (not depicting degeneracies):
\[ \xymatrix{ \cdots    \ar@<6pt>[r]  \ar@<2pt>[r]   \ar@<-2pt>[r]  \ar@<-6pt>[r]   & A \otimes A     \ar@<4pt>[r]  \ar[r] \ar@<-4pt>[r] & A  \ar@<2pt>[r]  \ar@<-2pt>[r] & 1 } \]
in which the non-degenerate active morphisms are given by the multiplication in $A$, the inert morphisms by the (canonical) augmentations, and the degeneracies are given by inserting units. This will be discussed in more detail in Section~\ref{SECTCOBARALG}.
\end{PAR}

\subsection{General classical cobar --- existence}

\begin{SATZ}\label{THEOREMEXISTENCECOBAR}
Assume that $\mathcal{C} \to \mathcal{S}$ is an $L$-admissible cofibration of operads (i.e.\@ has cocomplete fibers and the push-forward functors respect colimits in each variable separately) 
then for each small $\infty$-operad $I$ and $S \in \mathcal{S}^I$, the classical bar functor (Definition~\ref{DEFCOBARCLASS})
\[ \barconst = \widetilde{\pi_1^*}:   \mathcal{C}_{S}^{I}  \to  (\mathcal{C}^\vee)_{\pi_2^* S^{\op}}^{ \twop  I }  \]
has a left adjoint $\cobarconst$ (classical cobar construction) given by the restriction of the relative left Kan extension
\[ \pi_{3,!}^{(S)}:   \mathcal{C}_{\pi_3^* S}^{\twcop I, 3-\cocart}   \to  \mathcal{C}_{S}^{I}.   \]
The $\infty$-categories $ (\mathcal{C}^\vee)_{\pi_2^* S^{\op}}^{ \twop  I}$ and $\mathcal{C}_{\pi_3^* S}^{\twcop I, 3-\cocart}$ are canonically equivalent. 
\end{SATZ}
We stated the Theorem in its most simple form. It is easy to extract from the proof, exactly which colimits need to exist and have to be preserved by the push-forward functors. 

Note that $\pi_3: \twcop I \to I$ is a cofibration and thus the Kan extension $ \pi_{3,!}^{(S)}$ exists by Corollary~\ref{KOREXISTFIBERWISERELKANEXT} (assuming $L$-admissibility) and is computed fiber-wise. For $I=\mathcal{S}=\OOO$, for instance, where there is only one fiber, is is given by a colimit
over $\twop \Delta_{\act}^{\op}$ which is essentially the category of necklaces (\ref{PARNECKLACE}). See Section~\ref{SECTCOBARALG} for a thorough discussion of classical cobar for $I=\mathcal{S}=\OOO$.

There is a morphism $\widetilde{\pi_1}: (\twcop I, \pi_3^*S) \to (I, S)$ in $\Dia(\mathcal{S})$ (cf.\@ \ref{DEFDIAS}) given by the obvious diagram
\[ \xymatrix{ \twcop I \ar[rr]^{\pi_1} \ar[rd]_{\pi_3^*S} &\ar@{}[d]|-{\Downarrow^{S(\mu)}} & I  \ar[ld]^{S} \\
& \mathcal{S}  } \] 
Be aware that $\pi_1$ is only a generalized morphism of $\infty$-operads (in the sense of Definition~\ref{DEFGENMOR}). Nevertheless, $\widetilde{\pi_1}$ induces a pull-back functor
\[ \widetilde{\pi_1}^* = S(\mu)_\bullet \pi_1^*: \mathcal{C}_{S}^{I}   \to \mathcal{C}_{\pi_3^* S}^{\twcop I} \]
as explained in \ref{PARFIBDER}.

\begin{LEMMA}\label{LEMMACOBAR}
We have a commutative diagram: 
  \[  \footnotesize
  \xymatrix{ \mathcal{C}_{S}^{I}  \ar[dd]^{\sim}  \ar[rd]^{\widetilde{\pi_1^*}} \ar[rr]^{\widetilde{\pi_1}^* = S(\mu)_\bullet \pi_1^*} & &  \mathcal{C}_{\pi_3^* S}^{\twcop I, 3-\cocart} \ar[d]^{\sim}\\
&   (\mathcal{C}^\vee)_{\pi_2^* S^{\op}}^{\twop I} \ar[d]^{\sim} \ar@{-->}[ru]^{\widetilde{\pi_{12}^*}}  &  \Hom^{\lax,\inert,3-\pseudo}_{(\Cat_{\infty}^{\times})^{\twcop I}}(\cdot, \pi_3^*\Xi) \ar[d]^{\sim}_{\pi_{234}^*} \\
  \Hom^{\lax, \inert-\pseudo}_{(\Cat_{\infty}^{\times})^I}(\cdot, \Xi) \ar[r] & \Hom^{\oplax, \inert-\pseudo}_{(\Cat_{\infty}^{\times})^{\tw I}}(\cdot, \pi_2^*\Xi) \ar[r]^-{\sim}  & \Hom^{\lax,\inert,1,4-\pseudo}_{(\Cat_{\infty}^{\times})^{\twtw I}}(\cdot, \pi_4^*\Xi) 
   }  \]
where $\Xi: I \rightarrow \mathcal{S} \to  \Cat_{\infty}$ is the morphism classifying the pull-back of the cofibration $\mathcal{C} \to \mathcal{S}$ via $S$, 
and where the lower horizontal morphisms are given by Proposition~\ref{PROPTWIST} using the obvious isomorphism $\tw(\tw I) = \twtw I$. The dashed morphism is defined by this diagram.
\end{LEMMA}
\begin{proof}Follows from Propositions~\ref{PROPTRANSLATE} and \ref{PROPTWIST} examining the construction. For the fact that $\pi_{234}^*$ is an isomorphism use Lemma~\ref{LEMMATWLOC} and reason as in the proof of Lemma~\ref{LEMMAOPTWIST}.
\end{proof}

 \begin{LEMMA}\label{LEMMACOBAR2}Let $I$ be a small $\infty$-operad.
 Consider the $\infty$-coCartesian\footnote{Because both are cofibrations this means just: it maps coCartesian to coCartesian morphisms} functor
 \[ \pi_{13}: \twcop I  \to \ddd I  \]
of cofibrations over $I$ (via $\pi_3$).  
It has a transpose over $I^{\op}$:
 \[ \pi_{13}^{\vee}:  (\twcop I)^{\vee} \to (\ddd I)^{\vee} \]
 (which, by construction, can be identified with the functor of cooperads
 $\pi_{13}:  {}^{\downarrow\uparrow\uparrow} I \to  {}^{\downarrow\uparrow} I$).
\begin{enumerate}
\item
 The functor 
 \[ \xymatrix{   \mathcal{C}_{\pi_3^* S}^{\ddd I}   \ar[rr]^-{\pi_{13}^*} & &  \mathcal{C}_{\pi_3^* S}^{\twcop I}}  \]
 has a left adjoint $\pi_{13,!}^{(\pi_3^*S)}$  (i.e.\@ a relative left Kan extension). 
 \item An object $\mathcal{E} \in \mathcal{C}^K_{\pi_3^*S}$ for either category $K \in \{\ddd I, \twcop I \}$ is 3-Cartesian\footnote{i.e.\@ type-3 morphisms (choosing the same indexing in both categories) in $K$ are mapped to Cartesian morphisms} if and only if for all active $\alpha: i' \to i$ in $I$
 \[ \alpha_{\mathcal{C}, \bullet} (\iota_{i'})^* \mathcal{E} \to (\alpha_{\bullet,K})^* (\iota_{i})^* \mathcal{E}    \]
 is an isomorphism where $\iota_{i'}: K_{i'} \hookrightarrow K$, resp.\@ $\iota_i: K_i \hookrightarrow K$ are the inclusions of the respective fibers. 

\item
 The functor $\pi_{13}^{\vee}$ is $\infty$-coCartesian (Definition~\ref{DEFCART}) as well, i.e.\@ for all active $\alpha: i' \to i$ in $I$
 \[  \alpha_{\bullet, \twcop I} \,{}^t\!\pi_{13} \cong \,{}^t\!\pi_{13}\, \alpha_{\bullet, \ddd I}.   \]
 Notice $\alpha_{\bullet, \ddd I} = \alpha^{\bullet}_{{}^{\downarrow\uparrow} I}$ and similarly for $\twcop I$.
 
 \item
 The functor $\pi_{13,!}^{(\pi_3^*S)}$ preserves 3-coCartesian objects.
\end{enumerate}
 \end{LEMMA}
 \begin{proof}
 1.\@ follows from Corollary~\ref{KOREXISTFIBERWISERELKANEXT}.
 
 2.\@ follows from the definitions.
 
3.\@ Consider an active morphism $\alpha: i \to i'$ in $I$. The assertion is the statement that the commutative diagram
\begin{equation} \label{EQEXACTFIBER} 
\vcenter{ \xymatrix{ 
\twop I \times_{/I} i \ar[rr]^{\alpha_{\bullet, \twcop I}} \ar[d]_{\pi_{13}} & & \twop I \times_{/I} i' \ar[d]^{\pi_{13}} \\
I \times_{/I} i \ar[rr]_-{\alpha_{\bullet, \ddd I}} & & I \times_{/I} i'
 } }
 \end{equation}
be exact. By Lemma~\ref{PROPEXACT}, 4.\@, this is the case, if for all objects $i_0 \to i_1 \to i'$  of $\tw I \times_{/I} i'$ and objects $i_2 \to i$   of $I \times_{/I} i$ the morphism
\[ \{ i_0 \to i_1 \to i' \} \times_{/ \tw (I \times_{/I} i')} \tw (I \times_{/I} i)  \times_{/(I \times_{/I} i)}  \{i_2 \to i\}  \to \Hom_{I \times_{I} i'}(i_0 \to i', i_2 \to i') \]
is an equivalence after forming groupoids. The left hand side is a category of diagrams in $I$ of the form
\[ \xymatrix{ i_0 \ar[r]  & x \ar[r] \ar[d] & y \ar[d] \ar[r] & i_1 \ar[d]  \\
& i_2 \ar[r] & i \ar[r] & i' 
}\]
contravariant in $x$ and covariant in $y$. There is chain of adjunctions to the category of diagrams of the form
\[ \xymatrix{ i_0 \ar@{=}[r]  & i_0 \ar@{=}[r] \ar[d] & i_0 \ar[d] \ar[r] & i_1 \ar[d]  \\
& i_2 \ar[r] & i \ar[r] & i' 
}\]
which is equivalent to $\Hom_{I \times_{/I} i'}(i_0 \to i', i_2 \to i')$. Therefore the diagram (\ref{EQEXACTFIBER}) is exact.

4.\@ Let $\mathcal{E}$ be a 3-coCartesian object. 
Using 2.\@, we have to prove that  
\[ \alpha_{\mathcal{C}, \bullet} (\iota_{i'})^* \pi_{13,!}  \mathcal{E} \to (\alpha_{\bullet,K})^* (\iota_{i})^* \pi_{13,!} \mathcal{E}   \]
is an isomorphism. Since $\pi_{13,!}$ is computed fiber-wise (w.r.t.\@ $\pi_3$) this is the same as
\[ \alpha_{\mathcal{C}, \bullet} \pi_{13,!}  (\iota_{i'})^*  \mathcal{E} \to (\alpha_{\bullet,K})^* \pi_{13,!}  (\iota_{i})^* \mathcal{E}.   \]
Examining the proof of Proposition~\ref{PROPDAYADJOINT} we have a commutative diagram: 
\[ \xymatrix{
 \alpha_{\mathcal{C}, \bullet} \pi_{13,!}  (\iota_{i'})^*  \mathcal{E}  \ar[r] & (\alpha_{\bullet,K})^* \pi_{13,!}  (\iota_{i})^* \mathcal{E}  \\
\pi_{13,!} \alpha_{\mathcal{C}, \bullet}   (\iota_{i'})^*  \mathcal{E} \ar[r] \ar[u] & \pi_{13,!} (\alpha_{\bullet,K})^*  (\iota_{i})^* \mathcal{E}  \ar[u]
}   \]
where the lower morphism is $\pi_{13,!}$ applied to the isomorphism expressing the 3-coCartesianity of $\mathcal{E}$ and the vertical maps are the obvious mates. 
The left hand side morphism is an isomorphism because of the $L$-admissibility of $\mathcal{C}$ and the right hand side is an isomorphism by 3. Therefore $\pi_{13,!}^{(\pi_3^*S)} \mathcal{E}$ is 3-coCartesian. 
 \end{proof}
 
\begin{proof}[Proof of Theorem~\ref{THEOREMEXISTENCECOBAR}] By Lemma~\ref{LEMMACOBAR}
 the functor $\barconst = \widetilde{\pi_1^*}$ is isomorphic to the functor 
\[ \xymatrix{  \mathcal{C}_{S}^{I}  \ar[rr]^-{\widetilde{\pi_1}^* = S(\mu)_\bullet \pi_1^*} & &  \mathcal{C}_{\pi_3^* S}^{\twcop I, 3-\cocart}}.  \]
 It is not reasonable to expect that  $\widetilde{\pi_1}^*$ has a relative left Kan extension on all of $\mathcal{C}_{\pi_3^* X}^{\twcop I}$ (notice that $\widetilde{\pi_1}$ is not of diagram type). 
In the statement of the Theorem it is claimed that the relative left Kan extension $\pi_{3,!}^{(S)}$ is left adjoint to $\widetilde{\pi_1}^*$ {\em when restricted to the full subcategory of 3-coCartesian objects. }

Observe that the functor in question factors: 
\[ \xymatrix{  \mathcal{C}_{S}^{I}  \ar[rr]^-{\widetilde{\pi_1}^* = S(\mu)_\bullet \pi_1^*} & &  \mathcal{C}_{\pi_3^* S}^{\ddd I}   \ar[rr]^-{\pi_{13}^*} & &  \mathcal{C}_{\pi_3^* S}^{\twcop I}}  \]
where we have numbered the indices in $\ddd I$ and $\twcop I$ in a coherent way. 
The left functor is an equivalence onto the subcategory of 3-coCartesian objects of the category in the middle and $\widetilde{\pi_1}^* \pi_{3,!} = \widetilde{\pi_1}^* \delta^* = \pi_1^*\delta^*$ is, in fact, a {\em right} coCartesian projector, or in other words, $\pi_{3,!}$ is a {\em right} adjoint to the fully-faithful inclusion $\widetilde{\pi_1}^*$. 
The relative left Kan extension $\pi_{13,!}^{(\pi_3^*S)}$ exists and is computed fiber-wise over $\pi_3$ by Lemma~\ref{LEMMACOBAR2}, 1.
It thus suffices to see that $\pi_{13,!}^{(\pi_3^*S)}$ preserves the condition of being 3-coCartesian. This is Lemma~\ref{LEMMACOBAR2}, 4.
\end{proof}

\subsection{Classical cobar for coalgebras}\label{SECTCOBARALG}

In this section, we discuss the classical  cobar construction (Definition~\ref{DEFCOBARCLASS}) associated with the operad $\OOO$ for {\em arbitrary monoidal ($\infty$-)categories},
i.e.\@ letting $(\mathcal{C}, \otimes) \to \OOO$ a monoidal ($\infty$-)category (considered as cofibration of operads) and $I:= \OOO$ and $S:=\id$ in the abstract setting (\ref{PARCOBARSETTING}). 

Assume that $\mathcal{C}$ is countably cocomplete and such that $\otimes$ commutes with countable colimits. Assume also that $\mathcal{C}$ is finitely complete, and such $\otimes$ commutes with $\dec_*$ (in the Abelian 1-categorical case this means only that $\otimes$ has to commute with finite products which we always assume anyway).

Recall the abstract classical (co)bar adjunction:
\[ \xymatrix{  ((\mathcal{C}, \otimes)^{\vee})^{\twop \OOO} \ar@<4pt>[rr]^-{\cobarconst} & & \ar@<4pt>@{_{(}->}[ll]^-{\barconst} (\mathcal{C}, \otimes)^{\OOO} . } \]
We will always apply $\cobarconst$ to an object in the image of $\rho^*$:
\begin{equation*}
\xymatrix{  ((\mathcal{C}, \otimes)^{\vee})^{(\Delta,*)^{\op}}   \ar[rr]^-{\rho^*} & & ((\mathcal{C}, \otimes)^{\vee})^{\twop \OOO} } 
\end{equation*}
(cf.\@ \ref{EXDAY2} and \ref{PARRHOSTAR}).  

 \begin{PAR}\label{PARCOALG}The source of $\rho^*$ is just $\Coalg(\mathcal{C}^{\Delta^{\op}}, \tildeotimes)$ under suitable $R$-admissibility.
If $\mathcal{C}$ is an Abelian 1-category, an object in $\Coalg(\mathcal{C}^{\Delta^{\op}}, \tildeotimes)$  is thus a coalgebra in non-negatively graded complexes (a dg-coalgebra) w.r.t.\@ to the usual tensor product of complexes.
In $((\mathcal{C}, \otimes)^{\vee})^{(\Delta,*)^{\op}}$ it translates to an object $A \in \mathcal{C}^{\Delta^{\op}}$ with comultiplication and counit 
\[ \mu_{n,m}: A_{[n] \ast [m]} \to A_{[n]} \otimes A_{[m]} \qquad  A_{[n]} \to 1  \]
in such a way that 
\[ \xymatrix{  A_{[n] \ast' [m]}  \ar[r]^-{s_{\can}} \ar@{->>}[d]  & A_{[n] \ast [m]}  \ar[r]^-{\mu_{n,m}} & A_{[n]} \otimes A_{[m]}   \ar@{->>}[d]  \\
A_{n+m} \ar[rr]_-{m_{n,m}} & & A_n \otimes A_m } \]
is commutative, where in the bottom row, $m_{n,m}$ is the component of the comultiplication in the complexes viewpoint (cf.\@ Proposition~\ref{PROPEXPLICITAB}). 

Back in the general case, pulling back under the morphism $\rho: \twop \OOO = (\Delta^{\op}_{\act},*') \to (\Delta,*)^{\op}$, we get an object $\rho^*A$ in
$((\mathcal{C}, \otimes)^{\vee})^{\twop \OO}$, i.e.\@ the same $A_{[n]}$ with structure maps
\[ \mu_{n,m}': A_{[n] \ast' [m]} \to A_{[n]} \otimes A_{[m]} \qquad  A_{[0]} \to 1  \]
which are nothing else then the composition with the canonical degeneracy (considered before in the Abelian case) and
the counit is the corresponding restriction to $A_{[0]}$. If 1 is final in $\mathcal{D}$, Lemma~\ref{LEMMARHO} below implies, that $A$ can be recovered from $\rho^* A$. 

To compute $\cobarconst \rho^* A$ using Theorem~\ref{THEOREMEXISTENCECOBAR}, we have to identify
 \[ ((\mathcal{C}, \otimes)^{\vee})^{\twop \OO} \cong (\mathcal{C}, \otimes)^{\twcop \OO, 3-\cocart}  \]
and then take the relative Kan extension $\pi_{3,!}$. See \ref{PARNECKLACE} below for a more thorough discussion of $\twcop \OO$ which is essentially the category of necklaces cf.\@ \cite{DS11, Riv22, BS23}. 

Since $\pi_3$ is a cofibration this is computed fiber-wise (Corollary~\ref{KOREXISTFIBERWISERELKANEXT}) and thus its underlying object in $\mathcal{C}$ is given by a colimit over the category 
$\mathcal{C}^{\twcop \OO}_{[1]} = \mathcal{C}^{\twop \Delta^{\op}_{\act}}$.
$\rho^*A$ corresponds to the following object $B \in \mathcal{C}^{\twop \Delta^{\op}_{\act}}:$
Active morphisms are mapped 
\[  [n] \leftarrow [m] \quad  \mapsto \quad   A_{[n_1]} \otimes \cdots \otimes A_{[n_m]} \]
where $[n] = [n_1] \ast' \cdots \ast' [n_m]$ is the induced decomposition of $[n]$ (which corresponds to taking fibers identifying $\Delta^{\op}_{\act} \cong \Delta_{\emptyset}$ via Lemma~\ref{LEMMADUAL}).
A type-1 morphism
\[ \xymatrix{ {[n]} \ar@{<-}[r] \ar@{<-}[d] &  {[m]} \ar@{=}[d] \\
{[n']} \ar@{<-}[r] &  {[m]}
} \]
maps to the factor-wise structure morphism
\[ A_{[n_1]} \otimes \cdots \otimes A_{[n_m]}  \to A_{[n_1']} \otimes \cdots \otimes A_{[n_m']} \]
while a type-2 morphism corresponds to applying the structure morphism of the algebra, for instance
\[ \xymatrix{ {[n_1] \ast' [n_2]} \ar@{<-}[r] \ar@{=}[d] &  {[1]} \ar[d]^{\delta_1} \\
{[n_1] \ast' [n_2]} \ar@{<-}[r] & {[2]}
} \]
maps to the component of the comultiplication: 
\[ \mu_{n_1, n_2}': A_{[n_1]} \otimes  A_{[n_2]}  \to A_{[n_1] \ast' [n_2]}. \]
 \end{PAR}
 
\begin{PAR}\label{PARCOBAREXPLICIT}
The colimit over $\twop \Delta^{\op}_{\act}$ can be computed using the observation \ref{PARFINALCOEND}. 
If $\mathcal{C}$ is a 1-category then this colimit is very simple: In \ref{PARFINALCOEND} it suffices to restrict to the
cokernel of the last two maps and  we may restrict to a generating set of morphisms and get: 
\begin{gather}
 \boxed{ \cobarconst \rho^* A =    
  \colim \left(  \substack { \coprod A_{[1]} \otimes \cdots \times A_{[2]} \otimes \cdots \otimes A_{[1]} \\ \amalg \\ \coprod A_{[1]} \otimes \cdots \times A_{[0]} \otimes \cdots \otimes A_{[1]}   } \rightrightarrows  \coprod_{n=0}^{\infty} A_{[1]}^{\otimes n}  \right) }  \label{eqformulacobar}
  \end{gather}
  where the two morphisms are induced by
  \[  A_{[2]} \to A_{[1]} \otimes A_{[1]} \]
  (comultiplication) and by
  \[ A(\delta_1):  A_{[2]} \to  A_{[1]}  \]
 resp.\@ 
  \[  A_{[0]} \to 1 \]
  (counit) and
  \[  A(s_0): A_{[0]} \to A_{[1]}. \]
\end{PAR}

In Chapter~\ref{CHAPTERBASICEXAMPLES} this colimit is computed explicitly for a number of mostly 1-categorical examples.

\subsection{Functoriality of (co)bar --- non-connected (co)bar}\label{SECTCOBARFUNCT}

\begin{PAR}
Let $\mathcal{C} \to \mathcal{S}$, $\mathcal{D} \to \mathcal{S}$ be cofibrations of $\infty$-operads and $R: \mathcal{C} \to \mathcal{D}$ be a functor of $\infty$-operads over $\mathcal{S}$. 
Assume that $R$ has a fiber-wise left adjoint. Then those assemble to a functor of fibrations
\[ Q: \mathcal{D}^{\vee} \to \mathcal{C}^{\vee}  \]
over $\mathcal{S}^{\op}$.  
For $S \in \mathcal{S}^I$, we get functors
\[ R: \mathcal{C}_S^I \to \mathcal{D}_S^I \qquad Q: (\mathcal{D}^{\vee})_{S^{\op}}^{I^{\op}} \to (\mathcal{C}^{\vee})_{S^{\op}}^{I^{\op}}   \]

which are --- in terms of the classifying functors $\Xi_{\mathcal{C}}, \Xi_{\mathcal{D}}: I \to \mathcal{S} \to (\Cat_{\infty}, \times)$ given by composition with 
$R: \Xi_{\mathcal{C}} \Rightarrow \Xi_{\mathcal{D}}$ which is a lax morphism:
\[ \Hom^{\lax, \inert-\pseudo}_{(\Cat_{\infty}^\times)^{I}}(\cdot, \Xi_{\mathcal{C}}) \to \Hom^{\lax, \inert-\pseudo}_{(\Cat_{\infty}^\times)^{I}}(\cdot, \Xi_{\mathcal{D}})  \]
and gives via composition with 
$Q: \Xi_{\mathcal{D}} \Rightarrow \Xi_{\mathcal{C}}$ (its mate), which is an oplax morphism, a morphism
\[ \Hom^{\oplax, \inert-\pseudo}_{(\Cat_{\infty}^\times)^{I}}(\cdot, \Xi_{\mathcal{D}}) \to \Hom^{\oplax, \inert-\pseudo}_{(\Cat_{\infty}^\times)^{I}}(\cdot, \Xi_{\mathcal{C}}).  \]
\end{PAR}
\begin{PAR}\label{PARCOBARSETTINGFUNCT}
In this situation we can refine the (co)bar diagrams: 
\begin{equation}\label{eqcobar0funct} \vcenter{ \xymatrix{
& (Q \dashv R)^{\downarrow \uparrow} \ar[ld]_{\rho_1} \ar[rd]^{\rho_2} \\
\mathcal{D} \ar[rd] & & (\mathcal{C}^\vee)^{\op} \ar[ld] \\
& \mathcal{S}
}} \end{equation}
where $(Q \dashv R)^{\downarrow \uparrow }$ is classified by the functor
\begin{eqnarray*}
\mathcal{S} & \to & \Cat_{\infty}\\
S & \mapsto & \mathcal{X}_S
\end{eqnarray*}
where $\mathcal{X}_S$ is the pairing defined by the adjunction $Q_S \dashv R_S$, i.e.\@ classified by
\begin{eqnarray*}
\mathcal{D}_S^{\op} \times \mathcal{C}_S & \to & \Cat_{\infty}\\
D, C & \mapsto & \Hom(Q(D), C) \cong \Hom(D, R(C))
\end{eqnarray*}

For each small $\infty$-operad $I$ and $S \in \mathcal{S}^I$ this gives rise to a diagram
\begin{equation}\label{EQBARCOBARFUNCT} \vcenter{ \xymatrix{
& ((Q \dashv R)^{\downarrow \uparrow })^I_S \ar[ld]_{\rho_1} \ar[rd]^{\rho_2} \\
\mathcal{D}^I_{S}  & & ((\mathcal{C}^\vee)^{I^{\op}}_{S^{\op}})^{\op}  \\
}}\end{equation}
\end{PAR}

\begin{PAR}\label{PARCOBARSETTING2FUNCT}
For $S \in \mathcal{S}^I$, there are two natural diagrams
\begin{equation}\label{EQBARFUNCT} 
\vcenter{ \xymatrix@C=1pc{
& (\mathcal{D}^{\vee})^{\twop I}_{\pi_2^* S^{\op}} \ar@{<-}[ld]_{\widetilde{\pi_1^*}} \ar[r]^Q &  (\mathcal{C}^{\vee})^{\twop I}_{\pi_2^* S^{\op}} \ar@{<-}[rd]^{\pi_2^*}\\
\mathcal{D}^I_{S}  & & & ((\mathcal{C}^\vee)^{I^{\op}}_{S^{\op}})^{\op}  \\
}}
 \quad 
 \vcenter{\xymatrix@C=1pc{
& \mathcal{D}^{\tw I}_{\Pi_2^* S} \ar@{<-}[ld]_{\Pi_2^*}  & \ar[l]_R \mathcal{C}^{\tw I}_{\Pi_2^* S}   \ar@{<-}[rd]^{\widetilde{\Pi_1^*}} \\
\mathcal{D}^I_{S}  & & & ((\mathcal{C}^\vee)^{I^{\op}}_{S^{\op}})^{\op}   \\
}  }
\end{equation}
We have then similarly: 
\end{PAR}

\begin{PROP}\label{PROPCOBARFUNCT}
The pairing (\ref{EQBARCOBARFUNCT}) is isomorphic to the pairing defined by the two (isomorphic) groupoid valued functors
\begin{align*}
(\mathcal{D}^I_{S})^{\op} \times (\mathcal{C}^\vee)^{I^{\op}}_{S^{\op}} &\to \Gpd_{\infty} \\
 (D, C) &\mapsto   \Hom_{(\mathcal{C}^\vee)^{\twop I}_{\pi_2^* S^{\op}}}(Q\, \widetilde{\pi_1^*}D, \pi_2^*C) \cong \Hom_{\mathcal{C}^{\tw I}_{\pi_2^* S}}(\Pi_2^*D, R\, \widetilde{\Pi_1^*} C).
\end{align*}
\end{PROP}
The proof is omitted for the moment. 

\begin{DEF}
We call
\[ \barlurie_Q := \pi_{2,!}^{(S^{\op})} \circ Q \circ \widetilde{\pi_1^*}  \qquad \cobarlurie_R := \Pi_{2,!}^{(S)} \circ R \circ \widetilde{\Pi_1^*}   \]
the derived (co)bar constructions w.r.t.\@ the adjunction $Q \dashv R$, when the respective Kan extensions exist. 
\end{DEF}
As before, the Proposition shows that if $\barlurie_Q$ and $\cobarlurie_R$ exist, then the pairing (\ref{EQBARCOBARFUNCT}) is left
and right representable with those functors associated. 

For the special case $I= \mathcal{S} = \OOO$ one can deduce exactly as  in Corollary~\ref{KORLURIE}:

\begin{KOR}
If $\mathcal{D}$ is a monoidal $\infty$-category such that 1 is final and which admits geometric realizations, and If $\mathcal{C}$ is a monoidal $\infty$-category such that 1 is initial and
such that $\mathcal{C}^{\op}$ admits geometric realizations, then the derived bar and cobar constructions w.r.t.\@ the adjunction $Q \dashv R$ exist and form an adjunction
\[ \xymatrix{ \Alg(\mathcal{D}) \ar@<3pt>[rr]^{\barlurie_Q} & & \ar@<3pt>[ll]^{\cobarlurie_R} \Coalg(\mathcal{C}) } \]
with $\barlurie_Q$ left adjoint. 
\end{KOR}

Notice that, if $Q$ preserves  the monoidal product (not only laxly), then $\widetilde{\pi_1^*} \circ Q \cong Q \circ \widetilde{\pi_1^*}$ and thus
$\barlurie_Q = \barlurie \circ Q$ and we gain nothing new.

\begin{BEISPIEL}[Non-connected (co)bar]\label{PARNONCONNECTED}
Let $\mathcal{C}$ be a monoidal $\infty$-category such that 1 is final (but not necessarily initial) and let $X \in \Coalg(\mathcal{C})$ be a coalgebra. 
Consider the following adjunction between monoidal $\infty$-categories: 
\[ \xymatrix{ \mathrm{Bimod}_{X}(\mathcal{C})_{X/} \ar@<3pt>[r]^-{Q}  & \ar@<3pt>[l]^-{R} \mathcal{C}_{X/} }    \]
Here $\mathrm{Bimod}_{X}(\mathcal{C})$ is equipped with the monoidal product $- \otimes_X -$ described in \cite[4.4.2]{Lur11}, cf.\@ \cite[Proposition~4.4.3.12]{Lur11}, and unit $X$ (considered as $X$-bimodule) and 
$\mathrm{Bimod}_{X}(\mathcal{C})_{X/}$ is, as undercategory under the unit, monoidal itself. 
However, also the undercategory $\mathcal{C}_{X/}$ is monoidal because $X$ is a coalgebra (two morphisms $X \to Y_1$ and $X \to Y_2$ are mapped to the
composition $X \to X \otimes X \to Y_1 \otimes Y_2$).
$Q$ is the forgetful functor which forgets the bimodule structure but keeps the coaugmentation. $R$ is the free-bialgebra functor \cite[Proposition~4.3.3.12]{Lur11} transferring the coaugmentation 
 by means of the adjunction
\[ \Hom_{\mathrm{Bimod}_{X}}(X, RY) \cong \Hom_{\mathcal{C}}(QX, Y) \]
where $X$ is considered as bimodule over itself (= unit in $\mathrm{Bimod}_{X}$).

Then $X$ is initial in $\mathrm{Bimod}_{X}(\mathcal{C})_{X/}$ and $X \to 1$ is final in $\mathcal{C}_{X/}$, hence if $\mathcal{C}^{\op}$ has geometric realizations 
and $\mathrm{Bimod}_{X}(\mathcal{C})_{X/}$ has geometric realizations (it has by the dual of \cite[4.3.3.9]{Lur11} as soon as $\mathcal{C}$ has) then 
we have a derived (co)bar adjunction
\[ \xymatrix{ \Alg(\mathrm{Bimod}_{X}(\mathcal{C})_{X/}) \ar@<3pt>[rr]^-{\barlurie_Q} & & \ar@<3pt>[ll]^-{\cobarlurie_R} \Coalg(\mathcal{C}_{X/}) } \]
\end{BEISPIEL}

\begin{DEF}
We call $\Alg(\mathrm{Bimod}_{X}(\mathcal{C})_{X/})=\Alg(\mathrm{Bimod}_{X}(\mathcal{C}))$ the category of {\bf category objects} of $\mathcal{C}$ with objects $X$ in this case\footnote{I just became aware of the articles \cite{Mer24, LM25} which contain related theory.}.
\end{DEF}
See Section~\ref{SECTCART} for the Cartesian case, explaining the notion. 

We proceed to discuss a functoriality of the {\em classical} (co)bar which will be important for comparisons between different such construction (Section~\ref{SECTFUNCT}).

\begin{PROP} \label{PROPFUNCTCOBAR}There is an adjunction
\[ \cobarconst \circ  Q \dashv \barconst \circ R \]
which is to say there are morphisms
 \[ Q \circ  \barconst \circ R \to \barconst \qquad \cobarconst \to R \circ \cobarconst \circ Q   \]
 satisfying suitable compatibilities. 
\end{PROP}
\begin{proof}[Proof (sketch).]

We have for the fibers 
\[ \mathcal{C}^{I}_{S} \cong (\mathcal{C}^{\vee})^{\twop I, 2-\cocart}_{\pi_2^* S^{\op}} \] 
   the identification being given by the composition:
  \[ \xymatrix{ \mathcal{C}_{S}^{I} \ar[d]^{\sim} &  (\mathcal{C}^{\vee})_{\pi_2^* S^{\op}}^{ \twop  I, 2-\cocart} \ar[d]^{\sim} \\
   \Hom^{\lax}_{(\Cat_{\infty}^{\times})^{I}}(\cdot, \Xi) \ar[r]^-{\sim}  & \Hom^{\oplax,\inert-\pseudo,2-\pseudo}_{(\Cat_{\infty}^{\times})^{\tw I}}(\cdot, \pi_2^* \Xi) 
   }  \]
 and similarly dually. 
There is a commutative diagram where the natural transformation is given (essentially) by the units of the point-wise adjunction. 
  \[ \xymatrix{  
   \Hom^{\lax}_{(\Cat_{\infty}^{\times})^{ I}}(\cdot, \Xi_{\mathcal{D}}) \ar[r] \ar[d]_{R} \ar@{}[rd]|{\Uparrow} & \Hom^{\oplax,\inert-\pseudo}_{(\Cat_{\infty}^{\times})^{\tw I}}(\cdot, \pi_2^*\Xi_{\mathcal{D}})  \\
   \Hom^{\lax}_{(\Cat_{\infty}^{\times})^{ I}}(\cdot, \Xi_{\mathcal{C}}) \ar[r]  & \Hom^{\oplax,\inert-\pseudo}_{(\Cat_{\infty}^{\times})^{\tw I}}(\cdot, \pi_2^*\Xi_{\mathcal{C}}) \ar[u]_{Q} 
   }  \]
and similarly with (op)lax exchanged.

This defines the  2-morpisms in the following diagrams
\[ \vcenter{ \xymatrix@C=1pc{
\mathcal{C}^I_{S} \ar@{}[rrd]|{\Uparrow} \ar[d]_R & & (\mathcal{C}^{\vee})^{\twop I}_{\pi_2^* S^{\op}} \ar@{<-}[ll]_-{\widetilde{\pi_1^*}} \\ 
\mathcal{D}^I_{S}  & &  (\mathcal{D}^{\vee})^{\twop I}_{\pi_2^* S^{\op}} \ar@{<-}[ll]^-{\widetilde{\pi_1^*}}  \ar[u]_{Q} 
}  }
\]
and

 \[ 
 \xymatrix{  %
   (\mathcal{C}^{\vee})_{\pi_2^*S^{\op}}^{\twop I} \ar[rr]^{\widetilde{\pi_{12}^*}} \ar@{}[rrd]|{\Uparrow} &  & \mathcal{C}_{\pi_3^*S}^{\twcop I}  \ar[rr]^{\pi_{3,!}^{(S)}}  \ar[d]^{R}   \ar@{}[rrd]|{\Uparrow} & &  \mathcal{C}_{S}^{I}  \ar[d]^R \\
   (\mathcal{D}^{\vee})_{\pi_2^*S^{\op}}^{\twop M} \ar[rr]_{\widetilde{\pi_{12}^*}} \ar[u]^{Q} & & \mathcal{D}_{\pi_3^*S}^{\twcop I} \ar[rr]_{\pi_{3,!}^{(S)}} & &  \mathcal{D}_{S}^{I}
}  
  \]
  with the morphism $\widetilde{\pi_{12}^*}$ from Lemma~\ref{LEMMACOBAR}.
  The proof that these two 2-morphisms yield an adjunction is omitted for the moment. 
\end{proof}

\subsection{Classical and derived (co)bar in the Cartesian case} \label{SECTCART}

We consider again the classical and derived (co)bar for $I = \mathcal{S} = \OOO$ in which case $(\mathcal{C}, \otimes) \to \OOO$ is just a monoidal $\infty$-category. 
We investigate the Cartesian case in this section, i.e.\@ $\otimes=\times$. 

Recall that if $(\mathcal{C}, \times)$ is a Cartesian monoidal $\infty$-category  then we have
\[ \Alg(\mathcal{C}, \times) \cong \mathrm{Mon}(\mathcal{C}), \qquad \Coalg(\mathcal{C}, \times) \cong \mathcal{C}. \]
We will consider the non-connected case (cf.\@ \ref{PARNONCONNECTED}) right away. 
The connected case can be obtained as the special case $X = \cdot$. 
Let $X$ be an object of $\mathcal{C}$ which we can see as ``coalgebra'' in this case. $\mathrm{Bimod}_X(\mathcal{C})$ is, in this case, nothing but 
\[ \mathcal{C}_{/X \times X} \] 
with product $M_1, M_2 \mapsto M_1 \times_{\pr_2, X, \pr_1} M_2$.
Objects in $\Alg(\mathrm{Bimod}_X(\mathcal{C})_{X/}) = \Alg(\mathrm{Bimod}_X(\mathcal{C}))$  are thus {\bf category objects} in $\mathcal{C}$ in the usual sense with objects $X$. 
In the adjunction 
\[ \xymatrix{ \mathrm{Bimod}_{X}(\mathcal{C})_{X /} \ar@<3pt>[r]^-{Q}  & \ar@<3pt>[l]^-{R} \mathcal{C}_{X /} }    \]
the functor $Q$ is given by forgetting the morphism $M \to X \times X$ (but remembering the unit = coaugmentation)  and $R$ associates with $X \to Y$ the ``free bimodule'' $X \times Y \times X$ with obvious augmentation 
$X \to X \times Y \times X$.

\begin{PROP}\label{PROPLOOP}
We have for $X \to Y$ in the category $\mathcal{C}_{X /}$  
\[ \boxed{ \cobarlurie_R(X \to Y) = X \times_Y X } \]
with product given by the canonical ``concatenation'' operation:
\[  (X \times_Y X) \times_X (X \times_Y X) \cong X \times_Y X \times_Y X \to X \times_Y X. \]
\end{PROP}
In particular, for $X= \cdot$, we obtain the functor $\mathcal{L}$ (loop space). 

Consider the morphism $\rho: \Delta \cong \Delta_{\act}^{\op} \subset \Delta^{\op}$ (cf.\@ \ref{PARRHOSTAR}). Considering $[n]$ as discrete category $\{0, \dots, n\}$, $\rho^{\op}$ gives rise to 
a fibration $p: \overline{\Delta} \to \Delta$ with fiber over $[n]$ being $\{0, \dots, n+1\}$.
There is a morphism 
\[ \xi:  \overline{\Delta} \to \righthalfcup  \]
mapping fiberwise $0$ to $(0,1)$, $n+1$ to $(1,0)$ and everything else to $(1,1)$.
\begin{LEMMA}\label{LEMMAXIFINAL} The morphism $\xi$ is $\infty$-final. 
\end{LEMMA}

\begin{proof}[Proof (Sketch) of Proposition~\ref{PROPLOOP}.]
Consider the non-connected cobar diagram (\ref{PARCOBARSETTING2FUNCT}) in this case:
\[  \xymatrix{  (\mathcal{C}_{X/}, \times)^{(\Delta_{\act}, \ast')}  \ar[r]^-R & (\mathrm{Bimod}_X(\mathcal{C})_{X/}, \times_X)^{(\Delta_{\act}, \ast')}   \\
\mathcal{C}_{X/} \cong \Coalg(\mathcal{C}_{X/}, \times)  \ar[u]_{\barconst^{\vee}} &   \ar[u]  \Alg(\mathrm{Bimod}_X(\mathcal{C})_{X/}, \times_X) } \]
where $\barconst^{\vee}$ is the dual classical bar construction or, in other words the functor $\widetilde{\Pi_2^*}$ from Definition~\ref{DEFCOBARLURIE}.
In the monoidal $\infty$-category $\mathrm{Bimod}_X(\mathcal{C})_{X/}$, we have (dual of Corollary~\ref{KORRHO}) because the unit $X$ is initial: 
\[ \rho^*: (\mathrm{Bimod}_X(\mathcal{C})_{X/})^{(\Delta, \ast)} \cong (\mathrm{Bimod}_X(\mathcal{C})_{X/})^{(\Delta_{\act}, \ast')} \]
and we have by construction of $\barconst^\vee$ --- for the underlying objects 
\[ (\rho^*)^{-1} \circ R \circ \barconst^{\vee} (X \to Y) =  p_* \xi^* \left( \vcenter{ \xymatrix{ & X \ar[d] \\ X \ar[r] & Y  } } \right) \]
(note that $p$ is a discrete fibration and thus $p_* \xi^*$ computes fiber-wise the product $X \times Y^n \times X$) 
and therefore
\begin{gather*} \cobarlurie_R(X \to Y) = \lim_{\Delta} \circ (\rho^*)^{-1} \circ R \circ \barconst^{\vee} (X \to Y) \\ 
\cong \lim_{\Delta} p_* \xi^*  \left( \vcenter{ \xymatrix{ & X \ar[d] \\ X \ar[r] & Y  } } \right) \cong \lim_{\righthalfcup}  \left( \vcenter{ \xymatrix{ & X \ar[d] \\ X \ar[r] & Y  } } \right) 
\end{gather*}
(Lemma~\ref{LEMMAXIFINAL}) or in other words:
\[ \cobarlurie_R(X \to Y) = X \times_Y X \]
and a bit more refined argument shows that the induced algebra structure is as claimed. 
\end{proof}

\begin{PAR}
Let $A \in \Alg(\mathrm{Bimod}_X(\mathcal{C})_{X/}, \times_X)$. Then $\barconst(A)$ is a coalgebra $B$ of shape $\Delta^{\op}_{\act}$  of the form
\begin{equation}
 \xymatrix{ \cdots  \ar@<4pt>[r]  \ar[r] \ar@<-4pt>[r] &A \times_X A\times_X A \ar@<2pt>[r]  \ar@<-2pt>[r] & A \times_X A \ar[r] & A  & \ar@{-->}[l] X } 
 \end{equation}
where the morphisms are the structure morphisms of the algebra
with canonical coalgebra structure
\[ B_{[i+j]} \cong B_{[i]} \times_X B_{[j]}. \]
Then apply the forgetful functor $Q$ and then $(\rho^*)^{-1}$ (in $\mathcal{C}_{X/}$ the unit becomes final!) to get a simplicial diagram (with redundant coalgebra structure by the trivial Eilenberg-Zilber Theorem~\ref{SATZEZTRIVIAL})
\[ \xymatrix{ \cdots  \ar@<8pt>[r]  \ar[r] \ar@<-8pt>[r] \ar@<4pt>[r]  \ar[r] \ar@<-4pt>[r] &A \times_X A\times_X A \ar@<2pt>[r]  \ar@<-2pt>[r]\ar@<6pt>[r]  \ar@<-6pt>[r] & A \times_X A \ar[r]  \ar@<4pt>[r] \ar@<-4pt>[r]  & A \ar@<4pt>[r] \ar@<-4pt>[r] & \ar@{-->}[l] X. } \]
This is a simplicial diagram in $\mathcal{C}_{X/}$, not in $\Alg(\mathrm{Bimod}_X(\mathcal{C}))$!
Examining the definitions, one can easily determine the morphisms. For instance, the two morphisms $A \to X$ are given by the left and right comodule structure, i.e.\@ by the given morphism $A \to X \times X$ (which are not bimodule morphisms themselves, for instance). 
$\barlurie_Q(A)$ is then the colimit over this diagram. 
\end{PAR}

\begin{BEM}
If $\mathcal{C} = \Gpd_{\infty}$ and for $X = \cdot$, we have $1 \times_Y 1  = \Hom_Y(1, 1)$, where $1 \in Y$ denotes the distinguished object given by the coaugmentation with 
the product given by composition. 
\end{BEM}

\begin{KOR}[Connected case, cf.\@ also {\cite[Corollary~5.2.2.13]{Lur11}}] 
There are  adjunctions
\[ \xymatrix{  \mathcal{C}  \ar@/^20pt/[rrrr]^{\Sigma} \ar@<3pt>[rr]^-{F}  & & \ar@<3pt>[ll] \mathrm{Mon}(\mathcal{C})  \ar@<3pt>[rr]^-{\barlurie}  &  & \ar@<3pt>[ll]^-{\cobarlurie} \mathcal{C}  \ar@/^20pt/[llll]^{\mathcal{L}}  }\]
where $\Sigma: X \mapsto 1 \amalg_{X} 1$ is the suspension, i.e.\@ the left adjoint of $\mathcal{L}: X \mapsto 1 \times_X 1$, and $F$ is the free monoid functor. 
\end{KOR}

Of course dual statements are true for a coCartesian monoidal $\infty$-category $(\mathcal{C}, \amalg)$, i.e.\@ where $\amalg$ is the coproduct.

\begin{PAR}
Consider the diagram
\[ \xymatrix{ ((\mathrm{Bimod}_X(\mathcal{C})_{X/}, \times_X)^{\vee})^{\twop \OOO} \ar[d]^{Q} \\
 ((\mathcal{C}_{/X}, \times)^{\vee})^{\twop \OOO} & \ar[l]_-{\rho^*}^-{\sim} ((\mathcal{C}_{/X}, \times)^{\vee})^{(\Delta, *)^{\op}} \ar[r]^{\sim} & \Coalg(\mathcal{C}_{/X}^{\Delta^{\op}}, \times) \ar[r]^-{\sim} &  \mathcal{C}_{/X}^{\Delta^{\op}} }   \]
where the middle equivalence is given by the trivial Eilenberg-Zilber Theorem~\ref{SATZEZTRIVIAL}. 
Unraveling the definitions, we see that the composition induces an isomorphism: 
\[ ((\mathrm{Bimod}_X(\mathcal{C})_{X/}, \times_X)^{\vee})^{\twop \OOO, 2-\cart} \cong  (\mathcal{C}_{X/}^{\Delta^{\op}})_{\mathrm{Segal}, 0} \] 
where the index  ``0'' means the full subcategory in which $X \to Y_{[0]}$ is an isomorphism
and where ``$\mathrm{Segal}$'' means the full subcategory of the $Y$ satisfying the {\bf Segal condition}, that is, the induced morphism
\[ Y_{[n+m]} \to Y_{[n]} \times_{Y_{[0]}} Y_{[m]}     \]
is an isomorphism for all morphisms $[n+m] = [n] \ast' [m] \to  [n] \ast [m]$ given by the canonical degeneracy. In particular $Y_{[n]} \cong Y_{[1]} \times_{Y_{[0]}} \cdots \times_{Y_{[0]}} Y_{[1]}$. Notice that $((\mathrm{Bimod}(\mathcal{C})_{X/}, \times_X)^{\vee})^{\twop \OOO, 2-\cart}$ is precisely the image of the (fully faithful) classical bar construction 
in this case and thus isomorphic to $\Alg(\mathrm{Bimod}(\mathcal{C})_{X/}, \times_X)$ (category objects).
In the connected case, this implies: 
\end{PAR}

\begin{PROP}
The connected (i.e.\@ for $X=\cdot$) classical (co)bar adjunction  has the following form in the Cartesian case: 
\[ \xymatrix{ \mathrm{Mon}(\mathcal{C}_{\cdot/})  \ar[rr]^-{\barconst}_-{\sim}  & &  (\mathcal{C}_{\cdot/}^{\Delta^{\op}})_{\mathrm{Segal}, 0} \ar@{^{(}->}[rr] & & \ar@/^20pt/[llll]^{\cobarconst}   \mathcal{C}_{\cdot/}^{\Delta^{\op}} } \]
\end{PROP}
Note that the left adjoint ``$\cobarconst$'' exists by Theorem~\ref{THEOREMEXISTENCECOBAR}, if $\mathcal{C}$ has geometric realizations and $\times$ commutes with them in each variable, which is true in any $\infty$-topos, for example. 
The Proposition is a special case of results of \cite{BS23} (see also \cite{HS23}).

\begin{PROP}
Via the identification $Q \circ (\rho^*)^{-1} \circ \barconst$ of category objects with Segal objects, we have for $X \to Y \in \mathcal{C}_{X/}$ 
\[ \boxed{  \cobarlurie_R (X \to Y)  \cong  \text{\v{C}ech nerve of $X \to Y$}  }  \]
and if $\mathcal{C}$ is an $\infty$-topos (e.g.\@ $\mathcal{C}=\Gpd_{\infty}$), the restriction of the derived (co)bar adjunction is an isomorphism
\[ \xymatrix{   (\mathcal{C}_{X/}^{\Delta^{\op}})_{\mathrm{Gpd}, 0}   \ar@<3pt>[rr]^-{\barlurie_Q}_{\sim}  &  & \ar@<3pt>[ll]^-{\cobarlurie_R} (\mathcal{C}_{X/})_{(-1)-\mathrm{conn}}   } \]
where $\mathrm{Gpd}$ indicates the full subcategory of $(\mathcal{C}_{X/}^{\Delta^{\op}})_{\mathrm{Segal}, 0}$ of groupoid objects\footnote{meaning, in the picture above, that, in addition to the Segal condition, $Y_{[n+m]} \to Y_{[n]} \times_{Y_{[0]}} Y_{[m]}$ is an isomorphism induced by {\em any} partition of $[n+m]$ sharing one vertex.   }
and where $(\mathcal{C}_{X/})_{(-1)-\mathrm{conn}}$ is the full subcategory of morphisms $X \to Y$ which are $(-1)$-connected. 
\end{PROP}
\begin{proof}
The first part is clear from Proposition~\ref{PROPLOOP} and the definitions. The fully-faithfulness of $\barlurie_Q$ on groupoid objects is thus actually part of the axioms of an $\infty$-topos. The characterization of
the essential image is in \cite[Proposiiton~6.2.3.15]{Lur09}). 
\end{proof}

\subsection{Using cobar to construct Kan extensions --- e.g.\@ the cofree coalgebra}\label{SECTCOFREECOALG}

In Section~\ref{SECTKAN3} an idea to construct relative (operadic) Kan extensions has been discussed, even in situations which are not $L$- resp.\@ $R$-admissible. 
This leads precisely to the (dual) classical cobar construction discussed in this chapter and is included as an illustration. It is otherwise not related to the main aims of these lectures.  

We will prove the following classical theorem using (the dual) cobar, noting that this method to construct problematic relative Kan extensions can be applied in far greater generality.
\begin{SATZ}
Let $(\mathcal{C}, \otimes)$ be an Abelian tensor category. Assume 
\begin{enumerate}
\item $\otimes$ is exact;
\item the mate
\[ (\prod_{i \in \N} X_i) \otimes Y  \to \prod_{i \in \N} (X_i \otimes Y) \]
 is a {\em monomorphism};
\item $\otimes$ commutes with countable intersections. 
\end{enumerate}

Then the forgetful functor
\[ \Coalg(\mathcal{C}, \otimes) \rightarrow \mathcal{C} \]
has a right adjoint (cofree coalgebra functor).
\end{SATZ}
\begin{proof}
Denote $c: \cdot \hookrightarrow \OOO^{\op}$. 
As in the proof of Proposition~\ref{PROPKAN} we have
an adjunction in the $2$-category $\Dia^{\op}(\mathcal{\OOO}^{\op})$ 
 \[ \xymatrix{  (\OOO^{\op} \times_{/\OOO^{\op}} \cdot, \pi_1)  \ar@<3pt>[rr]^-{\widetilde{\pi_2}=(\pi_2, S(\mu))} & &  \ar@<3pt>[ll]^-{(\iota, \id)}  (\cdot, c) } \]
 and an isomorphism
 \[ c_* \cong \pi_{1,*}\, \iota_*^{(\pi_1)}   \]
 in the strong sense that the existence of the RHS adjoints implies the existence of the LHS. By the adjunction, we see that 
 \[ \iota_*^{(\pi_1)} = S(\mu)^{\bullet} \pi_2^*. \]
 We have $\OOO^{\op} \times_{/\OOO^{\op}} \cdot \cong (\N_0, +)^{\vee}$ as discrete cooperad and $X':=S(\mu)^{\bullet} \pi_2^* X$ is the object
 \[ X'_{n} = X^{\otimes n} \]
 with the coalgebra structure given by the isomorphisms
 \[ X'_{i+j} \cong X'_i \otimes X'_j   \qquad  X'_0 \cong 1. \]
A right relative (or operadic) Kan extension $\pi_{1,*}$ exists in general only if $\otimes$ commutes with infinite products. 
However, 
in \ref{SECTKAN3} we saw that
 it suffices to construct a (dual) cobar functor for the object $X'$ pulled back to $(\tw \OOO) \times_{\OOO^{\op}} (\N_0, +)^{\vee}$ 
 and then right Kan extended along the fibration (!) of operads
 \[ (\tw \OOO) \times_{\OOO^{\op}} (\N_0, +)^{\vee} \to \tw \OOO. \]
 
 By the dual of Theorem~\ref{THEOREMEXISTENCECOBAR}, we have $(\mathcal{C}, \otimes)^{\tw \OOO} = ((\mathcal{C}, \otimes)^{\vee})^{\twc \OOO, 3-\cart}$, and the dual cobar is given by the Kan extension $\pi_{3,*}$ which is computed fiber-wise, i.e.\@ as limit over $((\mathcal{C}, \otimes)^{\vee})^{\twc \OOO}_{[1]} \cong \mathcal{C}^{\tw \Delta^{\op}_{\act}}$, {\em if $\otimes$ commutes with the relevant limits}. We will argue that the same colimit works in this case as well without that hypothesis. 
 The corresponding object $X''' \in \mathcal{C}^{\tw \Delta^{\op}_{\act}}$ is the following:
 \[ \left( [n_1] \ast' \cdots \ast' [n_m] \leftarrow [m] \right) \  \mapsto \  X''_{[n_1]} \otimes \cdots \otimes X''_{[n_m]} \]
 (the $[n_1] \ast' \cdots \ast' [n_m]$ is the induced decomposition, cf.\@ \ref{PARNECKLACE}) with
 \[ X''_{[n]} = \prod_{k_1,\cdots,k_n} X^{\otimes k_1} \otimes \cdots \otimes X^{\otimes k_n}. \]
 Thus the dual of the cobar formula (\ref{PARCOBAREXPLICIT}) would give
\begin{gather}\label{cobardual}
 \boxed{ \cobarconst^{\vee}  X' =    
  \lim \left(   \coprod_{n=0}^{\infty} (X''_{[1]})^{\otimes n}  \rightrightarrows  \substack { \prod X''_{[1]} \otimes \cdots \times X''_{[2]} \otimes \cdots \otimes X''_{[1]} \\ \amalg \\ \coprod X''_{[1]} \otimes \cdots \times X''_{[0]} \otimes \cdots \otimes X''_{[1]}   } \right) } 
  \end{gather}
  From 2.\@ follows that the object $X'''$ has the property that type-2 morphisms go to monomorphisms. 
  Hence (\ref{cobardual}) is essentially a countable intersection. 
 
The general existence Theorem~\ref{THEOREMEXISTENCECOBAR} (better: its dual) cannot be applied as it stands because $\otimes$ does not commute with limits (here countable products would be sufficient, which we do not want to assume either).
Examining the proof, we see that this is used at two places:
\begin{enumerate}
\item to construct the Kan extension along $\pi_3: \twc \OOO \to \OOO^{\op}$. 
However the relevant object $X'''$ has the property that type-2 morphisms go to monomorphisms. To have a fiber-wise Kan extension the limit over the fiber must commute with $\otimes$. This is Lemma~\ref{LEMMACOFREECOALG}, 2.\@
\item to see that the (fiber-wise for $\pi_{3}$) Kan extension $\pi_{1,*}: \tw \Delta^{\op}_{\act} \to \Delta_{\act}$ commutes with $\otimes$. This is Lemma~\ref{LEMMACOFREECOALG}, 1.\@ \qedhere
\end{enumerate}
\end{proof}

\begin{LEMMA}\label{LEMMACOFREECOALG}
Let $(\mathcal{C}, \otimes)$ be an Abelian tensor category with $\otimes$ exact and commuting with countable intersections. 
If $\mathcal{E} \in \Ab^{\tw \Delta^{\op}_{\act}}$ maps type-2 morphisms to monomorphisms
then
\begin{enumerate}
\item the Kan extension $\mathcal{E} \mapsto \pi_{1,*} \mathcal{E} \in \Ab^{\Delta_{\act}}$ commutes with $- \otimes X$ for any $X \in \mathcal{C}$,
\item the limit $\mathcal{E} \mapsto \lim \mathcal{E}$ commutes with $- \otimes X$ for any $X \in \mathcal{C}$.
\end{enumerate}
\end{LEMMA}
\begin{proof}Exercise.
\end{proof}

\begin{BEM}
In \cite{Ane14}, for instance, it is shown for a large class of Abelian tensor categories that, in fact, 
already
\[  \cobarconst^{\vee} X'''  = \lim \left( \vcenter{ \xymatrix{  & X''_{[1]} \ar[d] \\
X''_{[1]} \otimes X''_{[1]} \ar@{^{(}->}[r] & X''_{[2]} } } \right). \]
\end{BEM}

\section{Non-Abelian foundations} \label{CHAPTERNONAB}

This chapter discusses basic properties of simplicial sets and simplicial objects in categories culminating in the non-Abelian Eilenberg-Zilber Theorem~\ref{SATZEZ}.
It deals almost exclusively with 1-categories. Some facts from Sections~\ref{SECTSIMPLEX}--\ref{SECTBASICCOOP} have already been used in Chapter~\ref{CHAPTERABSTRACT}.

\subsection{The simplex category}\label{SECTSIMPLEX}

\begin{PAR}
Recall from \ref{PARSIMPLEX} the definition of the simplex category, the  notions of active and inert morphism, and the definition of the concatenation products $\ast$ and $\ast'$.
\end{PAR}

\begin{PAR}\label{PARFILT}
Let $\mathcal{C}$ be a Barr-exact category (\cite{Bar71}, we will here only be interested in $\Set_{\bullet}$ and Abelian categories), and let $X \in \mathcal{C}^{\Delta^{\op}}$.
Then there is a filtration by subobjects 
\[ F^0 X_{[n]} \subset  \cdots \subset F^n X_{[n]} = X_{[n]} \]
such that each $F^k X_{[n]}$ is the smallest subobject over which all $X(\nu)$ for degeneracies $\nu: [m] \to [n]$ with $m \le k$ factor. We define
\[ X_{[n]}^{\nd} := X_{[n]} / F^{n-1}X_{[n]}. \]
the quotient of {\bf non-degenerate elements}. In fact, the $F^k X_{[n]}$ assemble to simplicial objects $F^k X$ themselves and are called the {\bf skeletal filtration}.

For any $n$ and $k$, consider a subset $S \subset \{s: [k] \twoheadleftarrow [n]\}$ of the degeneracies. Since each such degeneracy can be written 
uniquely as $s = s_{i_1} \cdots s_{i_{n-k}}$ with $i_1 < i_2 < \cdots < i_{n-k}$, the lexicographic ordering of the $i_j$ defines a total ordering on $S$.
\end{PAR}

\begin{LEMMA}\label{LEMMAEZ1}
Fix integers $m$ and $n$.
For each subset $S \subset \{s: [m] \twoheadleftarrow [n]\}$ of degeneracies, there is a face morphism $D: [n] \hookleftarrow [m]$ such that $D\,s = \id$ for $s \in S$ precisely if $s$ is the smallest 
element of $S$. 
\end{LEMMA}
\begin{proof}
Let $s_{i}$ be the smallest degeneracy occurring as a left-most factor of  $s \in S$ and decompose
\[ S =  s_{m} S' \cup S'' \]
where $S'$ is a set of degeneracies $[m+1] \twoheadleftarrow [n]$ and $S''$ is a set of degeneracies in which only $s_{j}$ for $j > i$ occur. 
By induction, there is a face $D': [n] \hookleftarrow [m+1]$ which is a section of the first element of $S'$ and not of any other. 
Then $D:=D' d_i$ does the job, because $D' d_i s_i s = D' s = \id$ precisely for the first element of $s \in S'$ and not for any other, while $D' d_i s_{j_1} \cdots s_{j_m} s = D' s_{j_1-1} \cdots s_{j_m-1} d_i$ for $j_k>i$ cannot be the identity. 
\end{proof}

\begin{PAR}\label{PAREZWEAK}
The Lemma is usually stated in a weaker form saying ``the set of sections of a degeneracy $s: [m] \twoheadleftarrow [n]$ is non-empty and determines $s$''.
This follows obviously from the previous Lemma, because it shows non-emptyness for $\#S = 1$, and, for $\# S = 2$, that for each pair of degeneracies there is a section of the first which is not a section of the second. 
\end{PAR}

\begin{LEMMA}[Eilenberg-Zilber]\label{LEMMAEZ2}
If $\mathcal{C}$ is Barr-exact and has a zero object (we will be interested in $\Set_{\bullet}$ and Abelian Categories only), then 
\begin{equation}\label{eqgraded} \gr^i X_n \cong \coprod_{s: [m] \twoheadleftarrow [n]} X_m^{\nd} \end{equation}
where the coproduct runs over all degeneracies. 
\end{LEMMA}
\begin{proof}
We will show this for $\mathcal{C}$ Abelian (by suitably embedding a suitable small subcategory of $\mathcal{C}$, this actually suffices, but we won't discuss the technical details).
Consider the morphism 
\[ \bigoplus_{s: [m] \twoheadleftarrow [n]} X_m^{\nd} \to F^m X_n / F^{m-1} X_n. \]
It is an epimorphism by definition of the filtration, so we have to see that it is a monomorphism. Let $K$ be the kernel. If it is non-trivial, consider a minimal subset $S \subset \{s: [m] \twoheadleftarrow [n]\}$ with the property that
\[K \subset \bigoplus_{s \in S} X_m^{\nd}. \]
By the Lemma, there is $D$ such that $D s = \id$ precisely for the smallest $s \in S$. All other $D s$ factor thus through a degeneracy and hence the composition  with $D$ 
\[ \bigoplus_{s \in S} X_m^{\nd} \to F^m X_n / F^{m-1} X_n \to X_m^{\nd} \]
maps the first (w.r.t.\@ the lexicographic ordering) $X_m^{\nd}$ isomorphically to $X_m^{\nd}$ and all others to zero\footnote{Note that $D$ must map $F^{i-1} X_n$ to $F^{i-1}X_m$ and hence the morphism is well-defined.}. Thus 
\[ K \subset \bigoplus_{s' \in S \setminus \{s\} } X_m^{\nd} \]
in contradiction to the minimality. This shows that $K$ is trivial. 
\end{proof}

Let $X \in \Set^{\Delta^{\op}}$ be a simplicial set. 
Applied to $\mathcal{C}:=\Set_{\bullet}$, the category of pointed sets, and $X_{\bullet}$, the statement of Lemma~\ref{LEMMAEZ2} translates to
\begin{equation*} X_n = \coprod_{m} F^m X_n \setminus F^{m-1} X_n \cong \coprod_{[m] \to [n]} \underbrace{X_m \setminus F^{m-1}X_m}_{=:X_m^{\nd}}, \end{equation*}
i.e.\@ that every element in 
$X_n$ is in the image of the non-degenerate elements of $X_m$ for a unique degeneracy $[m] \twoheadleftarrow [n]$. (This follows even more directly from the weaker form \ref{PAREZWEAK} of Lemma~\ref{LEMMAEZ1}.)

The properties discussed here are shared by other diagrams than $\Delta^{\op}$ as well. These are condensed, for example, in the notions of {\bf elegant Reedy category} and {\bf Eilenberg-Zilber category} \cite{BM11, BR13}.

\subsection{Some basic (co)operads}\label{SECTBASICCOOP}

This section discusses the following zoo of very basic ($\infty$-exponential) operads over $\OOO$:

\begin{equation}\label{eqfundamentalop} \vcenter{ \xymatrix{ & (\FinSet_{\emptyset}, \coprod) \ar@{<-}[dd]^(.35)i & \\
(\Delta_{\emptyset}, \coprod) \ar[ru]^{\iota_{\delta}} \ar@{<-}[dd]^i  &  & (\Delta_{\emptyset}, \ast) ={}^{\downarrow \downarrow} \OOO  \ar[lu]_{\iota_{\dec}} \ar@{<-}[dd]^i \ar[ll]^(.35){\mathrm{forget}}   & (\Delta_{\emptyset}^{\op}, \ast) =   \tw \OOO  \ar[ldd]^{\rho^{\op}}  \\
& (\FinSet, \coprod) & \\
(\Delta, \coprod) \ar[ru]^{\iota_{\delta}} &  & (\Delta, \ast) \ar[lu]_{\iota_{\dec}}  \ar[ll]^{\AW^{\op} =  \mathrm{forget}}
} } 
\end{equation}
These operads (or their opposites) will occur in various related and unrelated places, most notably:
\begin{enumerate}
\item The cooperad $(\Delta, \ast)^{\op}$  describes the canonical simplicial enrichment on simplicial objects in a category $\mathcal{C}$ in a different way, which is very convenient to 
construct homotopies. See, in particular, Proposition~\ref{PROPTWOENRICHMENTS}.
\item The cooperads $(\Delta, \ast)^{\op}$ and $(\Delta, \coprod)^{\op}$ give, via Day convolution, rise to the two different tensor products $\otimes$ (point-wise for simplicial objects) and $\tildeotimes$ (usual tensor product of complexes)
on simplicial objects = non-negatively graded complexes in Abelian categories (cf.\@ \ref{EXDAY2}). The forgetful map gives rise to the Alexander-Whitney map. In fact, the diagram above extends (cf.\@ Theorem~\ref{SATZCOHEZ}) in the case of $\Ab$-enriched operads so as to include a model
of the Eilenberg-Zilber map as well. 
\item The morphism $\rho^{\op}$ can be described by generators and relations (cf.\@ Lemma~\ref{LEMMARHO}) and plays an important role in the (co)bar construction (cf.\@ Theorem~\ref{THEOREMLURIE}), more precisely, it takes care of the augmentation. 
\item The product in $(\Delta, \ast)^{\op}$, also denoted $\dec$ (decalage), connects (via pullback $\dec^*$) the classical cobar construction (Definition~\ref{DEFCOBARCLASS}) with the geometric cobar construction (or Kan's loop group), and Adams cobar construction, respectively. Dually, 
its adjoint $\dec_*$ (Artin-Mazur codiagonal, or total complex in the Abelian case) connects the classical bar construction with the classifying space, and Eilenberg-MacLane bar construction, respectively. 
\end{enumerate}

\begin{PAR}\label{PARIOTACOCART}
In the diagram (\ref{eqfundamentalop}) the $i$'s are $\infty$-Cartesian (for the right hand side $i$ this is Lemma~\ref{LEMMAIOTACART}, for the others it is trivial), the $\iota_{\delta}$'s are $\infty$-Cartesian, and the $\iota_{\dec}$'s are 1-coCartesian (for the upper trivial, it is even $\infty$-coCartesian, while for the lower, it boils down to the 1-finality of $\iota^{\op}$, Lemma~\ref{LEMMAIOTAFINAL}.)
\end{PAR}

\begin{PAR}
Consider the category $\FinSet_{\emptyset}$ of finite sets. It is monoidal w.r.t.\@ to the coproduct $\coprod$ giving rise to an operad
\[ (\FinSet_{\emptyset}, \coprod). \]
Since $\coprod: \FinSet_{\emptyset} \times \FinSet_{\emptyset} \to \FinSet_{\emptyset}$ has the right adjoint $\delta: \FinSet_{\emptyset} \to \FinSet_{\emptyset} \times \FinSet_{\emptyset}$
the morphism
\[ (\FinSet_{\emptyset}, \coprod) \to \OOO \]
is thus a {\em fibration and cofibration} of operads. 
On $\Delta_{\emptyset}$ (totally ordered finite sets) there is the concatenation product $\ast$ (obtained from $\coprod$ by choosing one of the two canonical total orderings) which is not symmetric anymore.
$\Delta_{\emptyset}$ can also be turned into an operad $(\Delta_{\emptyset}, \coprod)$ setting $\Hom(x, y; z) = \Hom(x, z) \times \Hom(y, z)$. This operad is not monoidal, of course, because $\Delta_{\emptyset}$ does not have a coproduct, but
\[ (\Delta_{\emptyset}, \coprod) \to \OOO \]
is also a {\em fibration} of operads. 
There are two morphisms of operads:
\[ \iota_{\delta}: (\Delta_{\emptyset}, \coprod) \to (\FinSet_{\emptyset}, \coprod)  \]
which is a morphism of fibrations (i.e.\@ maps Cartesian (active) morphisms to Cartesian morphisms) and
\[ \iota_{\dec}: (\Delta_{\emptyset}, \ast) \to (\FinSet_{\emptyset}, \coprod)  \]
which is a morphism of cofibrations (i.e.\@ maps coCartesian (active) morphisms to coCartesian morphisms).
\end{PAR}

The concatenation product $\ast$ restricts from $\Delta_{\emptyset}$ to $\Delta$, but the unit doesn't. 
However, $(\Delta, \ast)$ is still pro-monoidal (even in the $\infty$-categorical sense), or in other words: 
\begin{LEMMA}\label{LEMMAIOTACART}
The restriction
\[ i: (\Delta, \ast) \hookrightarrow (\Delta_{\emptyset}, \ast) \]
 is $\infty$-Cartesian and thus (cf.\@ Lemma~\ref{LEMMAFLAT}) $(\Delta, \ast) \to \OOO$ is an $\infty$-exponential fibration of operads. 
Furthermore, the morphism $(\Delta, \ast) \to \OOO$ is $\infty$-Cartesian. 
\end{LEMMA}

\begin{proof}
Let $\alpha: [n] \leftarrow [m] \in \Delta^{\op}_{\act}$ be an active morphism. 
For the pull-back pro-functors $\alpha^{\bullet}$ in $(\Delta, \ast)$, and $\alpha^{\bullet}_{\emptyset}$ in $(\Delta_{\emptyset}, \ast)$, we have
\[ \alpha^{\bullet} =  {}^t i_m \alpha^{\bullet}_{\emptyset} i_n.  \]
 Since the latter is a cofibration, the $\alpha^{\bullet}_{\emptyset}$ are of the form
${}^t \beta$. They are, in particular, determined by $(\delta_1)^{\bullet}_{\emptyset} = {}^t \dec_{\emptyset}$ and $(s_0)^{\bullet}_{\emptyset} = {}^t [\emptyset]$ where $\dec_{\emptyset} = \ast: \Delta_{\emptyset} \times \Delta_{\emptyset} \to \Delta_{\emptyset}$ is
the monoidal product and $[\emptyset]: [\emptyset] \hookrightarrow \Delta_{\emptyset}$ is the embedding.
$\dec_{\emptyset}$ restricts to a morphism $\dec: \Delta \times \Delta \to \Delta$. Thus $(\delta_1)^{\bullet} = {}^t\! \dec$.
Furthermore $[\emptyset]: [\emptyset] \hookrightarrow \Delta_{\emptyset}$ has the right adjoint $\pi_{\emptyset}: \Delta_{\emptyset} \to \cdot$ (because $[\emptyset]$ is initial). Therefore ${}^t [\emptyset] = \pi_{\emptyset}$.
This restricts to $\pi: \Delta \to \cdot$, i.e.\@ $(s_0)^{\bullet} = \pi$. $i$ being $\infty$-Cartesian boils down to 
\[  i\, {}^t\!\dec_{\infty}  \cong {}^t\!\dec\, i  \qquad \pi \cong   \pi_{\emptyset} \, i  \]
The second is trivial, while the first is shown (dually) in Lemma~\ref{LEMMAEXACT4}. 

The $\infty$-Cartesianity of $(\Delta, \ast) \to \OOO$ is trivial over degeneracies (because locally Cartesian morphisms exist) while for $\ast$, it is the $\infty$-finality (cf.\@ (dually) Lemma~\ref{LEMMAEXACT3}). 
\end{proof}

Dually, we obtain, that 
\[ (\Delta, \ast)^{\op} \to \OOO^{\op}  \]
is an $\infty$-exponential fibration of cooperads. 
Notice that $(\Delta_{\emptyset}, \ast)^{\op} = (\Delta_{\emptyset}^{\op}, \ast)^{\vee}$ is again the cooperad associated with a monoidal  category, namely $\Delta_{\emptyset}^{\op}$ with the dual structure. 
Restricting the {\em operad} $(\Delta_{\emptyset}^{\op}, \ast)$ to $\Delta^{\op}$ though gives an operad, which is somewhat unnatural, and is never considered.

\begin{LEMMA}[Duality between (finite) ordered sets and intervals]\label{LEMMADUAL}
There is an  equivalence of operads (or, equivalently, monoidal categories):
\[ (\Delta_{\act}^{\op},\ast') \cong (\Delta_{\emptyset}, \ast) \]
given by\footnote{The total ordering on $\Hom_{\Delta_{\emptyset}}([n], [1])$ is the natural one with minimal element the constant morphism $0$ and with maximal element the constant morphism $1$. The one on $\Hom_{\Delta_{\act}}([n], [1]) $ is its restriction.}
\[ [n] \mapsto \Hom_{\Delta_{\act}}([n], [1]) \cong [n-1]  \]
\[ [m] \mapsto \Hom_{\Delta_{\emptyset}}([m], [1]) \cong [m+1]  \]
This equivalence exchanges face and degeneracy maps. 
\end{LEMMA}
\begin{proof}
A combinatorial exercise. 
\end{proof}

\begin{PAR}\label{PARNECKLACE}
Recall from Section~\ref{SECTTW} the notions of (twisted) arrow operads $\tw I$ and ${}^{\downarrow\downarrow} I$ for an operad $I$. 
For $I=\OOO$ those are, being cofibered over $I$, in particular monoidal categories, and 
we have isomorphisms
\begin{equation} \label{necklace} \tw \OOO \cong (\Delta_{\act}, *') \cong (\Delta_{\emptyset}^{\op}, *) 
 \end{equation}
and also
\[ {}^{\downarrow\downarrow} \OOO \cong (\Delta_{\act}^{\op}, *') \cong (\Delta_{\emptyset}, *)  \]
in such a way that (for the first isomorphism) a multi-morphism $[n_0], \dots, [n_k] \to [m]$ in $(\Delta_{\act}, *')$ corresponding to
\[ [n_0] \ast' \cdots  \ast'  [n_k] \rightarrow [m]  \]
(using the arrow direction in $\Delta^{\op}$) is mapped to 
\[\xymatrix{ 
 [n_1] \ast' \cdots  \ast'  [n_k] \ar@{<-}[d] & \ar@{<-}[l]  [m] \ar@{<-}[d]    \\
[k]    \ar@{<-}[r] &  [1]  } \]
with the obvious maps.
In the other direction an active morphism in $\tw \OOO$ respresented by a  diagram
\[\xymatrix{ 
 [x] \ar@{<-}[d] & \ar@{<-}[l] [m]  \ar@{<-}[d]  \\
[k]  \ar@{<-}[r]  & [1]     } \]
determines a decomposition of $[x]$ into $[x_1] \ast' \cdots  \ast'  [x_k]$, and is mapped to
the corresponding multi-morphism $[x_1], \dots, [x_k] \to [m]$ in $(\Delta_{\act}, *')$.

The category (of operators) $(\tw \OOO)_{\act} \cong (\tw \Delta_{\act}) \cong \tw \Delta_{\emptyset}$ is essentially {\bf the category of necklaces},  cf.\@ \cite{DS11, Riv22, BS23}). It is itself 
monoidal, which can be expressed easiest by saying that  
\[ {}^{\uparrow \downarrow \downarrow} \OOO \to \OOO \]
(or dually $\twcop \OOO$) is cofibered, and hence the left hand side is monoidal and its fiber over $[1]$ is the category $(\tw \OOO)_{\act}$. 
The associated monoidal product is just $\ast'$ applied to source and destination of the morphism. 
All that was said can be translated replacing (via Lemma~\ref{LEMMADUAL}) $(\Delta_{\act}^{\op}, *')$ by the perhaps more tractable $(\Delta_{\emptyset}, *)$ everywhere.
For instance, the process of associating to 
\[ [x-1] \leftarrow [k-1] \]
a decomposition $[x_1-1], \dots, [x_k-1]$ is then just {\em collecting preimages} of the morphism in $\Delta_{\emptyset}$.

The  functor
\begin{align*}
 \Xi_{\emptyset}: \Delta_{\emptyset} &\to \Cat \\
 [n] &\mapsto (\Delta^{\op}_{\emptyset})^{n+1} 
\end{align*}
classifying the operad (\ref{necklace}) (as covariant functor) but composing into $\Cat^{\PF}$, it is rather associated with the cooperad functor:
\[ (\Delta_{\emptyset}, \ast)^{\op} \cong {}^{\uparrow \uparrow} \OOO \to \OOO^{\op} \]
\end{PAR}
There is also the functor
\begin{align*}
 \Xi: \Delta_{\emptyset} &\to \Cat^{\PF} \\
 [n] &\mapsto (\Delta^{\op})^{n+1} 
\end{align*}
where the transitions are the $\dec$ and the ${}^t\! \pr$. It is associated with the cooperad $(\Delta, \ast)^{\op}$.

\begin{LEMMA}\label{LEMMACOF}
All functors in the image of $\Xi_{\emptyset}$ are cofibrations with discrete fibers. 
\end{LEMMA}
\begin{proof}
They are of the form
\[ (\Delta_{\emptyset} \times_{/\Delta_{\emptyset}} [n])^{\op} \to (\Delta_{\emptyset} \times_{/\Delta_{\emptyset}} [m])^{\op} \]
given by composition with $[n] \to [m]$.
\end{proof}

\begin{PAR}\label{PARRHOSTAR}
 Denote by  $(\Delta_{\act}^{\op}, \ast')^{\vee}$ the associated cooperad, i.e.\@ 
 \[\Hom([n]; [k_0], \dots, [k_m]) = \Hom_{\Delta_{\act}^{\op}}([n], [k_0] \ast' \cdots \ast' [k_m]).  \]
The sets of 0-ary morphisms are $\Hom([n]; ) = \emptyset$ for $n > 0$ and $\Hom([0]; ) = \{ \rho_0 \}$ with $\rho_0$ corresponding to $\Delta_0 = \Delta_0$.
The transformation (\ref{eqcan}) induces (by composition) a functor of cooperads
\[ \rho:  (\Delta_{\emptyset}, \ast)^{\vee} \cong (\Delta_{\act}^{\op}, \ast')^{\vee} \rightarrow  (\Delta, \ast)^{\op}. \]
\end{PAR}

\begin{LEMMA}\label{LEMMARHO}
$(\Delta, \ast)^{\op}$ is obtained from $(\Delta_{\act}^{\op}, \ast')^{\vee}$ by adjoining a
0-ary morphism
 \[ \rho_1 \in \Hom([1]; ) \]
 subject to the relation that there be only one 0-ary morphism from any object: $\Hom([n]; ) = \{ \rho_n \}$  (including  $n=0$). 
\end{LEMMA}
Notice that there are new 1-ary morphisms introduced, for example the unique morphism in $\Hom([1]; [1], [1])$ composed with $\rho_1$ at one of the slots. 
 
\begin{proof}
In the category obtained from $(\Delta_{\act}^{\op}, \ast')^{\vee}$, as described in the Lemma, we have $\rho_n=\rho_1 p_n$ where $p_n: [n] \leftarrow [1]$ is the
morphism to the final object in $\Delta_{\act}^{\op}$ (also for $n=0$). 
The category 
thus has multi-morphisms
\[  \Hom([m]; [k_0], \dots , [k_n]) = \{ [m] \leftarrow [l_0] \ast' \cdots \ast' [l_{s}]  \}  \]
where the $k_i$ are an ordered subset of the $l_{i}$, such that no two consecutive elements are missing, and such that all missing $l_i$ are equal to $1$.
We can map it to the unique morphism in $\Hom_{(\Delta, \ast)^{\op}}([m]; [k_0], \dots, [k_n])$ corresponding to
\[  [m]  \leftarrow [l_0] \ast' \cdots \ast' [l_{s}] \leftarrow [k_0] \ast \cdots \ast [k_n]  \]
mapping the $[k_i]$ to the corresponding factor. If there is no $k_i$ then we map it to the unique 0-ary morphism in $\Hom_{(\Delta, \ast)^{\op}}([m]; )$.

On the other hand, consider an $n$-ary morphism in $\Hom_{(\Delta, \ast)^{\op}}([m]; [k_0], \dots, [k_n])$ given by
\[  \alpha:  [m] \leftarrow [k_0] \ast \cdots \ast [k_n]. \]
Now add factors $[1]$ everywhere, where the condition that the extremal points go to the same point, resp.\@ to the inital and final points, is violated: 
\[  [m] \leftarrow ([1]) \ast [k_0] \ast ([1]) \ast \cdots \ast ([1]) \ast  [k_n] \ast ([1])  . \]
This morphism factors
\begin{gather*}
   [m] \leftarrow    ([1]) \ast' [k_0] \ast' ([1]) \ast' \cdots \ast' ([1]) \ast'  [n_k] \ast' ([1]) \\
    \leftarrow ([1]) \ast [k_0] \ast ([1]) \ast \cdots \ast ([1]) \ast  [k_n] \ast ([1]) 
\end{gather*}
and thus corresponds to a multi-morphism in 
\[ \Hom_{(\Delta_{\act}^{\op}, \ast')^{\vee}}([m];  ([1]), [k_0], ([1]), \dots, ([1]),  [k_n], ([1])) \]
Compose this with at the various $[1]$ with the $\rho_1: [1] \rightarrow \emptyset$ adjoint, to yield a morphism
\[ [n] \to  [k_0], \dots, [k_n]. \]
Finally, the unique 0-ary morphism in $\Hom([n]; )$ in $(\Delta, \ast)^{\op}$ is mapped to $\rho_n$.
The two constructions obviously determine an isomorphism of cooperads. 
\end{proof}

\begin{KOR}\label{KORRHO}
For any monoidal $\infty$-category $\mathcal{C}$ in which the unit is a final object, $\rho^*$ induces an isomorphism
\[  (\mathcal{C}^{\vee})^{(\Delta, \ast)^{\op}} \cong  (\mathcal{C}^{\vee})^{(\Delta_{\act}^{\op}, \ast')^{\vee}}.   \]
\end{KOR}
\begin{proof}That is an obvious consequence of Lemma~\ref{LEMMARHO} in the 1-categorical case but needs some careful argument for $\infty$-categories that is omitted for the moment. 
\end{proof}

\begin{PAR}
$(\FinSet_{\emptyset}, \coprod)$ restricts to $(\FinSet, \coprod)$, not as cofibration over $\OOO$ (the unit is lacking) but as fibration, in particular still $\infty$-exponential. 
Similarly $(\Delta_{\emptyset}, \coprod)$ restricts to $(\Delta, \coprod)$ as a fibration and  $(\Delta_{\emptyset}, \ast)$ restricts to $(\Delta, \ast)$, which is neither a cofibration
nor a fibration, but $\infty$-exponential by Lemma~\ref{LEMMAIOTACART}. Furthermore there are obvious forgetful maps
$(\Delta_{\emptyset}, \ast) \to (\Delta_{\emptyset}, \coprod)$ etc.\@
\end{PAR}

\begin{PAR}\label{EXDAY2} Cf.\@ also Example~\ref{EXDAY}. Let $\mathcal{C}$ be a finitely complete category with monoidal product $\otimes$ commuting with finite limits. It defines a cooperad
\[ (\mathcal{C}, \otimes)^{\vee}. \]
We will be interested in the Day convolution cooperads
\[  D((\Delta, \ast)^{\op}, (\mathcal{C}, \otimes)^{\vee}) \qquad \text{and} \qquad
  D((\Delta, \coprod)^{\op}, (\mathcal{C}, \otimes)^{\vee}). \]
By Proposition~\ref{PROPDAY}, those are again cooperads on $\mathcal{C}^{\Delta^{\op}}$ associated with monoidal products
\[ - \tildeotimes - := \dec_* - \boxtimes - \qquad  \text{and} \qquad
 - \otimes - := \delta^* - \boxtimes -, \]
while in both cases the unit is given by $\pi^* 1$, where $1$ is the unit in $\mathcal{C}$. Note that finitely complete is sufficient for the monoidality, because $\dec_*$ is effectively computed by finite limits (see Proposition~\ref{PROPEXPLICIT} and also \ref{PROPEXPLICITAB}). 
The product $ - \otimes - $ is just the point-wise extension of $\otimes$ on $\mathcal{C}^{\Delta^{\op}}$.
If $(\mathcal{C}, \otimes)$ is an Abelian tensor category, then $- \tildeotimes -$ translates via Dold-Kan into the usual tensor product of complexes (and commutation with finite products is sufficient). This will be discussed at length in Section~\ref{SECTEZAB}.
\end{PAR} 
\begin{PAR}
The cooperad $(\mathcal{C}^{\Delta^{\op}}, \otimes)^{\vee}$ is always fibered over $\OOO^{\op}$ (i.e.\@ comes from a monoidal product), regardless of any properties of $(\mathcal{C}, \otimes)$. In fact, this is the situation of \ref{EXDAY2}.
We can thus equally well describe the {\em operad}
\[ (\mathcal{C}^{\Delta^{\op}}, \otimes)  = D((\Delta^{\op}, \coprod), (\mathcal{C}, \otimes)) \]
as a Day convolution. 
Notice that $((\Delta, \coprod)^{\op})^{\vee} = (\Delta^{\op}, \coprod)$.
The {\em operad} $(\mathcal{C}^{\Delta^{\op}}, \tildeotimes)$, however, cannot be described as a Day convolution\footnote{Notice that it involves morphisms departing from a {\em right} Kan extension.}. 
\end{PAR}

\subsection{Canonical enrichments in pre-sheaves}\label{SECTALT}

Let $I$ be a small category and
let $\mathcal{C}$ a category. We have recalled in
 \ref{PARENRICHMENTPRESHEAVES} that $\mathcal{C}^{I^{\op}}$ has a canonical enrichment in pre-sheaves
 $\Set^{I^{\op}}$ (with particularly easy left tensors --- if $\mathcal{C}$ is cocomplete --- namely just given by the point-wise tensoring of $\mathcal{C}$ with $\Set$).
 There is a second description of the same enrichment, which will be very convenient in Section~\ref{SECTACYCLIC} to construct homotopies, as, for example, the homotopy between $\Ez \Aw$ and $\id$.
 In these applications, we will always have $I = \Delta$ and the second description amounts to the fact that 
 \[ \boxed{ \uHom(X, Y)_{[n]} \cong \Hom_{\mathcal{C}^{(\Delta^{\op})^{n+1}}}(\dec_{n+1}^* X, \dec_{n+1}^*Y) } \]
 and thus (under suitable (co)completeness assumptions)
 $\Delta_n \otimes X$ can also be described as $\dec_{n+1,!} \dec_{n+1}^* X$ with
the functoriality in $\Delta_n$ (essentially) given by certain adjunction units, and $\mathcal{HOM}(\Delta_n, X)$ as $\dec_{n+1,*} \dec_{n+1}^* X$.
 This second description is well-known (cf.\@  \cite{OR20}) but seems not to be widely used in the literature.

\begin{DEF}\label{DEFENRICHMENTPRESHEAVES}
Let $I$ be a small category with final object $i$. 
Consider a functor
\[ F: I^{\op} \to \Cat \]
and define
\begin{eqnarray*} 
\uHom_F: I^{\op} \times F(i)^{\op} \times F(i) &\to& \Set \\
 (j, X, Y) &\mapsto& \Hom_{F(j)}(F(p_j) X, F(p_j)Y)  
\end{eqnarray*} 
where $p_j: j \to i$ denotes the unique morphism to $i$.
\end{DEF}

\begin{PROP}\label{PROPENRICH}
The $\uHom_F$ of Definition~\ref{DEFENRICHMENTPRESHEAVES}, defines an enrichment on $F(i)$  in $\Set^{I^{\op}}$ w.r.t.\@ the point-wise
product
\[ \Set^{I^{\op}} \times \Set^{I^{\op}} \to  \Set^{I^{\op}}. \]
If the functors w.r.t.\@ all $p_j$ for all objects in $I$ have a left (resp.\@ right adjoint), and $F(i)$ is cocomplete (resp.\@ complete) then this is left, resp.\@ right tensored. 

More generally a functor
\[ F: I^{\op} \to \Cat_{/J} \qquad\text{or}\qquad  F: I^{\op} \to \mathrm{(co)Op}_{/J} \]
 gives rise to a $\Set^{I^{\op}}$ enriched category or (co)operad $F(i) \to J$. 
 If the $F(i) \to J$ are (co)fibrations, then this only yields again a (co)fibration of simplicially enriched categories (i.e.\@ induces a structure of simplicially enriched functor on the pull-backs, resp.\@ push-forwards) if $F$ has
 values in (co)fibrations and (co)Cartesian functors. 
\end{PROP}
\begin{proof}
For fixed $j$ the morphisms in $\uHom_F(-,-)$ can obviously be composed and the composition is functorial in $j$.
If all $F(p_j)$ have left adjoints $G(p_j)$, then we can define 
\begin{align*}
 I \times F(i) &\to F(i) \\
 j, X &\mapsto j \ltimes X := G(p_j)  F(p_j) X 
\end{align*} 
This gives rise to a left tensoring defined by
\begin{align*} \Set^{I^{\op}} \times F(i) &\to F(i) & \\
 C, X &\mapsto \int^i C(i) \times (i \ltimes X)  & & \qedhere
\end{align*} 
\end{proof}

\begin{PAR}
If the $F(p_j)$ have left adjoints $G(p_j)$, then the isomorphisms
$F(\alpha) F(p_k) \Rightarrow F(p_j)$ for $\alpha: j \to k$ yield morphisms $j \ltimes X \to k \ltimes X$  by means of applying Yoneda to the commutative diagram
\[ \xymatrix{ \Hom(F(p_k)X, F(p_k)Y) \ar[r]^-{F(\alpha)} \ar[d]^{\sim} & \Hom(F(\alpha)F(p_k)X, F(\alpha)F(p_k)Y)  \ar[d]^{\sim} \\
\Hom(k \ltimes X, Y) \ar@{-->}[r] & \Hom(j \ltimes X, Y) } \]
Note that the upper horizontal arrow is the functoriality of the construction in the last Proposition. 
These morphisms can be described equivalently as follows: 
\begin{enumerate}
\item 
$G(p_k) F(p_k) \cong G(p_k) F(\alpha) F(p_j) \Rightarrow   G(p_j) F(p_j)$
where $G(p_k) F(\alpha) \Rightarrow G(p_j)$ is the mate of $F(\alpha) F(p_k) \cong F(p_j)$.
\item If the $F(\alpha)$ also have left adjoints then this may  be described even more easily as
\[ G(p_k) F(p_k) \cong G(p_j) G(\alpha) F(\alpha) F(p_j) \Rightarrow G(p_j) F(p_j) \]
given by the counit $G(\alpha)F(\alpha) \to \id$.
\end{enumerate}
We leave the proof as exercise. 
\end{PAR}

\begin{PAR}\label{PARIFUNCT}
Let
\[ F, F': I^{\op} \to \Cat \]
be functors and assume that the $F(p_j)$ (resp.\@ $F'(p_j)$) have left adjoints $G(p_j)$ (resp.\@ $G'(p_j)$). 
A natural transformation
\[ \mu: F' \Rightarrow F \]
yields a natural transformation
\[ k \ltimes \mu X \rightarrow \mu (k \ltimes X) \]
defined as the composition
\[ \xymatrix{ 
G(p_k) F(p_k) \mu(k) \ar[r]^{\sim}  &  G(p_k) \mu(i) F'(p_k) \ar[r] & \mu(k) G'(p_k) F'(p_k) } \]
where the right hand side morphism is the mate of the isomorphism featuring in the natural transformation. 
\end{PAR}

\begin{LEMMA}\label{LEMMAIFUNCT}
The transformation $k \ltimes (\mu -) \Rightarrow \mu (k \ltimes -)$ defined in \ref{PARIFUNCT} is natural in $k$, even if not all $F(\alpha)$ have left adjoints. 
\end{LEMMA}
\begin{proof}
We have to show that the outer square in 
\[ \footnotesize \xymatrix{ 
G(p_k) F(p_k) \mu \ar[rr] \ar[d]^{\sim} & &  G(p_k) \mu F'(p_k) \ar[r] \ar[d]^{\sim} & \mu G'(p_k) F'(p_k) \ar[d]^{\sim} \\
G(p_k) F(\alpha) F(p_l) \mu \ar[d] \ar[r] & G(p_k) F(\alpha) i F'(p_l) \ar[d] \ar[r] & G(p_k) \mu F'(\alpha) F'(p_l)  \ar[r] & \mu G'(p_k) F'(\alpha) F'(p_l) \ar[d] \\
G(p_l) F(p_l) \mu \ar[r] & G(p_l) \mu F'(p_l) \ar[rr] & & \mu G'(p_l) F'(p_l) 
} \]
commutes. Here obviously everything commutes, except  
\[ \xymatrix{ 
 G(p_k) F(\alpha) \mu  \ar[d] \ar[r] & G(p_k) \mu F'(\alpha)  \ar[r] & \mu G'(p_k) F'(\alpha)  \ar[d] \\
 G(p_l) \mu  \ar[rr] & & \mu G'(p_l).
} \]
This is an exercise with mates, that is left to the reader. 
\end{proof}

\begin{PROP}\label{PROPENRICHTWARROW}
Let $I$ be a small category with a final object $i$ and let $\mathcal{C}$ be a category. 
Consider the functor
\begin{align*} f: I &\to \Cat  \\
 j &\mapsto (I \times_{/I} j)^{\op}
\end{align*}
classifying the fibration ${}^{\uparrow \uparrow} I \to I^{\op}$ (or, equivalently, the cofibration $\tw I \to I$). It yields via composition with $L_{\mathcal{C}}$ (see \ref{PARPFLR}) a functor
\[ F:=L_\mathcal{C} \circ f^{\op}: I^{\op} \to \Cat . \]
Notice that $\mathcal{C}$ does not have to be cocomplete, because the composition consists entirely of pull-backs. 
The enrichment in $\Set^{I^{\op}}$ obtained by this functor $F$ in Definition~\ref{DEFENRICHMENTPRESHEAVES} is naturally isomorphic to the canonical enrichment in pre-sheaves (\ref{PARENRICHMENTPRESHEAVES}) (i.e.\@ to the one right adjoint to the
{\em point-wise tensoring}, if the latter exists), given by
\begin{align*} (\mathcal{C}^{I^{\op}})^{\op} \times (\mathcal{C}^{I^{\op}}) &\to \Set^{I^{\op}} \\
 C, D &\mapsto j \mapsto \int_k \Hom(\Hom_I(k, j), \Hom_{\mathcal{C}}(C(k), D(k)).
\end{align*}
\end{PROP}
\begin{proof}
We have
\[ \Hom(p_j^* C, p_j^* D) \cong \lim_{k' \to k \to j \in \tw(I \times_{/I} j)} \Hom_{\mathcal{C}}(C(k), D(k')) \]
and there is a canonical isomorphism, functorial in $j$, $F$ and $G$:
\[  \lim_{k' \to k \to j \in \tw(I \times_{/I} j)} \Hom_{\mathcal{C}}(C(k), D(k')) \cong \lim_{k' \to k \in \tw I} \Hom( \Hom_I(k, j), \Hom_{\mathcal{C}}(C(k), D(k')) \]
because $\tw(I \times_{/I} j) \to \tw I$ is a fibration with discrete fibers $\Hom_I(k, j)$. 
\end{proof}

\begin{PAR}\label{PARASTSIMPLSTRUCTURE}
Applying Proposition~\ref{PROPENRICHTWARROW} to the functor (cf.\@ \ref{PARNECKLACE})
\[ \Xi_{\emptyset}: \Delta_{\emptyset} \to \Cat \]
with $\Xi_{\emptyset}([0]) = \Delta^{\op}_{\emptyset}$, classifying the monoidal cooperad $(\Delta_{\emptyset}, \ast)^{\op}$,
we see that the enrichment on $\mathcal{C}^{\Delta^{\op}_{\emptyset}}$ defined by
\[ [n] \mapsto \Hom(\dec_{\emptyset, n+1}^*X, \dec_{\emptyset, n+1}^* Y) \]
(i.e.\@ obtained by letting $F:=L_{\mathcal{C}} \circ \Xi_{\emptyset}^{\op}$ in Definition~\ref{DEFENRICHMENTPRESHEAVES})
 is the same as the canonical enrichment in $\Set^{\Delta^{\op}_{\emptyset}}$.

Now assume that $\mathcal{C}$ is cocomplete. 
We have seen (cf.\@ \ref{PARNECKLACE}), that there is a functor
\[ \Xi: \Delta_{\emptyset} \to \Cat^{\PF} \] 
with $\Xi([0]) = \Delta^{\op}$, classifying the $\infty$-exponential cooperad $(\Delta, \ast)^{\op}$. 
We have a transformation
\[ {}^t\!i: \Xi_{\emptyset} \Rightarrow \Xi \]
which is natural (and not only oplax) by Lemma~\ref{LEMMAIOTACART}.
By applying $L_{\mathcal{C}}$ it gives a functor
\[ L_{\mathcal{C}} \circ \Xi^{\op} \circ i: \Delta^{\op} \to \Cat \]
and a natural transformation
\[  L_{\mathcal{C}}(i):  L_{\mathcal{C}}  \circ \Xi^{\op} \circ i \Rightarrow L_{\mathcal{C}} \circ \Xi^{\op}_{\emptyset} \circ i \]
and the the enrichment on $\mathcal{C}^{\Delta^{\op}}$ defined by\footnote{Note that, although this seems similar to before, the face maps now act as partial colimits!}
\[ [n] \mapsto \Hom(\dec_{n+1}^*X, \dec_{n+1}^* Y) \]
(i.e.\@ obtained by letting $F:=L_{\mathcal{C}} \circ \Xi^{\op} \circ i$ in Definition~\ref{DEFENRICHMENTPRESHEAVES}) is {\em again} the same 
as the canonical enrichment in $\Set^{\Delta^{\op}}$:
\end{PAR}

\begin{PROP}\label{PROPTWOENRICHMENTS}
The structure on $\mathcal{C}^{\Delta^{\op}}$ obtained by letting $F:=L_{\mathcal{C}} \circ \Xi^{\op} \circ i$ in Definition~\ref{DEFENRICHMENTPRESHEAVES} is again naturally isomorphic to the canonical enrichment (\ref{PARENRICHMENTPRESHEAVES}) in pre-sheaves (i.e.\@, here, simplicial sets).
\end{PROP}

\begin{proof}Cf.\@ \cite{OR20} for a similar proof (for $n=1$). 
From Lemma~\ref{LEMMAIFUNCT} applied to the natural transformation $L_{\mathcal{C}}(i)$ (\ref{PARASTSIMPLSTRUCTURE}), we get a  morphism {\em functorial in $n$}
\[ [n] \ltimes (i_! -) = \dec_{\emptyset,n+1,!} \dec_{\emptyset,n+1}^* i_! \cong \dec_{\emptyset,n+1,!}  i_!  \dec_{n+1}^* \cong i_! \dec_{n+1,!}  \dec_{n+1}^* = i_! ([n] \ltimes -) \]
which is an {\em isomorphism} (because the mate is just the adjoint of the equality  $\dec_{n+1}^* i^* = i^* \dec_{\emptyset,n+1}^*$).
We have thus functorially in $[n] \in \Delta^{\op}$:
\begin{equation*} \begin{array}[b]{llll} 
\uHom_{\mathcal{C}^{\Delta^{\op}}}(X, Y)_{[n]} &\cong& \uHom_{\mathcal{C}^{\Delta^{\op}_{\emptyset}}}(i_! X, i_* Y)_{[n]}  & \text{Lemma~\ref{LEMMAHOM}, 2.} \\
&\cong& \Hom_{\mathcal{C}^{\Delta^{\op}_{\emptyset}}}([n] \ltimes (i_! X),  i_* Y) & \text{Proposition~\ref{PROPENRICHTWARROW}} \\
&\cong& \Hom_{\mathcal{C}^{\Delta^{\op}_{\emptyset}}}(i_! ([n] \ltimes X),  i_*  Y) & \text{as seen above}   \\
&\cong& \Hom_{\mathcal{C}^{\Delta^{\op}}}( [n] \ltimes X, Y) & \text{$i$ is fully-faithful.}  
\end{array} \qedhere    \end{equation*}
\end{proof}

\begin{BEM}
The enrichment of $\mathcal{C}^{\Delta^{\op}}$ in $\Set^{\Delta^{\op}}$ given by procedure \ref{DEFENRICHMENTPRESHEAVES} 
(i.e.\@ applied to $L_{\mathcal{C}} \circ \Xi^{\op}$ instead of $L_{\mathcal{C}} \circ \Xi^{\op} \circ i$)
is in fact an enrichment in $\Set^{\Delta_{\emptyset}^{\op}}$. It follows from
the definition, that we have
\[ \uHom(X, Y)_{[-1]} = \Hom_{\mathcal{C}}(\pi_! X, \pi_! Y). \]
\end{BEM}

 \begin{KOR}\label{KORHOMOTOPY}
 If a left tensoring exists, 
 the diagram of functors 
 \[ \xymatrix{ \id \cong \pr_{i,!} \dec^* \ar@<3pt>[r] \ar@<-3pt>[r] &  \dec_! \dec^* \ar[r] & \id } \]  
 is canonically equivalent to the diagram
 \[ \xymatrix{ \id \cong \Delta_0 \otimes - \ar@<3pt>[r]^-{\delta_0} \ar@<-3pt>[r]_-{\delta_1} & \Delta_1 \otimes - \ar[r]^-{s_0} & \Delta_0 \otimes - }. \]
 Furthermore, for a morphism $f: \dec^* X \to \dec^* Y$, denote by $\overline{f}$ the corresponding $\Delta_1 \otimes X \to Y$. Then the following is commutative 
  \[ \xymatrix{ \Delta_1 \otimes X \ar@/^20pt/[rrr]^-{\overline{fg}} \ar[r]_-{\delta \otimes X} & \Delta_1 \otimes \Delta_1 \otimes X \ar[r]_-{\Delta_1 \otimes \overline{g}} & \Delta_1 \otimes Y  \ar[r]_-{\overline{f}} & Z   } \]
  \end{KOR}

\begin{proof}
The transition pro-functors for $s_0: [1] \to [0]$ and $\delta_0, \delta_1: [0] \to [1]$ 
 in $\Xi$ are given by $\dec: (\Delta^{\op})^2 \to \Delta^{\op}$ and ${}^t\!\pr_i: \Delta^{\op} \to (\Delta^{\op})^2$, respectively, which are mapped by $L_{\mathcal{C}}$ to $\dec^*$, and $\pr_{i,!}$, respectively.
\end{proof}

\begin{PAR}\label{PARDECEXPLICIT}
The isomorphism (cf.\@ Corollary~\ref{KORHOMOTOPY})
\[ \Hom( \dec^* X, \dec^*Y) \cong \Hom(\Delta_1 \otimes X, Y) \]
 can be made very explicit:
A morphism
\[ F: \dec^*X \to \dec^*Y \]
which is in components given by
\[ F_{[i],[j]}: X_{[i]\ast[j]} \to Y_{[i]\ast[j]} \]
also determines morphisms $F_{[-1],[j]}$ and $F_{[i],[-1]}$ via the colimit. 

In turn, an element
\[ G: \Delta_1 \times X \to Y  \]
is given in degree $n$ by
\[ (G_{[n],i})_i: \Hom_{\Delta^{\op}}([1],[n]) \times X_{[n]} \to Y_{[n]} \]
where $i=-1, \dots, n$ index the elements of $\Hom_{\Delta^{\op}}([1],[n]) $.
If $F$ and $G$ correspond under the isomorphism above, we have explicitly:
\[ \boxed{ G_{[n],i} = F_{[i],[n-i-1]}}.   \]
This can be seen, for instance, from the formula 
\[ (\dec_! K)_{[n]} = \coprod_{\substack{i+j=n-1\\i,j \ge -1}} K_{[i],[j]} \] established in Proposition~\ref{PROPEXPLICIT} below. 
\end{PAR}

There is a monoidal version of the above: 

\begin{PROP}\label{PROPSIMPLICIALOPERAD}
Let $(\mathcal{C}, \otimes)$ be a cocomplete monoidal category such that $\otimes$ commutes with colimits. 
The category of simplicial objects is equipped with the point-wise monodial structure. The corresponding operad $(\mathcal{C}^{\Delta^{\op}}, \otimes)$ can also be seen as the 
Day convolution $D((\Delta^{\op}, \coprod), (\mathcal{C}, \otimes))$, and the corresponding cooperad $(\mathcal{C}^{\Delta^{\op}}, \otimes)^{\vee}$ as the Day convolution $D((\Delta, \coprod)^{\op}, (\mathcal{C}, \otimes)^\vee)$ (cf.\@ \ref{EXDAY2}).
The simplicially enriched structure on $\otimes$ given by
\begin{equation}\label{eqenrich} K \otimes (X \otimes Y) \to   (K \otimes X) \otimes (K \otimes Y) \end{equation}
induced by the diagonal $K \to K \times K$ (all $\otimes$ are computed point-wise) can be equally given (using the procedure in Definition~\ref{DEFENRICHMENTPRESHEAVES} above) by a 
functor:
\[ \Xi: \Delta_{\emptyset} \to \Hom^{1-\oplax, 1-\inert-\pseudo}(\OOO^{\op}, (\Cat^{\PF}_{}, \times)^{\vee})   \]
with $\Xi([0])$ classifying $(\Delta^{\op}, \coprod) \to \OOO$ ($\Xi$ consists actually of natural transformation --- not only oplax) 
and applying $L_{(\mathcal{C}, \otimes)}$ (Proposition~\ref{PROPDAYLAX}). Likewise, it can be given by a  functor
\[\Xi^{\vee}: \Delta_{\emptyset} \to \Hom^{1-\oplax, 1-\inert-\pseudo}(\OOO, (\Cat^{\PF}_{}, \times))   \]
with $\Xi^{\vee}([0])$ classifying $(\Delta, \coprod)^{\op} \to \OOO^{\op}$ ($\Xi^{\vee}$ does not consist of natural transformations) and applying $L_{(\mathcal{C}, \otimes)^{\vee}}$ (Proposition~\ref{PROPDAYLAX}).
\end{PROP}

\begin{proof}
We need to observe that the exponential operad $(\Delta^{\op}, \coprod)$ extends to a functor 
$\Xi: \Delta_{\emptyset} \to \Hom^{1-\oplax, 1-\inert-\pseudo}(\OOO^{\op}, (\Cat^{\PF}_{}, \times)^{\vee})$
 or, in other words, that there is a functor\footnote{Shifting slightly the perspective and considering  $\Delta_{\emptyset}$ instead of the usual $\Delta_{\act}^{\op}$ via Lemma~\ref{LEMMADUAL}.}
\[ \Xi: \Delta_{\emptyset} \times \Delta_{\emptyset}^{\op} \to \Cat^{\PF}_{} \]
sending $\ast$ to products in both variables, and roughly speaking, encoding the operad $\Z[(\Delta^{\op}, \coprod)]$ vertically and the cooperad $\Z[(\Delta, \ast)^{\op}]$ horizontally. 
For this one one has to construct the following commutative diagram:
\[ \xymatrix{
 (\Delta^{\op})^2 \ar[r]^{\dec} \ar[d]_{\delta_{12,34}}  &  (\Delta^{\op}) \ar[d]^{\delta}  \\
 (\Delta^{\op})^4 \ar[r]_{\dec_{13,24}} & (\Delta^{\op})^2  }\]
which is obvious, and 
\[ \xymatrix{
 (\Delta^{\op}) \ar[d]_{\delta}  \ar@{<-}[r]^-{{}^t\! \pi}  &  \cdot \ar@{=}[d]  \\
 (\Delta^{\op})^2  \ar@{<-}[r]_-{{}^t\! \pi} & \cdot  }
\quad \xymatrix{
 (\Delta^{\op})  \ar[d]_{ \pi} \ar@{<-}[r]^-{{}^t\! \pi}   &  \cdot \ar@{=}[d] \\
 \cdot  \ar@{=}[r] & \cdot  }\]
 which is the cofinality of $\delta$ (Lemma~\ref{LEMMAEXACT5}), and the cofinality of $\pi$, i.e.\@ the contractibility of $\Delta^{\op}$, and a fourth type which is trivial. 
 The fact that this yields the structure defined by (\ref{eqenrich}) follows basically by unraveling the definition, and is left to the reader. 
 For this functor all diagrams commute (not only oplaxly), i.e.\@ it is morphism-wise Cartesian, reflecting the fact that this yields a simplicially enriched structure on the push-forward $\otimes$.
 The functor
 \[\Xi^{\vee}: \Delta_{\emptyset} \to \Hom^{1-\oplax, 1-\inert-\pseudo}(\OOO, (\Cat^{\PF}_{}, \times))   \]
 is formed by applying ${}^t$ in the $\delta$-direction and passing to mates. We get only oplax commutativity.
\end{proof}

\subsection{Naive homotopy}

\begin{PAR}\label{SENRICHEDNATTRANS}Recall the notions of {\bf simplicially enriched category, functor, and natural transformation} (see, for instance \cite{Kel82}). 
Let $\mathcal{C}$, $\mathcal{D}$ be simplicially enriched categories, and $\mu: A \Rightarrow B$ a natural transformation between simplicially enriched functors. 
$\mu$ is simplicially enriched, if  
\[ \xymatrix{  \uHom(i, j) \ar[d]_-B \ar[rr]^-A && \uHom(A(i), A(j)) \ar[d]^{ \mu(j) \circ }   \\
 \uHom(B(i), B(j)) \ar[rr]_{ \circ \mu(i) } &&  \uHom(A(i), B(j))  }  \]
commutes. The set of simplicially enriched natural transformations can be enhanced to a simplicial set turning the category $\Fun(\mathcal{C}, \mathcal{D})$ 
into a simplicially enriched category, as follows: 
\end{PAR}

\begin{PAR}\label{PARUHOM}
Given a simplicially enriched category $\mathcal{C}$,
for each simplicial set $K$, we may define a simplicially enriched category $\mathcal{C}_K$ with the same objects and morphisms
\[ \uHom(K, \uHom(X, Y)) = \uHom(K \times X, Y) \]
with composition induced by the diagonal $K \to K \times K$.
There is a simplicially enriched inclusion functor
\[ \mathcal{C} \hookrightarrow \mathcal{C}_K \]
given by composition with $K \to \Delta_0$ and a 
 simplicially enriched functor $F: \mathcal{C} \to \mathcal{D}$  induces a functor
\[ F_K:  \mathcal{C} \to \mathcal{D} \hookrightarrow \mathcal{D}_K. \]
For a second such functor $G$ there is a simplicial set $\uNat(F, G)$ characterized by
\[ \Hom(K, \uNat(F, G)) = \Hom(F_K, G_K)  \]
which can also be expressed using the {\em enriched} end: 
\[ \uNat(F, G) \cong \int_c \uHom(F(c), G(c)). \]
This can also be seen as the natural enriched $\Hom$-object in the enriched functor category of functors from $\mathcal{C} \to \mathcal{D}$.
\end{PAR}

\begin{PAR}\label{PARNAIVE1}
Let $I$ be a small category or (co)operad and let $\mathcal{C} \to I$ and  $\mathcal{D} \to I$ be simplicially enriched cofibrations (of categories or (co)operads) classified by functors
$\Xi_{\mathcal{C}}, \Xi_{\mathcal{D}}: I \to \SCat$ (resp.\@ $I \to (\SCat, \times))$ where $\SCat$ denotes the 2-category of simplicially enriched categories. 
We have
\[ \Hom_I(\mathcal{C}, \mathcal{D}) = \Hom^{\lax}(\Xi_{\mathcal{C}}, \Xi_{\mathcal{D}}). \]
The compatibility with the simplicially enriched structure is as follows:  
 $\Hom_I(\mathcal{C}, \mathcal{D})$  consists of simplicially enriched functors. 
For $\mu \in \Hom^{\lax}(\Xi_{\mathcal{C}}, \Xi_{\mathcal{D}})$ the $\mu(i)$ have to be simplicially enriched and  
\[ \xymatrix{ \Xi_{\mathcal{C}}(i) \ar[r]^{\mu(i)} \ar[d]_{\Xi_{\mathcal{C}}(\alpha)} \ar@{}[rd]|{\Leftarrow^{\mu(\alpha)}}  & \Xi_{\mathcal{D}}(i) \ar[d]^{\Xi_{\mathcal{D}}(\alpha)} \\  
\Xi_{\mathcal{C}}(j) \ar[r]_{\mu(j)} & \Xi_{\mathcal{D}}(j)}\]
the natural transformations $\mu(\alpha)$ have to be simplicially enriched, in the sense of \ref{SENRICHEDNATTRANS}.

Let $I^{\flat}$ the underlying set of objects (only those in $I_{[1]}$ in the (co)operad case). Denote $()^{\flat}$ the composition with (or pull-back to) $I^{\flat} \hookrightarrow I$.
Fix an object $X \in \Hom(\mathcal{C}^{\flat}, \mathcal{D}^{\flat}) = \Hom^{(\lax)}(\Xi_{\mathcal{C}}^{\flat}, \Xi_{\mathcal{D}}^{\flat})$ {\em that is simplicially enriched, } i.e.\@  a collection, for all objects $i \in I$, 
of a simplicially enriched functor $X(i): \mathcal{C}_i \to \mathcal{D}_i$.

Denote by $\Hom_I(\mathcal{C}, \mathcal{D})_X \cong \Hom^{\lax}(\Xi_{\mathcal{C}}, \Xi_{\mathcal{D}})_X$ the {\em set}\footnote{This is a set, because the functors between fibers are fixed, only the values on morphisms, resp.\@ the laxness constraint, are variable.} of functors (resp.\@ lax natural transformations) that map to $X$.

Each simplicial set $K$ define categories 
$\Hom_I(\mathcal{C}, \mathcal{D}_K) = \Hom^{\lax}(\Xi_{\mathcal{C}}, (\Xi_{\mathcal{D}})_K)$
where $(\Xi_{\mathcal{D}})_K$ is the functor that maps $i$ to the simplicially enriched category $\mathcal{D}_{i,K}$.
\end{PAR}

\begin{DEF}\label{DEFNAIVE}
We define a simplicial set of {\bf naive deformations}
\[ \underline{\mathrm{Def}}_X(\Xi_{\mathcal{C}}, \Xi_{\mathcal{D}})   \]
such that 
\[ \Hom(K, \uDef_X(\Xi_{\mathcal{C}}, \Xi_{\mathcal{D}})) =   \Hom_I(\mathcal{C}, \mathcal{D}_K)_X = \Hom^{\lax}(\Xi_{\mathcal{C}}, (\Xi_{\mathcal{D}})_K)_X. \]
\end{DEF}

The adjective ``naive'' is used to distinguished these deformations from the coherent transformations discussed in Section~\ref{SECTCOHTRANS}, which are more sophisticated, and encode
higher coherence as well. Sometimes, however, this higher coherence data can be built from a naive deformation, see Proposition~\ref{PROPCOHERENT}.

\begin{PAR}
Let us unravel the definition: For a simplicial set $K$ the following set of data are the same
\begin{enumerate}
\item $\Hom(K, \underline{\mathrm{Def}}_X)$
\item For each morphism $f: x \to y$ in $\mathcal{C}$ a morphism
\[  \mu_f: K \otimes X(x) \to X(y)   \]
which if $f$ lies over an identity of $I$ is the constant morphism $X(f): X(x) \to X(y)$. If $\mathcal{C}$ is simplicially enriched, then we have a morphism
\[ \uHom(x, y) \to \uHom(K \times X(x), X(y)). \]
mapping inert morphisms to inert morphisms, and compatible with composition.
\item For each morphism $f: i \to j$ in $I$ a $K$-natural transformation: 
\[ \mu_f: \Hom(K, \uNat(X(j) F(f), G(f) X(i)). \]
(invertible if $f$ is inert) compatible with composition.
\end{enumerate}

If the simplicially enriched structure on $J=\mathcal{C}$ is discrete --- and we will be mainly interested in this case ---  the deformations $F: J \to \mathcal{D}$ over $X: J^{\flat} \to \mathcal{D}^{\flat}$ can also be described as follows: 
\end{PAR}

\begin{PAR}\label{PARENRICHEDMONAD}
Let $\mathcal{D}$ be a left tensored simplicially enriched category. 
Let $(M \in \End(\mathcal{D}), m, u)$ be a simplicially enriched monad. Then for each simplicial set $K$ and object $G \in \mathcal{D}$, we may define a set of $M \times K$-algebra structures on $G$ by morphisms
$\mu_K: MG \times K \to G$ such that
\[ \xymatrix{ G \times K \ar[r]^-{(u, \id)} \ar[rd]_{\pr} &  MG \times K \ar[d]^{\mu_K} \\
& G
}\]
and
\[ \xymatrix{
M(MG \times K) \times K \ar@{<-}[r]^-{\delta} \ar[d]_{(M\mu_K, \id)} &  M^2G \times K \ar[r]^{(m, \id)} & M G \times K \ar[d]^{\mu_K}  \\
MG \times K \ar[rr]_-{\mu_K }& & G
}\]
commute, where in the arrow $\delta$ also the simplicially enriched structure of $M$ is involved. This is clearly functorial in $K$ and thus, 
in particular, given $G$, we get a simplicial set of $M$-algebra structures 
\[ \underline{\Alg}_M(G)_{[n]} = \{ \Delta_n \times M\text{-algebra structures on $G$} \}   \]
which represents the previous functor on simplicial sets.
\end{PAR}

\begin{PAR}\label{PARCOFIB}
Let $I$ be a small category or (co)operad and
let $\mathcal{D} \to I$ be a simplicially enriched cofibration and $p: J \to I$ a usual (discrete) cofibration. 
Let $I^{\flat}$ be the set of objects of $I$ and $\nu: I^{\flat} \hookrightarrow I$ the canonical functor and $J^{\flat}$, resp.\@ $\mathcal{D}^{\flat}$ the pull-backs. 
We have a pull-back diagram
\[ \xymatrix{ J^{\flat} \ar[r]^{\nu} \ar[d]_{p_{}^{\flat}} & J  \ar[d]^{p} \\
I^{\flat} \ar[r]^{\nu} & I 
} \]
We assume that 
$\nu^*: \mathcal{D}^{J}_{p} \to \mathcal{D}^{J^{\flat}}_{\nu^* p} = (\mathcal{D}^{\flat})^{J^{\flat}}_{p^{\flat}}$ has a left adjoint $\nu_!^{(p)}$ (relative left Kan extension, see Proposition~\ref{PROPKAN} for sufficient criteria for its existence).
It is well-known that the adjunction is monadic and thus $\mathcal{D}^{J}_{p}$ can be reconstructed as the algebras in $(\mathcal{D}^{\flat})^{J^{\flat}}_{p^{\flat}}$ 
w.r.t.\@ the monad
\[ M = \nu^* \nu_!^{(p)}. \]
In the special case that $J$ is an operad and $I=\OOO$, this is the interpretation of algebras over the operad $J$ as algebras for a  monad. 

The monad is simplicially enriched w.r.t.\@ the simplicially enriched structure on $\mathcal{D} \to I$.
We have then 
\[ \uDef_X(J, \mathcal{D}) = \uAlg_{ \nu^* \nu_!^{(p)}}(X). \]
\end{PAR}

\subsection{Acyclicity (Construction of homotopies)}\label{SECTACYCLIC}

\begin{PAR}\label{PARHOMOTOPYPOINT}
A functor
\[ Y: \Delta_{\emptyset} (\cong \Delta_{\act}^{\op}) \to \Cat^{\PF}. \]
 defines a structure of simplicially enriched category (in fact enriched over $\Set^{\Delta^{\op}_{\emptyset}}$) already on 
\[ \Hom_{\Cat^{\PF}}(Y_{[0]}, X)    \]
for any category $X$ as follows: Its objects are pro-functors $\alpha: Y_{[0]} \to X$ and we define simplicial sets of morphisms by
\[ \uHom(\alpha, \beta)_{[n]} := \Hom( \alpha Y(p_n), \beta Y(p_n)). \]

If $\mathcal{C}$ is a (finitely) cocomplete category, then 
$L_{\mathcal{C}} \circ Y$ equips $L_{\mathcal{C}}(Y_{[0]}) = \mathcal{C}^{Y_{[0]}}$ with the structure of simplicially enriched category by means of Proposition~\ref{PROPENRICH}, and we have a map of simplicial sets: 
\[ \uHom(\alpha, \beta) \rightarrow \uHom(L_{\mathcal{C}}(\alpha), L_{\mathcal{C}}(\beta)) \]
turning $L_{\mathcal{C}}$ into a simplicially enriched functor. Here the $\uHom$ on the RHS is the simplicial enrichment of functor categories $\Hom(\mathcal{C}^{X}, \mathcal{C}^{Y_{[0]}})$ (\ref{PARUHOM}), but {\em with the trivial (discrete) structure on $\mathcal{C}^X$!}

The standard example is the functor $Y=\Xi$ obtained from the cooperad $(\Delta, \ast)^{\op}$ determined by $\Xi_{[0]} = \Delta^{\op}$, $\Xi(p_1) = \dec$ and $\Xi(p_{\emptyset}) = {}^t \pi$.
We have seen in Proposition~\ref{PROPTWOENRICHMENTS} that $L_{\mathcal{C}} \circ \Xi$ yields the {\em canonical enrichment} in simplicial sets for simplicial objects in $\mathcal{C}$. 

For now it was not used that $Y$ comes from a cooperad, i.e.\@ that it is compatible with the product $\ast$ on $\Delta_{\emptyset}$ and $\times$ on $\Cat^{\PF}_{}$.
This circumstance has
strong acyclicity implications for the simplicial sets considered above (and also their cousins $\uDef(\cdots)$ discussed in \ref{PARFUNCT} below) that are eventually the source
of all homotopies constructed in these lectures. This is in some sense parallel to the theory of acyclic models \cite{Bar02}, but has the advantage of being entirely constructive, and not limited to Abelian situations. 

In particular, two homotopic transformations $\mu, \nu: \alpha \Rightarrow \beta$ induce homotopic transformations $L_{\mathcal{C}}(\alpha) \Rightarrow L_{\mathcal{C}}(\beta)$.
In particular, if $Y=\Xi$ is the standard example, and $\mathcal{W}$ is a class of weak equivalences in $\mathcal{C}^{\Delta^{\op}}$ satisifying (W1) of \ref{PARAXIOMSW} then
for the induced
\[ L(\alpha)', L(\beta)': \mathcal{C}^{X} \to \mathcal{C}^{\Delta^{\op}}[\mathcal{W}^{-1}] \]
we have $L(\mu)' = L(\nu)'$. 
\end{PAR}

\begin{PAR}\label{PARPROP}
Assume that {\em also $X$} extends to a functor \[ X: \Delta_{\emptyset} (\cong \Delta_{\act}^{\op})\to \Cat^{\PF} \]
and assume that $X$ and $Y$ have the following properties: 
\begin{enumerate}
\item They are monoidal w.r.t.\@ the concatenation product $\ast$ on $\Delta_{\emptyset}$ and $\times$ on $\Cat^{\PF}$. 
(This is the case if and only if they correspond to  exponential fibrations of cooperads $I \to \OOO^{\op}$. )
\item The unit $\cdot \cong X_{\emptyset} \rightarrow X_{0}$ is given by ${}^t \pi$. (By 1.\@ this determines all values of $X$ under injective maps.)
\end{enumerate}
\end{PAR}

\begin{PROP}\label{PROPCREATIONHOMOTOPY}
If $\alpha$ and $\beta$ extend to an oplax transformation, and lax transformation, respectively: 
\[ \widetilde{\alpha} \in \Hom^{\oplax, p_\emptyset-\pseudo, \times}(Y, X) \qquad \widetilde{\beta} \in \Hom^{\lax, p_\emptyset-\pseudo, \times}(Y, X) \]
((op)lax, pseudo on $p_{\emptyset}$\footnote{Unraveling the definition, in view of assumption 2.\@, this means that $\alpha$ and $\beta$ must be cofinal, i.e.\@ satisfy $\alpha\,{}^t\! \pi = {}^t\! \pi$ and $\beta\,{}^t\! \pi = {}^t\! \pi$.} , and compatible with the product structure)  then 
\[ \uHom(\alpha, \beta) \]
 is empty or contractible. 
\end{PROP}

\begin{PAR}\label{PARFUNCT}
Before we give the somewhat technical proof, we discuss  a functorial version of \ref{PARHOMOTOPYPOINT}. Consider a category or (co)operad $I$ and functors
\[ C, D \in \Delta_{\emptyset} \to \Hom^{1-\oplax, 1-\inert-\pseudo}(I^{\op}, (\Cat^{\PF}, \times))  \]
which satisfy  the properties of \ref{PARPROP} point-wise in $i$, such that also the oplax-ness constraints are compatible with the product structure. 
Let 
\[ F:  D|_{I^{\flat}} \to C|_{I^{\flat}}   \]
be a natural transformation (i.e.\@ a collection of pro-functors $F(i): D(i) \to C(i)$) compatible with product.
Then define 
\[ \uDef(D, C)_{F,[n]} := \Hom^{\oplax, \inert-\pseudo}(D_{[n]}, C_{[0]})_{ F_{[0]} \circ D^{\flat}(p_n)} \] 
writing $D^{\flat} :=  D|_{I^{\flat}}$.
The functoriality for $\alpha: [n] \to [m]$ is given by the pre-composition maps
\[   \Hom^{\oplax, \inert-\pseudo}(D_{[m]}, C_{[0]}) \to \Hom^{\oplax, \inert-\pseudo}(D_{[n]}, C_{[0]}) \]
induced by $D(\alpha) \in \Hom^{\oplax}(D_{[n]}, D_{[m]})$.
Note that this yields a set, because the ``values'' of the natural transformations are fixed; the only variable  is the oplaxness constraint. 
Notice: For this definition, the structure extension of $C_{[0]}$ to $\Delta_{\emptyset}$ does not play a role, yet!
\end{PAR}

\begin{PAR}\label{PARDAYDEF}
If $\mathcal{C} \to I$ is a functor of categories or (co)operads with (finitely) cocomplete fibers, then
$D$ equips $L_{\mathcal{C}}(D_{[0]}) = D(D_{[0]}, \mathcal{C})$ (the Day convolution) with the structure of simplicially enriched category over $I$ (cf.\@ Proposition~\ref{PROPENRICH}), and we have a map of simplicial sets (cf.\@ Definition~\ref{DEFNAIVE} for the RHS): 
\[ \uDef(D, C)_{F} \rightarrow \uDef(D(C_{[0]}, \mathcal{C}), D(D_{[0]}, \mathcal{C}))_{L_{\mathcal{C}}(F)} \]
{\em in which $D(C_{[0]}, \mathcal{C})$ is considered discrete} and $D(D_{[0]}, \mathcal{C})$ receives its simplicial structure from the functor $L_{\mathcal{C}} \circ D$.
Warning: The Day convolution $D(D_{[0]}, \mathcal{C}) \to I$ is only a {\em simplicially enriched} cofibration over $I$, if $\mathcal{C} \to I$ is a cofibration and 
$D$ is actually a functor $\Delta_{\emptyset} \to \Hom(I^{\op}, \Cat^{\PF})$ i.e.\@ maps morphisms to natural transformations (as opposed to oplax ones). 

Whether $D(D_{[0]}, \mathcal{C}) \to I$ is a simplicially enriched {\em fibration}, cannot be in general detected with the calculus of pro-functors, because this concerns $R$ and the simplicial structure is encoded via $L$. If $D_{[0]}$ actually corresponds to a fibration (i.e.\@ has morphisms in $\im {}^t \iota$), we may form ${}^{t}\! D_0$, and this can be done in certain cases.
\end{PAR}

\begin{PAR}
For $I= [1]$, as a special case, we obtain the situation in \ref{PARHOMOTOPYPOINT}, in case $\widetilde{\alpha}$ and $\widetilde{\beta}$ are actually natural transformations (not only lax or oplax), 
with the following data:
$C$ and $D$ then encode $\widetilde{\alpha}$ and $\widetilde{\beta}$ as cofibrations over $I=[1]$ and $F$ is the identity. 
Then we have 
\[ \uHom(\alpha, \beta) = \uDef(D, C)_{\id}.  \]
\end{PAR}

\begin{PAR}
For $I=[2]$ and $\alpha = \alpha_2 \alpha_1$ and $\beta = \beta_2 \alpha_1$, encoded by $C$ and $D$, $\uDef(D, C)_{\id}$ consists of those simplices 
in
\[ \uHom(\alpha_1, \beta_1) \times \uHom(\alpha_2, \beta_2) \times \uHom(\alpha, \beta) \]
 such that the pasting
\[ \uHom(\alpha_1, \beta_1) \times \uHom(\alpha_2, \beta_2) \to  \uHom(\alpha_2 \alpha_1, \alpha_2 \beta_1) \times \uHom(\alpha_2  \beta_1, \beta_2 \beta_1) \to \uHom(\alpha, \beta)      \]
maps the elements (of the same degree) in the former pair of simplicial sets to the one in the latter.
For the map $\uHom(\alpha_2, \beta_2) \to \uHom(\alpha_2  \beta_1, \beta_2 \beta_1)$  to exist, it is essential, that $\beta_1$ is ``simplicially enriched'' meaning that it extends to a functor $\Delta_{\emptyset} \times [1] \to \Cat^{\PF}$.
\end{PAR}

\begin{PAR}
There is a canonical element in $\uDef(D, C)_{F,[-1]}$ given by 
\[ \Hom^{\oplax}(D_{\emptyset}=\cdot, C_{[0]})_{\underbrace{F_{[0]} \circ D^{\flat}(p_{\emptyset})}_{\cong C^{\flat}(p_{\emptyset})}} \]
given by $C(p_{\emptyset})$. If that one consists of isomorphisms then it is the unique element in there. 
\end{PAR}

\begin{PROP}\label{PROPCREATIONHOMOTOPYFUNCT}
In the situation of \ref{PARFUNCT}, the fiber of 
\[ \uDef(D, C)_{F}  \]
over the canonical element $C(p_{\emptyset})$ in $\uDef(D, C)_{F,[-1]}$ is empty or contractible. 
\end{PROP}

\begin{proof}[Proof of Proposition~\ref{PROPCREATIONHOMOTOPY}] 
By Lemma~\ref{LEMMADECCONTRACTIBLE}, 2.\@ below it suffices to construct a section of the natural map
\[ \dec^* \uHom(\alpha, \beta) \rightarrow \uHom(\alpha, \beta) \boxtimes \uHom(\alpha, \beta).   \]
Given $\mu_1: \alpha Y(p_n) \Rightarrow \beta Y(p_n)$ and $\mu_2: \alpha Y(p_m) \Rightarrow \beta Y(p_m)$
we have to define a morphism
\[  \nu:  \alpha Y(p_{n+m+1}) \to   \beta Y(p_{n+m+1})  \]
which applied to $Y(\delta_l)$ and $Y(\delta_r)$ gives back $\mu_1$ and $\mu_2$.
It is defined as the pasting
\begin{equation} \label{eqdia1} \vcenter{ \xymatrix{ & &  \ar[rrdd]^{Y(p_{n+m+1})} \ar@/_8pt/[dd]_{\widetilde{\alpha}_{n+m+1}} \ar@/^8pt/[dd]^{\widetilde{\beta}_{n+m+1}}_{\Rightarrow}  \ar[lldd]_{Y(p_{n+m+1})}   \\
\\
\ar[rrdd]_{\alpha} & \Rightarrow &   \ar[dd]|{X(p_{n+m+1})} & \Rightarrow & \ar[lldd]^{\beta} \\
\\
& &   }  }
\end{equation}
where the middle natural transformation is $(\mu_1, \mu_2)$ exploiting the product structure on $\widetilde{\alpha}_{n+m+1}$ and $\widetilde{\beta}_{n+m+1}$.
We have to show that be get back $\mu_1$ after pre-composition with $Y(\delta_l)$ (and similarly, $\mu_2$ after pre-composition with $Y(\delta_r)$). 

{\em Claim:} The pasting
\[\xymatrix{ 
\ar[r]^{Y(\delta_l)} &  \ar@/_8pt/[rr]_{\widetilde{\beta}_{n+m+1}} \ar@/^8pt/[rr]^{\widetilde{\alpha}_{n+m+1}}_{\Downarrow} & & \\
 } \]
is  equal to the pasting: 
\begin{equation} \label{eqdia2} \vcenter{
\xymatrix{ & \ar[rrdd]^{\widetilde{\alpha}_{n+m+1}} \\
& \Downarrow \\
\ar[ruu]^{Y(\delta_l)} \ar[rdd]_{Y(\delta_l)}  \ar@/_8pt/[rr]_{\beta Y(p_n)} \ar@/^8pt/[rr]^{\alpha Y(p_n)}_{\Downarrow^{\mu_1}}  & &  \ar[r]^{X(\delta_l)}  & \\
& \Downarrow \\
& \ar[rruu]_{\widetilde{\beta}_{n+m+1}}  }  }
\end{equation} 

Both diagrams are actually products, namely
\[ \xymatrix{ 
\ar[r]^{Y(\id_n)} &  \ar@/_8pt/[rr]_{Y(p_n)\beta} \ar@/^8pt/[rr]^{Y(p_n)\alpha}_{\Downarrow^{\mu_1}} & & \\
 } \quad \times \quad \xymatrix{ 
\ar[r]^{Y(p_{\emptyset}^{m+1})} &  \ar@/_8pt/[rr]_{Y(p_m)\beta} \ar@/^8pt/[rr]^{Y(p_m)\alpha}_{\Downarrow^{\mu_2}} & & \\
 }  \]
and
\[ \vcenter{ \xymatrix{ & \ar[rrdd]^{Y(p_n)\alpha} \\
& \Downarrow \\
\ar[ruu]^{Y(\id_n)} \ar[rdd]_{Y(\id_n)}  \ar@/_8pt/[rr]_{Y(p_n)\beta} \ar@/^8pt/[rr]^{Y(p_n)\alpha}_{\Downarrow^{\mu_1}}  & &  \ar[r]^{X(\id_n)}  & \\
& \Downarrow \\
& \ar[rruu]_{Y(p_n)\beta}  } } 
\times 
\vcenter{ \xymatrix{ & \ar[rrdd]^{Y(p_m)\alpha} \\
& \Downarrow \\
\ar[ruu]^{Y(p_{\emptyset}^{m+1})} \ar[rdd]_{Y(p_{\emptyset}^{m+1})}  \ar@{=}@/_8pt/[rr]_{} \ar@{=}@/^8pt/[rr]^{}_{=}  & &  \ar[r]^{X(p_{\emptyset}^{m+1})}  & \\
& \Downarrow \\
& \ar[rruu]_{Y(p_m)\beta}  }} \]
For the left parts the assertion is thus clear. For the right parts note that by assumption $X(p_{\emptyset}) = Y(p_{\emptyset}) = {}^t\!\pi$ and that $\widetilde{\alpha}$ and $\widetilde{\beta}$ are pseudo on 
$p_{\emptyset}$, thus $Y(p_{\emptyset}) \alpha$ and $Y(p_{\emptyset}) \beta$ are thus both isomorphic to ${}^t\!\pi$ which is a final object. This proves the claim. 

The conclusion is made by pre-composing (\ref{eqdia1}) with $Y(\delta_l)$ and inserting (\ref{eqdia2}).
 \end{proof}

\begin{proof}[Proof of Proposition~\ref{PROPCREATIONHOMOTOPYFUNCT}]
By Lemma~\ref{LEMMADECCONTRACTIBLE}, 2.\@ below, it suffices to construct a section of the natural map
\[ \dec^* \uDef(D, C)_F \rightarrow \uDef(D, C)_F \boxtimes \uDef(D, C)_F.   \]
Given $\mu_1  \in \Hom^{\oplax}(D_n, C_0)_{F_0 \circ D^{\flat}(p_n)}$ and  $\mu_2  \in \Hom^{\oplax}(D_m, C_0)_{F_0 \circ D^{\flat}(p_m)}$, we have to produce 
$\nu \in \Hom^{\oplax}(D_{n+m+1}, D_0)_{F_0 \circ D^{\flat}(p_{n+m+1})}$ connecting the two.
This is more or less the same construction as in the proof of Proposition~\ref{PROPCREATIONHOMOTOPY}:
We have given for each $\alpha: i \to j$:
\[ \xymatrix{ \ar[rr]^-{F_0(j) \circ D(p_n,j)} \ar[d]_{D_n(\alpha)}  \ar@{}[rrd]|{\Downarrow^{\mu_1}} & & \ar[d]^{C_0(\alpha)} \\
\ar[rr]_-{F_0(i) \circ D(p_n,i)} & &  }  \quad 
\xymatrix{ \ar[rr]^-{F_0(j) \circ D(p_m,j)} \ar[d]_{D_m(\alpha)}  \ar@{}[rrd]|{\Downarrow^{\mu_2}} & & \ar[d]^{C_0(\alpha)} \\
\ar[rr]_-{F_0(i) \circ D(p_m,i)} & &  } \]
or which is the same:
\[ \xymatrix{ \ar[rr]^-{C(p_n,j) \circ F_n(j) } \ar[d]_{D_n(\alpha)}  \ar@{}[rrd]|{\Downarrow^{\mu_1}} & & \ar[d]^{C_0(\alpha)} \\
\ar[rr]_-{C(p_n,i) \circ F_n(i) } & &  }  \quad 
\xymatrix{ \ar[rr]^-{C(p_m,j) \circ F_m(j)  } \ar[d]_{D_m(\alpha)}  \ar@{}[rrd]|{\Downarrow^{\mu_2}} & & \ar[d]^{C_0(\alpha)} \\
\ar[rr]_-{C(p_m,i) \circ F_m(i) } & &  } \]
and have to construct a 2-morphism in 
\[ \xymatrix{ \ar[rr]^-{D(p_{n+m+1},j)} \ar[d]_{D_{n+m+1}(\alpha)}  \ar@{}[rrrrd]|{\Downarrow^{\nu}}   & &  \ar[rr]^-{F_0(j)}&  & \ar[d]^{C_0(\alpha)} \\
\ar[rr]_-{D(p_{n+m+1},i)} & &  \ar[rr]_-{F_0(i)} & &    } \]
The square is isomorphic to the outer square in 
\[ \xymatrix{ \ar[rr]^-{C(p_n \ast p_m,j) F_{n+m+1}(j) } \ar[d]_{D_{n+m+1}(\alpha)}  \ar@{}[rrd]|{\Downarrow^{(\mu_1, \mu_2)}} & &  \ar[rr]^{C(p_{1},j)}  \ar@{}[rrd]|{\Downarrow^{}} \ar[d]|{C_{1}(\alpha)} &  & \ar[d]^{C_{0}(\alpha)} \\
\ar[rr]_-{C(p_n \ast p_m,i) F_{n+m+1}(i) } & & \ar[rr]_{C(p_{1},i)} & &   } \]
and it is clear from the
functorialities that the constructed element yields a valid element of $\uDef(D, C)_{F,n+m+1}$.

We have to see that we get back $\mu_1$ and $\mu_2$ after applying $\delta_l, \delta_r$. Investigate the case of $\delta_l$, the other is analog: 
\[ \xymatrix{ \ar[rr]^-{D(\delta_l, j) } \ar[d]_{D_n(\alpha)}  \ar@{}[rrd]|{\Downarrow} &&  \ar[rr]^{C(p_n \ast p_m,j) F_{n+m+1}(j)}  \ar@{}[rrd]|{\Downarrow^{(\mu_1, \mu_2)}} \ar[d]|{D_{n+m+1}(\alpha)} &  & \ar[d]|{C_{1}(\alpha)} \ar[rr]^{C(p_{1},j)}  \ar@{}[rrd]|{\Downarrow}  & & \ar[d]^{C_0(\alpha)}  \\
\ar[rr]_-{D(\delta_l, i)}  & & \ar[rr]_{C(p_n \ast p_m,i) F_{n+m+1}(i)} & &  \ar[rr]_{C(p_{1},i)} & &    } \]
The left squares are the product  
\[ \vcenter{ \xymatrix{  \ar@{=}[rr]   \ar[d]_{D_n(\alpha)} &  & \ar[d]|{D_n(\alpha)} \ar[rr]^{C(p_n,j) F_{n}(j) } \ar@{}[rrd]|{\Downarrow^{\mu_1}}  & & \ar[d]^{C_0(\alpha)}   \\
 \ar@{=}[rr] & &  \ar[rr]_{C(p_n,i) F_{n}(i)} & &    } } \quad  \times \quad  \vcenter{ 
 \xymatrix{  \ar[rr]^{D(p_{\emptyset}^{m+1},j)}   \ar@{=}[d] &  & \ar[d]|{D_m(\alpha)} \ar[rr]^{C(p_m,j) F_{m}(j) } \ar@{}[rrd]|{\Downarrow^{\mu_2}}  & & \ar[d]^{C_0(\alpha)}   \\
 \ar[rr]_{D(p_{\emptyset}^{m+1}, i)} & &  \ar[rr]_{C(p_m,i) F_{m}(i)} & &    } }  \] 
 in which the second factor is isomorphic to a diagram of the form
\[   \xymatrix{  \ar[rr]^{C(p_{\emptyset},j)}  \ar@{}[rrd]|-{\Downarrow} \ar@{=}[d] &  &  \ar[d]^{C_0(\alpha)}   \\
 \ar[rr]_{C(p_{\emptyset},i)} & &   }  \]
 and we have by assumption that $\mu_1$ and $\mu_2$ map to the oplaxness-constraint given by $C(p_{\emptyset})$. (If that is an isomorphism, then both compositions are the final object and thus there
 is only one 2-morphism anyway. )
 We can thus write the pasting as 
\[ \xymatrix{ \ar[rr]^-{C(p_n,j) F_{n}(j) } \ar[d]_{D_n(\alpha)}  \ar@{}[rrd]|{\Downarrow^{\mu_1}} &&  \ar[rr]^{C(p_1 \ast p_{\emptyset},j) }  \ar@{}[rrd]|{\Downarrow} \ar[d]|{C_{0}(\alpha)} &  & \ar[d]|{C_{1}(\alpha)} \ar[rr]^{C(p_{1},j)}  \ar@{}[rrd]|{\Downarrow}  & & \ar[d]^{C_0(\alpha)}  \\
\ar[rr]_-{C(p_n,i) F_{n}(i)}  & & \ar[rr]_{C(p_1 \ast p_{\emptyset},i) } & &  \ar[rr]_{C(p_{1},i)} & &    } \]
which gives back $\mu_1$. 
 \end{proof}

The following Lemma and its Corollary, which is well-known, can also be obtained easily without calculation from the description of homotopies via $\dec$ but is a bit different, and also more elementary than
the Propositions discussed before.

\begin{LEMMA}\label{LEMMADECCONTRACTIBLE1}
The canonical map
\[  P: \dec_n \Rightarrow [n] \circ \pi: \Delta^{\op}  \to \Delta^{\op}  \]
is a homotopy equivalence in the sense of \ref{PARHOMOTOPYPOINT} (for $Y=\Xi$ the standard example).
\end{LEMMA}
\begin{proof}(cf.\@ also the proof of Lemma~\ref{LEMMAEXACT3}). 
Define 
\[ S: [n] \circ \pi \Rightarrow \dec_n \]
by the degeneracy $[n] \twoheadleftarrow [x] \ast [n]$ which maps $[x]$ to the minimal element in $[n]$.
Obviously, $PS = \id$. We have to define a morphism
\[ \xi: \dec_n \dec \to \dec_n \dec \]
such that the pre-composition with ${}^t\!\pr_0$ gives back the identity and the pre-composition with ${}^t\!\pr_1$ gives back $SP$.
We define the morphism 
\[ \xi: [a] \ast  [b] \ast [n] \to  [a] \ast  [b] \ast [n] \]
as the identity on $[a]$ and $[n]$ and mapping $[b]$ to the minimal element of $[n]$. 
We may in fact extend $\xi$ to a morphism  
\[ \xi_{\emptyset}: \dec_{\emptyset, n} \dec_{\emptyset}  \to \dec_{\emptyset, n} \dec_{\emptyset}.  \]
That one has the {\em obvious} property that composition with $p_0 = ([\emptyset], \id): \Delta_{\emptyset}^{\op} \to \Delta_{\emptyset}^{\op} \times \Delta_{\emptyset}^{\op}$ gives back (the restriction of) $SP$ and composition with $p_1$ gives back the identity. Thus
\[ {}^t\!i \dec_{\emptyset, n} \dec_{\emptyset} p_i i \to {}^t\!i \dec_{\emptyset, n} \dec_{\emptyset} p_i i  \]
is $SP$, and $\id$, respectively. But ${}^t\!i \dec ([n], \id) \dec \cong \dec ({}^t\!i, {}^t\!i) ([n], \id) \dec \cong \dec  ([n], \id) {}^t\!i \dec  \cong  \dec ([n], \id) \dec ({}^t\!i , {}^t\!i)$ using Lemma~\ref{LEMMAEXACT4} and finally
$({}^t\!i , {}^t\!i) p_i i = \pi_i$ using Lemma~\ref{LEMMAEXACT2}.
\end{proof}

\begin{KOR}\label{LEMMADECCONTRACTIBLE}
\begin{enumerate}
\item The canonical map
\[ \dec^* \Rightarrow \pr_1^*: \mathcal{C}^{\Delta^{\op}} \to \mathcal{C}^{\Delta^{\op} \times \Delta^{\op}}  \]
is, point-wise in the second variable, a homotopy equivalence, i.e.\@ all the canonical maps $\dec_n^* X \to X_n$ where $X_n$ is considered constant, 
are homotopy equivalences. 
\item If $X$ is a simplicial set such that there is a section of the natural morphism
\[ \dec^* X \rightarrow X \boxtimes X   \]
then $X$ is either empty or contractible. 
\end{enumerate}
\end{KOR}
\begin{proof}
1.\@ is an obvious consequence of Lemma~\ref{LEMMADECCONTRACTIBLE1}. For 2.\@ consider the morphism 
\[ \dec_0^* X \rightarrow X \times X_0   \]
which has a section by assumption.  The fiber over any element of $X_0$ of the left hand side is contractible by 1. That is thus also the case
for $X$.
\end{proof}

\subsection{Weak equivalences and geometric realization}\label{SECTGEOMREAL}

The non-Abelian Eilenberg-Zilber theorem that will be investigated in the next two sections is a statement ``up to weak equivalence''. Indeed, even
the Eilenberg-Zilber {\em morphism} that will be constructed is only defined as a morphism in the homotopy category (unless the situation is {\em strongly symmetric}, as in the Abelian case). We do not want to 
restrict the situation unnecessarily, allowing for simplicial objects $\mathcal{C}^{\Delta^{\op}}$ in a general 1-category (being mainly interested in $\mathcal{C} = \Set$ and $\mathcal{C}$ Abelian in these lectures).
The proper language for such a general situation would be, of course, that of model categories. However, as long as we are in a situation (as in these main examples), where ``potentially'' every simplicial object is cofibrant, {\em much} less is needed for our purposes: 
 
\begin{PAR}\label{PARAXIOMSW}
Let $\mathcal{C}$ be a (finitely) complete and cocomplete category. We consider the following axioms on a class of ``weak equivalences'' $\mathcal{W}$ in $\mathcal{C}^{\Delta^{\op}}$.
\begin{enumerate}
\item[(W1)] $\Delta_1 \otimes X  \to X$ (or equivalently $X \to \mathcal{HOM}(\Delta_1, X)$) is a weak equivalence for all $X$\footnote{where $\otimes$ and $\mathcal{HOM}$ are the (co)tensoring of the
canonical simplicial enrichment \ref{PARENRICHMENTPRESHEAVES}.}.
\item[(W2)] $\delta^*: \mathcal{C}^{\Delta^{\op} \times \Delta^{\op}} \to \mathcal{C}^{\Delta^{\op}}$ maps point-wise weak equivalences (in either direction) to weak equivalences.
\end{enumerate}
\end{PAR}

\begin{LEMMA}\label{LEMMAW1}
The two versions of (W1) are equivalent and also imply that 
$X \to \Hom(\Delta_n, X)$, and $\Delta_n \otimes X \to X$ are weak equivalences for all $n$.
\end{LEMMA}
\begin{proof}
Any of the two statements in (W1) implies that homotopy equivalences are weak equivalences, and,
since $\otimes$ is associative, $- \otimes X$ and $\Hom(-, X)$ map homotopy equivalences in $\Set^{\Delta^{\op}}$ to homotopy equivalences in $\mathcal{C}^{\Delta^{\op}}$, and $s: \Delta_n \to \Delta_0$ is a homotopy equivalence in $\Set^{\Delta^{\op}}$.
\end{proof}

\begin{LEMMA}\label{LEMMAW2}
The axioms of \ref{PARAXIOMSW} hold true for $\mathcal{C}=\Set$, and for $\mathcal{C}$ Abelian, and the class $\mathcal{W}$ of usual weak equivalences, and the class of quasi-isomorphisms, respectively. 
\end{LEMMA}
\begin{proof}
Case $\mathcal{C}=\Set$.

(W1) holds true because $\otimes=\times$ is a left Quillen bifunctor (cf.\@ \cite{GJ09}), every object is cofibrant, and $\delta_0, \delta_1: \Delta_0 \hookrightarrow \Delta_1$ are trivial cofibrations, therefore $\id \times \delta_0: X \to X \times \Delta_1$ is a trivial cofibration, and therefore the first map of (W1) which is a section of $\id \times \delta_0$ is a weak equivalence.
(W2) is shown e.g.\@ in \cite[IV, Proposition 1.7]{GJ09}.

Case $\mathcal{C}$ Abelian.

(W1) If two morphisms $f, g: X \to Y$ are $\otimes$-homotopic, i.e.\@ if there is a morphism
$\Delta_1 \otimes X \to Y$
such that the restrictions give $f$ and $g$, respectively, then the composition $\Delta_1 \tildeotimes X \to \Delta_1 \otimes X \to Y$
with the Eilenberg-Zilber map yields  a homotopy in the usual sense (cf.\@ also Lemma~\ref{LEMMAYONEDA}) which implies that $f$ and $g$ induce the same map on homology groups.
Hence homotopy-equivalences w.r.t.\@ $\otimes$ are quasi-isomorphisms as well. Then argue as in the proof of Lemma~\ref{LEMMAW1}.

(W2) The proof of the Eilenberg-Zilber theorem does not need (W2) in the strongly symmetric case (cf.\@ Definition~\ref{DEFSYMMETRIC}).
Thus it suffices to show that $\dec_* \cong \tot$ (cf.\@ Proposition~\ref{PROPEXPLICITAB}) preserves quasi-isomorphisms in any direction. This is true in {\em any} Abelian category for non-negatively graded complexes. 
\end{proof}

Note that we have (for the examples $\mathcal{C}$ Abelian, and $\Set$, and more generally in case that $(\mathcal{C}^{\Delta^{\op}}, \mathcal{W})$ is part of a model category structure) an equivalence of
localizations: \[ (\mathcal{C}^{\Delta^{\op}}[\mathcal{W}^{-1}])^{\Delta^{\op}} \cong \mathcal{C}^{\Delta^{\op} \times \Delta^{\op}}[\mathcal{W}_v^{-1}] \] at the vertical weak equivalences (\cite[Proposition~7.9.2.]{Cis19}).
This is not at all true for the 1-categorical localizations. The statement below makes sense for them as well, if one takes $\mathcal{C}^{\Delta^{\op} \times \Delta^{\op}}[\mathcal{W}_v^{-1}]$ and the homotopy colimit, i.e.\@ the
derived functor of the colimit. 

\begin{PROP}\label{PROPHOCOLIM}
If \ref{PARAXIOMSW} hold then we have a commutative diagram of $\infty$-categories:
\[ \xymatrix{
\mathcal{C}^{\Delta^{\op} \times \Delta^{\op}} \ar[rr]^{\delta^*} \ar[d] & &  \mathcal{C}^{\Delta^{\op}} \ar[d] \\
(\mathcal{C}^{\Delta^{\op}}[\mathcal{W}^{-1}])^{\Delta^{\op}} \ar[rr]_{\colim_{\Delta^{\op}}} & &  \mathcal{C}^{\Delta^{\op}}[\mathcal{W}^{-1}]
}\]
in which the vertical functors are the canonical ones into the localizations which we consider as $\infty$-categorical localizations. 
In other words 
\[ \boxed{ \delta^* = \text{ geometric realization. } }  \]
\end{PROP}
\begin{proof}
Let $\delta: [n] \mapsto \Delta_n$ be the canonical cosimplicial simplicial object. We have
\[ \delta^* X \cong \int^n \Delta_n \otimes X_{\bullet, n}.  \]
This is left adjoint to the functor
\[ (\delta_* X)_{\bullet, n} \cong \mathcal{HOM}(\Delta_n, X).  \]
By \ref{PARAXIOMSW} and Lemma~\ref{LEMMAW1} both functors preserve weak equivalences and thus descend to an adjunction between the localizations. 
However, by Lemma~\ref{LEMMAW1} the morphism $\pr_2^* \to \delta_*$ induced by $\Delta_n \to \Delta_0$ is a  (point-wise) weak equivalence.
The statement follows thus from the uniqueness of adjoints. 
\end{proof}

\begin{BEM}
In a situation  in which the Eilenberg-Zilber Theorem holds (e.g.\@ if $\mathcal{C}^{\Delta^{\op}}$ is symmetric w.r.t.\@ $\mathcal{W}$), thus also $\dec_*$ computes geometric realization in the localization.

In the Abelian case, the same proof gives also {\em directly} that $\dec_* =$ geometric realization:
We have, denoting $\Delta_\bullet^{\circ} := \Z[\Delta_\bullet]$ the cosimplicial object in $\Ch_{\ge 0}(\Ab)^{\Delta}$:
\[ \dec_* X \cong \int^n \Delta_n^{\circ} \tildeotimes X_{\bullet, n}.  \]
(because $\dec_* = \tot$ obviously commutes with colimits) and this has the right adjoint (cf.\@ also Section~\ref{SECTMOREADJOINTS})
\[ \dec^?: X \mapsto \mathcal{HOM}^{\tildeotimes}(\Delta^{\circ}, X). \]
Again both functors preserve quasi-isomorphisms and the second functor is isomorphic to $\pr_2^*$ in the localization. 
\end{BEM}

\begin{BEM}
The statement in Proposition~\ref{PROPHOCOLIM} is obviously also equivalent (using that $\delta$ is $\infty$-cofinal, see Lemma~\ref{LEMMAEXACT5}) to the commutativity of the diagram 
\[ \xymatrix{\mathcal{C}^{\Delta^{\op}} \ar[rrd]  \ar[d]  \\
 (\mathcal{C}^{\Delta^{\op}}[\mathcal{W}^{-1}])^{\Delta^{\op}} \ar[rr]_{\colim_{\Delta^{\op}}}& & \mathcal{C}^{\Delta^{\op}}[\mathcal{W}^{-1}] } \]
 where the vertical morphism is induced by the composition $\mathcal{C} \to \mathcal{C}^{\Delta^{\op}} \to \mathcal{C}^{\Delta^{\op}}[\mathcal{W}^{-1}]$.
 In other words, every object in $\mathcal{C}^{\Delta^{\op}}[\mathcal{W}^{-1}]$ is the ``colimit of itself considered as simplicial diagram''.
\end{BEM}

\subsection{Symmetry}\label{SECTSYMM}

This section contains a discussion of the symmetries of simplicial objects that are at the heart of the Eilenberg-Zilber theorems, 
Theorem~\ref{SATZEZ} in the general case, and Theorem~\ref{SATZCOHEZ} in the Abelian case.
These symmetries are more tractable and efficient in the Abelian case. 
Let $\iota: \Delta^{\op} \to \FinSet^{\op}$, where $\FinSet$ is the category of non-empty finite sets\footnote{for simplicity, the equivalent category with the same objects as $\Delta^{\op}$}, be the inclusion.

\begin{DEF}\label{DEFSYMMETRIC}
We say that $\mathcal{C}^{\Delta^{\mathrm{op}}}$ (with chosen class of weak equivalences $\mathcal{W}$) is (strongly) symmetric if
there is a functor:
\[ \Cfrak: \mathcal{C}^{\Delta^{\mathrm{op}}} \to \mathcal{C}^{\FinSet^{\mathrm{op}}} \]
with a natural transformation
\[ \id \to \iota^* \Cfrak  \]
which is object-wise a weak equivalence (resp.\@ an isomorphism). 
\end{DEF}

\begin{PAR}
Recall the axioms \ref{PARAXIOMSW}. 
(W2) is only needed in the non-strongly symmetric case.
\end{PAR}

\begin{LEMMA}\label{LEMMADELTADEC}
If $\mathcal{W}$ is a class of weak equivalences such that (\ref{PARAXIOMSW}) hold then
there is a weak equivalence $\delta^* \dec^* \to \id$. In particular, $\delta^* \dec^*$ preserves weak equivalences.
\end{LEMMA}
\begin{proof}
Lemma~\ref{LEMMADECCONTRACTIBLE} shows that $\dec^* \to \pr_1^*$ is point-wise in the second variable a homotopy equivalence. Thus by (W1) it is point-wise in the second variable a weak equivalence. 
Thus by  (W2) 
\[ \delta^* \dec^* \to \delta^* \pr_1^* = \id \]
is a weak equivalence. 
\end{proof}

\begin{LEMMA}\label{LEMMASYM}
 $\Set^{\Delta^{\op}}$ is symmetric and  $\mathcal{C}^{\Delta^{\op}}$ is strongly symmetric for any Abelian category $\mathcal{C}$. 
\end{LEMMA}
\begin{proof}
The standard cosimplicial objects 
\[ \Delta_n^T = \{(x_0, \dots, x_n)  \in \R^{[n]}\ | \ \sum x_i = 1 \} \]
in $\mathcal{TOP}$ and 
\[ (\Delta^{\circ}_n)_i = \Z[\Hom_{\Delta}([i], [n])]  \]
in $(\Ab^{\fg})^{\Delta^{\op}}$  extend to functors $\Delta_n^{T,s}$ and $\Delta^{\circ,s}$ from $\FinSet$ (the first by obvious action of $S_n$ on the coordinates and the second using the action of Lemma~\ref{LEMMAFINSET}).
Defining $\Cfrak:=N_{\Delta_n^{T,s}} R_{\Delta_n^T}$, and $\Cfrak:= N_{\Delta^{\circ, s}}R_{\Delta^{\circ}}$, respectively, we have
$\iota^*\Cfrak = N_{\Delta_n^{T}}R_{\Delta_n^T}$, and $\iota^*\Cfrak = N_{\Delta_n^{\circ}}R_{\Delta_n^{\circ}}$, respectively, which come equipped with a
natural transformation (unit of the $N$, $R$ adjunction) $\id \to \iota^*\Cfrak$. It is a weak equivalence of simplicial sets, and an isomorphism in $\mathcal{C}^{\Delta^{\op}}$, respectively. 
\end{proof}

\begin{LEMMA}\label{LEMMACPROFUNCTOR}
Let $\mathcal{C}$ be Abelian. 
For the functor $\Cfrak$ of Lemma~\ref{LEMMASYM} we have
\[ \Cfrak =  L_{\mathcal{C}}(({}^t\!C)^{\op}) = R_{\mathcal{C}}(C^{\op}) \]
where $C$ and ${}^t C$ are the $\Ab$-enriched pro-functors $C: \FinSet \to \Delta$  (resp.\@ ${}^t\!C: \Delta \to \FinSet$) with $C$ left adjoint:
\begin{eqnarray*} 
  C: n, m &\mapsto& \Hom(\Delta^{\circ}_n, \Delta^{\circ}_m) \in \Hom(\Delta^{\op} \times \FinSet, \Ab)     \\
 {}^t C: n, m &\mapsto& \Hom(\Delta^{\circ}_n, \Delta^{\circ}_m) \in \Hom(\FinSet^{\op} \times \Delta, \Ab)     
  \end{eqnarray*}
and there is an isomorphism of pro-functors $C\, \iota \cong \id$ (resp.\@ ${}^t\!\iota\, {}^t\!C \cong \id$). 
\end{LEMMA}
In the definition of $C$ (resp.\@ ${}^t\!C$) $\Delta^{\circ}$ in the right slot (resp.\@ left slot) is considered as functor on $\FinSet$ as defined by Lemma~\ref{LEMMAFINSET}. 
\begin{proof}
Recall that $N_{\Delta_n^{\circ}}$ commutes with colimits and $\otimes$ (see proof of Theorem~\ref{SATZDOLDKANII}).
Hence we have
\[\Cfrak(X)_n = \Hom(\Delta_n^{\circ}, \int^m \Delta_m^{\circ} \otimes X_m) = \int^m \Hom(\Delta_n^{\circ},  \Delta_m^{\circ}) \otimes X_m = L_{\mathcal{C}}(({}^t\!C)^{\op})(X)   \]
(which is isomorphic to $X_n$ as object in $\mathcal{C}$, of course, but now gives a $\FinSet^{\op}$-valued functor).

We also have: 
\[ \Cfrak(X)_n = \int_m \mathcal{HOM}(\Hom(\Delta_m^{\circ},  \Delta_n^{\circ}), X_m) = R_{\mathcal{C}}(C^{\op})(X) \]
which is also isomorphic to $X_n$. Applying this to $\mathcal{C} = \Ab$, we see that $L_{\Ab}(C^{\op})$ is left adjoint to $L_{\Ab}(({}^t\!C)^{\op})$, thus
$C^{\op}$ right adjoint to $({}^t\!C)^{\op}$ and thus $C$ left adjoint to ${}^t\!C$.
\end{proof}
This shows also that there is a similar operator: 
\[ \Cfrak^n: \mathcal{C}^{\Delta^{\op} \times \cdots \times \Delta^{\op}} \to \mathcal{C}^{\FinSet^{\op} \times \cdots \times \FinSet^{\op}} \]
by applying the $\Ab$-enriched pro-functor $C \times \cdots \times C$.

For {\em general} symmetric $\mathcal{C}^{\Delta^{\op}}$, we have something weaker: 

\begin{LEMMA}\label{LEMMAC2} Let $\mathcal{C}^{\Delta^{\op}}$ be symmetric. 
There is an  operator\footnote{and similar for higher powers, which will not be needed. }
\[ \Cfrak^2: \mathcal{C}^{\Delta^{\op} \times \Delta^{\op}} \to \mathcal{C}^{\FinSet^{\op} \times \FinSet^{\op}}  \]
with a natural transformation
\[ \id \to (\iota,\iota)^* \Cfrak^2 \]
which is a composition of natural transformations which are point-wise in one or the other direction weak equivalences. 
\end{LEMMA}
\begin{proof}
For $X \in \{\Delta^{\op}, \FinSet^{\op}\}$ denote by
\[ (\Cfrak, \id): \mathcal{C}^{\Delta^{\op} \times X} \to \mathcal{C}^{\FinSet^{\op} \times X}   \]
the point-wise application of $\Cfrak$ and similarly for $(\id, \Cfrak)$.
We then set $\Cfrak^2 = (\Cfrak, \id)(\id, \Cfrak)$, noting that the factors  do not commute in general. 
However, $(\id, \iota)^*$ commutes with $(\Cfrak, \id)$. And thus
\[ (\iota,\iota)^* \Cfrak^2 = (\iota, \id)^* (\Cfrak, \id) (\id, \iota)^* (\id, \Cfrak)    \]
and there is a morphism
\[ \id \to (\iota,\iota)^* \Cfrak^2 \]
defined as the compositon
\[ \id \to  (\id, \iota)^* (\id, \Cfrak)  \to   (\iota, \id)^* (\Cfrak, \id) (\id, \iota)^* (\id, \Cfrak) = (\iota,\iota)^* \Cfrak^2    \]
The two morphisms are point-wise in the first (resp.\@ second variable a weak equivalence). 
\end{proof}

\begin{PROP}\label{PROPIOTA}
For any cocomplete category $\mathcal{C}$, the composition
\[ \xymatrix{ \iota^* \iota_! \iota^* \ar[r]^-{\iota^* c} \ar[r] & \iota^* \ar[r]^-{u \iota^*} & \iota^* \iota_! \iota^*   } \]
of functors $\mathcal{C}^{\FinSet^{\op}} \to \mathcal{C}^{\Delta^{\op}}$
is homotopic to the identity. 
\end{PROP}
\begin{proof}[Proof.]
Recall diagram
\[ \xymatrix{   
& (\FinSet, \coprod)^{\op} & \\
(\Delta, \coprod)^{\op} \ar[ru]^{\iota_{\delta}} &  & (\Delta, \ast)^{\op} \ar[lu]_{\iota_{\dec}}  \ar[ll]^{\mathrm{forget}}
}\]

All cooperads (being exponential) translate to functors 
\[ \Delta_{\emptyset}(\cong \Delta^{\op}_{\act}) \to \Cat^{\PF} \]
which are compatible with products, $p_{\emptyset}$ is mapped to  ${}^t \pi$, 
and the oplax transformations $\iota_{\delta}$ and $\iota_{\dec}$ (also compatible with products) have lax mates (consisting point-wise of ${}^{t}\!\iota$) denoted
${}^{t}\!\iota_{\delta}$ and ${}^{t}\!\iota_{\dec}$. Since $\iota_{\delta}$ is Cartesian (\ref{PARIOTACOCART}), ${}^{t}\!\iota_{\delta}$ is natural, and since  $\iota_{\dec}$ is coCartesian (\ref{PARIOTACOCART}), $\iota_{\dec}$ is natural. Therefore, the following are lax and oplax extensions, respectively, of of $\iota\, {}^t\! \iota\, \iota$:
\[ \iota_{\dec} \,{}^t\! \iota_{\dec} \, \iota_{\dec} \qquad  \iota_{\delta}  \,{}^t\!  \iota_{\delta} \, \iota_{\dec}. \]

Thus Proposition~\ref{PROPCREATIONHOMOTOPY} shows that $\uHom(\iota \,{}^t\! \iota\, \iota, \iota \,{}^t\! \iota\, \iota)$ is connected and thus
any two natural transformations $\iota^* \iota_! \iota^* \Rightarrow \iota^* \iota_! \iota^*$ given by morphisms of pro-functors are homotopic. 
Notice that the above being ``pseudo on $p_{\emptyset}$'' translates to the 1-cofinality of $\iota$ (Lemma~\ref{LEMMAIOTAFINAL}).
\end{proof}

\begin{KOR}
If $\mathcal{W}$ is a class of weak equivalences such that \ref{PARAXIOMSW}, 1.\@ holds then
the unit $\id \to \iota^* \iota_!$ is a weak equivalence on objects in the image of $\iota^*$. 
Furthermore
\[ \id  \to (\iota, \iota)^* (\iota, \iota)_!  \]
is a composition of morphisms which are point-wise in one or the other direction weak equivalences, on objects in the image of $(\iota, \iota)^*$.
\end{KOR}
\begin{proof}
The first statement follows directly from Proposition~\ref{PROPIOTA}, while the second follows from a similar reasoning as in the proof of Lemma~\ref{LEMMAC2}.
\end{proof}

\subsection{The non-Abelian Eilenberg-Zilber theorem}

\begin{PAR}\label{PARCONJ}
There is a diagram
\[ \xymatrix{ \Delta \ar[d]_{\iota} \ar@<3pt>[r]^-{\delta} & \ar@<3pt>[l]^-{\dec} \ar[d]^{(\iota, \iota)} \Delta \times \Delta  \\
\FinSet  \ar@<3pt>[r]^-{\delta_s} & \ar@<3pt>[l]^-{\dec_s}  \FinSet \times \FinSet }
 \]
 in which the corresponding squares commute. The functors $\dec_s = \coprod$  and $\delta_s$ are adjoint on $\FinSet$ (as the coproduct is always adjoint to the diagonal by definition) with unit 
 \[ u: \id \Rightarrow \delta_s \dec_s    \]
 \[  ([n], [m]) \to ([n] \ast [m],[n] \ast [m])  \]
 given by the obvious face maps and counit
 \[ c:   \dec_s \delta_s \Rightarrow  \id \]
 \[ [n] \ast [n]  \to [n]    \]
 given by (the opposite of) the diagonal.
 Notice that the morphism $u$ is point-wise in $\Delta$ whereas $c$ is not, reflecting the fact that $\dec$ is not a coproduct in $\Delta$.
 
 For $\FinSet^{\op}$-diagrams we get, in particular, an isomorphism $\delta_s^* \cong \dec_{s,*}$ (or equivalently $\delta_{s,!} \cong \dec_s^*$). 
 Assume $\mathcal{C}^{\Delta^{\op}}$ is symmetric w.r.t.\@ a class of weak equivalences $\mathcal{W}$ satisfying \ref{PARAXIOMSW}. 
 While $c$ is not a morphism in $\Delta$ we may may conjugate the morphism $c^{\op}: \id \to \delta^* \dec^*$ to a morphism $\mlq c^{\op} \mrq$
 in $\Fun(\mathcal{C}^{\Delta^{\op}}, \mathcal{C}^{\Delta^{\op}})[\mathcal{W}^{-1}]$ by means of the commutative diagram
\[ \xymatrix{ \id \ar[d]_{\in \mathcal{W}} \ar@{-->}[rr]^-{\mlq c^{\op} \mrq } & & \delta^*\dec^*   \ar[d]^{\in \mathcal{W}}  \\
 \iota^* \Cfrak  \ar[rr]_-{\iota^* c^{\op} \Cfrak } & &  \iota^* \delta_s^*\dec_s^* \Cfrak = \delta^*\dec^* \iota^* \Cfrak    }\] 
(using that $\delta^* \dec^*$ preserves weak equivalences by Lemma~\ref{LEMMADELTADEC}\footnote{of course not needed  in the strongly symmetric case}). If  $\mathcal{C}^{\Delta^{\op}}$ is strongly symmetric, e.g.\@ if $\mathcal{C}$ is Abelian, 
then $\mlq c^{\op} \mrq$ is an honest morphism (without need to invert weak equivalences). 
\end{PAR}
\begin{DEF}[Alexander-Whitney and Eilenberg-Zilber morphism]  \label{DEFAWEZ}
\begin{enumerate}
\item The morphism
 \[ \Awfrak: \delta^* \to \dec_* \]
is defined as the composition
 \[ \xymatrix{ \delta^* \ar[r] & \dec_*  \dec^*  \delta^*  \ar[rr]^-{\dec_* u^{\op}}  & & \dec_* }. \]
 \item Assume $\mathcal{C}^{\op}$ is symmetric w.r.t.\@ a class of weak equivalences $\mathcal{W}$ satisfying \ref{PARAXIOMSW}. 
 The morphism
  \[ \Ezfrak: \dec_* \to  \delta^* \]
  in $\Fun(\mathcal{C}^{\Delta^{\op} \times \Delta^{\op}}, \mathcal{C}^{\Delta^{\op}})[\mathcal{W}^{-1}]$\footnote{By abuse of notation, here $\mathcal{W}$ denotes the class of morphisms that are object-wise weak equivalences.} is defined as the composition (cf.\@ \ref{PARCONJ}):
\[  \xymatrix{ \dec_*  \ar@{->}[rr]^-{\mlq c^{\op} \mrq \dec_*} & &  \delta^*\dec^*  \dec_* \ar[r] &   \delta^*. } \]
If $\mathcal{C}^{\Delta^{\op}}$ is strongly symmetric, e.g.\@ if $\mathcal{C}$ is Abelian, then we understand $\Ezfrak \in \Fun(\mathcal{C}^{\Delta^{\op} \times \Delta^{\op}}, \mathcal{C}^{\Delta^{\op}})$ 
as an honest natural transformation (no need to  invert anything).
 \end{enumerate}
\end{DEF}
If $\mathcal{C}^{\Delta^{\op}}$ is strongly symmetric then $\Ezfrak$ is still not inverse to $\Awfrak$.
If $\mathcal{C}$ is Abelian\footnote{and provided the symmetry operator $\Cfrak$ is the canonical one given by Lemma~\ref{LEMMASYM}}, then $\Awfrak$ is the familiar Alexander-Whitney morphism, and $\Ezfrak$ is  the familiar Eilenberg-Zilber morphism usually defined in terms of shuffles, see Proposition~\ref{PROPAWEZAB}.

\begin{PAR}
In the strongly symmetric case, $\Awfrak$ and $\Ezfrak$ are in morphisms in``standard form'' (\ref{PARSTANDARD}), image of the morphisms of correspondences 
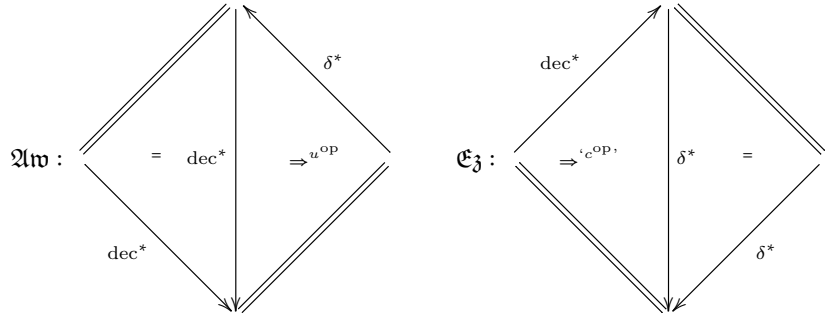
\[ 
\Awfrak: \vcenter{\xymatrix{ 
& & \ar[dddd]_{\dec^*} \\
&&&&\\
\ar@{=}[ruru]^{} \ar[rdrd]_{\dec^*}  \ar@{}[rr]|{=} && \ar@{}[rr]|{\Rightarrow^{u^{\op}}} &  & \ar[lulu]_{\delta^*} \ar@{=}[ldld]^{} \\
&&&&\\
& & 
}} \qquad
\Ezfrak: \vcenter{\xymatrix{ 
& & \ar[dddd]^{\delta^*} \\
&&&&\\
\ar[ruru]^{\dec^*} \ar@{=}[rdrd]  \ar@{}[rr]|{\Rightarrow^{\mlq c^{\op} \mrq}} && \ar@{}[rr]|= &  & \ar@{=}[lulu] \ar[ldld]^{\delta^*} \\
&&&&\\
& & 
}}\]
a fact that will be convenient to prove all kinds of compatibilities using compositions in the 2-categories $\Cor(-, -)$ (\ref{PARSTANDARD}). 
\end{PAR}

The goal of this section is to prove
\begin{SATZ}[Eilenberg-Zilber, general version]\label{SATZEZ}
Let $\mathcal{C}^{\Delta^{\op}}$ be symmetric with weak equivalences $\mathcal{W}$ satisfying \ref{PARAXIOMSW}, for example $\mathcal{C}$ Abelian (with $\mathcal{W}$ the quasi-isomorphisms), or $\mathcal{C} = \Set$ (with $\mathcal{W}$ the usual weak equivalences)\footnote{If $\mathcal{C}^{\Delta^{\op}}$ is strongly symmetric (W2) is not needed.}.  
Then $\Awfrak$ and $\Ezfrak$ are mutually inverse isomorphisms
\[ \dec_* \cong \delta^* \]
in $\Fun(\mathcal{C}^{\Delta^{\op} \times \Delta^{\op}}, \mathcal{C}^{\Delta^{\op}})[\mathcal{W}^{-1}]$. 
\end{SATZ}
For $\mathcal{C}=\Set$ there are several proofs in the literature, see \cite[Theorem 1.1]{CR05}, \cite[Theorem~1.1]{Ste12}, \cite{Zis15}.

In the non-strongly symmetric case, it follows also that $\dec_*$ preserves morphisms which are weak equivalences point-wise in one direction, because $\delta^*$ does so by assumption. 
For Abelian $\mathcal{C}$ we will make the Theorem more precise, and coherent, with a simpler proof, in section~\ref{SECTEZAB}.

\begin{PAR}
This theorem automatically implies a coherent version: The forgetful morphism of cooperads
\[ \AW:=\mathrm{forget}: (\Delta, \ast)^{\op} \Rightarrow (\Delta, \coprod)^{\op}   \]
can be seen as an oplax transformation in $\Hom^{1-\oplax}(\Delta_{\emptyset}, \Cat^{\PF})$ and induces as oplaxness constraint on $p_1$ a
morphism $\Aw: \dec \Rightarrow {}^t\!\delta$ such that $R(\Aw)$ is the map $\Awfrak$ above (cf.\@ also the proof of Theorem~\ref{SATZCOHEZ}).
If $(\mathcal{C}, \otimes)$ is monoidal,
on Day convolutions ``$\mathrm{forget}$'' induces 
\[ \AWfrak:=R(\mathrm{forget}): D((\Delta, \coprod)^{\op}, (\mathcal{C}, \otimes)^{\vee}) \rightarrow D((\Delta, \ast)^{\op}, (\mathcal{C}, \otimes)^{\vee})  \]
which is nothing but
\[ \AWfrak: (\mathcal{C}^{\Delta^{\op}}, \otimes)^{\vee} \rightarrow  (\mathcal{C}^{\Delta^{\op}}, \tildeotimes)^{\vee}  \]
with $\otimes = \delta^* (- \boxtimes - )$ the point-wise product and $\tildeotimes = \dec_* (- \boxtimes - )$ and the morphism on monoidal products is induced by the map $\Awfrak$.
If $\otimes$ (with a constant object) preserves weak equivalences, Theorem~\ref{SATZEZ} implies thus that these yield monoidal cooperads
\[ \AWfrak[\mathcal{W}^{-1}]: (\mathcal{C}^{\Delta^{\op}}[\mathcal{W}^{-1}], \otimes)^{\vee} \rightarrow  (\mathcal{C}^{\Delta^{\op}}[\mathcal{W}^{-1}], \tildeotimes)^{\vee}  \]
and the functor is an equivalence. If $\otimes$ was the Cartesian product $\times$, this is actually trivial, cf.\@ Theorem~\ref{SATZEZTRIVIAL} below.
\end{PAR}

The following theorem puts this into perspective, but does not imply Theorem~\ref{SATZEZ} directly. 
\begin{SATZ}\label{SATZFINSETEQ}
Let $\mathcal{C}^{\Delta^{\op}}$ be  symmetric with weak equivalences $\mathcal{W}$ satisfying \ref{PARAXIOMSW}. Set $\mathcal{W}_{s} := (\iota^*)^{-1}\mathcal{W}$.
Then the functor induced on the $\infty$-categorical localizations
\[ \iota^*: \mathcal{C}^{\FinSet^{\op}}[\mathcal{W}_{s}^{-1}] \to \mathcal{C}^{\Delta^{\op}}[\mathcal{W}^{-1}] \]
is an equivalence of $\infty$-categories.
\end{SATZ}
\begin{proof}
We show that $\iota^*$ and $\Cfrak$ constitute inverse functors up to a chain of point-wise weak equivalences. 
We have a weak equivalence
\[ \id \Rightarrow \iota^* \Cfrak \]
by definition. This also shows that $\Cfrak$ preserves weak equivalences. Furthermore, we have a chain of natural transformations 
\[  \Cfrak \iota^* \Leftarrow \iota_! \iota^* \Cfrak \iota^* \Leftarrow \iota_! \iota^* \Rightarrow \id \]
which are all point-wise in $\mathcal{W}_s$, for  apply $\iota^*$ and extend to a commutative diagram
\[ \xymatrix{ &  \iota^*\Cfrak \iota^* \ar[d]_{\mathcircled{1}}  & \ar[l]_{\mathcircled{5}} \iota^*   \ar[d]^{\mathcircled{3}} & \\ 
\iota^* \Cfrak \iota^* & \ar[l]^{\mathcircled{2}} \iota^*  \iota_! \iota^* \Cfrak  \iota^* & \ar[l] \iota^*  \iota_! \iota^* \ar[r]_{\mathcircled{4}} &  \iota^*} \]
The natural transformations $\mathcircled{1}$--$\mathcircled{4}$  are (point-wise) weak equivalences by Proposition~\ref{PROPIOTA} and map $\mathcircled{5}$ is a weak equivalence by definition of ``symmetric''. 
\end{proof}

\begin{PROP}\label{PROPEZ}
Assume $\mathcal{C}^{\Delta^{\op}}$ is symmetric w.r.t.\@ a class of weak equivalences $\mathcal{W}$ satisfying \ref{PARAXIOMSW}\footnote{If $\mathcal{C}^{\Delta^{\op}}$ is strongly symmetric \ref{PARAXIOMSW}, 2.\@ is not needed.}. 
Then the compositions\footnote{Note that $\dec_* \dec^*$ preserves weak equivalences, being isomorphic to $\mathcal{HOM}(\Delta_1, -)$ by Proposition~\ref{PROPTWOENRICHMENTS}.}
\[ \xymatrix{ \dec_* \dec^*  \ar[rr]^-{\dec_*  \dec^* \mlq c^{\op} \mrq } & & \dec_*  \dec^* \delta^* \dec^*   \ar[rr]^-{\dec_* u^{\op} \dec^*} & & \dec_* \dec^* } \]
and
\[ \xymatrix{ \delta^*  \ar[rr]^-{\mlq c^{\op} \mrq \delta^*} & & \delta^* \dec^* \delta^*   \ar[rr]^-{\delta^* u^{\op}} & & \delta^* } \]
are the identity in $\Fun(\mathcal{C}^{\Delta^{\op}}, \mathcal{C}^{\Delta^{\op}})[\mathcal{W}^{-1}]$, and in $\Fun(\mathcal{C}^{\Delta^{\op} \times \Delta^{\op}}, \mathcal{C}^{\Delta^{\op}})[\mathcal{W}^{-1}]$,
respectively. 

If $\mathcal{C}^{\Delta^{\op}}$ is strongly symmetric, then 
\[ \xymatrix{  \dec^*  \ar[rr]^-{ \dec^* \mlq c^{\op} \mrq } & &  \dec^* \delta^* \dec^*   \ar[rr]^-{ u^{\op} \dec^*} & & \dec^* } \]
(which makes sense now, because $\mlq c^{\op} \mrq$ is an honest morphism) is even the identity without inverting weak equivalences. 
\end{PROP}
\begin{proof}
We have the following commutative diagram 
\[ \footnotesize \xymatrix{ 
\dec_* \iota^*  \dec^*_s   \Cfrak  \ar@{=}[r] \ar[d]^{c^{\op}} & \dec_*  \dec^*  \iota^* \Cfrak  \ar@{<-}[rr]^-\sim \ar[d]^{c^{\op}} & & \dec_* \dec^* \ar@{-->}[d]^{\mlq c^{\op} \mrq}   \\
\dec_* \iota^*   \dec^*_s \delta^*_s \dec^*_s  \Cfrak  \ar[d]^{u^{\op}}  \ar@{=}[r]  & \dec_*  \dec^* \iota^* \delta^*_s \dec^*_s  \Cfrak  \ar@{=}[r]   & \dec_*  \dec^* \delta^* \dec^*  \iota^* \Cfrak  \ar@{<-}[r]^-\sim  \ar[d]^{u^{\op}} & \dec_*  \dec^* \delta^* \dec^* \ar[d]^{u^{\op}}   \\
\dec_* \iota^* \dec^*_s   \Cfrak \ar@{=}[rr]  & & \dec_*  \dec^* \iota^*  \Cfrak   \ar@{<-}[r]^{\sim}   & \dec_* \dec^*  \\
}\]
in which the left vertical composition is the identity and the upper horizontal and lower horizontal compositions are equal. This treats the first composition. In the strongly symmetric case, we can omit the  $\dec_*$ and get a similar diagram with isomorphisms instead of weak equivalences. 

We also have the following commutative diagram:
\[ \xymatrix{ 
  \iota^* \delta^*_s    (\iota,\iota)_!   \ar[d]_{c^{\op}} & \ar[l]_{\mathcircled{3}} \iota^* \iota_!  \delta^* \ar[d]^{c^{\op}} \ar[r]  &\iota^* \Cfrak  \delta^* \ar[d]^{c^{\op}} & \delta^* \ar@{-->}[dd]^{\mlq c^{\op} \mrq} \ar[l]^-{\sim}  \ar@/_10pt/[ll]_{\mathcircled{1}}\\
 \iota^* \delta^*_s\dec^*_s   \delta^*_s   (\iota,\iota)_!   \ar[dd]_{u^{\op}} & \ar[l] \iota^*  \delta^*_s \dec^*_s \iota_! \delta^* \ar[r] \ar@{=}[d] &  \iota^*  \delta^*_s \dec^*_s \Cfrak \delta^*  \ar@{=}[d]& \\
  & \delta^*\dec^*  \iota^* \iota_!  \delta^*  \ar[r] & \delta^*\dec^*  \iota^* \Cfrak \delta^* & \delta^* \dec^* \delta^* \ar@/^10pt/[ll]^{\mathcircled{2}} \ar[d]^{u^{\op}} \ar[l]_-{\sim} \\ 
  \iota^* \delta^*_s    (\iota,\iota)_!   & \ar[l]^{\mathcircled{3}} \iota^* \iota_!  \delta^* &  & \delta^*  \ar[ll]^{\mathcircled{1}}
}\]
in which the left vertical composition is the identity.
By Lemma~\ref{LEMMAC2}, and since we assume that $\delta^*$ maps morphisms which are point-wise in either direction a weak equivalence to weak equivalences, it suffices to see that 
\[ \xymatrix{ \delta^*  \ar[rr]^-{\mlq c^{\op} \mrq \delta^*} & & \delta^* \dec^* \delta^*   \ar[rr]^-{\delta^* u^{\op}} & & \delta^* } \]
is the identity {\em on objects in the image of $(\iota, \iota)^*$}.  Then by Proposition~\ref{PROPIOTA}
$\mathcircled{1}$ and $\mathcircled{2}$ are weak equivalences on objects in the image in $(\iota, \iota)^*$, noting that $\delta^*$ of such is in the image of $\iota^*$. 
We are left to show that $\mathcircled{3}$ is a weak equivalence. However the following is commutative: 
\[ \xymatrix{
 \delta^* \ar[d]_{\mathcircled{5}} \ar[r]^{\mathcircled{4}} & \iota^* \iota_!  \delta^* \ar[d]^{\mathcircled{3}} \\
 \delta^*   (\iota, \iota)^* (\iota,\iota)_!  \ar@{=}[r] & \iota^* \delta_s^*    (\iota,\iota)_!    
} \]
in which $\mathcircled{4}$ and $\mathcircled{5}$ are weak equivalences on objects in the image of $(\iota, \iota)^*$ by Proposition~\ref{PROPIOTA}.
\end{proof}

\begin{LEMMA}\label{LEMMAALTAWEZ}
The following diagram is commutative
\[ \xymatrix{ \dec_*  \ar[rr]^-{\Ezfrak}   \ar[d]_{\in \mathcal{W}} & & \delta^* \ar[rr]^{\Awfrak}  & &  \dec_* \\
\dec_* \dec^* \dec_* \ar[rr]_-{\dec_* \dec^*  \mlq c^{\op} \mrq \dec_*}& & \dec_* \dec^* \delta^*\dec^*  \dec_* \ar[rr]_-{\dec_*u^{\op}\dec^* \dec_*} &&  \dec_* \dec^* \dec_*   \ar[u]_{\in \mathcal{W}}  } \] 
in which the indicated  morphisms are point-wise weak equivalences and inverse to each other, and the following diagram is commutative:
\[ \xymatrix{ \delta^*  \ar[r]^-{\Awfrak} \ar[d]_{\mlq c^{\op} \mrq \delta^*} & \dec_* \ar[d]^{\Ezfrak} \\
\delta^* \dec^* \delta^* \ar[r]_-{\delta^* u^{\op}} & \delta^*  } \]  
\end{LEMMA}
\begin{proof}
Follows from the commutativity of the following diagrams where the unlabeled morphisms are the various units and counits:
\[ \xymatrix{ \dec_*  \ar[rr]^-{\Ezfrak}   \ar[d]^{\sim}  & & \delta^*  \ar[d]^{\sim}  \ar[rrd]^{\Awfrak}     \\
\dec_* \dec^* \dec_* \ar[rrd]_-{\dec_* \dec^*  \mlq c^{\op} \mrq \dec_*} \ar[rr]^{\dec_* \dec^* \Ezfrak}  & &\dec_* \dec^* \delta^*   \ar[rr]^{\dec_* u^{\op} } & & \dec_* \\
 & & \dec_* \dec^* \delta^*\dec^*  \dec_* \ar[u] \ar[rr]_-{\dec_*u^{\op}\dec^* \dec_*} & & \dec_* \dec^* \dec_*   \ar[u]^{\sim}  } \] 
\[ \xymatrix{ & & \delta^* \ar@/_40pt/[lldd]_{\Awfrak} \ar[rd]^{\mlq c^{\op} \mrq \delta^*} \ar[ld] & \\
& \dec_* \dec^* \delta^* \ar[ld]^{\dec_* u^{\op}} \ar[rd]^{ \mlq c^{\op} \mrq \dec_* \dec^*  \delta^* } & &  \delta^* \dec^*  \delta^* \ar[ld] \ar@{=}[dd] \\
\dec_*  \ar@/_40pt/[rrdd]_{\Ezfrak} \ar[rd]^{\mlq c^{\op} \mrq \dec_*} & &  \delta^* \dec^* \dec_* \dec^* \delta^* \ar[ld]^{\delta^* \dec^* \dec_*  u^{\op}}  \ar[rd] \\
 & \delta^* \dec^* \dec_* \ar[rd] & & \delta^* \dec^* \delta^* \ar[ld]^{\delta^*   u^{\op}}  \\
& & \delta^* \\ } \]
\end{proof}

\begin{proof}[Proof of Theorem~\ref{SATZEZ}.]
Lemma~\ref{LEMMAALTAWEZ} shows that Proposition~\ref{PROPEZ} implies Theorem~\ref{SATZEZ}.
\end{proof}

\begin{SATZ}[Trivial Eilenberg-Zilber] \label{SATZEZTRIVIAL}
If $(\mathcal{C}, \times)$ is a {\em Cartesian} monodial ($\infty$-)category then the functor
\[ \AWfrak:=R(\mathrm{forget}): (\mathcal{C}^{\Delta^{\op}}, \times)^{\vee} \rightarrow  (\mathcal{C}^{\Delta^{\op}}, \widetilde{\times})^{\vee}  \]
is an equivalence (i.e.\@ no need to invert weak equivalences). 
\end{SATZ}
\begin{proof}Consider the diagram
\[ \xymatrix{
( \Delta^{\op} \times \Delta^{\op} ) \coprod  ( \Delta^{\op} \times \Delta^{\op} )  \ar@/_7pt/[rr]_-{\pr_1 \coprod \pr_2} \ar@/^7pt/[rr]^-{ \dec \coprod \dec}_-{\Downarrow} \ar[d]_{d} & &   \Delta^{\op} \coprod \Delta^{\op} \ar[d]^{d} \\ 
 \Delta^{\op} \times \Delta^{\op}  \ar[rr]_{\dec} & &  \Delta^{\op} 
} \]
in which the square with the upper horizontal morphism commutes. 
We have 
\[  A \times B = \delta^* d_*  (\pr_1 \coprod \pr_2)^* (A,B)  = d_* (A,B)\]
and
\[  \dec_* A \boxtimes B =  \dec_* d_*  (\pr_1 \coprod \pr_2)^* (A,B). \]
The morphism $\AWfrak$ is induced by the composition
\[ d_*  \to d_* (\dec \coprod \dec)_* (\dec \coprod \dec)^* \to d_* (\dec \coprod \dec)_* (\pr_1 \coprod \pr_2)^*.   \]
This however induces the {\em isomorphisms} $\id \cong \dec_* \pr_{1/2}^* $ (Lemma~\ref{LEMMAEXACT3}).

Alternative proof using Proposition~\ref{PROPDAY}. $(\mathcal{C}, \times)^{\vee} \to \OOO^{\op}$ is a {\em cofibration} of cooperads via the diagonal
$\mathcal{C} \to \mathcal{C} \times \mathcal{C}$. The Day convolution thus sees only the projections of
the pro-functors $\Delta^{\op} \times \Delta^{\op} \to \Delta^{\op}$ given by $\dec$ and ${}^{t} \delta$. Those projections are the identities in both cases by Lemma~\ref{LEMMAEXACT3}.
\end{proof}
The second proof shows that the statement holds more generally, if $\mathcal{C} \to \OOO$ is a {\em fibration} of operads.

\subsection{Coherent transformations}\label{SECTCOHTRANS}

\begin{PAR} 
For $F, G: \mathcal{C} \to \mathcal{D}$ usual functors between 1-categories, we have:
\[ \Hom(F, G) = \int_c \Hom(F(c), G(c))  = \lim_{\tw \mathcal{C}} \Hom(F(c'), G(c)) .  \]
We start by extending this formula to a relative setting of (co)fibrations of (co)operads and then to the enriched setting.  
\end{PAR}

\begin{DEF}\label{DEFCOPC} Let $p: \mathcal{C} \to I$ be a morphism of operads. The operad
 \[ ( \mathcal{C}^{\op} \times \mathcal{C} )_{/I}. \]
  consists of objects $(c',c,\nu)$ where $\nu:p(c) \to p(c')$ is an active morphism.
 Morphisms $(c',c,\nu) \to (e',e,\mu)$  are morphisms  in $\tw I$
\begin{equation} \label{eqmortwi} \vcenter{ \xymatrix{ p(c') \ar[d]_{\iota'}  \ar[r]^{\nu} & p(c) \ar[d]^{\iota} \\
x \ar[r] & y \ar[d]^{\alpha} \\
p(e') \ar[r]^{\mu} \ar[u]^{\alpha'} & p(e)  } } \end{equation}
 together with a lift $\iota_\bullet c \to e$ (over $\alpha$) and $e' \to \iota'_\bullet c'$ (over $\alpha'$). 
 (Notice that a push-forward along inert morphisms exists by the definition of operad. )
\end{DEF}
One checks that $( \mathcal{C}^{\op} \times \mathcal{C} )_{/I} \to \mathcal{C}$ is a cofibration and that $( \mathcal{C}^{\op} \times \mathcal{C} )_{/I}$ is an operad. 
If $I$ is a category then more simply
\[( \mathcal{C}^{\op} \times \mathcal{C} )_{/I} = (\mathcal{C}^{\op} \times \mathcal{C}) \times_{/I^{\op} \times I} \tw I.  \]
In any case there is a functor
\[ \iota: \tw \mathcal{C} \to (\mathcal{C}^{\op} \times \mathcal{C} )_{/I} \] 
which is a discrete cofibration of operads. It is classified by a functor of operads
\begin{equation} \label{eqhomoperads} \Hom: (\mathcal{C}^{\op} \times \mathcal{C} )_{/I} \to (\Set, \times) \end{equation}
such that we have
\[ \Hom(X, Y, \tau)^{\times} = \Hom_{\tau}(X, Y)  \]
where the RHS is the set of morphisms mapping to $\tau$ in $I$ and $(-)^{\times}$ is the obvious functor $\Set^{\times} \to \Set$ where $\Set^{\times}$ is the underlying category (of operators) of $(\Set, \times)$, i.e.\@ with $\Set^{\times}_{[n]} = \Set^n$.

\begin{DEF} 
Let 
\[ F: (\mathcal{C}^{\op} \times \mathcal{C} )_{/I} \to (\mathcal{D}, \times) \]
\[ c', c, \nu \mapsto F_{\nu} (c', c) \]
be a functor of operads.
Then the {\bf relative end} of $F$ is defined as
\[ \int_{\mathcal{C}/I} F := \lim_{\tw \mathcal{C}} (F \circ \iota)^{\times}. \]
 For $I = \OOO$  it is also called the {\bf operadic end}. Recall that $(-)^\times: \mathcal{D}^\times \to \mathcal{D}$ is the functor on
the category of operators given by $\times$. In the limit expression $\tw \mathcal{C}$ is just considered as category of operators. 
\end{DEF}

\begin{BEM}
By Lemma~\ref{LEMMAOPTWIST} there is a localization $\mathrm{tw}^{\op}\,\mathcal{C} \to \tw \mathcal{C}$. Therefore it is also the same as just the limit over the dual twisted arrow category of the 
category of operators. 
\end{BEM}

\begin{LEMMA}
We have
\[ \xymatrix{ & \int_{\mathcal{C}/I} F  \cong \lim ( \prod_{c \in \mathcal{C}_{[1]}} F_{\id_{p(c)}}(c, c) \\
\ar@<3pt>[r] \ar@<-3pt>[r] & \prod_{c' \in \mathcal{C}, c \in \mathcal{C}_{[1]}, \nu: p(c') \to p(c)} \Hom(\Hom_{\nu}(c', c), F_{\nu}(c', c))  )   } \]
where $\nu$ is an active morphism in $I$, where one of the arrows is induced by
\begin{gather*}
 \prod_{i} F_{\id_{p(c_i)}}(c_i, c_i) \times \Hom_{\nu}(c, c') \to F_{\id_{p(c)}}(c, c)^{\times} \times \Hom_{\nu}(c, c') \\
 \to  F_{\nu}(c, c') 
 \end{gather*}
where the $c_i$ are the components of $c$ and the isomorphism
\[ \prod_{i} F_{\id_{p(c_i)}}(c_i, c_i) \cong F_{\id_{p(c)}}(c, c)^{\times} \] 
follows because $F$ is a functor of operads. 
\end{LEMMA}

\begin{PROP} 
Let $F, G: \mathcal{C} \to \mathcal{D}$ be a functor between operads over an operad $I$.
Then 
\[ \Hom_{/I}(F, G) \cong \int_{\mathcal{C}/I} \Hom_{\nu}(F(c'), G(c)).  \]
\end{PROP}

We also have a relative Yoneda:

\begin{PROP}\label{PROPYONEDA}
Given a functor of operads $F_{\mu}(-,-): (\mathcal{C}^{\op} \times \mathcal{C} )_{/I} \to (\Set, \times)$ we have:
\begin{enumerate}
\item $\int_{c} \Hom_{\pi(\tau)}(\Hom_{\nu}(c, e), F_{\mu}(c',f)) = F_{\tau}(e,f)$
in which the end is a limit over 
\[ \xymatrix{ {}^{\uparrow} c \ar[d]  &  p(c) \ar[r]^{\mu} \ar[d]   & p(f)    \\
 {}^{\downarrow}  c'     & p(c') \ar[r]^{\nu} & p(e) \ar[u]^{\tau}   } \]
(covariantly in $c'$ and contravariantly in $c$)
\item $\int^{c} \Hom_{\nu}(e, c) \times F_{\mu}(c',f) = F_{\tau}(e,f)$
in which the coend is a colimit over
\[ \xymatrix{ {}^{\uparrow} c  & p(c) \ar[r]^{\mu}   & p(f)  \\
{}^{\downarrow} c'  \ar[u]    & p(c')  \ar[u] & p(e) \ar[u]^{\tau}  \ar[l]^{\nu}  } \]
(contavariantly in $c'$ and covariantly in $c$) and $\Hom_{\nu}(e, c') \times F_{\mu}(c,f)$ is shorthand for the collection of sets $\Hom_{\nu_i}(e_i, c'_i)^{\times} \times F_{\mu_i}(c_i,f_i)$, 
where the $f_i$ are the components of $f$.
\end{enumerate}
\end{PROP}

\begin{PAR}
From the Proposition the following formula follows:
\begin{equation}\label{eqyoneda} \int_{\mathcal{C}/I} F_{\nu}(c',c) \cong \int_{((\mathcal{C}^{\op} \times \mathcal{C})/I)/\OOO} \Hom_{\tau}(\Hom_{\nu'}(c', d'), F_{\nu}(c,d)).    \end{equation}
\end{PAR}

\begin{PAR}
Given simplicially enriched functors between simplicially enriched operads $F, G: \mathcal{C} \to \mathcal{D}$, we first define the structure of
simplicially enriched operad on $(\mathcal{C}^{\op} \times \mathcal{C})_{/I}$ defining with the notation of (\ref{eqmortwi})
\[ \uHom((c',c,\nu), (e',e,\mu)) := \prod_{\Hom_{ \tw I}(\nu, \mu) } \uHom_{\alpha}(\iota_\bullet c, e) \times \uHom_{\alpha'}(e', \iota'_\bullet c').  \] 
It is immediate that the functor
\[ \Hom: (\mathcal{C}^{\op} \times \mathcal{C})_{/I} \to (\Set^{\Delta^{\op}}, \times) \]
\[ c, c', \nu \mapsto \uHom_{\nu_0}(c'_0, c_0), \dots, \uHom_{\nu_n}(c'_n, c_n) \]
becomes simplicially enriched via composition in $\mathcal{C}$.
For a simplicially enriched functor
\begin{equation} \label{eqf} F: (\mathcal{C}^{\op} \times \mathcal{C})_{/I} \to (\Set^{\Delta^{\op}}, \times) \end{equation}
we define the {\bf relative enriched end} as 
\[ \xymatrix{ & \int_{\mathcal{C}/I} F  := \lim ( \prod_{c \in \mathcal{C}_{[1]}} F_{\id_{p(c)}}(c, c) \\
\ar@<3pt>[r] \ar@<-3pt>[r] & \prod_{c \in \mathcal{C}, c' \in \mathcal{C}_{[1]}, \nu: p(c) \to p(c')} \uHom(\uHom_{\nu}(c, c'), F_{\nu}(c, c'))  ) .  } \]

We also have 
 for the enriched natural transformations
\[ \uHom_I(F, G) \cong  \int_{ \mathcal{C}/I }  \uHom_{\nu}(F(c'), G(c)) \]
with the {\em relative enriched} end. 
\end{PAR}

The set of coherent transformations (Cordier-Porter) $\uCoh_I(F, G)$ is a coherent enhancement of this (cf.\@ \cite{CP97} for the case of simplicially enriched {\em categories}).
\begin{DEF}\label{DEFHOMHAT}
We define a simplicially enriched functor of operads
\[ \widehat{\uHom}: (\mathcal{C}^{\op} \times \mathcal{C})_{/I} \to (\Set^{\Delta^{\op}}, \times)  \]
specified for $d \in \mathcal{C}_{[1]}$ by\footnote{The $(...)_{\nu}$ means the fiber over $\nu$ of the morphism 
\[ \uHom_{\mathcal{C}_{\act}}(c, c_0)_{[n]} \times \uHom_{\mathcal{C}_{\act}}(c_0, c_1)_{[n]} \times \dots \times \uHom_{\mathcal{C}_{\act}}(c_n, d)_{[n]} \to \Hom_{I_{\act}}(p(c), p(d)) \]
induced by $p$ and composition.}:
\[ c, d, \nu, [n] \mapsto \prod_{c_0, \dots, c_n} ( \uHom_{\mathcal{C}_{\act}}(c, c_0)_{[n]} \times \uHom_{\mathcal{C}_{\act}}(c_0, c_1)_{[n]} \times \dots \times \uHom_{\mathcal{C}_{\act}}(c_n, d)_{[n]}  )_\nu \]
\end{DEF}

\begin{DEF}
Let $F$ be as in (\ref{eqf}). We define the {\bf relative coherent end} as
\[ \oint_{\mathcal{C}/I} F := \int_{((\mathcal{C}^{\op} \times \mathcal{C})/I)/\OOO} \uHom_{\tau}(\widehat{\uHom}_{\nu'}(c', d'), F_{\nu}(c,d))  \]
where the RHS is the operadic enriched end. 
\end{DEF}

\begin{DEF}
Given simplicially enriched functors between simplicially enriched operads $F, G: \mathcal{C} \to \mathcal{D}$ over $I$, define the simplicial set of {\bf coherent transformations} over $I$ as
\[ \uCoh_{I}(F, G) := \oint_{ \mathcal{C}/I} \uHom_{\nu}(F(c'), G(c))   \]
using the relative coherent end. 
\end{DEF}

\begin{PAR}
We will be interested mainly in the case in which the structure on $\mathcal{C}$ is {\em discrete} in which case 
\[ \widehat{\uHom}_{\nu}(c', c) = N(c' \times_{/\mathcal{C}_{\act}} \mathcal{C}_{\act} \times_{/\mathcal{C}_{\act}} c)_{\nu} \]
i.e.\@ the fiber of the nerve of the over-/undercategory associated with $c'$ and $c$ in $\mathcal{C}_{\act}$ over $\nu$, 
and the coherent end is defined in terms of the usual operadic end (not enriched). In this case a coherent transformation $F \Rightarrow G$ is given by
morphisms $F(c) \to G(c)$ for all objects $c$ and for each (active) morphism $c \to c'$ 1-simplices 
\[ \xymatrix{ F(c) \ar[r] \ar[d] \ar@{}[rd]|{\Rightarrow} & F(c') \ar[d] \\
G(c) \ar[r] & G(c') 
 } \] 
 connecting the two compositions and then for each composition $c \to c' \to c''$ a 2-simplex connecting appropriate compositions of the
 0-simplices associated with $c$, $c'$ and $c''$, and 1-simplices associated with $c \to c'$, $c' \to c''$ and $c \to c''$, and so forth. This is very similar to a lax natural transformation which is indeed the special case in which
 the enrichment is in 1-categories (considered as simplicial sets).  
\end{PAR}

\begin{LEMMA}
There is an injection 
\[ \uHom_{I}(F, G) \to \uCoh_{I}(F, G) \]
\end{LEMMA}
\begin{proof}
The morphism is induced by the obvious simplicially enriched natural transformation
\[ \widehat{\uHom}_{\nu}(-,-) \to \uHom_{\nu}(-,-) \]
given by composition and (\ref{eqyoneda}).
\end{proof}

There is the following important relation between coherent transformations and the naive deformations (Definition~\ref{DEFNAIVE}): 

\begin{PROP}\label{PROPCOHERENT}
Consider simplicially enriched operads $\mathcal{C}, \mathcal{D}, I$ with discrete enrichment on $I$.
Let $\mathcal{C} \to I$ and $\mathcal{D} \to I$ be cofibrations of simplicially enriched operads. 
Consider $Y: \mathcal{C}^{\flat} \to \mathcal{D}^{\flat}$  and two $Y_0, Y_1: \mathcal{C} \to \mathcal{D}$ over $I$ that extend $Y$ as in \ref{PARNAIVE1}.

There is a natural morphism
\[ \exp: \uHom(\Delta_1, \uDef_Y(\mathcal{C}, \mathcal{D}))_{Y_0, Y_1} \to \uCoh_I(Y_0, Y_1). \]
\end{PROP}
\begin{proof}
An element $\mu \in \Delta_n \times \Delta_1 \to \underline{\mathrm{Def}}_Y(\mathcal{C}, \mathcal{D})$  is given by morphisms  
\begin{equation} \label{eqinput} \mu_{c,c'}:  \uHom_{\mathcal{C}}(c, c') \to  \uHom_{\mathcal{D}}(\Delta_1 \times \Delta_n \times Y(c), Y(c'))  \end{equation}
compatible with their projections to $\Hom_I(p(c), p(c'))$ and constant equal to $Y$ over identities, 
compatible with composition (via the diagonals $\Delta_1 \times \Delta_n \to (\Delta_1 \times \Delta_n)^2$),
and specialising to $Y_0$, resp.\@ $Y_1$ when restricted to $\{0\}$ and $\{1\}$ (constant in $\Delta_n$).

It is mapped to the following element $\exp(\mu) \in \uCoh_I(Y_0, Y_1)_{[n]}$ of the (enriched) end: It maps the collection
\begin{equation}\label{eqinput2} \Delta_m \times c \to c_0, \Delta_m \times c_0 \to c_1, \dots,  \Delta_m \times c_m \to c' \end{equation}
of morphisms in $\mathcal{C}$
 to the following morphism
\[ \Delta_n \times \Delta_m \times Y_0(c)=Y(c) \to Y_1(c')=Y(c'). \]
Inserting the morphisms (\ref{eqinput2}) into (\ref{eqinput}) gives rise to a collection
\[ \Delta_1 \times \Delta_n \times \Delta_m \times Y(c_i) \to Y(c_{i+1}) \]
Composing gives a morphism
\[ (\Delta_1)^{m+2} \times (\Delta_n)^{m+2} \times (\Delta_m)^{m+2} \times Y(c) \to Y(c'). \]
This yields a morphism
\[ \Delta_n \times \Delta_m \times Y(c) \to Y(c') \]
mapping $\Delta_n$ to $(\Delta_n)^{m+2}$ diagonally, $\Delta_m$ to $(\Delta_m)^{m+2}$ diagonally
and to $(\Delta_1)^{m+2}$ via the $m+2$ morphisms $\Delta_m \to \Delta_1$ in their natural order, the first being constant 0 and the last constant 1. 

We have to show that this is compatible with face maps. 
We have a commutative diagram for $j=0, \dots, m$ (omitting for simplicity the $\Delta_n$-factor):
\begin{gather*} \xymatrix{ 
\Delta_{m-1} \times G(c) \ar@{^{(}->}[d]^{\delta_{j}} \ar[r] &   \Delta_1^{j} \times \Delta_1 \times \Delta_1^{m-j} \times G(c) \ar@{^{(}->}[d]^{\delta}  \ar[r] &  \cdots \ar[r] &  \\
\Delta_{m} \times G(c) \ar[r] & \Delta_1^{j} \times \Delta_1^2 \times \Delta_1^{m-j}    \times G(c) \ar[r] & \cdots \ar[r] & \\
} \\
 \footnotesize \xymatrix{ 
 \Delta_1^{j} \times  \Delta_1 \times G(c_{m-j-1})  \ar@{^{(}->}[d]^{\delta} \ar@{}[rrd]|{\mathcircled{1}} \ar[rr] &  & \Delta_1^{j}  \times G(c_{m-j+1}) \ar@{=}[d] \ar[r] & \cdots \ar[r] & G(c') \ar@{=}[d]  \\
  \Delta_1^{j}  \times \Delta_1^2 \times G(c_{m-j-1})  \ar[r] & \Delta_1^{j}  \times \Delta_1 \times G(c_{m-j})  \ar[r] & \Delta_1^{j} \times G(c_{m-j+1})  \ar[r] & \cdots \ar[r] & G(c') \\
} 
\end{gather*}
with the diagonal $\delta: \Delta_1 \to \Delta_1^2$, setting $c_{-1} = c$ and $c_{n+1} = c'$, in which the square $\mathcircled{1}$ is commutative by definition of $\underline{\mathrm{Def}}_Y$ (cf.\@ Definition~\ref{DEFNAIVE}). 
It is clear that the maps fulfill the defining compatibilies of the enriched end.
\end{proof}

\subsection{Composition of coherent transformations}\label{SECTCOMPCOH}

\begin{PAR}\label{PARCOMPCOH}
First we discuss the composition of coherent transformation --- for simplicity --- in the case in which the simplicial structure on the source is discrete. 
In this case two coherent transformation $f: X \Rightarrow Y$ and $g: Y \Rightarrow Z$ are determined by maps that associate to an $n$-simplex of the nerve of the following form
\[ \mu: i = i_0 \to \dots \to i_n = i' \]
elements
\[ f(\mu) \in \uHom(X(i), Y(i'))_{[n]}  \qquad   g(\mu) \in \uHom(Y(i), Z(i'))_{[n]}  \]
and it is not immediately possible to compose them. This is different after applying the AW components $\delta_l: [i] \to [n]$ and $\delta_r: [{n-i}] \to [{n}]$ which yields pairs:
\[ \delta_l \mu = (i=i_0 \to \cdots \to i_i \to i')  \qquad \delta_r \mu = (i \to i_i \to \cdots \to i_n = i')   \]
so, by the relations in the relative end, thus
\[ f(\mu)_i = f(\mu_i) \circ X(i_i \to i') \qquad g(\mu)_j = Y(i \to i_i) \circ g(\mu_j)  \]
where $\mu_i = i=i_0 \to \cdots \to  i_i = i_i$ and $\mu_j = i_i \to i_i \to \cdots \to i_n = i'$.
Now $f(\mu_i)$ and $g(\mu_j)$ are elements in $\uHom(X(i), Y(i_i))_{[i]}$, and $\uHom(Y(i_i), Z(i'))_{[n-i]}$, respectively.
 In the Abelian case, we can
compose with the Eilenberg-Zilber map and then compose
\begin{gather*} \Hom(X(i), Y(i_i))_i \otimes \Hom(Y(i_i), Z(i'))_j \\ \to 
 \Hom(X(i), Y(i_i))_n \otimes \Hom(Y(i_i), Z(i'))_n \to \Hom(X(i), Y(i'))_n. 
\end{gather*}
(Note that in case of the enrichment $F \Aw \Hom_{(\mathcal{C}^{\Delta^{\op}}, \tildeotimes)^{\vee}}^{\tildeotimes}(X, Y)$  the composition morphism even factors over $\Hom(X(i_0), Y(i_i))_i \times \Hom(Y(i_i), Z(i_n))_j$,  cf.\@ Lemma~\ref{LEMMAYONEDA}).
\end{PAR}

To discuss this a bit more conceptually, and without the assumption that simplicial structure on the source is discrete, first a Proposition:

\begin{PROP}\label{PROPCOMPCOHERENT}
The simplicial object $\widehat{\uHom}_{\nu}(c, d)$ (cf.\@ Definition~\ref{DEFHOMHAT}) has the following property: There is a morphism
\[ \rho:  \dec^* \widehat{\uHom}_{\tau}(c, e)  \to \int^d  \widehat{\uHom}_{\nu}(c, d') \boxtimes \widehat{\uHom}_{\mu}(d, e)   \]
(where the coend is the enriched coend) which is an isomorphism, if the structure is discrete.  Here the RHS is an enriched colimit over: 
\[ \xymatrix{ {}^{\uparrow} d  & p(d) \ar[r]^{\mu}   & p(e)     \\
{}^{\downarrow} d'  \ar[u]   & p(d')  \ar[u] & p(c) \ar[u]^{\tau}  \ar[l]^{\nu}  } \]

\end{PROP}
\begin{proof}
The LHS at $i, j$ is a union of
\begin{gather*} \cong \coprod_{c_0, \dots, c_i, d_0, \dots, d_j} \uHom_{\tau_0}(c, c_0)_n \times \uHom_{\tau_1}(c_1, c_2)_n \times \cdots \\
 \times \uHom_{\tau_i}(c_i, d_0)_n \times  \uHom_{\tau_{i+1}}(d_1, d_2)_n \times \cdots \times \uHom_{\tau_n}(d_n, e)_n      
\end{gather*} 
over all decompositions of $\tau$, which is
\begin{gather*} = \int^d  \coprod_{c_0, \dots, c_i, d_0, \dots, d_j} \uHom_{\nu_0}(c, c_0)_n \times \uHom_{\nu_1}(c_1, c_2)_n \times \cdots \\ \times \uHom_{\nu_i}(c_i, d)_n \times \uHom_{\mu_{0}}(d, d_0)_n \times \uHom_{\mu_{1}}(d_1, d_2)_n \times \cdots \times \uHom_{\mu_{j}}(d_n, e)_n 
\end{gather*}
by relative Yoneda, Proposition~\ref{PROPYONEDA}, 2. 
Applying the AW components $\delta_l, \delta_r$ (which is obviously an isomorphism  in the discrete case) gives:
\begin{gather*} \to \int^d \coprod_{c_0, \dots, c_i} \uHom_{\nu_0}(c, c_0)_i \times \uHom_{\nu_1}(c_1, c_2)_i  \times \cdots   \\
 \times \uHom_{\nu_i}(c_i, d)_i \times \coprod_{d_0, \dots, d_i}  \uHom_{\mu_{0}}(d, d_0)_j \times \uHom_{\mu_{1}}(d_1, d_2)_j \times \cdots \times 
\uHom_{\mu_j}(d_n, e)_j     \end{gather*}
Taking the union one arrives at
\[ \int^d  \widehat{\uHom}_{\nu}(c, d') \boxtimes \widehat{\uHom}_{\mu}(d, e) \tag*{\qedhere}  \]
\end{proof}

\begin{PAR}
Now we are in a position to formalize the composition discussed in \ref{PARCOMPCOH}: 
\begin{gather*} \uCoh_{I}(F, G) \times \uCoh_{I}(G, H) = \int_{c, d, e, f} \uHom_{\pi(\tau)}(\widehat{\uHom}_{\nu'}(c', d'), \uHom_{\nu}(F(c),G(d))) \times  \\
  \uHom_{\pi(\tau')}(\widehat{\uHom}_{\mu'}(e', f'), \uHom_{\mu}(G(e),H(f))) \\
\end{gather*}
where the end is a relative enriched end, an enriched limit over: 
\[ \xymatrix{ c \ar[d]   & d  & e \ar[d]  & f  & p(c) \ar[r]^{\nu} \ar[d]   & p(d)  & p(e) \ar[r]^{\mu} \ar[d]  & p(f)   \\
c'  & d' \ar[u] & e' & f'  \ar[u]  & p(c') \ar[r]^{\nu'} & p(d') \ar[u]^{\tau} & p(e') \ar[r]^{\mu'} & p(f')  \ar[u]^{\tau'}   } \]
This can be extended to an enriched limit over:
\[ \xymatrix{ c \ar[d]   & d  \ar[r] & e \ar[d]  & f   & p(c) \ar[r]^{\nu} \ar[d]   & p(d)   \ar[r]^{\rho} & p(e) \ar[r]^{\mu} \ar[d]  & p(f)   \\
c'  & d' \ar[u]  \ar[r] & e' & f'  \ar[u]  & p(c') \ar[r]^{\nu'} & p(d') \ar[u]^{\tau}  \ar[r]^{\rho'} & p(e') \ar[r]^{\mu'} & p(f')  \ar[u]^{\tau'}  } \]
And the map to 
\[ \xymatrix{ c \ar[d]   &  f   \\
c'  &  f'  \ar[u]   } \qquad 
\xymatrix{ p(c) \ar[r] \ar[d]  & p(f)   \\
p(c') \ar[r] & p(f')  \ar[u]   } \]
is a fibration.

We arrive at an end over 
\begin{gather*} 
 \to  \int_{c, d \to e, f} \uHom_{\pi(\tau)}(\widehat{\uHom}_{\nu'}(c', d') \times \widehat{\uHom}_{\mu'}(e', f'),\\
  \uHom_{\nu}(F(c),G(d)) \times \uHom_{\mu}(G(e),H(f)))  
\end{gather*}
where in the products $\times$ the factors are assembled using the new maps $\rho$ and $\rho'$. Using the composition in the enrichment of $\mathcal{D}$ this maps to:
\begin{gather*} 
\to \int_{c, f} \uHom_{\pi(\tau)}( \int^d \widehat{\uHom}_{\nu'}(c', d') \times \widehat{\uHom}_{\mu'}(d, f'),  \Hom_{\mu}(F(c),H(f))).
\end{gather*}
Composing with $\delta^* $ applied to the map
\[ \rho: \dec^* \widehat{\uHom}_{\nu'}(c, d') \to \int^d \widehat{\uHom}_{\nu'}(c, d') \boxtimes \widehat{\uHom}_{\mu'}(d, e) \]
of Proposition~\ref{PROPCOMPCOHERENT} we have a map to 
\begin{gather*} 
\to \int_{c, f} \uHom_{\pi(\tau)}( \delta^* \dec^* \widehat{\uHom}_{\nu'}(c', f'),  \uHom_{\mu}(F(c),H(f))).
\end{gather*}

Defining a cosimplicial simplicial  set by
\[ X_{n,m} :=  \int_{c, f} \Hom_{\pi(\tau)}( \widehat{\uHom}_{\nu'}(c', f')_n,   \uHom(F(c),H(f))_m )  \]
this can be writen as 
\[  \int_n \uHom(\delta^* \dec^* \Delta_n, X_{n,\bullet})  \]
(This is also
\[ = \int_n X_{\dec \delta n, m} \]
in degree 0. )

Hence to define a composition one must specify a morphism
\begin{equation} \label{eqdeltadec}  \int_n \uHom(\delta^* \dec^* \Delta_n, X_{n,\bullet})  \to \int_n \uHom(\Delta_n, X_{n,\bullet}).  \end{equation}

\begin{enumerate}
\item Case $\uHom = F \Aw \uHom^{\tildeotimes}_{\mathcal{C}^{\Delta^{\op}}, \tildeotimes}$ with $\mathcal{C}$ Abelian (i.e.\@ the usual simplicial enrichment of complexes):

Here $X_{n,\bullet}$ is of the form $F X_{n,\bullet}'$  where $X_{n,\bullet}'$ is a cosimplicial object in complexes of Abelian groups, and
\[ \int_n \uHom(\delta^* \dec^* \Delta_n, F X_{n,\bullet}')  = \int_n \uHom(\delta^* \dec^* \Z[ \Delta_n ], X_{n,\bullet}')  \]
and we may just compose with the ``EZ-counit'' (cf.\@ \ref{PARCONJ} and Definition~\ref{DEFAWEZ}): $\mlq c^{\op} \mrq: \delta^* \dec^* \Z[ \Delta_n ] \to \Z[ \Delta_n ]$.

This composition can be described differently.
In this case, we arrive at an end over 
\begin{gather*} 
 \to  \int_{c, d \to e, f} \uHom_{\pi(\tau)}(\Z[\widehat{\uHom}_{\nu'}(c', d')] \boxtimes \Z[\widehat{\uHom}_{\mu'}(e', f')], \\
 \uHom_{\nu}^{\tildeotimes}(F(c),G(d)) \boxtimes \uHom_{\mu}^{\tildeotimes}(G(e),H(f)))
\end{gather*}
and can apply $\dec_*$ to arrive at
\begin{gather*} 
 \to  \int_{c, d \to e, f} \uHom_{\pi(\tau)}(\dec_* \Z[\widehat{\uHom}_{\nu'}(c', d')] \boxtimes \Z[\widehat{\uHom}_{\mu'}(e', f')], \\ \uHom_{\nu}^{\tildeotimes}(F(c),G(d)) \tildeotimes \uHom_{\mu}^{\tildeotimes}(G(e),H(f)))
\end{gather*}
Now apply the composition in $\uHom^{\tildeotimes}_{\mathcal{C}^{\Delta^{\op}}, \tildeotimes}$ directly to arrive at
\begin{gather*} 
 \to  \int_{c, e} \uHom_{\pi(\tau)}(\int^d \dec_* \Z[\widehat{\uHom}_{\nu'}(c', d)] \boxtimes \Z[\widehat{\uHom}_{\mu'}(d', e')], \\
 \uHom_{\nu}^{\tildeotimes}(F(c),H(e)))
\end{gather*}
Then, as $\dec_* \cong \tot$ clearly commutes with colimits in the category of complexes of Abelian groups, using Proposition~\ref{PROPCOMPCOHERENT}, this maps to
\begin{gather*} 
 \to  \int_{c, e} \uHom_{\pi(\tau)}(\dec_* \dec^* \Z[\widehat{\uHom}_{\nu'}(c', e')], \uHom_{\nu}^{\tildeotimes}(F(c),H(e)))
\end{gather*}
and after compositing with the unit at
\begin{gather*} 
 \to  \int_{c, e} \uHom_{\pi(\tau)}(\Z[\widehat{\uHom}_{\nu'}(c', e')], \uHom_{\nu}^{\tildeotimes}(F(c),H(e)))
\end{gather*}
It is an exercise to see that this yields the same composition, and also that it is the abstract description of the construction in \ref{PARCOMPCOH}. 

\item General case. 
In  \cite[Theorem~4.4]{CP97} you may find a construction of a map (\ref{eqdeltadec}), if $X_{n,\bullet}$ is (point-wise in $n$) a weak Kan complex (i.e.\@ a quasi-category). 
Since the ``EZ-counit'' always exists (without passing to the homotopy category) {\em on symmetric objects}, i.e.\@ in objects in the image of $\iota^*$, one could try to proceed as follows:
Resolve the canonical cosimplicial object $\Delta_n \hookrightarrow \iota^* \Delta_{n,s}$ into something symmetric, e.g.\@ $\Delta_n \hookrightarrow \iota^* \mathfrak{C} \Delta_n$ (Lemma~\ref{LEMMASYM}).
If this can be done in such a way that there is a lift (functorial in the cosimplicial direction)
\[ \xymatrix{ \delta^* \dec^* \Delta_n \ar[r]  \ar@{^{(}->}[d] & X_{n,\bullet} \\
\delta^* \dec^* \iota^* \Delta_{n,s} \ar@{.>}[ru]
 } \]
 assuming, say,  that $X_{n,\bullet}$ is a (weak) Kan complex, then we can compose the lifted morphism with the  ``EZ-counit'' also here
 \[ \xymatrix{  \Delta_n \ar@{^{(}->}[r] & \iota^* \Delta_{n,s} \ar[r]_-c \ar@/^20pt/[rr]^-{\mlq c \mrq}  & \iota^* \delta_s^* \dec_s^*  \Delta_{n,s} \ar[r] & \delta^* \dec^* \iota^* \Delta_{n,s} }  \]
  to get the morphism 
 \[  \int_n \uHom(\delta^* \dec^* \Delta_n, X_{n,\bullet})  \to \int_n \uHom(\Delta_n, X_{n,\bullet}).  \]
 I ignore whether this is possible.  Notice that {\em point-wise (fixing $n$)} such a lift exists, if we take $\Delta_{n,s} := \mathfrak{C} \Delta_n$ (Lemma~\ref{LEMMASYM}) and if $X_{n,\bullet}$ is a Kan complex, for the morphism $\delta^* \dec^* \Delta_n \to \delta^* \dec^* \iota^* \Delta_{n,s}$ is clearly injective and a weak equivalence because $\delta^* \dec^*$ preserves weak equivalences by Lemma~\ref{LEMMADELTADEC}. It is therefore a trivial cofibration. 
\end{enumerate}
\end{PAR}

\subsection{Explicit formul\ae}

The following gives explicit formulas for $\dec_!$ and $\dec_*$ in the general case. 
If $\mathcal{C}$ is Abelian, they can be simplified, cf.\@ Proposition~\ref{PROPEXPLICITAB}.

\begin{PROP}\label{PROPEXPLICIT}
Let $\mathcal{C}$ be a ($\infty$-)category. 
\begin{enumerate}
\item Assume that $\mathcal{C}$ is finite cocomplete. Then a left  adjoint $\dec_!$ of $\dec^*$ exists and is given point-wise by:
\[ (\dec_! X)_n = \coprod_{\substack{[i] \ast [j] = [n] \\ [i],[j] \in \Delta^{\op}_{\emptyset}}} (i_! X)_{i,j} \]
where $(i_! X)_{i,j} = X_{i,j}$ for $[i] \not= \emptyset$ and $[j] \not= \emptyset$ and $(i_! X)_{i,\emptyset} = \colim X_{i,-}$ and the same switched.
\item Assume that $\mathcal{C}$ is finite complete. Then a right adjoint $\dec_*$ of $\dec^*$, which is often called the {\bf Artin-Mazur codiagonal} or, when $\mathcal{C}=\Set$, the {\bf total simplicial set}, exists, and is point-wise given by:
\begin{gather*} (\dec_* X)_{[n]} = \\ 
 \lim \left( \vcenter{ 
\xymatrix@C=1.5pc{
X_{[n],[0]} \ar[rd] & & \ar[ld] X_{[n-1],[1]}   \ar[rd] & \ \cdots \  &   X_{[0],[n]}  \ar[ld] \\
& X_{[n-1],[0]} &  &\  \cdots \  
}
} \right)  \end{gather*}
where the maps are $(\id, \delta_0): X_{[i],[n-i]} \to X_{[i],[n-i-1]}$ and $(\delta_{i+1}, \id): X_{[i+1],[n-i-1]} \to X_{[i],[n-i-1]}$.
\end{enumerate}
\end{PROP}
\begin{proof}
1.\@ follows from $\dec_{\emptyset}$ being a cofibration with discrete fibers (cf.\@ Lemma~\ref{LEMMACOF}) and the formula $\dec_! = \iota^* \dec_{\emptyset,!} (\iota, \iota)_!$ (Lemma~\ref{LEMMAEXACT4}).  
For 2.\@ we have to establish a final functor
\[     Z \to  [n]\times_{/\Delta^{\op}}  (\Delta^{\op})^2    \]
where $Z$ is the zig-zag shaped category depicted in the statement. This can be done by mapping
$(i,j), i+j=n$ to the obvious composition $[n] \cong [i] \ast' [j] \twoheadleftarrow [i] \ast [j]$ and
 $(i,j), i+j=n-1$ to the isomorphism $[n] \cong [i] \ast [j]$. Morphisms are mapped to the obvious face maps. To 
see that it is final, factor
\[ Z \to ([n] \times_{/\Delta^{\op}} (\Delta^{\op})^2 )' \to [n] \times_{/\Delta^{\op}} (\Delta^{\op})^2 \]
where $([n] \times_{/\Delta^{\op}} (\Delta^{\op})^2 )'$ is the full subcategory of those morphisms $\alpha: [n] \leftarrow [i] \ast [j]$ such that the composition with the faces $[i] \ast [j] \hookleftarrow [i]$ and
$[i] \ast [j] \hookleftarrow [j]$ is a face. Notice that this is a poset. 
This inclusion has a right adjoint given by factoring $[n] \leftarrow [i] \ast [j]$ into $[n] \hookleftarrow [i'] \ast [j'] \twoheadleftarrow [i] \ast [j]$, the  epi-mono factorization, and is thus $\infty$-final. 
The posets
\[ \alpha  \times_{ / ([n] \times_{/\Delta^{\op}} (\Delta^{\op})^2 )' } Z   \]
are all isomorphic to connected subposets of $Z$ which are obviously contractible. 
 \end{proof}

\section{Abelian foundations}


\subsection{Dold-Kan}\label{SECTDOLDKAN}

\begin{SATZ}[Dold-Kan]
For an Abelian category $\mathcal{C}$, there is an equivalence of categories
\[ \Ch_{\ge 0}(\mathcal{C}) \cong \mathcal{C}^{\Delta^{\op}} . \]
\end{SATZ}
One elegant way of formulating this starts with the observation that functors
\[\Delta^{\op} \to \mathcal{C} \]
are the same as additive functors
\[ \Z[\Delta^{\op}] \to \mathcal{C} \]
and, since in $\mathcal{C}$ all idempotents split, the same as additive functors
\[ \Z[\Delta^{\op}]^{\Kar} \to \mathcal{C} \]
where $\Z[\Delta^{\op}]^{\Kar}$ is the Karoubi closure, in which splittings for all idempotents are formally adjoined to $\Z[\Delta^{\op}]$.
Similarly, $\Ch_{\ge 0}(\mathcal{C})$ is the same as the category of additive functors
\[ \mathcal{CH}_{\ge 0} \to \mathcal{C} \]
where $\mathcal{CH}_{\ge 0}$ is the additive category consisting of finite coproducts (in fact biproducts) of objects $D_i$, $i \ge 0$ with 
\begin{eqnarray} \label{eqd} \Hom(D_i, D_j) = \begin{cases} \Z & \text{if } i \in \{ j, j+1\} \\ 0 & \text{else} \end{cases}.  \end{eqnarray}
One can therefore state the Dold-Kan theorem as
\begin{SATZ}[Dold-Kan (variant)]There is an equivalence of categories: 
\[ \Z[\Delta^{\op}]^{\Kar} \cong \mathcal{CH}_{\ge 0}.   \]
\end{SATZ}

Instead of trying to prove this directly, it is, however, easier to realize the dual categories $\Delta$ and $\mathcal{CH}^{\op}_{\ge 0}$ inside
$\Ch_{\ge 0}(\Ab^{\fg})$. This also gives explicit functors. 
\begin{PAR}We first define
\[ \Gamma: \mathcal{C}^{\Delta^{\op}} \to \Ch_{\ge 0}(\mathcal{C})  \]
called functor of {\bf normalized chains}, 
 setting: $\Gamma(C)_n := C^{\nd}_n$ and 
$\dd:= \sum_{i=0}^n (-1)^i \delta_i$. It is an easy verification that $\dd^2 = 0$ and that $\dd$ descends to a map $C^{\nd}_n \to C^{\nd}_{n-1}$. 
The functoriality in $C$ is clear. 
\end{PAR}

This allows us to define the following cosimplicial object in $\Ch_{\ge 0}(\mathcal{\Ab^{\fg}})$, applying $\Gamma$ w.r.t.\@ $\mathcal{C}=\Ab^{\fg}$:
\begin{DEF}\label{DEFDELTACIRC}
\[ \Delta^{\circ}_n := \Gamma(\Z[\Delta_n]). \]
\end{DEF}

The following follows directly from the definition: 
\begin{LEMMA}\label{LEMMADELTACIRC}
We have explicitly
\[ (\Delta^{\circ}_n)_m = \Z[\Hom_{\Delta}^{\inj}([m], [n])] \]
whose basis can be identified with subsets $S \subset \{0, \dots, n\}$ of Cardinality $m+1$.
The differential is given on basis elements by
\[ \dd: [S] \mapsto \sum_{i} (-1)^i [S \setminus \{x_i\}] \]
where $S= \{x_0, \dots, x_m\}$ with $x_0< \dots< x_m$. 
\end{LEMMA}
This not only defines a cosimplicial object but even 
\begin{LEMMA}\label{LEMMAFINSET}
The association (\ref{DEFDELTACIRC}), a priori a functor $\Delta \to \Ch_{\ge 0}(\Ab^{\fg})$, extends to a functor
\[ \Delta^{\circ}: \FinSet \to \Ch_{\ge 0}(\Ab^{\fg}). \]
\end{LEMMA}
\begin{proof}
An arbitrary map $\alpha: \{0, \dots, n\} \to \{0, \dots, n'\}$ induces
\[ (\Delta^{\circ}_n)_m \to (\Delta^{\circ}_{n'})_m\]
mapping (a basis element identified as in Lemma~\ref{LEMMADELTACIRC})
\[ [S] \mapsto \begin{cases} [\alpha(S)] &  \text{if $\#\alpha(S) = \#S$ and the induced permutation is even,} \\ -[\alpha(S)] & \text{if $\#\alpha(S) = \#S$ and the induced permutation is odd,} \\ 0 & \text{otherwise.} \end{cases} \] 
(The ``induced permutation'' is the permutation of $\alpha(S)$ bringing the elements $\alpha(x_0), \dots, \alpha(x_n)$ into the correct order). 
One has to check that this respects the differential. This holds by construction if $\alpha$ is order preserving, i.e.\@ a morphism of $\Delta$. Furthermore,  for the neighboring transposition $(i\ i+1)$ the formula
\begin{equation}\label{eqfinset} (i\ i+1) = -\id + d_is_i + d_{i+1}s_i \end{equation}
holds true --- which is easily checked on a basis ---
and morphisms in $\FinSet$ are obviously generated by morphisms in $\Delta$ and neighboring transpositions. Therefore this respects the differential in any case. 
\end{proof}
\begin{PAR}
Similarly, also $\mathcal{CH}_{\ge 0}^{\op}$ can be realized inside $\Ch_{\ge 0}(\Ab^{\fg})$, defining
\[ D_i  := \left(  \cdots \to 0 \to \Z \to \Z \to 0 \to \cdots \to 0 \right)  \]
where the $\Z$ are in degree $i$ and (for $i>0$) $i-1$. The equality (\ref{eqd}) is obvious. Note that 
\begin{equation}  \label{eqtrivial} \Hom(D_i, C) \cong C_i  \end{equation}
 for all $i\ge 0$ and $C \in \Ch_{\ge 0}(\Ab)$.
\end{PAR}

To prove Dold-Kan using $\Delta^{\circ}$, the following is crucial: 
\begin{KEYLEMMA}\label{LEMMAKEY} We have\footnote{$\Delta^{\circ}$ is a cosimplicial object. 
The non-degenerate part, and also references to Lemma~\ref{LEMMAEZ2}, are meant w.r.t.\@  the {\em simplicial} object $(\Delta^{\circ})^{\op}: \Delta^{\op} \to \Ch_{\ge 0}(\Ab^{\fg})^{\op}.$}
\[ \Delta^{\circ,\nd}_n \cong D_n. \]
\end{KEYLEMMA}
\begin{proof}
It follows from formula (\ref{eqfinset}) that, on the non-codegenerate part $\Delta^{\circ,\nd}_n$, i.e.\@ the joint kernel of all the degeneracies, $S_n$ acts by the sign character. 
$(\Delta^{\circ}_n)_n$ is one dimensional, generated by $[\{0, \dots, n\}]$ which does transform according to the sign character. 
$(\Delta^{\circ}_n)_{n-1}$ is generated by $[\{0, \dots, \widehat{i}, \dots, n\}]$, $i=0, \dots, n$. Applying the transposition $(i\ j)$ we see that (for an element transforming according to the sign character) the coefficient 
of $[\{0, \dots, \widehat{i}, \dots, n\}]$ must be $(-1)^{i-j}$ times the coefficient of $[\{0, \dots, \widehat{j}, \dots, n\}]$. Hence $(\Delta^{\circ, \nd}_n)_{n-1}$ is one dimensional as well, generated by
$\dd [\{0, \dots, n\}]$. The basis element $[S]$ for any $S \subset \{0, \dots, n\}$ with at least two elements $i, j$ missing is fixed by the transposition $(i\ j)$. Hence
$(\Delta^{\circ, \nd}_n)_{m} = 0$ for $m < n-1$.
\end{proof}

\begin{KOR}\label{KORDOLDKANFF}
The induced functor
\[ \Z[\Delta^{\circ}]: \Z[\Delta] \to \Ch_{\ge 0}(\Ab) \]
is fully-faithful. 
\end{KOR}
\begin{proof}
Fix $n$ and consider the morphism of simplicial Abelian groups
\[  \Z[\Hom_{\Delta}([m], [n])] \to \Hom(\Delta^{\circ}_m, \Delta^{\circ}_n). \]
By Lemma~\ref{LEMMAEZ2} it suffices to show that it is an isomorphism for the non-degenerate quotient, i.e.\@
\[    \Z[\Hom^{\inj}_{\Delta}([m], [n])] \cong \Hom(\Delta^{\circ, \nd}_m, \Delta^{\circ}_n). \]
Since $\Delta^{\circ, \nd}_m \cong D_m$ by  Lemma~\ref{LEMMAKEY}, this follows directly from the Definition of $\Delta^{\circ}_n$ (cf.\@ Lemma~\ref{LEMMADELTACIRC}).
\end{proof}

 Lemma~\ref{LEMMAKEY} together with Lemma~\ref{LEMMAEZ2} (applied to the {\em simplicial} object $(\Delta^{\circ})^{\op}: \Delta^{\op} \to \Ch_{\ge 0}(\Ab^{\fg})^{\op}$)  gives an isomorphism
\[ \gr \Delta^{\circ}_m \cong \bigoplus_{\Delta_m \twoheadrightarrow \Delta_n} D_n \]
and since the $D_n$ are obviously projective by (\ref{eqtrivial}), we have {\em non-canonically}
\begin{eqnarray}\label{eqdoldkandecomp} \Delta^{\circ}_m \cong \bigoplus_{\Delta_m \twoheadrightarrow \Delta_n} D_n. 
\end{eqnarray}
This gives a proof of Dold-Kan in the first variant because this decomposition exists thus in $\Z[\Delta]^{\Kar}$ by Corollary~\ref{KORDOLDKANFF}.

\begin{PAR}\label{PARSKELETAL1}
As a complex $C= \Delta_n^{\circ}$ has also the canonical filtration: 
\[ F^m C := \left( \cdots\to C_{m+2} \to C_{m+1} \to \dd(C_{m+1}) \to 0 \to \cdots \to 0. \right) \]
Isomorphism~(\ref{eqdoldkandecomp}) shows a posteriori that on $\Delta^\circ_n$, the filtration by codegeneracy degree and the canonical filtration agree. 
\end{PAR}

We now proceed to give a second, more explicit, proof of Dold-Kan in the following variant
\begin{SATZ}[Dold-Kan (2nd variant)] \label{SATZDOLDKANII}
The functors
\[  \xymatrix{ \mathcal{C}^{\Delta^{\op}} \ar@<3pt>[r]^-{R} & \ar@<3pt>[l]^-{N} \Ch_{\ge 0}(\mathcal{C}) } \]
given by
\[ N(C): \Delta_n \mapsto \mathcal{HOM}_r(\Delta_n^{\circ}, C) \]
and
\[ R(A):= \int^n \Delta_n^{\circ} \otimes A_{[n]}\]
constitute an adjoint equivalence of categories. 
\end{SATZ} 
\begin{PAR}\label{PAR3ADJUNCTION}
Here the following is used. We have an obvious functor in two variables
\[ \mathcal{HOM}_l:   \mathcal{C}^{\op} \times  \Ch_{\ge 0}(\mathcal{C})   \to \Ch_{\ge 0}(\Ab) \]
(induced by the canonical enrichment $\mathcal{C}^{\op}  \times \mathcal{C} \to \Ab$) which has a (partial) left adjoint
\[ \otimes:   \Ch_{\ge 0}(\Ab^{\fg}) \times \mathcal{C}  \to  \Ch_{\ge 0}(\mathcal{C})   \]
(induced by the tensoring $\mathcal{C} \times \Ab^{\fg} \to \mathcal{C}$) which has another right adjoint
\[ \mathcal{HOM}_r:   \Ch_{\ge 0}(\Ab^{\fg})^{\op} \times \Ch_{\ge 0}(\mathcal{C})  \to \mathcal{C}. \]
(if $\mathcal{C}=\Ab^{\fg}$, then this is just the usual group of morphisms of complexes). 
\end{PAR}

\begin{proof}
The fact that $R$ (realization) is left adjoint to $N$ (nerve) holds in much greater generality and nothing has to be assumed about the cosimplicial object $\Delta^{\circ}: \Delta \to \Ch_{\ge 0 }(\Ab)$.
The equivalence needs more assumptions. In the more general setting, usually $\Delta^{\circ}$ is called {\em dense}, if $N$ is fully-faithful.

Here the statement follows immediately from the following three facts:
\begin{enumerate}
\item[(i)] $N$ is conservative.
\item[(ii)] $\mathcal{HOM}_r(\Delta^{\circ}_n, -)$ commutes with $\otimes$\footnote{in the sense that $\mathcal{HOM}_r(\Delta^{\circ}_n, A \otimes C) \cong \mathcal{HOM}(\Delta^{\circ}_n, A) \otimes C$ via the canonical morphism, for $A \in \Ch_{\ge 0}(\mathcal{AB}) $ and  $C \in \Ch_{\ge 0}(\mathcal{C}) $} and colimits. 
\item[(iii)] $\Z[\Delta^{\circ}]: \Z[\Delta] \to \Ch_{\ge 0 }(\Ab^{\fg})$ is fully-faithful. 
\end{enumerate}
Proof of the facts: (i) means that the functors $\mathcal{HOM}_r(\Delta_n^{\circ}, -)$ are jointly conservative. This is
certainly the case for the functors  $\mathcal{HOM}_r(D_n, -)$. However, by (\ref{eqdoldkandecomp}), there are split epimorphisms $\Delta_n^{\circ} \to D_n$.
(ii) follows also immediately from (\ref{eqdoldkandecomp}) because obviously $\mathcal{HOM}_r(D_n, -)$ commutes with colimits and $\otimes$.
(iii) is Corollary~\ref{KORDOLDKANFF}. 

To prove the equivalence, it suffices to see that (ii) and (iii) imply that $R$ is fully-faithful: We have (functorially in $n$):
\begin{align*} 
& \mathcal{HOM}_r(\Delta_n^{\circ}, \int^m \Delta_m^{\circ} \otimes C_{[m]})     \\
 \cong &\int^m \mathcal{HOM}(\Delta_n^{\circ}, \Delta_m^{\circ}) \otimes C_{[m]}  & \text{(by ii)} \\
 \cong &\int^m \Z[\Hom_{\Delta}([n], [m])] \otimes C_{[m]} & \text{(by iii)} \\
 \cong &\int^m \Hom_{\Delta}([n], [m]) \times C_{[m]} \\
 \cong & \, C_{[n]} & \text{(coend Yoneda)} &  \qedhere
 \end{align*}
\end{proof}

\begin{BEM}
The functor $R$ (realization) is nothing else then the functor $\Gamma$ introduced earlier. In fact we have proven that
\[ \mathcal{HOM}(\Delta_n^{\circ}, R(C)) = C_{[n]} \]
functorially in $[n]$ and $D_n = \ker(\Delta_n^{\circ} \to \Delta_k^{\circ})$ joint kernel of all
degeneracies $\Delta_n \twoheadrightarrow \Delta_k$. 
Therefore  
\[ R(C)_n = \Hom(D_n, R(C)) = C_{[n]}^{\nd} \]
and the differential, which comes from the morphism $D_n \to D_{n+1}$, is given by the alternating sum of the $\delta_i$. 
\end{BEM}

\subsection{The coherent Eilenberg-Zilber Theorem for Abelian categories}\label{SECTEZAB}

We have seen in Lemma~\ref{LEMMACPROFUNCTOR} that the symmetry operator $\mathfrak{C}$ (cf.\@ Lemma~\ref{LEMMASYM}) can, in the Abelian case, be given by an Abelian pro-functor (in two ways). Likewise, the natural transformations $\Awfrak$ and $\Ezfrak$  can be given by morphisms of Abelian pro-functors
where, by abuse of notation, $\dec$ denotes $\Z[\dec]: \Z[\Delta^{\op}] \otimes \Z[\Delta^{\op}] \to \Z[\Delta^{\op}]$ and $\delta$ denotes $\Z[\delta]: \Z[\Delta^{\op}] \to \Z[\Delta^{\op}] \otimes \Z[\Delta^{\op}]$:
\begin{eqnarray} \label{eqawpro}
\Aw: && \xymatrix{  \dec  \ar[r]^-{u^{\op}} & \dec \,{}^t\! \dec\, \,{}^t\!\delta    \ar[r]^-{} &  {}^t\!\delta }  \\
 \label{eqezpro}
 \Ez:&& \xymatrix{  \,{}^t\!\delta  \ar[r]  &  \ar[d]^{\sim}  {}^t\!\delta \,{}^t\!\dec \, \dec     &\ar[d]^{\sim}  \dec  \\
& \,{}^t\!\delta  \,{}^t\!\dec \,  {}^t\!\iota  \,   C^{\op}\,   \dec    =  \,{}^t\!\iota     \,{}^t\!\delta_s  \,{}^t\!\dec_s \,  C^{\op}\,     \dec    \ar[r]^-{c^{\op}} &  \,{}^t\!\iota \, C^{\op}\, \dec   } \end{eqnarray}
It is immediate that for an Abelian category $\mathcal{C}$, $R_{\mathcal{C}}(\Aw)$ is the map $\Awfrak$, and $R_{\mathcal{C}}(\Ez)$ is the map $\Ezfrak$, both defined in Definition~\ref{DEFAWEZ}. 
We will prove later (cf.\@ Proposition~\ref{PROPAWEZAB}) that the latter are nothing but the classical Alexander-Whitney and Eilenberg-Zilber maps.

We have $\Ez\Aw = \id$ already on the level of profunctors and:
\begin{DEFLEMMA}\label{DEFSHIH}The following map $\Xi_1 \in \uHom(\delta, \delta)_{[1]}$ (\ref{PARHOMOTOPYPOINT})
\begin{eqnarray} \label{eqshih}
\xymatrix{  \delta \dec \ar[r]^-{\sim} & \dec_{13,24} \delta_{12,34} \ar[rrr]^{\dec_{13,24} \circ (\id, {}^t\!(\Aw \Ez))} & &  &   \dec_{13,24} \delta_{12,34} \ar[r]^-{\sim} &  \delta \dec   }  
\end{eqnarray}
is called {the \bf Shih-operator} and constitutes a homotopy  between ${}^t\!(\Aw \Ez)$ and $\id$.
\end{DEFLEMMA} 
The fact that this yields a homotopy as claimed follows from (the proof of) Proposition~\ref{PROPCREATIONHOMOTOPY}. It will be reproven in the following Theorem which yields a coherent version. Later we will derive an explicit formula (Proposition~\ref{PROPSHIHEXPL})

\begin{SATZ}\label{SATZCOHEZ}
The (appropriate part of the) diagram of operads (\ref{eqfundamentalop}) extends to a diagram of $\Ab$-enriched operads and oplax pro-functors (cf.\@ \ref{PARCONVOLUTION})
\[ \xymatrix{ 
& \Z[(\FinSet, \coprod)] & \\
\Z[(\Delta, \coprod)] \ar[ru]^{\iota_{\delta}} &  & \Z[(\Delta, \ast)] \ar@<3pt>[lu]^{\iota_{\dec}}  \ar@{<-}@<-3pt>[lu]_{C}  \ar[ll]^{\mathrm{forget}}
}\]
with $C\, \iota_{\dec} \cong \id$ and setting
\[ \AW := \mathrm{forget}^{\op} \qquad \EZ := (C\, \iota_{\delta})^{\op} \]
we have:
\begin{enumerate}
\item Both morphisms are the identity on objects.
\item Over $\delta_{1}: [2] \to [1]$, we reobtain the maps $\Aw$ (\ref{eqawpro}) and $\Ez$ (\ref{eqezpro}) as the oplaxness constraint.
\item We have $\EZ \AW = \id$ and there is a homotopy (see \ref{PARFUNCT} for the Definition of $\uDef$ on the level of pro-functors):
\[ \Xi \in \uDef(\Z[(\Delta^{\op}, \coprod)], \Z[(\Delta^{\op}, \coprod)])_{\id,[1]} \]
 connecting $(\EZ \AW)^{\vee}$ and $\id$, which yields over $\delta_1$ the Shih-operator defined in Definition~\ref{DEFSHIH}. There is also a homotopy 
 \[ \Xi^{\vee} \in \uDef(\Z[(\Delta, \coprod)^{\op}], \Z[(\Delta, \coprod)^{\op}])_{\id,[1]} \] 
 connecting $\EZ \AW$ and $\id$. (In both cases the simplicial structure comes from Proposition~\ref{PROPSIMPLICIALOPERAD}.)
\end{enumerate}
\end{SATZ}
\begin{proof}
All operads, being exponential, are given by functors
\[ \OOO = \Delta^{\op} \to (\AbCat^{\PF}, \times) \]
mapping $\delta_1$ to $\delta$, ${}^t\! \dec$ and $\delta_s = {}^t\! \dec_s$ (on $\FinSet$).
The identification for the pro-functors associated with $\delta_1$ in the bottom row are more precisely:
\begin{align*}
 \iota(\delta): [n],[m];[k] \mapsto &\Hom_{(\Delta, \coprod)}([n],[m];[k])  \\
\cong & \Hom([n],[k]) \times \Hom([m],[k]) \cong \Hom(([n],[m]);\delta([k])) \\
  {}^t\!\iota(\dec):  [n],[m];[k] \mapsto &\Hom_{(\Delta, \ast)}([n],[m];[k]) \cong \Hom(\underbrace{[n] \ast [m]}_{=:\dec([n],[m])};[k])
\end{align*}
and the morphism ``forget'' is induced by the inclusions $[n] \hookrightarrow [n] \ast [m]$ and $[m] \hookrightarrow [n] \ast [m]$ (morphisms in $\Delta$). 
The morphism (\ref{eqawpro}) is the same map.

The existence of $C$ and the isomorphism $C \iota_{\dec} \cong \id$ will be shown in Lemma~\ref{LEMMACCOH} below.
In the following diagram the left hand side composition is the laxness constraint from $C\, \iota_{\delta}$ and the right hand side composition is the previous map (\ref{eqezpro}):
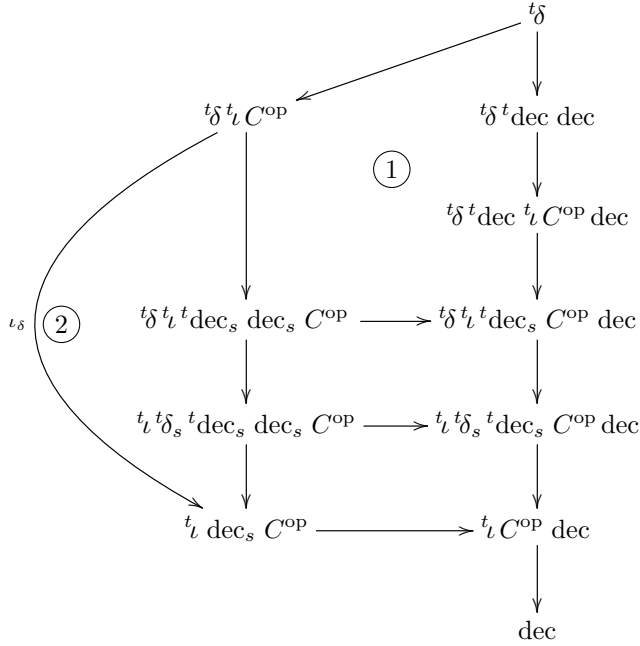
\[ \xymatrix{  
&  {}^t\!\delta  \ar[d]  \ar[ld]   \ar@{}[dddl]|{\mathcircled{1}}  \\
 {}^t\!\delta\,   {}^t\!\iota\,  C^{\op}  \ar@/_80pt/[dddd]_{\iota_{\delta}}^{\mathcircled{2}}  \ar[dd]  &  {}^t\!\delta\, {}^t\!\dec\,  \dec  \ar[d] \\
&   {}^t\!\delta\,   {}^t\!\dec\, {}^t\!\iota\,  C^{\op}  \dec  \ar[d] \\ 
{}^t\!\delta\,  {}^t\!\iota\,  {}^t\!\dec_s\,  \dec_s\, C^{\op}\, \ar[d]  \ar[r] &   {}^t\!\delta\,   {}^t\!\iota\,  {}^t\!\dec_s\,  C^{\op}\, \dec  \ar[d] \\
 {}^t\!\iota\,  {}^t\!\delta_s\, {}^t\!\dec_s\,   \dec_s\,  C^{\op}   \ar[d] \ar[r] & {}^t\!\iota\, {}^t\!\delta_s\,   {}^t\!\dec_s\, C^{\op}   \dec   \ar[d] \\
 {}^{t}\!\iota\, \dec_s\,  C^{\op} \ar[r]  &  {}^t\!\iota\, C^{\op}\,  \dec\,   \ar[d] \\
& \dec } \]
Here diagram $\mathcircled{1}$ commutes because of relation $C \,\iota_{\dec} = \id$ and diagram $\mathcircled{2}$ by definition of the constraint in $\iota_{\delta}$ (notice that ${}^t\!\delta_s = \dec_s$).
The assertion $\EZ \AW = \id$ is clear because we have $\iota_{\delta}\, \mathrm{forget} = \iota_{\dec}$. The operator $\Xi$ is constructed applying Proposition~\ref{PROPCREATIONHOMOTOPYFUNCT}.

First, we have seen in Proposition~\ref{PROPSIMPLICIALOPERAD} that the ``simplicially enriched structure'' on $\Delta^{\op}$ extends to the exponential operad $(\Delta^{\op}, \coprod)$ by means of a functor
\[ \Delta_{\emptyset} \to \Hom^{1-\oplax, 1-\inert-\pseudo}(\OOO^{\op}, (\Cat^{\PF}_{}, \times)^{\vee})  \]
such that  $p_{\emptyset}$ is mapped to ${}^t\! \pi$ (point-wise) and all oplaxness constraints are actually isomorphisms, to the extent that the condition that 
$(\EZ \AW)^{\vee}$ (i.e.\@ the mate of $\EZ\AW$ passing to the right adjoints in the $\delta$-direction) maps to the canonical element is vacuous. By Proposition~\ref{PROPCREATIONHOMOTOPYFUNCT} there exists therefore 
\[ \Xi \in \uDef(\Z[(\Delta^{\op}, \coprod)], \Z[(\Delta^{\op}, \coprod)])_{\id,[1]} \]
connecting $\id$ and $(\EZ \AW)^{\vee}$. By construction it yields the operator defined in Definition~\ref{DEFSHIH} as oplaxness-constraint for $\delta_1: [2] \hookleftarrow [1]$. 
Similarly, the ``simplicially enriched structure'' on $\Delta^{\op}$ extends to the exponential cooperad $(\Delta, \coprod)^{\op}$ by means of a functor
\[ \Delta_{\emptyset} \to \Hom^{1-\oplax, 1-\inert-\pseudo}(\OOO, (\Cat^{\PF}_{}, \times))  \]
such that  $p_{\emptyset}$ is mapped to ${}^t\! \pi$ (point-wise) and the oplaxness constraints {\em for $p_{\emptyset}$} are isomorphisms.
There exists therefore  
\[ \Xi^{\vee} \in \uDef(\Z[(\Delta, \coprod)^{\op}], \Z[(\Delta, \coprod)^{\op}])_{\id,[1]} \] 
connecting $\id$ and $\EZ \AW$. (Actually, this is just the mate of $\Xi$ passing to the right adjoints in the $\delta$ direction).
\end{proof}

Let $(\mathcal{C}, \otimes)$ be a tensor Abelian category. Applying Day convolution (cf.\@ also \ref{EXDAY2}) to the cooperad $(\mathcal{C}, \otimes)^{\vee}$, we get:
\begin{KOR}[Coherent Eilenberg-Zilber for Abelian categories] \label{KOREZAB}
The morphisms of cooperads\footnote{Note that $\AW$ and $\EZ$, being the identity on objects, have trivial lax mates ${}^t\! \AW$ and ${}^t\! \EZ$ (with {\em the same} constraints) and the equality of $L$ and $R$ has been discussed in (\ref{PARCONVOLUTION}). It does not involve forming ${}^t\! \Aw$ for the transformation $\Aw: {}^t \delta \Rightarrow \dec$ where the operation would not make sense! } $\AWfrak:=R_{(\mathcal{C},\otimes)^{\vee}}(\AW) = L_{(\mathcal{C},\otimes)^{\vee}}({}^t\!\AW)$ and $\EZfrak:=R_{(\mathcal{C},\otimes)^{\vee}}(\EZ) = L_{(\mathcal{C},\otimes)^{\vee}}({}^t\!\EZ)$ 
\[ \xymatrix{   (\mathcal{C}^{\Delta^{\op}}, \otimes)^{\vee} \ar@/^10pt/[rr]^{\AWfrak} & & \ar@/^10pt/[ll]^{\EZfrak} (\mathcal{C}^{\Delta^{\op}}, \tildeotimes)^{\vee} }\]
 are the identity on objects, and satisfy $\AWfrak\,  \EZfrak=\id$ and 
\[ L_{(\mathcal{C},\otimes)^{\vee}}(\Xi^{\vee}) \in \uDef((\mathcal{C}^{\Delta^{\op}}, \otimes)^{\vee}, (\mathcal{C}^{\Delta^{\op}}, \otimes)^{\vee})_{\id, [1]} \]
constitutes a homotopy\footnote{The simplicially enriched structure to define $\uDef$ is the {\em discrete one on the source} and the {\em point-wise tensoring on the target} (cf.\@ also Proposition~\ref{PROPSIMPLICIALOPERAD}).
While this implies thus a certain coherence, it does not say that $\EZfrak\, \AWfrak$, let along $\Xi$, is a simplicially enriched natural transformation w.r.t.\@ the latter enrichment, which it is certainly not. }
 between $\EZfrak\, \AWfrak$ and $\id$. 
Likewise
\[ L_{(\mathcal{C},\otimes)}(\Xi) \in \uDef((\mathcal{C}^{\Delta^{\op}}, \otimes), (\mathcal{C}^{\Delta^{\op}}, \otimes))_{\id, [1]} \]
constitutes the same element by means of interpreting the deformations as deformations of $\otimes$. 
\end{KOR}

In the proof of Theorem~\ref{SATZCOHEZ} the following was used: 

\begin{LEMMA}\label{LEMMACCOH}
The point-wise application of $C$ yields an  oplax pro-functor (cf.\@ \ref{PARCONVOLUTION}) of $\Ab$-enriched operads
\[  C: \Z[(\FinSet, \coprod)]  \to \Z[(\Delta, \ast)]  \]
and the point-wise application of ${}^t\!C$ a {\em coCartesian} lax pro-functor,
satisfying:
\begin{equation}\label{eqciota}   C \, \iota_{\dec}  \cong \id.  \end{equation}
\end{LEMMA}
\begin{proof}
It suffices to construct natural isomorphisms
\[   \dec\, C \cong C\, \dec_s  \]
(the constraint for $C$ is then the mate $C\, {}^t\! \dec_s \rightarrow {}^t\! \dec\, C$, whose mate in turn is the adjoint 
$ {}^t\!\dec_s\, {}^t\!C\ \cong {}^t\!C\, {}^t\!\dec $)
and natural isomorphisms
\[      \pi \, C\cong\pi.  \]
Equivalently, we may construct:
\begin{eqnarray*} L_{\Ab}(({}^t\!\dec)^{\op}) L_{\Ab}(({}^t\!C)^{\op}) &\cong& L_{\Ab}(({}^t\!C)^{\op})   L_{\Ab}(({}^t\!\dec)^{\op}) \\
 \rho: \dec^*_s \mathfrak{C} &\cong& \mathfrak{C} \dec^*  
\end{eqnarray*}
and
 where $\dec$ now denotes $\Delta^{\op} \times \Delta^{\op} \to \Delta^{\op}$ and $\dec_s$ denotes $\FinSet^{\op} \times \FinSet^{\op} \to \FinSet^{\op}$ (we omit the $\op$).

Given a complex $X$, we get the object $\mathfrak{C} X$ in $\Ab^{\FinSet^{\mathrm{op}}}$ as 
\[ (\mathfrak{C} X)_{[n]} = \Hom(\Delta_{n}^\circ, X) \]
with the natural action of $\FinSet$ on $\Delta_{n}^\circ$. We have
\[ (\mathfrak{C} \dec^* X)_{[i],[j]} = \Hom(\Delta_{i}^\circ \boxtimes \Delta_{j}^\circ, \dec^* X) = \Hom(\dec_! \Delta_{i}^\circ \boxtimes \Delta_{j}^\circ, X) \]
with the natural action of $\FinSet \times \FinSet$ on $\Delta_{i}^\circ \boxtimes \Delta_{j}^\circ$. On the other hand, we have
\[ (\dec^*_s \mathfrak{C} X)_{[i],[j]} = \Hom(\Delta_{i+j+1}^\circ, X) \]
with the action of $\FinSet \times \FinSet$ on $\Delta_{i+j+1}^\circ$ via $\dec_s$.
There is a canonical isomorphism 
\[ \dec_! \Delta_{i}^\circ \boxtimes \Delta_{j}^\circ \cong \Delta_{i+j+1}^\circ \]
and the assertion boils down to its compatibility with the actions of $\FinSet \times \FinSet$.

Furthermore, we have
\[  \pi \cong \pi \, \iota \, C  \cong   \pi \, C  \]
This yields the morphism. 
It has a mate which can also be described as
\[  \pi \cong   {}^{t}\! C\, {}^t\!\iota \,\pi  \cong     {}^{t}\! C\, \pi \]
It is also an isomorphism due to the finality of $\iota^{\op}$.
We leave to the reader to check the compatibilities. 

(\ref{eqciota}) holds point-wise, and the compatibility of oplaxness constraints translates to the commutativity of
\[ \xymatrix{ 
\dec^* \ar[d]^{\sim} \ar[r]^{\sim} &  \iota^* \mathfrak{C} \dec^* \ar[d]^{\rho} \\
 \dec^* \iota^* \mathfrak{C} \ar@{=}[r] &  \iota^* \dec^*_s \mathfrak{C}  \\
} \]
which holds by construction. 
\end{proof}

\subsection{The higher Shih operators}

Theorem~\ref{SATZCOHEZ} states that the value of
\[ L_{(\mathcal{C}, \otimes)^{\vee}}(\Xi^{\vee}) \in  \uDef((\mathcal{C}^{\Delta^{\op}}, \otimes)^{\vee}, (\mathcal{C}^{\Delta^{\op}}, \otimes)^{\vee})_{\id,[1]} \]
at $\delta_1: [2] \to [1]$ is the Shih operator defined in Definition~\ref{DEFSHIH}. 
We will also discuss briefly the ``higher information'' contained in $\Xi^{\vee}$. The most 
useful form is to map the transformation of $L_{(\mathcal{C}, \otimes)^{\vee}}(\Xi^{\vee})$ as follows:

\begin{PAR}\label{PARHIGHERSHIH}
Recall (Corollary~\ref{KOREZAB}) that the two morphisms of cooperads
\[ \xymatrix{ (\mathcal{C}^{\Delta^{\op}}, \otimes)^{\vee} \ar@<3pt>[r]^{\id}_{}  \ar@<-3pt>[r]_{\EZfrak \circ \AWfrak} & (\mathcal{C}^{\Delta^{\op}}, \otimes)^{\vee}  } \]
are connected by a naive homotopy
\[ L_{(\mathcal{C},\otimes)^{\vee}}(\Xi^{\vee}) \in \uDef((\mathcal{C}^{\Delta^{\op}}, \otimes)^{\vee}, (\mathcal{C}^{\Delta^{\op}}, \otimes)^{\vee})_{\id, [1]} \]
(w.r.t.\@ the trivial enrichment on the left and $F \uHom^{\otimes}$ on the right).
This is mapped via (\ref{PROPCOHERENT}) to a coherent transformation
\[ H := \exp(L_{(\mathcal{C},\otimes)^{\vee}}(\Xi^{\vee})) \in \uCoh(\id, \EZfrak \circ \AWfrak)_{[0]} \]
(with the same enrichments). To this coherent transformation we may apply $\Ezfrak^*$ (Proposition~\ref{PROPSIMPLICIALLYENRICHED}, 1.\@)  to change the enrichment to $F \uHom^{\tildeotimes, \otimes}$ which is
a weak enrichment. The transport works because  the trivial enrichment is chosen on the source.

We ignore units, that is, consider the cooperads as cooperads over $\OOO^{\circ, \op}$ forgetting the structure involving the counit.  
The components of this coherent tranformation are thus determined in dimension $[k]$ by a map which sends an element 
\[ \xymatrix{  [k+1] \ar@{=}[r]  & [k+1] \ar[r]^{\delta_{i_0+1}} &  [k] \ar[r] & \cdots  \ar[r]^{\delta_{i_{k-1}+1}} & [1]   \ar@{=}[r] &  [1]  } \]
 in $N([k+1] \times_{/\Delta_{\act}^{\circ, \op}} \Delta_{\act}^{\circ, \op} \times_{/\Delta_{\act}^{\circ, \op}} [1])_{[k]}$
--- which we identify with the vector $\underline{i} = (i_0, \dots, i_{k-1})$ ---
 to a map
\[ (\Ezfrak^* H)_{\underline{i}}: \Delta_{k+1} \tildeotimes X^{\otimes k} \to X^{\otimes k}  \]
which corresponds (via \ref{BEMHOMCLASSICAL}) to a degree $k+1$-map between complexes:
\[ (\Ezfrak^* H)_{\underline{i}}^n:  (X^{\otimes k})_{n-k} \to (X^{\otimes k})_{n}.  \]
\end{PAR}

\begin{PAR}[Higher Shih operators]\label{HIGHERSHIH}
For a sequence $b_0 > b_1 > \cdots >  b_{k-1} \subset \{0, \dots, n-1\}$, and $i_0, \dots, i_{k-1}$, with $0 \le i_j \le k-j-1$, we would like to iterate the morphism
 $H^n_{(b),(0)} :=  (\id_{[b]} \ast \Ezfrak \, \Awfrak) s_b$.
Define\footnote{we have the factorization $\Delta^{\op} \to (\Delta^{\op})^{k-1} \to (\Delta^{\op})^{k}$ where the second map doubles the $i_0$-th entry. 
This gives a transport
\[ (- \circ \delta_{i_0}^*): \Hom(\delta_{k-1}^*, \delta_{k-1}^*) \to \Hom(\delta_{k}^*, \delta_{k}^*).  \]} inductively: 
\begin{align}  H^n_{\underline{b}, \underline{i}}: & (\delta_k^* X)_{[n-k]} \to (\delta_k^* X)_{[n]} \nonumber \\
\label{highershih} := & ( \underbrace{1}_{i_{0} \text{ times}}, (\id_{[b_{0}]} \ast \Ezfrak \, \Awfrak), \underbrace{1}_{k-i_{0}-1 \text{ times}} )s_{b_0} (- \circ \delta_{i_0}^*)  H^{n-1}_{\underline{b}', \underline{i}'}   
\end{align}
where $\underline{b}' = (b_1, \dots, b_{k-1})$ and $ \underline{i}' := (i_1, \dots, i_{k-1})$. 
\end{PAR}

\begin{PAR}\label{PARSHUFFLEB}
Each such sequence $\underline{b} = (b_0, \dots, b_{k-1})$ determines a shuffle $\sigma_{\underline{b}}: [k] \to [n]$, $\tau_{\underline{b}}: [n-k] \to [n]$ such that
$\tau_{\underline{b}} = s_{b_0}s_{b_1} \cdots s_{b_k-1}$ is degenerate precisely at the intervals $b_i+1$ and $\sigma_{\underline{b}}$ is degenerate precisely at the other intervals in $\{1, \dots, n\}$.
\end{PAR}

\begin{PROP}\label{PROPHIGHERSHIH}
We have
\[ \boxed{ (\Ezfrak^* H)_{\underline{i}}^n = \sum_{\underline{b}}\sgn(\sigma_{\underline{b}}) H_{\underline{b}, \underline{i}}^n.  } \]
\end{PROP}
\begin{proof}
This follows from the construction of $\exp$ (cf.\@ Proposition~\ref{PROPCOHERENT}) and from the fact that $\Ezfrak$ is the classical Eilenberg-Zilber map defined in terms of shuffles (Proposition~\ref{PROPAWEZAB}).
\end{proof}

We will derive explicit formulas for the $H_{\underline{b}, \underline{i}}^n$ in Proposition~\ref{PROPHIGHERSHIHEXPL} and relate them to the classical Szczarba operators.

\subsection{The different simplicial (weak) enrichments on complexes}

Let $(\mathcal{M}, \otimes)$ be a symmetric monoidal category. Recall the definition of {\bf operad enriched in $(\mathcal{M}, \otimes)$}.
Each enriched operad $\mathcal{C}$ has an underlying usual operad with $\Hom$-sets being $\Hom(1, \uHom(X,Y))$ for $Y \in \mathcal{C}_{[1]}$ and it induces a functor
\[ \uHom: (\mathcal{C}^{\op} \times \mathcal{C})_{/\OOO} \to (\mathcal{M}, \otimes) \]
(for the definition of $(\mathcal{C}^{\op} \times \mathcal{C})_{/\OOO} $, see Definition~\ref{DEFCOPC}). An enriched operad such that the underlying operad comes from
a monoidal category is not necessarily cofibered as enriched category over $\OOO$. This is the case if and only if it is classified by a functor of operads
\[ \OOO \to (\SCat, \times) \]
in which case the monoidal product has the structure of simplicially enriched functor. 

\begin{DEF}We say that an operad $\mathcal{C}$ is {\bf weakly enriched} over $(\mathcal{M}, \otimes)$, if there is a functor of operads
\[ \uHom: (\mathcal{C}^{\op} \times \mathcal{C})_{/\OOO} \to (\mathcal{M}, \otimes) \]
such that the composition with $\Hom(1, -)$ gives back the usual operadic $\Hom$ (\ref{eqhomoperads}).
\end{DEF}
Equivalently, it may be seen as a collection of $\uHom(X, Y)$ for pairs of objects with $Y \in \mathcal{C}_{[1]}$, but where composition is only defined partially with 
``constant morphisms'' for $Z \in \mathcal{C}_{[1]}$:
\begin{align} \label{eqenrich1} \Hom(1, \uHom(X_1, Y_1)) \times \cdots \times  \Hom(1, \uHom(X_n, Y_n))  \times \uHom(Y, Z) &\to \uHom(X, Z) \\
 \label{eqenrich2} \uHom(X_1, Y_1) \otimes \cdots \otimes  \uHom(X_n, Y_n)  \times \Hom(1, \uHom(Y, Z)) &\to \uHom(X, Z)  
  \end{align}
  
Notice that, for a weak enrichment on a {\em category} $\mathcal{C}$, the monoidal structure on $\mathcal{M}$ is irrelevant (except for the unit). 
However, this
is not true for weak enrichments of {\em operads}, in (\ref{eqenrich2}) $\otimes$ is the tensor product w.r.t.\@ which the weak enrichment is defined. 
Similarly, there is the notion of {\bf weakly enriched functor}. 

\begin{PAR}For a morphism of operads $G: (\mathcal{M}, \otimes) \to (\mathcal{M}', \otimes')$ (i.e.\@ a lax monoidal functor) and an operad $\mathcal{C}$ (weakly) $(\mathcal{M}, \otimes)$-enriched  
there is a (weak) $(\mathcal{M}', \otimes')$-enrichment given by
\[ G\uHom(X, Y). \]
For a functor $F: \mathcal{C} \to \mathcal{C}'$ of cooperads (weakly) enriched over $(\mathcal{M}, \otimes)$ there is an induced functor $GF$ of $(\mathcal{M}', \otimes')$-enriched (co)operads. 
\end{PAR}

\begin{PAR}
Everything  has a dual counterpart for cooperads. For example a weak enrichment on a cooperad $\mathcal{C}$ is a functor of {\em operads}: 
\[ \uHom: (\mathcal{C} \times \mathcal{C}^{\op})_{/\OOO} \to (\mathcal{M}, \otimes). \]
Equivalently, it may be seen as a collection of $\uHom(X, Y)$ for pairs of objects with $X \in \mathcal{C}_{[1]}$, but where composition is only defined partially with 
``constant morphisms'':
\begin{align*}    \uHom(X, Y) \times \Hom(1, \uHom(Y_1, Z_1)) \times \cdots \times  \Hom(1, \uHom(Y_n, Z_n))  &\to \uHom(X, Z) \\
    \Hom(1, \uHom(X, Y)) \times  \uHom(Y_1,  Z_1) \otimes \cdots \otimes  \uHom(Y_n, Z_n) &\to \uHom(X, Z)  
  \end{align*}
\end{PAR}

\begin{PAR}
Let $(\mathcal{C}, \otimes)$ be an Abelian tensor category. Recall the Day convolutions
\begin{align}
(\mathcal{C}^{\Delta^{\op}}, \otimes)^{\vee} & =  D((\Delta, \coprod)^{\op}, (\mathcal{C}, \otimes)^{\vee}) \\
(\mathcal{C}^{\Delta^{\op}}, \tildeotimes)^{\vee} & = D((\Delta, \ast)^{\op}, (\mathcal{C}, \otimes)^{\vee}) 
\end{align}
(i.e.\@ $\otimes := \delta^* - \boxtimes -$ is the point-wise product of simplicial objects, and $\tildeotimes := \dec_* - \boxtimes -$, where $\dec_* \cong \tot$, is the usual tensor product of complexes).
We will define natural (weak) $(\Ab^{\Delta^{\op}}, \otimes)$-enrichments, as well as $(\Ab^{\Delta^{\op}}, \tildeotimes)$-enrichments on these cooperads: 
\end{PAR}

\begin{DEF}
\label{DEFENRICHCOMPOSE}
Let $(\mathcal{C}, \otimes)$ be an Abelian tensor category. 

For objects $X \in \mathcal{C}^{\Delta^{\op}}_{[1]}$ and  $Y \in \mathcal{C}^{\Delta^{\op}}_{[m]}$ we define the following objects in $\Ab^{\Delta^{\op}}$ whose $n$-simplices are given by 
\begin{align}
 \uHom^{\otimes}_{(\mathcal{C}^{\Delta^{\op}}, \otimes)^{\vee}}(X, Y)_{[n]} &= \Hom(\Delta_n \otimes X, Y_1 \otimes \cdots \otimes Y_m)  \\
 \uHom^{\tildeotimes}_{(\mathcal{C}^{\Delta^{\op}}, \otimes)^{\vee}}(X, Y)_{[n]} &= \Hom(\Delta_n \tildeotimes X, Y_1 \otimes \cdots \otimes Y_m)  \\
 \uHom^{\tildeotimes}_{(\mathcal{C}^{\Delta^{\op}}, \tildeotimes)^{\vee}}(X, Y)_{[n]} &= \Hom(\Delta_n \tildeotimes X, Y_1 \tildeotimes \cdots \tildeotimes Y_m) 
\end{align}

We define a composition (in $(\Ab^{\Delta^{\op}}, \otimes)$ ) given\footnote{for simplicity, for $1$-ary compositions on the left, the others are defined completely analogously} for $f \in \uHom^{\ast}_{\ast}(X;Y_1, Y_2)_{[n]}$ and
$g_1 \in \uHom^{\ast}_{\ast}(Y_1;Z_1)_{[n]}$ and $g_2 \in \uHom^{\ast}_{\ast}(Y_2;Z_2)_{[n]}$ by 
\begin{align}
(g_1, g_2) \circ f: &\ \Delta_n \otimes X \to \Delta_n \otimes \Delta_n \otimes \Delta_n \otimes X \to \Delta_n \otimes \Delta_n \otimes Y_1 \otimes Y_2 \nonumber \\
\to &\ \Delta_n \otimes   Y_1 \otimes \Delta_n \otimes Y_2 \to Z_1 \otimes Z_2 \label{eqotimescompose}   \\
(g_1, g_2) \circ f: &\ \Delta_n \tildeotimes X \to (\Delta_n \otimes \Delta_n) \tildeotimes \Delta_n \tildeotimes X \to (\Delta_n \otimes \Delta_n) \tildeotimes (Y_1 \otimes Y_2)\nonumber \\
 \to &\ (\Delta_n \tildeotimes   Y_1) \otimes (\Delta_n \tildeotimes Y_2) \to Z_1 \otimes Z_2  \label{eqtildeotimesotimescompose} \\
(g_1, g_2) \circ f: &\ \Delta_n \tildeotimes X \to \Delta_n \tildeotimes \Delta_n \tildeotimes \Delta_n \tildeotimes X \to \Delta_n \tildeotimes \Delta_n \tildeotimes Y_1 \tildeotimes Y_2\nonumber \\
 \to &\ \Delta_n \tildeotimes   Y_1 \tildeotimes \Delta_n \tildeotimes Y_2 \to Z_1 \tildeotimes Z_2   \label{eqtildeotimescompose}
\end{align}
where in (\ref{eqtildeotimesotimescompose}) the morphism 
\[ \switch: (\Delta_n \otimes \Delta_n) \tildeotimes (Y_1 \otimes Y_2) \to (\Delta_n \tildeotimes   Y_1) \otimes (\Delta_n \tildeotimes Y_2) \]
is defined in Definition~\ref{DEFSWITCH} and the morphisms $\Delta_n \to \Delta_n \otimes \Delta_n$ (resp.\@ $\Delta_n \to \Delta_n \tildeotimes \Delta_n$) are the diagonal (resp.\@ AW-diagonal).
The maps (\ref{eqotimescompose}--\ref{eqtildeotimescompose}) are bilinear in $f$, $g_1$, and $g_2$, and thus extend to the tensor product. 
\end{DEF}
There is an operad version of these enrichement as well, defined in the same way. 
Notice that for $n=0$ we get back the usual Abelian groups $\Hom$ in $(\mathcal{C}^{\Delta^{\op}}, \otimes)^{\vee}$ resp.\@ $(\mathcal{C}^{\Delta^{\op}}, \tildeotimes)^{\vee}$.

\begin{LEMMA}
The compositions (\ref{eqotimescompose}) and (\ref{eqtildeotimescompose}) are associative
and define an enrichment, denoted 
$\uHom^{\otimes}_{(\mathcal{C}^{\Delta^{\op}}, \otimes)^{\vee}}(X, Y)$, and\footnote{The decoration $'$ will become apparent shortly.} $\uHom^{\tildeotimes}_{(\mathcal{C}^{\Delta^{\op}}, \tildeotimes)^{\vee}}(X, Y)'$, respectively, 
of the cooperads $(\mathcal{C}^{\Delta^{\op}}, \otimes)^{\vee}$, and $(\mathcal{C}^{\Delta^{\op}}, \tildeotimes)^{\vee}$, respectively, such that the enrichments are again monoidal (i.e.\@ turn the tensor-product into a simplicially enriched functor).
The composition in $\uHom^{\tildeotimes}_{(\mathcal{C}^{\Delta^{\op}}, \otimes)^{\vee}}(X, Y)$ (\ref{eqtildeotimesotimescompose}) is {\em not} associative, in general, but defines a weak enrichment. 
\end{LEMMA}
\begin{proof}Left to the reader. Note that the symmetry of $(\Ch_{\ge 0}(\Ab), \otimes)$ and $(\Ch_{\ge 0}(\Ab), \tildeotimes)$ is used to formulate associativity for $n$-ary compositions. 
In the latter case, it is given explicitly by Lemma~\ref{LEMMADEC1}.
\end{proof}

\begin{BEM}\label{BEMHOMCLASSICAL}
In case of the enrichment $\uHom^{\tildeotimes}_{(\mathcal{C}^{\Delta^{\op}}, \tildeotimes)^{\vee}}$, the Hom-objects identify via Dold-Kan with the usual Hom-complexes identifying 
an element 
$f: \Delta_n \tildeotimes X \to Y_1 \tildeotimes \cdots \tildeotimes Y_k$  with its restriction
\begin{equation} \label{eqhomotopy} \widetilde{f}: \{0, \dots, n\} \tildeotimes X \to Y_1 \tildeotimes \cdots \tildeotimes Y_k  \end{equation}
which is a degree $n$ morphism of complexes. We have (by definition) a commutative diagram
 \[ \xymatrix{ 
\Hom(\Delta_n \tildeotimes X, Y) \ar[rrr]^-{\dd_l^* = \sum (-1)^i \circ (\delta_i \otimes \id)^* } \ar[d] & & &  \Hom(\Delta_{n-1} \tildeotimes X, Y) \ar[d] \\
\uHom^{\tildeotimes}(X, Y)_n \ar[rrr]^-{\partial} & & &  \uHom^{\tildeotimes}(X, Y)_{n-1} 
 } \]
 where $\partial$ is the differential in the Hom-complex, and also 
 \begin{equation} \label{eqhomotopyd} \boxed{  \partial \widetilde{f} =  \mathrm{d} \circ \widetilde{f} - (-1)^n \widetilde{f} \circ \mathrm{d} .}    \end{equation}
 which follows from the fact that (\ref{eqhomotopy}) is a morphism of complexes itself, i.e.\@
$f \dd_l + (-1)^n  f \dd_r = \dd f$ which translates into (\ref{eqhomotopyd}).
\end{BEM}

\begin{LEMMA}[Yoneda product] \label{LEMMAYONEDA}
The composition (\ref{eqtildeotimescompose}) 
defined  for $\uHom^{\tildeotimes}_{(\mathcal{C}^{\Delta^{\op}}, \tildeotimes)^{\vee}}$ (resp.\@ $\uHom^{\tildeotimes}_{(\mathcal{C}^{\Delta^{\op}}, \tildeotimes)}$)  factors through $\Aw$ in the sense that 
\[\xymatrix{
\underline{\Hom}^{\tildeotimes}(B, C) \tildeotimes  \underline{\Hom}^{\tildeotimes}(A, B)  \ar[rr]^-{\underline{\Hom}^{\tildeotimes}} &  & \underline{\Hom}^{\tildeotimes}(A, C) \ar@{=}[d]  \\
\underline{\Hom}^{\tildeotimes}(B, C) \otimes  \underline{\Hom}^{\tildeotimes}(A, B)  \ar[u]^{\Aw} \ar[rr]^-{(\underline{\Hom}^{\tildeotimes})'} &  & \underline{\Hom}^{\tildeotimes}(A, C)  
}\]
is commutative, where 
\begin{enumerate}
\item for cooperads: 
\[  \uHom^{\tildeotimes}:   \uHom^{\tildeotimes}(Y_1, Z_1) \tildeotimes \cdots \tildeotimes \uHom^{\tildeotimes}(Y_m, Z_m) \tildeotimes \uHom^{\tildeotimes}(X, Y)  \to \uHom^{\tildeotimes}(X, Z) \]
is a morphism of complexes defined as follows (for $m=2$ for simplicity). A pure tensor in complex degree $n$ on the left can be given as $i+j+k=n$ and  $f: Y_1 \to Z_1^{\tildeotimes}$ of degree $i$,  $g: Y_2 \to Z_2^{\tildeotimes}$ of degree $j$, $h: X \to Y_1 \tildeotimes Y_2$ of degree $k$. Then the composition in $\uHom^{\tildeotimes}(X, Z)_n$ is given by
\[ (f \tildeotimes g) \circ h  \]

\item for operads:
\[ \uHom^{\tildeotimes}: \uHom^{\tildeotimes}(Y; Z) \tildeotimes  \uHom^{\tildeotimes}(X_1, Y_1) \tildeotimes \cdots \tildeotimes \uHom^{\tildeotimes}(X_m, Y_m)    \to \uHom^{\tildeotimes}(X, Z) \]
is a morphism of complexes defined as follows (for $m=2$ for simplicity). A pure tensor in complex degree $n$ on the left can be given as $i+j+k=n$ and  $h: Y_1 \tildeotimes Y_2 \to Z$ of degree $i$,  $f: X_1^{\tildeotimes} \to Y_1$ of degree $j$, $g: X_2^{\tildeotimes} \to Y_2$ of degree $k$. Then the composition in $\uHom^{\tildeotimes}(X, Z)_n$ is given by
\[ h \circ (f \tildeotimes g) \]
\end{enumerate}
where in both cases $(f \tildeotimes g)$ is defined by (Koszul convention, cf.\@ also \ref{PARKOSZUL} below):
\[ (f \tildeotimes g)(x \tildeotimes y) := (-1)^{\deg(g) \deg(x)} f(x) \tildeotimes g(y). \]
\end{LEMMA}

\begin{proof}
The reason is that we {\em defined} the composition (\ref{eqtildeotimescompose}) using the Alexander-Whitney diagonal. Details are left to the reader. 
\end{proof}

From the Lemma follows that also 
\[\xymatrix{
\underline{\Hom}^{\tildeotimes}(B, C) \tildeotimes  \underline{\Hom}^{\tildeotimes}(A, B)  \ar[d]_{\Ez}  \ar[rr]^-{\underline{\Hom}^{\tildeotimes}} &  & \underline{\Hom}^{\tildeotimes}(A, C) \ar@{=}[d]  \\
\underline{\Hom}^{\tildeotimes}(B, C) \otimes  \underline{\Hom}^{\tildeotimes}(A, B)   \ar[rr]^-{(\underline{\Hom}^{\tildeotimes})'} &  & \underline{\Hom}^{\tildeotimes}(A, C)  
}\]
is commutative or, in other words, $\underline{\Hom}^{\tildeotimes} = \Ez^* (\underline{\Hom}^{\tildeotimes})'$.

\begin{PAR}\label{PARENRICHMENTS}
By transport via $F, \Aw, \Ez$ we thus get thus six enrichments
\[ \begin{array}{l|lll}
\hbox{\diagbox{\text{cooperad}}{\text{enriched in}}}&(\Set^{\Delta^{\op}}, \times) & (\Ab^{\Delta^{\op}}, \otimes) & (\Ab^{\Delta^{\op}}, \tildeotimes) \\
\hline
(\mathcal{C}^{\Delta^{\op}}, \otimes)^{\vee} &F\underline{\Hom}^{\otimes}&\underline{\Hom}^{\otimes} & \Ez^* \underline{\Hom}^{\otimes} \\ 
(\mathcal{C}^{\Delta^{\op}}, \tildeotimes)^{\vee} &F\Aw^*\underline{\Hom}^{\tildeotimes} & \Aw^*\underline{\Hom}^{\tildeotimes} & \underline{\Hom}^{\tildeotimes} 
\end{array}\]
and three weak enrichments
\[ \begin{array}{l|lll}
\hbox{\diagbox{\text{cooperad}}{\text{weakly enriched in}}}&(\Set^{\Delta^{\op}}, \times) & (\Ab^{\Delta^{\op}}, \otimes) & (\Ab^{\Delta^{\op}}, \tildeotimes) \\
\hline
(\mathcal{C}^{\Delta^{\op}}, \otimes)^{\vee}&F\underline{\Hom}^{\tildeotimes} & \underline{\Hom}^{\tildeotimes} & \Ez^* \underline{\Hom}^{\tildeotimes} 
\end{array}\]
where $F: (\Ab^{\Delta^{\op}}, \otimes) \to (\Set^{\Delta^{\op}}, \times)$ is the forgetful functor, which is lax monoidal:
\[  F(A) \times F(B) \to F(A \otimes B). \]
\end{PAR}

\begin{PROP}\label{PROPSIMPLICIALLYENRICHED}
\begin{enumerate}
\item There is a functor of weakly $(\Ab^{\Delta^{\op}}, \otimes)$-enriched cooperads, which is the identity on objects and underlying morphism-sets: 
\begin{equation} \label{eqaw1} \Ezfrak^*: \uHom^{\otimes}_{(\mathcal{C}^{\Delta^{\op}}, \otimes)^{\vee}}(X, Y) \to  \uHom^{\tildeotimes}_{(\mathcal{C}^{\Delta^{\op}}, \otimes)^{\vee}}(X, Y).  \end{equation}
\item There is a functor of weakly $(\Ab^{\Delta^{\op}}, \otimes)$-enriched cooperads, which is the identity on objects and induces the $\Awfrak$-morphism on underlying morphism-sets.
\begin{equation}  \label{eqaw2} \Awfrak: \uHom^{\tildeotimes}_{(\mathcal{C}^{\Delta^{\op}}, \otimes)^{\vee}}(X, Y) \to  \Aw^* \uHom^{\tildeotimes}_{(\mathcal{C}^{\Delta^{\op}}, \tildeotimes)^{\vee}}(X, Y)  \end{equation}
\item $\Ez^*$ applied to the composition is a  functor of $(\Ab^{\Delta^{\op}}, \tildeotimes)$-enriched cooperads 
\[ \Ez^*(\Awfrak): \Ez^* \uHom^{\otimes}_{(\mathcal{C}^{\Delta^{\op}}, \otimes)^{\vee}}(X, Y) \to \uHom^{\tildeotimes}_{(\mathcal{C}^{\Delta^{\op}}, \tildeotimes)^{\vee}}(X, Y) \]
(but the composition $\uHom^{\otimes}_{(\mathcal{C}^{\Delta^{\op}}, \otimes)^{\vee}}(X, Y) \to \Aw^*\uHom^{\tildeotimes}_{(\mathcal{C}^{\Delta^{\op}}, \tildeotimes)^{\vee}}(X, Y)$ is only a weakly $(\Ab^{\Delta^{\op}}, \otimes)$-enriched functor)
\end{enumerate}
\end{PROP}
\begin{proof} 
1.\@ Morphism (\ref{eqaw1}) is given by mapping $\Delta_n \otimes X \to Y_1 \otimes \cdots \otimes Y_n$ to the pre-composition
\[ \Delta_n \tildeotimes X \to \Delta_n \otimes X \to Y_1 \otimes \cdots \otimes Y_n \]
with the Eilenberg-Zilber morphism. The only non-trivial case to see that this a morphism of weak enrichments is the composition on the left with a 2-ary morphism in $\Hom(X; Y, Z)$, i.e.\@ to show that the following diagram
\[ \xymatrix{
X \tildeotimes \Delta_n \ar[d] \ar[rr]^{\Ezfrak}  & & X \otimes \Delta_n \ar[d] \\
(Y \otimes Z) \tildeotimes (\Delta_n \otimes \Delta_n) \ar[rr]^{\Ezfrak} \ar[d]^{\switch}  & &  Y \otimes Z \otimes \Delta_n \otimes \Delta_n \ar[d]^{\sigma} \\
( Y \tildeotimes  \Delta_n) \otimes (Z \tildeotimes \Delta_n) \ar[d] \ar[rr]^{\Ezfrak \otimes \Ezfrak} & &  Y \otimes \Delta_n \otimes Z \otimes \Delta_n \ar[d]  \\
Y' \otimes Z' \ar@{=}[rr] & & Y' \otimes Z'
}\]
is commutative. The commutativity of the middle square is Lemma~\ref{LEMMABATEAU2} below.

2.\@ Morphism (\ref{eqaw2}) is given by mapping $\Delta_n \tildeotimes X \to Y_1 \otimes \cdots \otimes Y_n$
to the composition
\[ \Delta_n \tildeotimes X \to Y_1 \otimes \cdots \otimes Y_n \to Y_1 \tildeotimes \cdots \tildeotimes Y_n \]
with the Alexander-Whitney morphism. 
The only non-trivial case to see that this a morphism of weak enrichments is to show that the following diagram
\[ \xymatrix{
X \tildeotimes \Delta_n \ar[d] \ar@{=}[rr] & & X \tildeotimes \Delta_n \ar[d] \\
(Y \otimes Z) \tildeotimes (\Delta_n \otimes \Delta_n) \ar[rr]^-{\Awfrak \tildeotimes \Awfrak} \ar[d]^{\switch} & & Y \tildeotimes Z \tildeotimes \Delta_n \tildeotimes \Delta_n \ar[d]^{\widetilde{\sigma}} \\
( Y \tildeotimes  \Delta_n) \otimes (Z \tildeotimes \Delta_n) \ar[d] \ar[rr]^{\Awfrak} & & Y \tildeotimes \Delta_n \tildeotimes Z \tildeotimes \Delta_n \ar[d]  \\
Y' \otimes Z' \ar[rr]^-{\Awfrak} & &Y' \tildeotimes Z'
}\]
is commutative. The commutativity of the middle square is Lemma~\ref{LEMMABATEAU3} below.

3.\@ is left to the reader. 
\end{proof}

\begin{LEMMA}[{\cite[Proposition 2.2.1]{Fra01}}]\label{LEMMABATEAU1}
The following is commutative: 
\begin{equation} \label{eqbateau} \vcenter{  \xymatrix{ &  \delta^* \delta_{12,34}^*   \ar[r]^{\sigma} &  \delta^* \delta_{13,24}^* \ar[rd]^{\Aw}  &. \\
\dec_* \delta_{12,34}^*  \ar[ru]^{\Ez} \ar[rd]_{\Aw_{12,34}} & & &\dec_* \delta_{13,24}^*   \\
&  \dec_* \dec_{12,34,*} \ar[r]_{\widetilde{\sigma}} &  \dec_* \dec_{13,24,*} \ar[ru]_{\Ez_{13,24}} } } \end{equation}
and hence for objects $A, B, C, D$ with an isomorphism
\[  A \boxtimes B \boxtimes C \boxtimes D \cong (23)^* A \boxtimes C \boxtimes B \boxtimes D  \]
the diagram 
\[ \xymatrix{ &   A \otimes B \otimes C \otimes D \ar[r]^{\sigma} &  A \otimes  C \otimes B \otimes D \ar[rd]^{\Awfrak}   \\
 (A \otimes B) \tildeotimes  (C \otimes D) \ar[ru]^{\Ezfrak} \ar[rd]_{\Awfrak \tildeotimes \Awfrak} & & & (A \otimes  C) \tildeotimes  (B  \otimes D) \\
&  A \tildeotimes B \tildeotimes  C \tildeotimes D \ar[r]_{\widetilde{\sigma}} &  A \tildeotimes  D \tildeotimes B \tildeotimes C \ar[ru]_{\Ezfrak \tildeotimes \Ezfrak} }\]
is commutative. 
\end{LEMMA}

\begin{proof}
(\ref{eqbateau}) is in standard form (cf.\@ \ref{PARSTANDARD}), hence it suffices to see that 
\[ \vcenter{ \xymatrix{ \dec^* \delta^* \dec^* \ar@{=}[r] \ar@{<-}[d]^{\dec^* \mlq c^{\op} \mrq} & \delta_{13,24}^* \dec^*_{12,34} \dec^* \ar@{<-}[d]^{\widetilde{\sigma}} \\
\dec^* \ar[r]_-{\mlq c_{13,24}^{\op}\mrq \dec^*} & \delta_{13,24}^* \dec_{13,24}^* \dec^* 
  }}\ \text{
  and } \ 
 \vcenter{\xymatrix{ \dec^* \delta^* \delta_{12,34}^* \ar@{=}[r] \ar@{=}[d]^{\sigma} & \delta_{13,24}^* \dec^*_{12,34} \delta_{12,34}^* \ar[d]^{\delta_{13,24}^* u_{12,34}^{\op}} \\
\dec^* \delta^* \delta_{13,24}^* \ar[r]_{u^{\op} \delta_{13,24}^*} & \delta_{13,24}^*
  }} \]
  commute. This translates into the following commutative diagrams of functors $\FinSet^2 \to \FinSet$ and $\Delta^2 \to \Delta^4$, respectively, noticing that $\dec^*$ commutes with $C$ (Lemma~\ref{LEMMACCOH}):
\[ \xymatrix{ \dec_s \delta_s \dec_s \ar@{=}[r] \ar[d]^{c \dec_s} &   \dec_s \dec_{s,12,34} \delta_{s,13,24} \ar[d]^{\widetilde{\sigma}} \\
\dec_s \ar@{<-}[r]_-{\dec_s c_{13,24}} &  \dec_s \dec_{s,13,24}  \delta_{s,13,24}  } \quad 
 \xymatrix{ \delta_{12,34} \delta  \dec \ar@{=}[r] \ar@{=}[d]^{\sigma} &   \delta_{12,34} \dec_{12,34} \delta_{13,24}  \ar@{<-}[d]^{u_{12,34} \delta_{13,24}} \\
\delta_{13,24} \delta \dec  \ar@{<-}[r]_-{\delta_{13,24} u} & \delta_{13,24}
  } \]
  whose commutativity is checked straightforwardly. 
\end{proof}

\begin{DEF}\label{DEFSWITCH}
We define the {\bf switch} map as the following composite
\[ \mathrm{switch}: \xymatrix{ \dec_* \delta_{12,34}^* \ar[r]^-{\Ez \delta_{12,34}^* } &  \delta^* \delta_{12,34}^*  \ar@{=}[r]^{\sigma} &  \delta^* \delta_{13,24}^* \ar[rr]^-{\delta^* \Aw_{13,24}} &  & \delta^* \dec_{13,24,*}  \\
 }\]
and hence for an appropriately symmetric collection of objects: 
\[ \tiny \xymatrix{ (A \otimes B) \tildeotimes  (C \otimes D) \ar[r]^{\Ezfrak} & A \otimes B \otimes C \otimes D  \ar[r]^{\sigma} & A \otimes C \otimes B \otimes D  \ar[rr]^-{\Awfrak \tildeotimes \Awfrak} & & (A \tildeotimes  C) \otimes  (B  \tildeotimes D) }   \\
\]
\end{DEF}

\begin{LEMMA}\label{LEMMABATEAU2}
The following is commutative: 
\[ \xymatrix{ \dec_* \delta_{12,34}^* \ar[d]_{\Ez \delta_{12,34}^*} \ar[r]^{\mathrm{switch}} &  \delta^* \dec_{13,24}^* \ar[d]^{\delta^* \Ez_{13,24}} \\
 \delta^* \delta_{12,34}^* \ar@{=}[r]_{\sigma} &\delta^* \delta_{13,24}^*   }\]
and hence for an appropriately symmetric object 
\[ \xymatrix{ (A \otimes B) \tildeotimes  (C \otimes D) \ar[d]^{\Ezfrak} \ar[r]^{\mathrm{switch}} &   (A \tildeotimes C) \otimes  (B \tildeotimes D) \ar[d]^{\Ezfrak \otimes \Ezfrak} \\
 A \otimes B \otimes  C \otimes D  \ar[r]^{\sigma} & A \otimes C \otimes  B \otimes D
 }\]
 is commutative. 
\end{LEMMA}
\begin{proof}

 It suffices to show the equality after precomposition with $\Aw$, i.e.\@ that
 \[ \tiny  \xymatrix{ \delta^* \delta_{12,34}^* \ar[d]_{\Ez \Aw \delta_{12,34}^* } \ar[rr]^-{\Ez \Aw \delta_{12,34}^* }  & &  \delta^* \delta_{12,34}^*  \ar@{=}[r]^{\sigma} &  \delta^* \delta_{13,24}^* \ar[rr]^-{\delta^* \Aw_{13,24}} &  &  \delta^* \dec_{13,24}^* \ar[d]^{\delta^* \Ez_{13,24}} \\
 \delta^* \delta_{12,34}^* \ar@{=}[rrrrr]_{\sigma} &&& & & \delta^* \delta_{13,24}^*   }\]
is commutative. Using Lemma~\ref{LEMMAALTAWEZ} this amount to the commutativity of the outer diagram in
\[ \footnotesize \xymatrix{ \delta^* \delta_{12,34}^* \ar[r]^-{\mlq c^{\op} \mrq} \ar[d]^-{\mlq c^{\op} \mrq} & \delta^* \dec^* \delta^*  \delta_{12,34}^* \ar@{=}[d]  \ar[r]^-{u^{\op}} & \delta^* \delta_{12,34}^* \ar@{=}[r] & \delta^* \delta_{13,24}^* \ar[d]^-{\mlq c^{\op} \mrq_{13,24}} \\
\delta^* \dec^* \delta^*  \delta_{12,34}^*  \ar[d]^{u^{\op}} \ar@{=}[r] & \delta^* \delta_{13,24}^* \dec_{13,24}^* \delta_{13,24}^* \ar@{}[ru]|-{\mathcircled{A}}   \ar@{=}[rr] \ar[rru]_{u^{\op}_{13,24}} &   &  \delta^* \delta_{13,24}^* \dec_{13,24}^* \delta_{13,24}^* \ar[d]^{u^{\op}_{13,24}} \\
\delta^* \delta_{12,34}^* \ar@{=}[rrr] \ar@{}[rrru]|-{\mathcircled{B}} &  & & \delta^* \delta_{13,24}^*  }
        \]
       Here all maps denoted by = are the canonical identifications. It suffices to see that A and B are commutative. Actually these are the same diagram. 
       We have to see that
\[ \xymatrix{ 
\delta_{12,34}\delta \dec \delta  \ar@{<-}[r] \ar@{<-}[d]_{\delta_{12,34} u \delta} & \delta_{13,24}\dec_{13,24} \delta_{13,24} \delta \ar@{<-}[d]^{u_{13,24} \delta_{13,24} \delta} \\
\delta_{12,34}\delta \ar@{<-}[r] &    \delta_{13,24} \delta
}\]
commutes as diagram of functors $\Delta \to \Delta^4$ (no FinSet symmetry is involved here). 
This is checked straightforwardly. 
\end{proof}

\begin{LEMMA}\label{LEMMABATEAU3}
The following is commutative: 
\[ \xymatrix{ \dec_* \delta_{12,34}^* \ar[d]_{\dec_* \Aw_{13,24} } \ar[r]^{\mathrm{switch}} &  \delta^* \dec_{13,24}^* \ar[d]^{\Aw \dec_{13,24}^* } \\
 \dec_* \dec_{12,34,*} \ar[r]_{\widetilde{\sigma}} & \dec_* \dec_{13,24,*}   }\]
and hence for an appropriately symmetric object 
\[ \xymatrix{ (A \otimes B) \tildeotimes  (C \otimes D) \ar[d]_{\Awfrak \tildeotimes \Awfrak} \ar[r]^{\mathrm{switch}} &   (A \tildeotimes C) \otimes  (B \tildeotimes D) \ar[d]^{\Awfrak} \\
 A \tildeotimes B \tildeotimes  C \tildeotimes D  \ar[r]_{\widetilde{\sigma}} & A \tildeotimes C \tildeotimes  B \tildeotimes D
 }\]
 is commutative. 
\end{LEMMA}
\begin{proof}
Follows from the following commutative diagram in which the hexagon is in Lemma~\ref{LEMMABATEAU1}:
\begin{equation*} \tiny  \vcenter{  \xymatrix{ 
&&&\delta^* \dec_{13,24}^* \ar[rd]^{\Aw \dec_{13,24}^*} \\
&  \delta^* \delta_{12,34}^*   \ar[r]^{\sigma} &  \delta^* \delta_{13,24}^* \ar[rd]_{\Aw \delta_{13,24}^*} \ar[ru]^{\delta^* \Aw_{13,24}}  &.&   \dec_* \dec_{13,24,*}  \\
\dec_* \delta_{12,34}^*  \ar[ru]_{\Ez \delta_{12,34}^*} \ar[rd]_{\dec_* \Aw_{12,34}} \ar@/^50pt/[rrruu]^{\switch} & & &\dec_* \delta_{13,24}^*  \ar[ru]^{\dec_* \Aw_{13,24}}  \\
&  \dec_* \dec_{12,34,*} \ar[r]^{\widetilde{\sigma}} &  \dec_* \dec_{13,24,*} \ar[ru]_{\dec_* \Ez_{13,24}} \ar@/_30pt/@{=}[rruu] & & & 
 } } \end{equation*}
\end{proof}

 \begin{DEF}\label{DEFTRUNC} For later applications we will need truncated versions of the (weak) simplicial enrichments discussed so far:
 We define truncated (weak) $(\Ab^{\Delta^{\op}}, \otimes)$-enrichments for the cooperads (in which the counits have been discarded):
\[ (\mathcal{C}^{\Delta^{\op}}, \otimes)^{\circ, \vee} \text{ and } (\mathcal{C}^{\Delta^{\op}}, \tildeotimes)^{\circ, \vee}   \]
setting:  
\begin{align*}
 (\uHom^{\otimes, t}_{(\mathcal{C}^{\Delta^{\op}}, \otimes)^{\circ, \vee}})_k(X; Y_1, \dots, Y_n) &:=  \begin{cases} (\uHom^{\otimes}_{(\mathcal{C}^{\Delta^{\op}}, \otimes)^{\circ, \vee}})_k(X; Y_1, \dots, Y_n) & k < n \\ 0 & \text{otherwise } \end{cases} \\
 (\uHom^{\tildeotimes, t}_{(\mathcal{C}^{\Delta^{\op}}, \otimes)^{\circ, \vee}})_k(X; Y_1, \dots, Y_n) &:=  \begin{cases} (\uHom^{\tildeotimes}_{(\mathcal{C}^{\Delta^{\op}}, \otimes)^{\circ, \vee}})_k(X; Y_1, \dots, Y_n) & k < n \\ 0 & \text{otherwise } \end{cases} 
 \end{align*}
with composition defined as for the untruncated case
and a truncated $(\Ab^{\Delta^{\op}}, \tildeotimes)$-enrichment for the cooperad $(\mathcal{C}^{\Delta^{\op}}, \tildeotimes)^{\circ, \vee}$ (in which the counits have been discarded):
\begin{align*}
 (\uHom^{\tildeotimes, t}_{(\mathcal{C}^{\Delta^{\op}}, \tildeotimes)^{\circ, \vee}})_k(X; Y_1, \dots, Y_n) &:=  \begin{cases} (\uHom^{\tildeotimes}_{(\mathcal{C}^{\Delta^{\op}}, \tildeotimes)^{\circ, \vee}})_k(X; Y_1, \dots, Y_n) & k < n \\ 0 & \text{otherwise } \end{cases} 
 \end{align*}
with composition defined as for the untruncated case. 
Note that this works only because we have discarded the counits (no map in $\Delta^{\circ, \op}$ decreases the arity).
\end{DEF}

\begin{LEMMA}
\begin{enumerate}
\item $\uHom^{\otimes, t}_{(\mathcal{C}^{\Delta^{\op}}, \otimes)^{\circ, \vee}}$ defines a $(\Ab^{\Delta^{\op}}, \otimes)$-enrichment;
\item $\uHom^{\tildeotimes, t}_{(\mathcal{C}^{\Delta^{\op}}, \otimes)^{\circ, \vee}}$ defines a weak $(\Ab^{\Delta^{\op}}, \otimes)$-enrichment;
\item $\uHom^{\tildeotimes, t}_{(\mathcal{C}^{\Delta^{\op}}, \tildeotimes)^{\circ, \vee}}$ defines a $(\Ab^{\Delta^{\op}}, \tildeotimes)$-enrichment.
\end{enumerate}
\end{LEMMA}
\begin{proof}
This follows from the fact (cf.\@ Proposition~\ref{PROPEXPLICITAB}) that if $X$ and $Y$ are such that $Y_i = 0$ for $i > n$ and $Y_i = 0$ for $i > m$ then $(X \otimes Y)_i = 0$ and $(X \tildeotimes Y)_i = 0$ for $i > n+m$. 
\end{proof}

\subsection{Explicit formul\ae}

Let $\mathcal{C}$ be an Abelian category.
We  identify $\mathcal{C}^{\Delta^{\op}}$ with $\Ch(\mathcal{C})_{\ge 0}$ via the adjoint equivalences $N$ and $R = \Gamma$ (Dold-Kan Theorem~\ref{SATZDOLDKANII}) and consider the functors
$\delta^*$, $\dec^*$,  $\dec_*$ and $\dec_!$ as functors between complexes and double complexes. 

\begin{PAR}\label{PARSKELETAL2}
Recall the skeletal filtration (\ref{PARFILT}) with 
\[ (F^k X)_{[n]} = \sum_{\sigma: [d] \twoheadleftarrow [n]} \sigma(X_{[d]}) \]
where the sum is over the degeneracies with $k \ge d$. It follows from the dual assertion (\ref{PARSKELETAL1}) that, considering $X$ as a complex, this coincides with the truncation:
\[ F^k X =  \{ \cdots \to 0 \to X_k \to X_{k-1} \to \cdots \to X_0 \}. \]
\end{PAR}

\begin{PROP}\label{PROPEXPLICITAB}
\begin{enumerate}
\item  
$(\dec_* X)_n \cong \tot (X)_{n} := \bigoplus_{i+j=n}X_{i,j}$
with the differential given by:  
\begin{equation}\label{eqtot} \mathrm{d} = \mathrm{d}_l + (-1)^i \mathrm{d}_r. \end{equation}
\item 
$(\dec^* A)_{i,j} \cong A_{i+j+1} \oplus A_{i+j}$ 
forming the double complex:
  \[ \footnotesize \xymatrix{   & j & &  j-1 \\
 i &  A_{i+j+1} \oplus A_{i+j} \ar[rr]^-{\mathrm{d}_r = \Mat{0 & 1 \\ 0 & 0}} \ar[dd]_-{\mathrm{d}_l = \Mat{\mathrm{d} & (-1)^i  \\ 0 & \mathrm{d}}} & & A_{i+j} \oplus A_{i+j-1} \ar[dd]^-{\mathrm{d}_l = \Mat{\mathrm{d} & (-1)^i \\ 0 & \mathrm{d}}} \\ 
 \\
i-1  & A_{i+j} \oplus A_{i+j-1} \ar[rr]_-{\mathrm{d}_r = \Mat{0 & 1 \\ 0 & 0}} & &  A_{i+j-1} \oplus A_{i+j-2}
 }\]
 \item 
 Let $X$ be a double complex and denote by $\widetilde{X}$ the complex extended to $-1,\N_0$ and $\N_0,-1$ by taking the cokernel of the last differential. 
 Then 
 \[ \dec_! X \cong \widetilde{\tot} (X)_{n} := \bigoplus_{\substack{i+j=n-1 \\ i,j \ge -1}} \widetilde{X}_{i,j}. \]
\item For the natural two-dimensional filtration $F^{i,j}\delta^* X := \delta^* F^{i,j} X$ (cf.\@ \ref{PARSKELETAL2}) we have
\[ (\gr^{i,j} \delta^* X)_n \cong \bigoplus_{\sigma, \tau} X_{i,j} \]
where $\sigma: [n] \twoheadrightarrow [i], \tau: [n] \twoheadrightarrow [j]$ runs through the {\em jointly injective pairs of surjections}.
In particular: $(F^{i,j}\delta^* X)_n = (\delta^* X)_n$ for $i \ge n$ and $j \ge n$ and $(F^{i,j}\delta^* X)_n = 0$ for $i+j<n$ and 
\[ (F^{i,n-i}\delta^* X)_n \cong \bigoplus_{\sigma, \tau} X_{i,n-i} \  \qquad (\gr^{n,n} \delta^* X)_n \cong X_{n,n} . \]
In this case ($i+j=n$) the jointly injective pairs of surjections $\sigma: [n] \twoheadrightarrow [i], \tau: [n] \twoheadrightarrow [j]$ are called $i,j$-{\bf shuffles}. 
The isomorphisms are determined by requiring that
 \[ \xymatrix{ (\delta^*X)_{[n]} = X_{[n],[n]} \ar@{<-}[r]^-{\sigma, \tau} \ar@{->>}[d] & X_{[i],[j]} \ar@{->>}[d] \\
 (\delta^*X)_n \ar@{<-^{)}}[r] & X_{i,j} }\]
commutes modulo $F^{i,j-1}(\delta^*X)_n + F^{i-1,j}(\delta^*X)_n$ (which is zero if $i+j=n$), where the vertical morphisms are the projections onto non-degenerate elements.
\end{enumerate}
\end{PROP}

We start by discussing the translation of the functors
\begin{align*} \dec^*:  \mathcal{C}^{\Delta^{\op}} &\to \mathcal{C}^{\Delta^{\op} \times \Delta^{\op}} \\
\delta^*:  \mathcal{C}^{\Delta^{\op} \times \Delta^{\op}}  &\to \mathcal{C}^{\Delta^{\op}}
\end{align*}
via Dold-Kan.  We have
\begin{align*} (\dec^* C)_{m,n} &= \Hom(D_{m, n}, \dec^* C) = \Hom(\dec_! D_{m,n}, C) \\
 (\delta^* C)_m &= \Hom(D_m, \delta^* C) = \Hom(\delta_! D_m, C) 
\end{align*}
Furthermore, by Yoneda and the commuation of $\Z[-]$ with colimits: 
\begin{align*}  \delta_! \Delta_m^{\circ} &= \Delta_{m,m}^{\circ} \\
 \dec_! \Delta_{m,n}^{\circ} &= \Delta_{m+n+1}^{\circ}   
\end{align*}
such that
\begin{align*}
 \delta_! D_m &= \ker( \Delta_{m,m}^{\circ} \to \Delta_{m',m'}^{\circ} )  \\
 \dec_! D_{m, n}  &= \ker( \Delta_{m+n+1}^{\circ}  \to \Delta_{m'+n'+1}^{\circ} )   
\end{align*}
where the kernel is joint over all degeneracies $[n] \twoheadrightarrow [n']$, resp.\@ $[m] \twoheadrightarrow [m']$. The filtrations on  $\Delta_{m,m}^{\circ}$ and $\Delta_{m+n+1}^{\circ}$ 
are respected by these morphisms in such a way to induce filtrations:
\begin{align*}
 F^{k,l}\delta_! D_m &= \ker( F^{k,l}\Delta_{m,m}^{\circ} \to F^{k,l}\Delta_{m',m'}^{\circ} )  \\
 F^k \dec_! D_{m, n}  &=  \ker( F^k \Delta_{m+n+1}^{\circ}  \to F^{k} \Delta_{m'+n'+1}^{\circ} )   
\end{align*}

We get
\begin{align*}
  \gr \delta_! D_m &= \bigoplus_{\substack{\{[m] \twoheadrightarrow [k], [m] \twoheadrightarrow [l] \}  \\ \text{that do not factor over $\im(\delta^*)$ }}} D_{k,l}  \\
  \gr \dec_! D_{m,n} &= \bigoplus_{\substack{ \{ [m+n+1| \twoheadrightarrow [k]\} \\ \text{that do not factor over $\im(\dec^*)$ }}} D_{k} 
\end{align*}
In the first case the sum goes over pairs $[m] \twoheadrightarrow [k]$ and $[m] \twoheadrightarrow [l]$ that are {\em jointly injective} and in the second case the sum has {\em only two} summands corresponding to
$\id: [m+n+1] \to [m+n+1]$ and the canonical degeneracy $s_{\can}: [m+n+1] \to [m+n]$ identifying the maximum of $[m]$ with the minimum of $[n]$.

\begin{proof}[Proof of Proposition~\ref{PROPEXPLICITAB}]
2.\@ We get an exact sequence
\[ 0 \to D_{n+m+1} \to \dec_! D_{m,n} \to D_{n+m} \to 0 \]
and it is convenient to choose the following explicit splittings $l, r$ of the surjective map induced by 
\[ (-1)^i \dec^*(\dd_l), \dec^*(\dd_r): \Delta_{m+n}^{\circ} \to \Delta_{m+n+1}^{\circ} \]
Notice that $(-1)^i \dec^*(\dd_l) s \cong \id$ and $\dec^*(\dd_r) s \cong \id$ modulo degeneracies. 

Dually for a complex $A$, we get a diagram: 
\[ \xymatrix{ & A_{[i+j]} \ar[r]^-{s_{\can}}  \ar@{->>}[d]   & A_{[i+j+1]} \ar@/_10pt/[l]_{(-1)^i \dec^*d_l,\dec^*d_r}  \ar@{->>}[d] \ar@{->>}[rd] \\
0 \ar[r] & A_{i+j} \ar[r]  & (\dec^*A)_{i,j} \ar@/_10pt/[l]_{l,r}  \ar[r] & A_{i+j+1}  \ar[r] & 0 } \]

The splitting $r$ induces  isomorphisms: $(\dec^*A)_{i,j}  \cong A_{i+j}  \oplus A_{i+j+1}$ with differentials identified with
  \[  \footnotesize \xymatrix{  
 A_{i+j+1} \oplus A_{i+j} \ar[rr]^-{\mathrm{d}_r = \Mat{0 & 1 \\ 0 & 0}} \ar[dd]_-{\mathrm{d}_l = \Mat{d & (-1)^i  \\ 0 & d}} & & A_{i+j} \oplus A_{i+j-1} \ar[dd]^-{\mathrm{d}_l = \Mat{d & (-1)^i \\ 0 & d}} \\ 
 \\
 A_{i+j} \oplus A_{i+j-1} \ar[rr]_-{\mathrm{d}_r = \Mat{0 & 1 \\ 0 & 0}} & &  A_{i+j-1} \oplus A_{i+j-2}
 }\]

The splitting $l$ induces  isomorphisms: $(\dec^*A)_{i,j}  \cong A_{i+j}  \oplus A_{i+j+1}$ with differentials identified with
  \[ \footnotesize \xymatrix{  
 A_{i+j+1} \oplus A_{i+j} \ar[rr]^-{\mathrm{d}_r = \Mat{(-1)^i\dd & -1 \\ 0 & \dd}} \ar[dd]_-{\mathrm{d}_l = \Mat{0 & (-1)^i  \\ 0 & 0}} & & A_{i+j} \oplus A_{i+j-1} \ar[dd]^-{\mathrm{d}_l = \Mat{0 & (-1)^i \\ 0 & 0}} \\ 
 \\
 A_{i+j} \oplus A_{i+j-1} \ar[rr]_-{\mathrm{d}_r = \Mat{(-1)^i \dd & -1 \\ 0 & \dd}} & &  A_{i+j-1} \oplus A_{i+j-2}
 }\]

In these notes, when interested in explicit formulas, we usually work with the splitting $r$ because it induces the more common convention on the differential on the total complex.

1.\@ Assertion 2.\@ shows that a morphism
\[ \Hom(\dec^* X, Y) \]
is given by morphisms
\[ \xymatrix{ X_{i+j+1} \oplus X_{i+j} \ar[rrr]^-{(\mathrm{d}_r \alpha_{i,j+1},  \alpha_{i,j})} & &  &   Y_{i,j} } \]
satisfying 
\[  \alpha_{i-1,j} \mathrm{d}  = \mathrm{d}_l \alpha_{i,j} + (-1)^{i-1}  \mathrm{d}_r \alpha_{i-1,j+1}   \]
which is the same as a morphism of complexes
\[  X \to \tot Y. \]
where $\tot X$ is equipped with the differential (\ref{eqtot}). 

3.\@ is shown the same way as 1.

4.\@ has been discussed above. 
Note that, for $\delta^*$, the lowest filtration steps are given by those pairs $\sigma: [m] \twoheadrightarrow [i]$ and $\tau: [m] \twoheadrightarrow [j]$ that are jointly injective and such that $i+j=n$ (if $i+j<n$, obviously joint injectivity cannot be achieved).
\end{proof}

\begin{PAR}\label{PARSHUFFLE}
For surjections $\sigma: [n] \twoheadrightarrow [i]$ and $\tau: [n] \twoheadrightarrow [j]$ that are jointly injective, and such that $i+j=n$, i.e.\@ in the case of $i,j$-shuffles, 
 $\sigma$ and $\tau$ determine each other, because we can index degeneracies by non-empty subsets $I, J \subset \{1, \dots, n\}$ (the intervals contracted by $\tau$, and $\sigma$, respectively) and under the two conditions, $\sigma$ and $\tau$ must correspond to complementary sets. 
In other words
\[ \sigma = s_{\sigma_j}  \cdots s_{\sigma_1} \qquad \tau = s_{\tau_i} \cdots s_{\tau_1} \]
for $J = \{\sigma_1+1, \dots, \sigma_j+1\}$, and $I = \{\tau_1+1, \dots, \tau_i+1\}$, respectively.  
 
  In particular, we get for $i, j$ with $i+j=n$ surjections
\[ p_{i,j}: \delta_! D_{n} \to D_{i} \boxtimes D_{j}  \]
induced by $s^l: [m] \twoheadrightarrow [i], s^r: [m] \twoheadrightarrow [j]$ the extremal degeneracies corresponding to the subsets $\{i+1, \dots, n\}$ and $\{1, \dots, i\}$, respectively.
\end{PAR}

\begin{LEMMA}
There are splittings
\[ s_{i,j}: D_{i} \boxtimes D_{j}  \to \delta_! D_{n} \]
(of the maps $p_{i,j}$) induced by $\delta_{l,i} \times \delta_{r,j}: [{i}] \times [j] \to [{n}] \times [{n}]$, the extremal faces. 
Every other map
\[ \delta_! D_{n}  \to D_{i'} \boxtimes D_{j'}  \]
given by pairs $\sigma: [{m}] \twoheadrightarrow [{i'}]$ and $\tau: [{m}] \twoheadrightarrow [{j'}]$ that are jointly injective and such that $l+k=m$, composes to zero with $s_{i,j}$.
\end{LEMMA}

\begin{proof}$\delta_{l,i} \times \delta_{r,j}$ composed with a diagonal degeneracy factors through a degeneracy of either $[i]$ or $[j]$, hence there is an induced map as indicated. 
The composition with any pair of surjective maps $[n] \times [n] \to [k] \times [l]$, jointly injective, corresponding to subsets $S$ and $T$, is a degeneracy unless $S \subseteq \{1, \dots, i\}$ and
$T \subseteq \{i+1, \dots, n\}$. If $\# S + \# T = n$ we must have equality.  
\end{proof}

\begin{PROP}[Alexander-Whitney and Eilenberg-Zilber explicit]\label{PROPAWEZAB}
For a double complex $X$, we have commutative diagrams (where $\Awfrak$ and $\Ezfrak$ were defined in Definition~\ref{DEFAWEZ}\footnote{And the operator $C$ used to define $\Ezfrak$ is understood to be the canonical one given by Lemma~\ref{LEMMASYM} for Abelian categories.}):
\begin{enumerate}
\item
\begin{gather*}
 \xymatrix{ 
 (\delta^* X)_{[n]} = X_{[n],[n]}  \ar[d] \ar[rr]^-{\sum_{i+j=n} (\delta_{i,l}, \delta_{j,r}) }  &  & \ar[d]  \bigoplus_{i+j=n} X_{[i],[j]}    \\
 (\delta^* X)_n  \ar[rr]^-{\Awfrak} & & (\dec_* X)_n = \bigoplus_{i+j=n} X_{i,j}  
 }  
  \end{gather*}
  On the smallest filtration step and on the factor $X_{i,j}$ corresponding to a shuffle $\sigma, \tau$ the map $\Awfrak$ is zero unless $\sigma = s^l$ and $\tau = s^r$. 
\item
\begin{gather*}
  \xymatrix{ 
 \bigoplus_{i+j=n} X_{[i],[j]} \ar[d] \ar[rr]^-{\sum_{\sigma, \tau} \mathrm{sgn}(\sigma, \tau)(\sigma, \tau) }  &  & \ar[d] (\delta^* X)_{[n]} = X_{[n],[n]}    \\
 (\dec_* X)_n = \bigoplus_{i+j=n} X_{i,j}   \ar[rr]^-{\Ezfrak} & & (\delta^* X)_n
 } 
 \end{gather*}
 where the sum runs over all shuffles $\sigma: [n] \twoheadrightarrow [i]$, $\tau: [n] \twoheadrightarrow [j]$ (i.e.\@ jointly injective with $i+j=n$).
 The morphism $\Ezfrak$ has image in the smallest filtration step and on the factor $X_{i,j}$ corresponding to the shuffle $\sigma, \tau$, the map is thus given by $\mathrm{sgn}(\sigma, \tau)$.
\end{enumerate}
\end{PROP}
\begin{proof}
1.\@ It suffices to see that the following commutes\footnote{where $s_{\can}$ is on each summand the adjoint of the restriction of the homonymous map (canonical degeneracy) $s_{\can}: \dec_! \Delta_i^{\circ} \boxtimes \Delta_j^{\circ} = \Delta_{i+j+1}^{\circ} \to \Delta_n^{\circ}$.}:
\[ \xymatrix{ 
\delta_!(D_n)  \ar@{<-}[rr]^-{\sum_i (\delta_{l,i}, \delta_{r,j})} \ar@{<-}[rrd]_{\Awfrak} & &  \bigoplus_{i+j=n} D_i \boxtimes D_j \ar[d]^{s_{\can}}_{\sim} \\
& &  \dec^*(D_n) }\]

First, observe that the morphisms $\delta_{l,i}$ and $\delta_{r,j}$ really induce a morphism
\[   D_i \boxtimes D_j \to \delta_!(D_n) \]
because for every diagonal degeneracy $[n] \times [n] \twoheadrightarrow [n'] \times [n']$ 
 there  are  degeneracies $[i] \twoheadrightarrow [i']$ and $[j] \twoheadrightarrow [j']$  such that we have a factorization
\[ \xymatrix{ [i] \times [j] \ar[rr]^{(\delta_{i,l}, \delta_{j,r})} \ar@{->>}[d] & &   [n] \times [n] \ar@{->>}[d] \\  
[i'] \times [j'] \ar[rr] & &  [n'] \times [n'] } \]

We are left to show that for all $i$ the outer shape commutes in the diagram
\[ \xymatrix{ 
\delta_!\Delta_n^{\circ}  \ar@{<-}@/^20pt/[rr]^-{(\delta_{l,i}, \delta_{r,j})} \ar@{<-}[rd]_-{\mathrm{counit}} \ar@{<-}[r]^-x & \delta_!\dec_! (\Delta_i^{\circ} \boxtimes \Delta_j^{\circ}) \ar[d]^{s_{\can}} \ar@{<-}[r]^-{u_!^{\op}} &   \Delta_{i}^{\circ} \boxtimes \Delta_{j}^{\circ}  \ar[d]^{s_{\can}} \\
& \delta_! \dec_! \dec^*(\Delta_n^{\circ}) \ar@{<-}[r]_-{u_!^{\op}} &  \dec^*(\Delta_n^{\circ}) }\]
commutes, where the map $x$ is $\delta_!$ of the canonical degeneracy, seen as a morphism $\dec_! \Delta_i^{\circ} \boxtimes \Delta_j^{\circ} = \Delta_{i+j+1}^{\circ} \to \Delta_n^{\circ}$. 
Here every other shape clearly commutes. For the upper triangle notice that this is the image under $\Z[-]$ of the following diagram of bisimplicial set maps
\[ \xymatrix{ 
\Delta_n \boxtimes \Delta_n  \ar@{<-}@/^20pt/[rrrr]^-{\delta_{l,i} \boxtimes \delta_{r,j}} \ar@{<-}[rr]_-{s_{\can} \boxtimes s_{\can}}  & & \Delta_{i+j+1} \boxtimes \Delta_{i+j+1}  \ar@{<-}[rr]_-{u^{\op}=\delta_{l,i} \boxtimes \delta_{r,j}} & &   \Delta_{i} \boxtimes \Delta_{j}. }\]

2.\@ It suffices to see that the following commutes:
\[ \xymatrix{ 
\delta_!(D_n)  \ar[rrr]^-{\sum_{\sigma, \tau} \mathrm{sgn}(\sigma, \tau)\cdot (\sigma, \tau)} \ar[rrrd]_{\Ezfrak} &  & &  \bigoplus_{i+j=n} D_i \boxtimes D_j \ar[d]^{s_{\can}}_{\sim} \\
& & &  \dec^*(D_n) }\]

First, observe that the morphism $\sigma, \tau: \Delta_n^{\circ} \boxtimes \Delta_n^{\circ} \to  \Delta_i^{\circ} \boxtimes \Delta_j^{\circ}$ ($i+j=n$) really induces a map
\[ \delta_!(D_n) \to D_i \boxtimes D_j \]
because for every degeneracy $[i] \times [j] \twoheadrightarrow [{i'}] \times [{j'}]$ (for $i+j=n$, and where either $i'<i$ {\em or} $j'<j$) there is a  degeneracy $[n] \twoheadrightarrow [{n'}]$ such that
we have a factorization
\[ \xymatrix{ [n] \times [n] \ar[r] \ar@{->>}[d] &  [i] \times [j] \ar@{->>}[d] \\  
[{n'}] \times [{n'}] \ar[r] & [{i'}] \times [{j'}] } \]

First observe that
\[ \delta_! (D_n)  \subset F^{n-1} \Delta_n^{\circ} \boxtimes \Delta_n^{\circ} \] 
because every pair of degeneracies $[n] \twoheadrightarrow [i]$ and $[n] \twoheadrightarrow [j]$ with $i+j\le n-1$ must factor through a diagonal degeneracy. 

Since the filtration by codegeneracy is the same as the canonical complex filtration (cf.\@ \ref{PARSKELETAL1}), it suffices to show therefore that 
\[ \xymatrix{ 
\Delta_n^{\circ} \boxtimes \Delta_n^{\circ}  \ar[rr]^-{\sum_{\sigma, \tau} \mathrm{sgn}(\sigma, \tau) \cdot (\sigma, \tau)} \ar[rrd]_{\Ezfrak} & &  \bigoplus_{i+j=n} \Delta_i^{\circ} \boxtimes \Delta_j^{\circ} \ar[d]^{s_{\can}} \\
& &  \dec^*(\Delta_n^{\circ}) }\]
commutes in double complex degree $\ge n$\footnote{ Contrary to the situation with the Alexander-Whitney map it is not commutative in general!}. 
Since $(\Delta_n^{\circ})_{i+j+1} = 0$ for $i+j=n$ we actually have
\begin{equation} \label{eqdr} (\dec^* \Delta_n^{\circ})_{i,j} \cong (\Delta_n^{\circ})_{i+j} \end{equation}
for $i+j=n$ and 0 for $i+j>n$, and thus we are left to show commutativity in complex degree $i,j$ with $i+j=n$. 

The map (\ref{eqdr}) may be described as $c \circ \dd_r$ (see the proof of Proposition~\ref{PROPEXPLICITAB}, 2.) where $c$ is the morphism fitting into
\[ \xymatrix{ (\dec^* \Delta_n^{\circ})_{[i],[j-1]} \ar@{=}[d] \ar@{->>}[r] &  (\dec^* \Delta_n^{\circ})_{i,j-1}  \ar[d]^c  \\
(\Delta_n^{\circ})_{[i+j]} \ar@{->>}[r]  & (\Delta_n^{\circ})_{i+j}  } \]
Thus we have to see that the composition
\begin{equation}\label{eqcomp1} c \circ \mathrm{EZ}  \circ \dd_r  : (\delta_! \Delta_n^{\circ})_{i,j} \to (\delta_! \Delta_n^{\circ})_{i,j-1} \to (\dec^* \Delta_n^{\circ})_{i,j-1} \to (\Delta_n^{\circ})_{i+j}.   \end{equation}
is equal to the composition
\begin{gather}\label{eqcomp2} c \circ s_{\can} \circ \dd_r \circ (\sum_{\sigma, \tau} \mathrm{sgn}(\sigma, \tau)\cdot (\sigma, \tau)) : \\
\nonumber (\delta_! \Delta_n^{\circ})_{i,j} \to  (\Delta_i^{\circ})_i \times (\Delta_j^{\circ})_j \to (\Delta_i^{\circ})_i \times (\Delta_j^{\circ})_{j-1}  \to (\dec^* \Delta_n^{\circ})_{i,j-1} \to (\Delta_n^{\circ})_{i+j}.   \end{gather}

The composition of the last two morphisms in (\ref{eqcomp1}) is by definition (cf.\@ Lemma~\ref{LEMMAFINSET}):
\begin{eqnarray*}
 \Z\left[ \substack{ \text{subsets of $\{0, \dots, n\}$} \\ \text{of cardinality $i+1$ and $j$} } \right]  & \to & \Z\left[\substack{ \text{subsets of $\{0, \dots, n\}$} \\ \text{ of cardinality $i+j+1$} } \right]    \\
{ [S] \otimes [T]}  & \mapsto &  \begin{cases}  \pm [S \cup T] & \text{if $S \cap T = \emptyset$}  \\ 0 & \text{otherwise} \end{cases}   
 \end{eqnarray*}
where the sign is determined by the parity of the permutation bringing $S \ast T$ into the order of $S \cup T$.
Given subsets  $I, J \subset \{0, \dots, n\}$ of cardinality $i+1$ and $j+1$ with $I \cap J = \{x\}$,
 $c \circ \mathrm{EZ}  \circ \dd_r $ maps $[I] \otimes [J]$ thus to $[\{0,\dots,n\}]$ with sign $(-1)^{k-1} \mathrm{sgn}(\kappa)$,
 where $\kappa$ is the permutation that brings $I \coprod (J \setminus \{x\})$ into the correct order, and $x$ is the $k$-th element of $J$.
 It maps subsets with larger intersection to zero.

The composition of the last three morphisms in  (\ref{eqcomp2}) maps the generator $[\{0, \dots, i\}] \otimes [\{0, \dots, j\}]$ to the generator $[\{0, \dots, n\}]$. By Lemma~\ref{LEMMACOMB} below, the first morphism maps
a pair $I$, $J$ either to zero, or, if $I \cap J = \{x\}$ , there is exactly one shuffle that maps it to $[\{0, \dots, i\}] \otimes [\{0, \dots, j\}]$ with the same sign $(-1)^{k-1} \mathrm{sgn}(\kappa)$ and all others map it to zero. 
\end{proof}

\begin{LEMMA}\label{LEMMACOMB}
Let $I, J \subset \{0, \dots, n\}$ be two subsets of Cardinality $i+1$ and $j+1$ with  $i+j = n$. 
\begin{enumerate}
\item If $I \cap J = \{x\}$ there
is exactly one $i,j$-shuffle $\sigma, \tau$ such that $\sigma([I]) = [\{0, \dots, i\}]$ and $\tau([J]) = [\{0, \dots, j\}]$. (All other shuffles satisfy either $\sigma([I]) = 0$ or $\tau([J])=0$.)
Then 
\[ \mathrm{sgn}(\sigma, \tau) = (-1)^{k-1} \mathrm{sgn}(\kappa), \]
 where $\kappa$ is the permutation that brings $I \coprod (J \setminus \{x\})$ into the correct order, and $x$ is the $k$-th element of $J$. 

\item If $\{x, y\} \subset I \cap J$ then 
\[ \sum_{\substack{\sigma,\tau \\ \sigma([I]) = [\{0, \dots, i\}] \\ \tau([J]) = [\{0, \dots, j\}]}} \mathrm{sgn}(\sigma, \tau) = 0.  \] 
\end{enumerate}
\end{LEMMA}
\begin{proof}
(i) In this case, the subsets are covering.
If $x\not=0$, we have $x-1 \in I$ and $x-1 \not \in J$ or vice versa. Thus $\sigma$ must not contract $x-1< x$ and hence $\tau$ must do so (or vice versa). By induction to the left, $\sigma$ and $\tau$ are determined for $i \le x$. 
If $x \not= n$, we have $x+1 \in I$ and $x+1 \not \in J$ or vice versa. Thus $\sigma$ must not contract $x<x+1$ and hence $\tau$ must do so (or vice versa). By induction to the right, $\sigma$ and $\tau$ are determined for $i \ge x$. 
The equality of the sign expressions follows from the definition of the sign of a shuffle\footnote{The shuffle corresponds to complementary subsets $I'$ (contracting intervals for $\tau$) and $J'$ (contracting intervals for $\sigma$) of $\{1, \dots, n\}$ and the sign is determined by the permutation $\kappa'$ that 
brings $I' \cup J'$ into the correct order. The procedure just described determines order-preserving isomorphisms $I' \cong I \setminus \{x\}$ and $J' \cong J \setminus \{x\}$ in such a way that the transposed permutation brings $I \setminus \{x\} \cup J \setminus \{x\}$ into the correct order. If $x$ was the $l$-th element of $I$ and $k$-th element of $J$, it has to move from position $l$ to $l+k-1$, hence an additional sign of $(-1)^{k-1}$.}. 

(ii) In this case, there is an element $z \in \{0, \dots, n\} \setminus  (I  \cup J)$.
The shuffles correspond to complementary subsets $I', J' \subset \{1, \dots, n\}$ of cardinality $i$ and $j$ (the contracting intervals of $\tau$ and $\sigma$ respectively). 
The shuffles satisfying the property from the sum induce jointly injective surjections $\sigma' = \sigma \delta_z: [{n-1}] \twoheadrightarrow [{i}]$ and $\tau' = \tau \delta_z: [{n-1}] \twoheadrightarrow [{j}]$. As such they correspond to subsets $I'', J'' \subset \{1, \dots, n-1\}$ of cardinality $i-1$ and $j-1$ with $I'' \cap J'' = \emptyset$. 

The interval $q$ of $\Delta_{n-1}$ which is neither in $I''$ nor $J''$ thus must correspond to a new interval that is not mapped to an interval by $\delta_z$ (in particular, $z \not= 0$ and $z \not= n$) because all others {\em are} contracted by one of the compositions. To each of these pairs $I''$, $J''$ thus correspond precisely two pairs $I'$ and $J'$ where the two intervals that $\delta_z(q)$ is composed of, are distributed evenly into $I'$ and $J'$. 
The corresponding shuffles have opposite sign. 
\end{proof}

\begin{PAR} 
Recall the {\em Shih-operator} $\Xi_1$ from \ref{DEFSHIH} (cf.\@ also Corollary~\ref{KOREZAB}). Applying $L_{\mathcal{C}}$, it induces 
\[ L_{\mathcal{C}}(\Xi):   \xymatrix{ \dec^*  \delta^* \ar@{=>}[r]^-{\sim} &  (\delta, \delta)^* \dec_{13,24}^*  \ar@{=>}[rr]^{(\id, \Ezfrak\, \Awfrak)} & &   (\delta, \delta)^* \dec_{13,24}^* \ar@{=>}[r]^-{\sim} &   \dec^*  \delta^*     }  \]
Its adjoint $L(\Xi_1)': \Delta_1 \otimes \delta^* \cong \dec_! \dec^* \delta^* \Rightarrow \delta^*$ (cf.\@ Corollary~\ref{KORHOMOTOPY}) yields via composition with $\Ezfrak: \Delta_1 \tildeotimes \delta^* \to \Delta_1 \otimes \delta^*$ a morphism
\[ \Ezfrak^* L(\Xi_1)' :  \Delta_1 \tildeotimes \delta^* \Rightarrow  \delta^* \]
which gives a homotopy in the usual sense between  $\Ezfrak \, \Awfrak$ and $\id$. (Notice the two very different instances of $\Ezfrak$ here!)
\end{PAR}

\begin{PAR} 
Define $\mathcal{Q}^n \in \Z[\Hom([{n}], [{n}])] \otimes \Z[\Hom([{n}], [{n}])]$
by the formula
\[ \boxed{ \mathcal{Q}^n := \sum_{\substack{ p+q=n \\ \sigma, \tau }} \mathrm{sgn}(\sigma, \tau) (\sigma \delta^{n}_{0,p} \otimes \tau \delta^n_{p,n}) }  \]
where $\sigma, \tau$ run through the $p,q$-shuffles. 
By the formula in Proposition~\ref{PROPAWEZAB} we have that $(\Ezfrak \, \Awfrak)_n \equiv \mathcal{Q}^n$ modulo degeneracies, i.e.\@ for $X \in \mathcal{C}^{\Delta^{\op} \times \Delta^{\op}}$
\[ \xymatrix{  (\delta^* X)_{[n]} = X_{[n],[n]} \ar@{->>}[d]   \ar[r]^-{\mathcal{Q}^n} &  (\delta^* X)_{[n]} = X_{[n],[n]} \ar@{->>}[d] \\
(\delta^* X)_{n} \ar[r]^-{  \Ezfrak \, \Awfrak} & (\delta^* X)_n
 } \]
 commutes. 
Define $H^n_b:= (\id_{[b]} \ast \Ezfrak \, \Awfrak) s_{\can}$. We have
\[ H_b^n s_i  = \begin{cases} s_i H_{b-1}^{n-1}  & i < b  \\ s_{i+1} H_b^{n-1}  &  i \ge b \end{cases}  \]
hence $H_b^n$ maps degenerate elements to degenerate elements.

Defining 
$\mathcal{H}^n_i \in \Z[\Hom([{n}], [{n-1}])] \otimes \Z[\Hom([{n}], [{n-1}])]$
by the formula
\[ \boxed{ \mathcal{H}^n_b:=(\id_{[b]} \ast \mathcal{Q}^{n-b-1}) s_{\can} = \sum_{\substack{p+q=n-b-1 \\ \sigma, \tau}} (\id_{[b]} \ast \, \sigma \delta_{0,p}^{n-b-1}) s_{\can} \otimes (\id_{[b]} \ast \,\tau \delta_{p,n-b-1}^{n-b-1}) s_{\can} }   \]
we have the commutative diagram for $X \in \mathcal{C}^{\Delta^{\op} \times \Delta^{\op}}$:
\[ \xymatrix{  (\delta^* X)_{[n-1]} = X_{[n-1],[n-1]} \ar@{->>}[d]   \ar[r]^-{\mathcal{H}^n_b} &  (\delta^* X)_{[n]} = X_{[n],[n]} \ar@{->>}[d] \\
(\delta^* X)_{n-1} \ar[r]^-{  H^n_b } & (\delta^* X)_n
 } \]
Note, however, that the lower map does not constitute a morphism of complexes!
\end{PAR}

\begin{PROP}[Shih's formula]\label{PROPSHIHEXPL}The following commutes for $X \in \mathcal{C}^{\Delta^{\op} \times \Delta^{\op}}$:
\[ \xymatrix{  (\delta^* X)_{[n-1]} = X_{[n-1],[n-1]} \ar@{->>}[d]   \ar[rr]^-{\sum_i (-1)^b \mathcal{H}^n_b } & & (\delta^* X)_{[n]} = X_{[n],[n]} \ar@{->>}[d] \\
\{0,1\} \tildeotimes (\delta^* X)_{n-1} \ar[rr]^-{  \Ezfrak^* L_{\mathcal{C}}(\Xi_1)'} & & (\delta^* X)_n
 } \]
\end{PROP}

\begin{PAR}\label{PARDECEXPLICIT2}Before discussing the proof, recall the explicit description of homotopies obtained via $\dec^*$ from \ref{PARDECEXPLICIT}.
A homotopy in the classical sense is extracted by the element
\[ \{0,1\} \tildeotimes X \to \Delta_1^{\circ} \tildeotimes X \to \Delta_1^{\circ} \otimes X   \]
where the second morphism is the EZ morphism. In degree $n$ the corresponding map $\{0,1\} \tildeotimes X \to Y$  is given, with the notation of \ref{PARDECEXPLICIT}, by (sum over $(1, n-1)$-shuffles)
\[ \{0, 1\} \otimes x_{n-1} \mapsto \alpha(x_{n-1}) := \sum_{i=0}^{n-1} (-1)^i G_{[n],i} (s_i x_{n-1}) =  \sum_{i=0}^{n-1} (-1)^i F_{[i], [n-i-1]} (s_i x_{n-1})  \]
where $x_{n-1} \in X_{[n-1]}$ denotes a representing element.  
\end{PAR}

\begin{proof}[Proof of Proposition~\ref{PROPSHIHEXPL}]
By Proposition~\ref{PROPAWEZAB} and the Definition~\ref{DEFSHIH}, we have for $r+s+1=n$
\begin{eqnarray*}
 L_{\mathcal{C}}(\Xi)_{[r],[s]}: (\dec^* \delta^* X)_{[r],[s]} = X_{[r]\ast[s],[r]\ast[s]} &\to& X_{[r]\ast[s],[r]\ast[s]} = (\dec^* \delta^* X)  \\
 x_{n} &\mapsto& \sum_{\substack{p+q = s \\ \sigma, \tau}} \mathrm{sgn}(\sigma, \tau) ( \sigma d_l  , \tau d_r  ) x_{n} 
\end{eqnarray*}
where $\sigma, \tau, \delta_l, \delta_r$ are intended w.r.t.\@ the $[s]$-variable and the $\sigma, \tau$ run through all $p,q$-shuffles. 

In view of the discussion in \ref{PARDECEXPLICIT2} this translates to\footnote{The formula for $L_{\mathcal{C}}(\Xi_1)$ is modulo degenerate elements. However, $s_b = s_{\can}$ translates simplicial degeneracies into degeneracies as element (in degree $(b,n-b-1)$) in the bisimplicial object $\dec^*X$.} 
\begin{gather*} 
\{0, 1\} \otimes x_{n-1} \mapsto   \sum_{b=0}^{n-1} (-1)^b L_{\mathcal{C}}(\Xi_1)_{[b], [n-b-1]} (s_b x_{n-1}) \\
= \sum_{b=0}^{n-1} (-1)^b \sum_{\substack{p+q = n-b-1 \\ \sigma, \tau}} \mathrm{sgn}(\sigma, \tau)  ((\id_{[b]} \ast \, \sigma \delta_{0,p}^{n-b-1}) s_{b}, (\id_{[b]} \ast \,\tau \delta_{p,n-b-1}^{n-b-1}) s_{b}) x_{n-1} 
 \end{gather*} 
\end{proof}

\begin{PAR}The formula for $\mathcal{Q}^n$ can be given in a different form: 
The summands of $\mathcal{Q}^n$ run over  $p,q$-shuffles with $p+q=n$.
Those are determined by subsets $Z \subset \{1, \dots, n\}$ giving the intervals where $\sigma$ is degenerate and thus $\tau$ is not. 
Equivalently, they are given by sequences $\underline{y} = (z_0, \dots, z_{n-1}) \in \{0,1\}^n$ and we have
\[ \sigma = s_0^{z_0} \ast' \cdots \ast' s_{0}^{z_{n-1}}  \]
where $s_0^0 = \id_{[1]}$ and $s_0^1: [1] \twoheadrightarrow [0]$ is the usual degeneracy.  
From this we get
\[ \sgn(\sigma) = (-1)^{\sum_{i<j} z_j (z_{i}+1) } \]
which may be derived from the formula
\begin{equation}\label{eqsignshuffle}
 \sgn(\sigma_0 \ast' \cdots \ast' \sigma_{n-1}) = \prod_i \sgn(\sigma_i) \prod_{i<j} (-1)^{|\sigma_j| \cdot |\overline{\sigma_i}|}  
 \end{equation}
where $\sigma_i: [n_i] \to [|\sigma_i|]$ are degeneracies and $\overline{\sigma}$ is the complementary shuffle. 
We also have for $\underline{b} = (b_0, \dots, b_{k-1}) \subset \{0, \dots, n-1\}$, with $b_0 > b_1 > \cdots > b_{k-1}$ the formula 
\begin{equation}\label{eqsignshuffle2}
   \sgn(\sigma_{\underline{b}}) = (-1)^{ \sum_{i=0}^{k-1} b_i-(k-1-i)} \qquad \sgn(\tau_{\underline{b}}) = (-1)^{ \sum_{i=0}^{k-1} n-k-i-b_i} 
\end{equation}
where $\sigma_{\underline{b}}, \tau_{\underline{b}}$ is the associated shuffle (\ref{PARSHUFFLEB}). 
\end{PAR}

\begin{LEMMA}\label{LEMMAH}
Let $n$ be a positive integer. Consider the set of $\underline{y} = (y_0, \dots, y_{n-1})$ with $y_i \in \{0, 1\}$.
We have
\begin{equation}\label{eqalt} \mathcal{Q}^{n} = \sum_{\underline{y}} (-1)^{\sum_{i<j} y_j (y_{i}+1) } (B_{\underline{y}}^0 \otimes B_{\underline{y}}^1) \end{equation}
with $B_{\underline{y}}^0 : [n+1] \to [m]$ for some $m$, and $B_{\underline{y}}^1 : [n+1] \to [n+1]$  recursively defined by:
$B_{()}^0 = B_{()}^1 = \id_{[0]}$ and
\begin{equation}\label{eqb}
 B_{\underline{y}}^j = \begin{cases} s_{n-1}B_{\underline{y}'}^j & j < y_0  \\
 B_{\underline{y}'}^j \ast \id_{[0]} & j = y_0  \\
s_{n-1} B_{\underline{y}'}^j  \delta_{0} & j > y_0 \end{cases}  
\end{equation}
Each summand is degenerate on the left at $i$, if $y_i = 0$ and degenerate on the right at $i$, if $y_i = 1$. 
Implicitly in the formula (\ref{eqalt}) the $B_{\underline{y}}^0$ are composed with the inclusion $[m] \hookrightarrow [n+1]$. 
\end{LEMMA}
\begin{proof}
An instructive exercise.
\end{proof}

\begin{PAR}
We would like an explicit formula also for the higher Shih operators \ref{HIGHERSHIH}, however, defining $\mathcal{H}^n_{\underline{b}, \underline{i}}$ in the same way as in (\ref{highershih}) would not be correct because the equality of
$H^n_{b}$ with $\mathcal{H}^n_b$ holds only up to degenerate elements. However we can define (for any $m$):
\[ \mathcal{H}^n_b: \Z[\Hom([n-1], [m])] \otimes \Z[\Hom([{n-1}], [m])]  \to \Z[\Hom([n], [m])] \otimes \Z[\Hom([n], [m])]   \]
by writing  generators as
$\tau \otimes \kappa = s_{i_1} \cdots s_{i_k}  \cdot (\tau' \otimes \kappa')$ with $(\tau' \otimes \kappa')$ non-degenerate and setting: 
\[ \mathcal{H}_b^n ( \tau \otimes \kappa ) :=  s_{i_1'} \cdots s_{i_k'}  \cdot \mathcal{H}_{b'}^{n-k'}  \cdot (\tau' \otimes \kappa')     \]
where the $i$ and $b$ have been transformed by the rule: 
\[ \mathcal{H}_b^n s_i  := \begin{cases} s_i \mathcal{H}_{b-1}^{n-1}  & i < b,  \\ s_{i+1} \mathcal{H}_b^{n-1}  &  i \ge b. \end{cases} \]
Defining then $\mathcal{H}^n_{\underline{b}, \underline{i}}$ by the same recursive formula as (\ref{highershih}) 
\begin{align*}  \mathcal{H}^n_{\underline{b}, \underline{i}}: & (\delta_k^* X)_{[n-k]} \to (\delta_k^* X)_{[n]}  \\
 := & ( \underbrace{s_{b_0}, \dots, s_{b_0}}_{i_{0} \text{ times}}, \mathcal{H}_{b_0}^n, \underbrace{s_{b_0}, \dots, s_{b_0}}_{k-i_{0}-1 \text{ times}} ) (- \circ \delta_{i_0}^*)  \mathcal{H}^{n-1}_{\underline{b}', \underline{i}'}. 
\end{align*}
with $\underline{b}' = (b_1, \dots, b_{k-1})$ and $\underline{i}' = (i_1, \dots, i_{k-1})$, 
we arrive at: 
\end{PAR}

\begin{PROP}\label{PROPHIGHERSHIHEXPL}
The following diagram is commutative
\[ \xymatrix{  (\delta^*_k X)_{[n-k]} = X_{[n-k],[n-k]} \ar@{->>}[d]   \ar[r]^-{\mathcal{H}^n_{\underline{b}, \underline{i}}} &  (\delta_k^* X)_{[n]} = X_{[n],[n]} \ar@{->>}[d] \\
(\delta^*_k X)_{n-k} \ar[r]^-{  H^n_{\underline{b}, \underline{i}} } & (\delta_k^* X)_n
 } \]
 Notice that, again, the lower map is not a morphism of complexes!
\end{PROP}

\begin{PAR}
There is a relation between the higher Shih operators $H_{\underline{b}, \underline{i}}^n$ and the Szczarba operators defined in \cite{Szc61} (cf.\@ also \cite{HT10}) as we will now explain.
The conceptual relation will be clarified by the discussion in Section~\ref{SECTFUNCT}. Cf.\@ also \cite{Fra21} for a different discussion of this relation. 
\end{PAR}

\begin{PAR}\label{PARP}
We say that an element in $\Hom([{n}], [m])^{k+1}$, considered as generating summand of $\Z[\Hom([n], [m])]^{\otimes (k+1)}$
is {\bf $i$-degenerate at $j$}, if it is of the form
\[ (\tau_0, \dots, s_i \tau_j, \dots, \tau_k). \]
Define an idempotent
\[ \mathcal{P}_i: \Z[\Hom([{n}], [m])] \to \Z[\Hom([{n}], [m])] \]
defined on a basis by
\[ \tau \mapsto \begin{cases} 0 & \text{ if $\tau$ is $i$-degenerate, } \\ \tau & \text{otherwise.} \end{cases}  \]
Define also 
\[ \mathcal{P}: \Z[\Hom([{n}], [m])]^{\otimes (k+1)} \to \Z[\Hom([n], [m])]^{\otimes (k+1)} \]
by
\[ \mathcal{P} = \prod_{i=0}^{k} 1^{\otimes i} \otimes \mathcal{P}_{n-k-1+i} \otimes 1^{\otimes k-i}  \]
i.e.\@ it throws away all summands that are $n-k-1+i$ degenerate at $i$ for $i=0, \dots, k$.
\end{PAR}

\begin{DEF}\label{DEFSZCZARBA}
Define an active morphism $\Sz^j_{\underline{i}}: [k] \to [j]$  (i.e.\@ $\Sz^j_{\underline{i}}(0) = 0$ and 
$\Sz^j_{\underline{i}}(k) = j$) by $\Sz^0_{()} := \id_{[0]}$, and the recursion
\begin{equation}\label{eqszczarba} \Sz^j_{\underline{i}} := \begin{cases} 
(\Sz^{j}_{\underline{i}'}  \ast \id_{[0]})  s_j  & j <  i_{0}+1\\ 
(\Sz^{j-1}_{\underline{i}'}  \ast \id_{[0]}) & j = i_{0}+1  \\
(\Sz^{j-1}_{\underline{i}'}  \ast \id_{[0]}) s_{j-1}  \delta_{i_{0}+1}   & j >  i_{0}+1 
\end{cases}
\end{equation}
in which $\underline{i}' = (i_1, \dots, i_{k-1})$. 
\end{DEF}
One can also write the recursion in a slightly different, equivalent form: 
\begin{equation}\label{eqszczarba2} \Sz^j_{\underline{i}} := \begin{cases} 
s_{k-1} \Sz^{j}_{\underline{i}'}  & j <  i_{0}+1\\ 
\Sz^{j-1}_{\underline{i}'}  \ast \id_{[0]} & j = i_{0}+1  \\
s_{k-1} \Sz^{j-1}_{\underline{i}'}   \delta_{i_{0}+1}   & j >  i_{0}+1 
\end{cases}
\end{equation}

\begin{BEM}\label{BEMINDEX}
This definition is similar to the one given by Szczarba in \cite{Szc61} used by Hess and Tonks in \cite[Definition 5]{HT10}. In fact
\[  \Sz^{j}_{\underline{i}}(x)  = {}^t\! D^{k+1}_{k-j,\underline{i}}   \]
where ${}^t x$ for a map $x: [k] \to  [j]$ is the conjugation with the reversal of simplices and $D$ is the operator defined in [loc.\@ cit.]. 
\end{BEM}

\begin{PAR}\label{PARREDUCT}
A sequence
\[ \xymatrix{  [1] \ar@{=}[r]  &  [1] \ar[r]^-{\delta_{i_{k-1}+1}} & \cdots   \ar[r]  & [k]  \ar[r]^-{\delta_{i_{0}+1}}  & [k+1] \ar@{=}[r] & [k+1]     } \]
of active faces, or written as vector $\underline{i} = (i_0, \dots, i_{k-1})$
can be interpreted as a connected leveled binary tree with $k$ vertices of degree 3: 
For example 
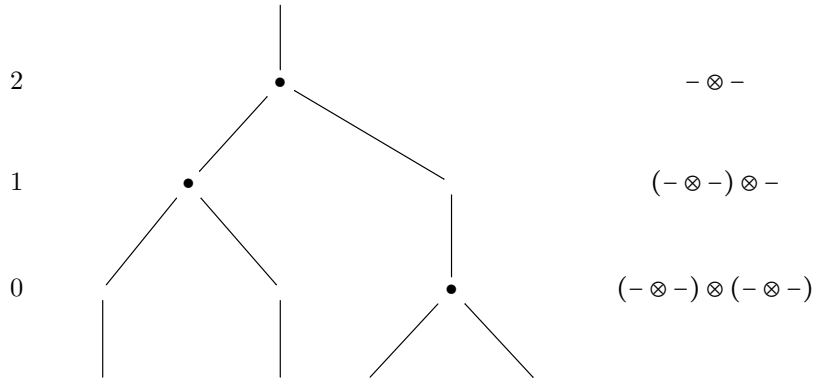
\[ \xymatrix{ 
& & &  \ar@{-}[d] & \\
2 & &  &   \bullet \ar@{-}[ld] \ar@{-}[rrd] & & & &  - \otimes -\\
1 & &  \bullet \ar@{-}[ld] \ar@{-}[rd]  & & & \ar@{-}[d]  & &  (- \otimes -) \otimes - \\
0 &  \ar@{-}[d]  &  &  \ar@{-}[d]  & &  \ar@{-}[ld]  \bullet \ar@{-}[rd] & & (- \otimes -) \otimes (- \otimes -)   \\
& & & & &  & 
} \]
corresponds to the vector $(2, 0, 0)$. 
For a vector $\underline{b}=(b_0, \dots, b_{k-1})$ with  $\widetilde{k} > b_0 > b_1 > \cdots > b_{k-1} \ge 0$, we say that $\underline{i}=(i_0, \dots, i_{k-1})$ is a {\bf $\underline{b}$-reduction of $\widetilde{\underline{i}}$}
if for $i \notin \{b_0,  \dots, b_{k-1}\}$ the vertex in row $i$ has no left children and deleting those vertices (with their corresponding left leaf and their row), we obtain the tree associated with $\underline{i}$.
The remaining leaves form a subsequence $c_0 < c_1 < \dots < c_{k}$ of $\{0, \dots, \widetilde{k}\}$.
For example the displayed tree has the $(2)$-reduction $(0)$, $(1,2)$-reduction $(0,0)$, $(0,2)$-reduction $(1,0)$, and $(0,1,2)$-reduction $(2,0,0)$, but does not have a $(0,1)$-reduction. 

The $\widetilde{\underline{i}}$ such that $\underline{i}=(i_0, \dots, i_{k-1})$ is a $\underline{b}$-reduction are in bijection with vectors $(y_0, \dots, y_{\widetilde{k}-k-1})$
with $0 \le y_0, \dots, y_{\widetilde{k}-2-b_0} \le k$, \dots, $0 \le y_{\widetilde{k}-k-b_{k-1}}, \dots, y_{\widetilde{k}-k-1} \le 0$

Example: 
\[ \begin{array}{r|l|l}
\widetilde{k}-1 & 0 \le y_0 \le k &  \\
\widetilde{k}-2 & 0 \le y_1 \le k &  \\
\widetilde{k}-3 & b_0 & 0 \le i_0 \le k-1 \\
\widetilde{k}-4 & 0 \le y_2 \le k-1 \\
\widetilde{k}-5 & b_1 & 0 \le i_1 \le k-2 \\
\widetilde{k}-6 & 0 \le y_3 \le k-2 \\
\widetilde{k}-7 & 0 \le y_4 \le k-2 \\
\widetilde{k}-8 & b_2 \\
\vdots & \vdots \\ 
1 & b_{k-1} & 0 \le i_{k-1} \le 0 \\
0 & 0 \le y_{\widetilde{k}-k-1} \le 0 
\end{array}
\]
The $i_i$ and $y_i$ together determine $\widetilde{\underline{i}}$. In each row the $y_i$ determines the position where a node (with no left children) is added (starting from $i=0$).
With a leveled tree given by $\underline{i} = (i_0, \dots, i_{k-1})$ we associate the sign
\begin{equation}\label{eqsgntree} \sgn(\underline{i}) := (-1)^{\sum_j i_j} \end{equation}
and we define $\underline{i}^{\vee} = (k-1-i_0, k-2-i_1, \dots, i_{k-1})$.
\end{PAR}

\begin{PROP}\label{PROPSHIHSZCZARBAGEN1} We have for arbitrary $n$, setting $\widetilde{k}:=n-k-1$:
\[  \mathcal{P} \mathcal{H}_{\underline{b}, \underline{i}}^n \equiv \begin{cases} \sum_{\substack{\widetilde{\underline{i}} \text{ such that } \\ \text{$\underline{i}$ is a $\underline{b}$-reduction of $\widetilde{\underline{i}}$}} }  \varepsilon (\Sz^{c_0}_{\widetilde{\underline{i}}} \ast' q_{0}) \otimes \cdots \otimes  (\Sz^{c_k}_{\widetilde{\underline{i}}} \ast' q_{k}) & \underline{b}_{0} < \widetilde{k}  \\
0 & \text{otherwise} \end{cases}    \]
modulo degenerates, where $q_{j} = s_{j,j+1}^{k+1}$. In particular the expession is zero unless $\widetilde{k} \ge k$ (i.e.\@ $n \ge 2k+1$).

Here the $\widetilde{\underline{i}}$ are vectors of length $\widetilde{k}$ and $c_0< \dots< c_k$ are determined by $\widetilde{\underline{i}}$ as in \ref{PARREDUCT}, and
\[ \varepsilon = (-1)^{(\widetilde{k}+1)k} \sgn(\underline{i}^{\vee}) \sgn(\widetilde{\underline{i}}^{\vee}) \sgn(\sigma_{\underline{b}}).  \]

The $\Sz_{\overline{i}}^j \ast' q_{j}$ have values in $[c_j] \ast' [1] = [c_j+1]$ and are considered as $[n-k]$-valued via
the inclusion $[c_j+1] \subset [n-k]$.
\end{PROP} 
The proof is a bit involved and has been shifted to appendix~\ref{APXSHIHSZCZARBA}. Unless the reader is interested in the
combinatorics of the formula, it is not very enlightening. 

For the special case $n = 2k+1$ (least non-zero case) we get a simpler form: 
\begin{KOR}\label{PROPSHIHSZCZARBA} We have for $n = 2k+1$: 
\[ \boxed{ \mathcal{P} \mathcal{H}_{\underline{b}, \underline{i}}^n \equiv   \begin{cases}  (\Sz^0_{\underline{i}}  \ast' q_{0}) \otimes \cdots \otimes  (\Sz^k_{\underline{i}}  \ast' q_{k}) & \underline{b} = (k-1, k-2, \dots, 0) \\
0 & \text{otherwise} \end{cases}  }  \]
modulo degenerates, where $q_j = s_{j,j+1}^{k+1}$ with the same conventions as in Proposition~\ref{PROPSHIHSZCZARBAGEN1}.
\end{KOR} 

Notice that in that case $\widetilde{k} = k$, $\widetilde{\underline{i}}=\underline{i}$, and $\varepsilon = \sgn(\sigma_{\underline{b}}) = 1$. 

\begin{PAR}
The formula in Proposition~\ref{PROPSHIHSZCZARBAGEN1} has the following important property.
Defining 
\[  \mathcal{K}_{\underline{b}, \underline{i}}^n :=  (-1)^{(\widetilde{k}+1)k} \sgn(\underline{i}^{\vee}) \sgn(\sigma_{\underline{b}})   \sum_{\substack{\widetilde{\underline{i}} \\ \text{$\underline{i}$ is a $\underline{b}$-reduction of $\widetilde{\underline{i}}$}} } \sgn(\widetilde{\underline{i}}^{\vee})  (\Sz^{c_0}_{\widetilde{\underline{i}}}  \otimes \cdots \otimes  \Sz^{c_k}_{\widetilde{\underline{i}}})    \]
(with the same inclusion $[c_j] \subseteq [\widetilde{k}]$, such that this is in $\Z[\Hom([\widetilde{k}], [\widetilde{k}])^{k+1}]$), i.e.\@ the essential part of $\mathcal{P} \mathcal{H}_{\underline{b}, \underline{i}}^n$ and define 
\[  \mathcal{K}^n_k :=   \sum_{\underline{i}, \underline{b}} \sgn(\underline{i}^{\vee}) \sgn(\sigma_{\underline{b}}) \mathcal{K}_{\underline{b}, \underline{i}}^n  \]
where the sum is over all (correct) vectors of length $n-k-1$ and $k$, respectively.

We can then write
\begin{equation}\label{eqglobalszczarba}
 \mathcal{K}^n_k=   \sum_{\substack{ \widetilde{\underline{i}} \\ c_0< \cdots< c_k }  \\ } \sgn(\widetilde{\underline{i}}^{\vee}) (\Sz^{c_0}_{\widetilde{\underline{i}}}  \otimes \cdots \otimes  \Sz^{c_k}_{\widetilde{\underline{i}}})   
\end{equation}
where the sum runs over all $\widetilde{\underline{i}}$ of length $n-k-1$ and $\{c_0, \dots, c_k\} \subset \{0, \dots, n-k-1\}$ over subsets of indices of leaves running first through nodes at their left child. 
Notice that these determine uniquely a $\underline{b}$ consisting of the corresponding row indices and a reduction $\underline{i}$.

\end{PAR}
\begin{PROP}\label{PROPSZCZARBACANCELLATION}
 For each integer $k \ge 0$ the following holds true\footnote{notice that the summands are zero for $n > k$, and that $\sum_{\underline{i}} \sgn(\underline{i})$ is 1 for $k\le 1$ and 0 otherwise. However, constants are degenerate for $k \ge 1$ anyway.  }
\[  (1+\varepsilon)^{\otimes (k+1)} \mathcal{K}^{2k+1}_k \equiv \sum_{n}   \mathcal{K}^{n+k+1}_{n}  \]
modulo degenerates and constants
or equivalently
\[ \sum_n (1-\varepsilon)^{\otimes (n+1)} \mathcal{K}^{n+k+1}_{n} \equiv \mathcal{K}^{2k+1}_k   \]
modulo degenerates and constants. Here $\varepsilon: \Z[\Hom([k], [k])] \to \Z$ is the augmentation and the calculation takes place in $\bigoplus_{n=0}^{\infty}  \Z[\Hom([k], [{k}])]^{\otimes n}$.
\end{PROP}
\begin{proof}
Follows from (\ref{eqglobalszczarba}) and the following Lemma~\ref{LEMMADEG}.
\end{proof}

\begin{LEMMA}\label{LEMMADEG}
Let $\underline{i} = (i_0, \dots, i_{k-1})$ be a vector which we see as a leveled tree as in (\ref{PARREDUCT}).
\begin{enumerate}
\item  
$\Sz_{\underline{i}}^j$ is non-degenerate at $e$ if and only if the path from the $j$-th leaf towards the root passes through a right child at $e$. 
 \item The expression 
 \[ \Sz_{\underline{i}}^{c_0} \otimes \cdots \otimes \Sz_{\underline{i}}^{c_m}  \]
 for $\{c_0, \dots, c_m\} \subset \{0, \dots, k\}$ is non-degenerate if and only if there is a $\underline{b} = (b_0, \dots, b_{k-m-1})$ such that 
  $\underline{i}$ has a $\underline{b}$-reduction with associated $\{c_j\}$, or in other words, if for each $j \not\in \{c_0, \dots, c_m\}$ the path from the $j$-node towards the root passes first through a left child ($\underline{b}$ consists then of the corresponding row indices). 
\end{enumerate}
\end{LEMMA}
\begin{proof}
1.\@ follows immediately from the recursive formula and 2.\@ follows from 1.
\end{proof}

\section{Basic examples of classical bar and cobar}\label{CHAPTERBASICEXAMPLES}


\subsection{Classical (co)bar  for $\Set$}

\begin{PAR}\label{PARCOBARSET}
Consider $(\Set, \times)$, the category of sets equipped with the product.
We will investigate the (very simple) classical bar construction $\barconst_{(\Set, \times) \to \OOO}$ and cobar construction $\cobarconst_{(\Set, \times) \to \OOO}$ (Definition~\ref{DEFCOBARCLASS}). 
This also determines --- in principle --- the bar and  cobar constructions
$\barconst_{(\Set^{\Delta^{\op}}, \times) \to \OOO}$ and $\cobarconst_{(\Set^{\Delta^{\op}}, \times) \to \OOO}$ of simplicial sets because the product, and thus the bar and cobar functors, are computed point-wise. 
Since the point (unit of $\times$) is final, by Corollary~\ref{KORRHO}, $\rho^*$ is an equivalence:
\[ \rho^*: ((\Set, \times)^{\vee})^{(\Delta,*)^{\op}} \cong ((\Set, \times)^{\vee})^{\twop \OOO} .   \]  
Furthermore, we have
\[ ((\Set, \times)^{\vee})^{(\Delta,*)^{\op}} \cong \Coalg(\Set^{\Delta^{\op}}, \times) \]
by the trivial Eilenberg-Zilber Theorem \ref{SATZEZTRIVIAL}, and finally
\[ \Coalg(\Set^{\Delta^{\op}}, \times) \cong \Set^{\Delta^{\op}} \]
because $\times$ is the product. 
\end{PAR}

\begin{LEMMA}\label{LEMMABARSET}
We have an isomorphism (making the identifications in \ref{PARCOBARSET})
\[ \boxed{N \cong  (\rho^*)^{-1} \circ \barconst }  \]
of functors
\[ \mathrm{Mon} \to \Set^{\Delta^{\op}} \]
where $N$ is the nerve. 
\end{LEMMA}
\begin{proof}
Clear from the definitions. 
\end{proof}

\begin{LEMMA}\label{LEMMACOBARSET}
For $A \in \Set^{\Delta^{\op}}$ a simplicial set, the monoid
$\cobarconst \rho^* A$ (making the identifications in \ref{PARCOBARSET}) is generated by $A_{[1]}$ modulo the relations
\begin{align*} a_0 = 1 &\quad \text{for $a_0 \in A_{[0]}$}  \\
\delta_1(a_2) = \delta_2(a_2)\cdot \delta_0(a_2)  &\quad \text{for $a_2 \in A_{[2]}$ } 
\end{align*}
In particular, if $A$ is path-connected, then this is the fundamental monoid of the simplicial set. 
\end{LEMMA}
\begin{proof}
Recall the discussion \ref{PARCOBAREXPLICIT}. The coequalizer (\ref{eqformulacobar}) with its induced algebra structure is obviously the monoid described in the statement. Notice that the ``comultiplication''
\[ A_{[2]} \to A_{[1]} \times A_{[1]}  \]
is given by $(\delta_2, \delta_0)$ in this case (trivial Eilenberg-Zilber Theorem~\ref{SATZEZTRIVIAL}). 
\end{proof}

Let $G: \mathrm{Mon} \to \mathrm{Grp}$ the functor ``group completion'' and $F: \Set \to \mathrm{Grp}$ the free group functor. 
 
  \begin{LEMMA}\label{LEMMAKANPOINTWISE}
For a simplicial set $A$ such that the map $A \to A_{[-1]}$ is simply connected (where $A_{[-1]} = \colim A$) and choosing a splitting $A_{[-1]} \to A_{[0]}$, 
we have an isomorphism\footnote{where the quotient has to be interpreted in the
only possible way as the quotient modulo the smallest normal subgroup containing $F(A_{[-1]})$ or, what is the same here, imposing the relations $[x] \sim 1$ for all $x \in A_{[-1]}$ or, in other words, taking $F(A_{[0]} \setminus A_{[-1]})$. }: 
\[ G \cobarconst \rho^* A = F(A_{[0]}) / F(A_{[-1]}) \]
 induced by 
 \begin{align*}
 \phi:  A_{[1]} & \to  F(A_{[0]}) / F(A_{[-1]}) \\
 x &\mapsto  (\delta_0 x) \cdot (\delta_1 x)^{-1}.
 \end{align*}
 \end{LEMMA}
 
  \begin{proof}
 Both sides commute with coproducts and thus it suffices to see the simply-connected case.
By Lemma~\ref{LEMMACOBARSET} the group $G \cobarconst \rho^* A$ is generated  by 
$A_{[1]}$ modulo the relations
\begin{align*} a_0 = 1 &\quad \text{for $a_0 \in A_{[0]}$}  \\
\delta_1(a_2) = \delta_0(a_2)\cdot \delta_2(a_2)  &\quad \text{for $a_2 \in A_{[2]}$ } 
\end{align*}
The morphism $\phi$ obviously respects these relations. 
We have to see that it is bijective.

If $\phi(x)=1$ then $x$ is a product of products $x_1^{\pm 1} \dots x_n^{\pm 1}$ with the property that $\delta_1(x_1^{\pm 1}) = \delta_0(x_n^{\pm 1})$ and $\delta_0(x_i^{\pm 1})=\delta_1(x_{i+1}^{\pm 1})$ where we understood $\delta_i(x^{-1}) = \delta_{1-i}(x)$. The $x_1^{\pm 1}, \dots, x_n^{\pm 1}$ thus constitute closed paths which means that $x_1^{\pm 1}  \cdots x_n^{\pm 1} \equiv 1$ modulo the relations by assumption (simple connectivity). 
 
For $x \in X_{[0]}$ there is a path $x_1^{\pm 1} \dots x_n^{\pm 1}$ from an element in the image of $X_{[-1]}$ to $x$. It follows then that 
 $x_1^{\pm 1} \dots x_n^{\pm 1}$ is mapped to $x$ under the map $\phi$. 
 \end{proof}

 \begin{PAR}\label{PARKANEXPLICIT}
By Lemma~\ref{LEMMADECCONTRACTIBLE}, for $A = \dec^*(X)_{[k],\bullet}$, we have $A_{[-1]} = X_{[k]}$, and the map $A \to A_{[-1]}$ is simply connected.
However, in this case, there are extra degeneracies: $s_{\can} = s_{k}:  A_{[-1]} \to A_{[0]}$ (which we take as splitting),  $s_{k}:  A_{[0]} \to A_{[1]}$,   $s_{k}:  A_{[1]} \to A_{[2]}$, and
 we can make this thus more explicit: 
Since for $x \in A_{[1]}$, 
\[ \underbrace{\delta_{k+3}}_{\delta_2^A}(s_{\can}x) = s_{\can}( \underbrace{\delta_{k+2}}_{\delta_1^A} x), \underbrace{\delta_{k+2}}_{\delta_1^A}(s_{\can}x) = s_{\can}(\underbrace{\delta_{k+1}}_{\delta_{0}^A} x), \text{ and } \underbrace{\delta_{k+1}}_{\delta_0^A}(s_{\can}x) = x, \]  
we have the relations
 \[ x = s_{\can}(\delta_0 x)  s_{\can}(\delta_1 x)^{-1}  \quad  x^{-1} = s_{\can}(\delta_0 x) s_{\can}(\delta_1 x)^{-1}   \]
 in $G \cobarconst \rho^* A$. Hence we see directly that $x_1^{\pm 1} \dots x_n^{\pm 1} = 1$ in $G \cobarconst \rho^* A$ (for $x_1, \dots, x_n$ as in the proof). Hence $\phi$ is injective.
 Furthermore, for $x \in A_{[0]}$, we have $\phi(s_{\can}x) = x \cdot (s_{\can}(\delta_{0}x))^{-1}$ where the second factor is in $F(A_{[-1]})$. Hence $\phi$ is surjective. 
 Furthermore
 \begin{align*} 
 F(A_{[0]})/F(A_{[-1]}) &\to G(\coprod_n A_{[1]}^{n} / \cdots) \\
 [x] &\mapsto s_{\can}x
 \end{align*}
 is an inverse of $\phi$.
\end{PAR}

\subsection{Classical (co)bar  for $\Set^{\Delta^{\op}}$}\label{SECTCOBARSSET}

\begin{PAR}\label{PARADJCOBAR}
We get an adjunction
\[  \xymatrix{  \Set^{\Delta^{\op}} \ar@<3pt>[rrr]^-{\cobarconst \circ \rho^* \circ \dec^*} & & &  \ar@<3pt>[lll]^-{ \dec_* \circ (\rho^*)^{-1} \circ \barconst} \mathrm{Mon}^{\Delta^{\op}} } \]
where now $\cobarconst = \cobarconst_{(\Set^{\Delta^{\op}}, \times) \to \OOO}$ and $\barconst = \barconst_{(\Set^{\Delta^{\op}}, \times) \to \OOO}$ are computed point-wise because $\times$ has this property. We will investigate this adjunction explicitly in this section. 
\end{PAR}

\begin{DEF}
The functor 
\[ M^{\mathrm{Kan}}:=\cobarconst \circ \rho^* \circ \dec^*: \Set^{\Delta^{\op}} \to \mathrm{Mon}^{\Delta^{\op}}   \]
 is called the {\bf geometric cobar construction}. 
\end{DEF}
The functor is closely related to Kan's loop group functor hence the notation. 
Let $F: \Set \to \mathrm{Grp}$ the free group functor which we denote by the same letter on diagrams.

\begin{DEF}
\begin{enumerate}
\item 
Let $X \in \Set^{\Delta^{\op}}$ be simplicial set. Define a simplicial group called {\bf Kan's loop group} by
\[ G^{\mathrm{Kan}} (X)_{[k]} := F(X_{[k+1]}) / F(X_{[k]})  \]
with the simplicial structure given by
\[ \delta_i'(x) = \begin{cases} (\delta_{k} x)(\delta_{k+1} x)^{-1} & i=k; \\ \delta_{i} x & i < k. \end{cases}    \]

\item Let $X \in \mathrm{Mon}^{\Delta^{\op}}$ be simplicial monoid. Define a simplicial set called the {\bf classifying space} by
\[ (\overline{W} X)_{[n]} := X_{[n-1]} \times X_{[n-2]} \times \cdots \times X_{[0]}  \]
where the simiplicial structure is given by
\begin{eqnarray*} \delta_i(x_{n-1}, \dots, x_0) &=& \begin{cases}(x_{n-2}, \dots,  x_0)  & i=0, \\
 \delta_i(x_{n-1}), \dots, \delta_1(x_{n-i-1}), x_{n-i-1}\cdot \delta_0(x_{n-i}), x_{n-i-2}, \dots, x_0)   & 1 \le i<n,
 \end{cases}    \\
s_i(x_{n-1}, \dots, x_0)  &=& \begin{cases}
 (1, x_{n-1}, \dots, x_0) & i=0, \\
 (s_{i-1}(x_{n-1}),\dots, s_{0} (x_{n-i}), 1, x_{n-i-1}, \dots, x_0) & 1 \le i<n.
  \end{cases}   
\end{eqnarray*}
\end{enumerate}
\end{DEF}

Let $G: \mathrm{Mon} \to \mathrm{Grp}$ be the point-wise group completion functor which we denote by the same letter on diagrams. 
\begin{PROP}We have for $\cobarconst = \cobarconst_{(\Set^{\Delta^{\op}}, \times) \to \OOO}$ an isomorphism
\[ \boxed{G^{\mathrm{Kan}} \cong G \circ \cobarconst \circ \rho^* \circ \dec^* } \]
of functors
\[  \Set^{\Delta^{\op}} \to \mathrm{Grp}^{\Delta^{\op}} \]
\end{PROP}
See also \cite[Proposition 5.3]{Ste12}.
\begin{proof}
This follows from Lemma~\ref{LEMMAKANPOINTWISE} (cf.\@ also \ref{PARKANEXPLICIT}). The simplicial structure can be (for $i=0, \dots, k$) read off from:  
\[ \raisebox{\dimexpr\depth-3\fboxsep}{\xymatrix{ G\coprod (\dec^*X)_{[k],[1]}^n  \ar[d]^{\delta_i} & &  \ar[ll]_-{s_{k}} F(X_{[k+1]})/F(X_{[k]}) \ar@{-->}[d]  \\ 
G\coprod (\dec^*X)_{[k-1],[1]}^n \ar[rr]_-{\delta_{k} \cdot \delta_{k+1}^{-1}}  & &  F(X_{[k]})/F(X_{[k-1]})
 }} \tag*{\qedhere} \]
\end{proof}

\begin{PROP}[Duskin \cite{Dus75}]We have for $\barconst = \barconst_{(\Set^{\Delta^{\op}}, \times) \to \OOO}$ an isomorphism
\[ \boxed{ \overline{W} \cong \dec_* \circ (\rho^*)^{-1} \circ \barconst } \]
of functors
\[  \mathrm{Mon}^{\Delta^{\op}} \to  \Set^{\Delta^{\op}} . \]
\end{PROP}
\begin{proof}
Let $X$ be a simplicial monoid. We have by Lemma~\ref{LEMMABARSET} applied point-wise that $Y = (\rho^*)^{-1} \circ \barconst (X)$ is
the bisimplicial set  $Y_{[n],[m]} = N(X_{[n]})_{[m]}$. By Proposition~\ref{PROPEXPLICIT}
\begin{gather*} (\dec_* Y)_{[n]} \cong 
 \lim \left( \vcenter{ 
\xymatrix@C=1.5pc{
Y_{[n],[0]} \ar[rd] & & \ar[ld] Y_{[n-1],[1]}   \ar[rd] & \ \cdots \  &   Y_{[0],[n]}  \ar[ld] \\
& Y_{[n-1],[0]} &  &\  \cdots \  
}
} \right)  \end{gather*}
and there are maps
\begin{eqnarray*}
 \alpha_i: (\overline{W}(X))_{[n]} &\to& Y_{[n-i],[i]} \\
(g_{n-1}, \dots, g_0) &\mapsto& (\delta_0^{i-1}(g_{n-1}), \delta_0^{i-2}(g_{n-2}), \dots, \delta_0(g_{n-i+1}), g_{n-i})   
\end{eqnarray*}
which yield bijections $(\overline{W}X)_{[n]} \cong (\dec_* Y)_{[n]}$ compatible with the simplicial structure. 
\end{proof}

\begin{KOR}
There is are adjunctions
\[ \xymatrix{ \mathrm{Grp}^{\Delta^{\op}} \ar@<3pt>@{^{(}->}[rr]^-{} && \ar@<3pt>[ll]^-{G} \mathrm{Mon}^{\Delta^{\op}} \ar@<3pt>[rr]^-{\overline{W}} && \ar@<3pt>[ll]^-{M^{\mathrm{Kan}}} \ar@/^20pt/[llll]^-{G^{\mathrm{Kan}}}  \mathrm{Set}^{\Delta^{\op}}  } \]
with $\overline{W}$ right adjoint. 
\end{KOR}
\begin{proof}The right hand side is the adjunction~\ref{PARADJCOBAR} and the left hand side the group completion adjunction.
\end{proof}

\begin{BEM}
From the non-Abelian Eilenberg-Zilber Theorem~\ref{SATZEZ} follows that for a simplicial monoid  $X$ there is a weak equivalence 
\[ \overline{W} X \cong \delta^* N X. \]
\end{BEM}

\subsection{$A_{\infty}$-algebras and coalgebras}

Let $(\mathcal{C}, \otimes)$ be an  Abelian tensor category. 
The classical bar and cobar constructions for $(\mathcal{C}, \otimes)$ are intimately related to the notions of $A_{\infty}$-algebra and -coalgebra. Hence we will pause the discussion of (co)bar to briefly discuss the latter. 
For an extensive introduction we refer to \cite{Kel01}.

Embed $\mathcal{C}^{\Delta^{\op}}$ as $\Ch(\mathcal{C})_{\ge 0}$ (via Dold-Kan) into $\Ch(\mathcal{C})$, the category of unbounded complexes which we see here
as $\Z$-graded objects $X$ in $\mathcal{C}$ with a differential, i.e.\@ a morphism
\[ \mathrm{d}: X \to X \]
of degree -1 that satisfies $\mathrm{d}^2 = 0$. We define, in the usual way, shift operators
\[ s, s^{-1}: \Ch(\mathcal{C}) \to \Ch(\mathcal{C}) \]
with $(sX)_{i} := X_{i-1}$.
Furthermore there are the adjunctions
\[ \xymatrix{   \Ch_{\le 0}(\mathcal{C}) \ar@{^{(}->}@<3pt>[r]  & \ar@<3pt>[l]^-{\tau_{\le 0}} \Ch(\mathcal{C}) }  \]
with $\tau_{\le 0}$ (truncation) right adjoint
and
\[ \xymatrix{   \Ch_{\ge 0}(\mathcal{C}) \ar@{^{(}->}@<3pt>[r]  & \ar@<3pt>[l]^-{\tau_{\ge 0}} \Ch(\mathcal{C}) }  \]
with $\tau_{\ge 0}$ (truncation) left adjoint. The inclusions also have the other adjoints which will not be used in the sequel. 

We let $\widetilde{\otimes}$ be the tensor product $\tot - \boxtimes -$ on $\Ch(\mathcal{C})$, i.e.\@ the usual tensor product of complexes. 
Motivated by Proposition~\ref{PROPEXPLICITAB}, explicitly normalize $A \tildeotimes B$
as the following complex 
\[ \bigoplus_{i+j=n} A_i \otimes B_j  \]
with differential
\begin{equation}\label{eqtot} \mathrm{d} = \mathrm{d}_l + (-1)^i \mathrm{d}_r.  \end{equation}
This will hardly ever be an infinite sum in our applications. Other-wise one should of course be more careful about whether to choose the product or coproduct (or a mixture) and assume that $\mathcal{C}$ has those. 

\begin{PAR}\label{PARKOSZUL} 
We will always use the {\bf Koszul sign rule} (cf.\@ also the Yoneda product Lemma~\ref{LEMMAYONEDA}, and how this is related to the realization of complexes as objects in $\mathcal{C}^{\FinSet^{\op}}$). For homogenous morphisms $f: A \to C$, $g: B \to D$ we let 
$f \otimes g: A \otimes B \to C \otimes D$ be the morphism, defined on graded pieces as $x \otimes y \mapsto (-1)^{\deg(x)\deg(g)} f(x) \otimes g(y)$ and similarly for higher tensors. 
For example the differential (\ref{eqtot}) on $A \tildeotimes B$ is expessed as $1 \otimes \dd + \dd \otimes 1$.
\end{PAR}

\begin{PAR}
An algebra object in $\Ch(\mathcal{C})$, i.e.\@ an object in $(\Ch(\mathcal{C}), \widetilde{\otimes})^{\OOO}$ can be seen as an algebra object in 
graded objects such that 
\[ \xymatrix{ A \otimes A \ar[d]_{\dd \otimes 1 + 1 \otimes \dd} \ar[r]^-{m} & A \ar[d]^{\dd}  \\
A \otimes A  \ar[r]_-{m} & A
} \]
commutes (with Koszul sign rule). Similarly for coalgebra objects.
Expressed with elements this reads: 
\[ \mathrm{d} (a \cdot b) = (\mathrm{d} a) \cdot b + (-1)^{\deg a} a \cdot  (\mathrm{d} b).  \]
\end{PAR}

\begin{PAR}
For $X$ a graded object, we let 
\[ T^{\coprod}(X) \quad T^{\prod}(X)  \]
be the tensor (co)algebras. To avoid any problems during the abstract discussion for now with infinite direct sums or products, in particular, with their existence and commutation with $\otimes$, 
we consider $\mlq T^{\coprod} \mrq(X) \in \Ind-\mathcal{C}$ and $\mlq T^{\prod} \mrq(X)  \in \Pro-\mathcal{C}$.

$\mlq T^{\prod} \mrq(X)$ is considered as algebra with  
\[ (x_1 \otimes \cdots \otimes x_i) \boxtimes (x_{i+1} \otimes \cdots \otimes x_n) \mapsto x_1 \otimes \cdots \otimes x_n     \]
$\mlq T^{\coprod} \mrq(X)$ is considered as coalgebra with the ``deconcatenation'' coproduct. 
\[ x_1 \otimes \cdots \otimes x_n \mapsto \sum_{i=0}^n (x_1 \otimes \cdots \otimes x_i) \boxtimes (x_{i+1} \otimes \cdots \otimes x_n)    \]
Actually, these constructions are precisely dual to each other. 

Denote by $\mlq T^{\prod,+} \mrq(X)$ the positive part (i.e.\@ without the unit object $1$).  
\end{PAR}

\begin{LEMMA} \label{LEMMATENSORCOALG}
\begin{enumerate}
\item 
Any homomorphism $\dd_1: X \to \mlq T^{\prod} \mrq(X)$ of degree -1 can be extended uniquely to a morphism
 $\dd: \mlq T^{\prod}\mrq(X) \to \mlq T^{\prod} \mrq(X)$ of degree -1 satisfying the graded Leibniz rule (called a {\bf derivation}). 
\item 
 Any $\alpha_1: X \to \mlq T^{\prod,+} \mrq(X)$ of degree 0 can be extended uniquely to an algebra homomorphism
 \[ \alpha: \mlq T^{\prod} \mrq(X) \to \mlq T^{\prod}\mrq (X). \]
\end{enumerate}
\end{LEMMA} 
\begin{proof}
1.\@ The components 
\[ \dd_n: X^{\otimes n} \to  \mlq T^{\prod} \mrq(X) \]
of the extension  
are given by
\[ x_1 \otimes \cdots  \otimes x_n \mapsto \sum_{i=1}^n (1^{\otimes i}  \otimes \dd   \otimes 1^{\otimes n-i-1}) (x_1 \otimes \cdots \otimes x_n) \]
with the Koszul sign convention (\ref{PARKOSZUL}).
To see that this is well-defined, observe that the components
$X^{\otimes n} \to  X^{\otimes m}$ are zero for $n \gg m$. 

2.\@ The components 
\[ \alpha_n: X^{\otimes n} \to  \mlq T^{\prod} \mrq(X) \]
of the extension 
are given by
\[ x_1 \otimes \cdots  \otimes x_n \mapsto \alpha(x_1) \otimes \cdots \otimes \alpha(x_n).  \]
To see that this is well-defined, observe that the components
$X^{\otimes n} \to  X^{\otimes m}$ are zero for $n > m$. 
\end{proof}
There is a corresponding dual version for the coalgebra $\mlq T^{\coprod} \mrq (X)$.

If $\mathcal{C}$ is complete\footnote{That $\mathcal{C}$ has countable products is enough (and sometimes even no restriction because we work with graded objects) for what we are doing. We leave it to the reader to make the necessary precisions.}, we have a lax monoidal functor
\[ \lim: \Pro-\mathcal{C} \to \mathcal{C} \]
and thus $\mlq T^{\prod} \mrq(X)$ gives rise to an algebra $T^{\prod} := (\lim \mlq T^{\prod} \mrq)(X)$. Similarly, if $\mathcal{C}$ is cocomplete, we have an oplax monoidal functor
\[ \colim: \Ind-\mathcal{C} \to \mathcal{C} \]
and thus $\mlq T^{\coprod} \mrq(X)$ gives rise to a coalgebra $ T^{\coprod}  := (\colim \mlq T^{\coprod} \mrq)(X)$. 

\begin{DEF}Let $(\mathcal{C}, \otimes)$ be an Abelian tensor category. 
\begin{enumerate}
\item
A graded object $X$ in $\mathcal{C}$ together with a differential $\dd$ on the graded coalgebra $\mlq T^{\coprod} \mrq (s X)$ 
is called {\bf an $A_{\infty}$-algebra. }
\item
Let $X, Y$ be $A_{\infty}$-algebras. A {\bf morphism of $A_{\infty}$-algebras $X \to Y$} is a morphism of (non-counital) dg-coalgebras $\mlq T^{\coprod,+} \mrq (s X) \to \mlq T^{\coprod,+} \mrq (s Y)$.
\item
A graded object $X$ in $\mathcal{C}$ together with a differential\footnote{i.e.\@ a derivation, satisfying $\dd^2 = 0$} $\dd$ on the graded algebra $\mlq T^{\prod} \mrq (s^{-1 }X)$ 
is called {\bf an $A_{\infty}$-coalgebra. }
\item
Let $X, Y$ be $A_{\infty}$-coalgebras. A {\bf morphism of $A_{\infty}$-coalgebras $X \to Y$} is a morphism of dg-algebras $\mlq T^{\prod,+} \mrq (s^{-1}X) \to \mlq T^{\prod,+} \mrq (s^{-1}Y)$.
\end{enumerate}
\end{DEF}
There are notions of (co)units for $A_{\infty}$-(co)algebras which will not be considered here.

\begin{PAR}\label{PARMYSTERIOUSSIGN}
Translating via Lemma~\ref{LEMMATENSORCOALG}, 1., an $A_{\infty}$-algebra is thus determined by a collection $\dd_i: sX^{\otimes i} \to sX$  of degree -1 satisfying 
\[ \sum_{n=r+s+t}  \dd_{r+1+t} (1^{\otimes r}\otimes \dd_s \otimes 1^{\otimes t}) = 0. \] 
 It is convenient to define maps $m_i: X^{\otimes i} \to X$ of degree $i-2$ by means of the commutative diagram (with Koszul sign rule \ref{PARKOSZUL})
\[ \xymatrix{  (sX)^{\otimes i} \ar[r]^{\dd_i}  & sX  \\
X^{\otimes i} \ar[u]^{s^{\otimes i}} \ar[r]^{m_i} & X \ar[u]_{s}
 } \]
 where $s$ is the canonical map of degree 1\footnote{ Explicitly $m_i(x_1, \dots, x_n) = (-1)^{\sum (n-i)\deg(x_i)} s^{-1} \dd_i(sx_1, \dots, sx_n) $}. Equivalently, an $A_{\infty}$-algebra is thus given by a family $m_i$, satisfying
the {\bf Stasheff identities}:
\[ \sum_{n=r+s+t} (-1)^{r+st} m_{r+1+t} (1^{\otimes r}\otimes m_s \otimes 1^{\otimes t}) = 0 \] 
 and components $\alpha_i: X^{\otimes i}  \to X$ give a morphism of  $A_{\infty}$-algebras, if and only if
\[ \sum_{n=r+s+t} (-1)^{r+st} \alpha_{r+1+t} (1^{\otimes r}\otimes m_s \otimes 1^{\otimes t}) = \sum_{\substack{1 \le r \le n \\ n=i_1+\cdots+i_r}} (-1)^s m_r (\alpha_{i_1} \otimes  \cdots \otimes \alpha_{i_r})  \] 
with $s = \sum_{j=1}^r (r-j)(i_j-1)$. 
\end{PAR}

\begin{LEMMA}\label{LEMMAAINFTYCORRESPONDENCE}
There are functors
\begin{align*} A_{\infty}: \Alg^{\circ}(\Ch(\mathcal{A}), \tildeotimes) &\to \Alg^{A_{\infty}}(\Ch(\mathcal{A}), \tildeotimes)    \\
  A_{\infty}:  \Coalg^{\circ}(\Ch(\mathcal{A}), \tildeotimes) &\to \Coalg^{A_{\infty}}(\Ch(\mathcal{A}), \tildeotimes)   
\end{align*}
giving as the identity on the underlying graded object and setting for a (co)algebra $(X, m)$ 
\[ m_1 := -\dd \qquad m_2 := m \qquad m_j := 0 \ (j \ge 2) \]
They are faithful, the image on morphisms are precisely those morphisms $\alpha: sX \to T^{\coprod}(sY)$ (resp.\@ $\alpha: T^{\prod}(s^{-1}X) \to s^{-1}Y$) that factor through $sX \to sY$ (resp.\@ $s^{-1}X \to s^{-1}Y$). 
Fixing a graded object $X$, the image is given precisely by those structures on $X$ that satisfy $m_j = 0$ for $j \ge 2$. 
\end{LEMMA}
\begin{proof}
The (co)associativity of $m$ is equivalent to the equation $(\mathrm{d}_2)^2 = 0$ and the equation
$ \mathrm{d}_1 \mathrm{d}_2 = - \mathrm{d}_2\mathrm{d}_1 $
is equivalent to the compatibility of $m$ with the differential. 
\end{proof}

\subsection{The Eilenberg-MacLane bar and Adams cobar construction}

\begin{PAR}\label{PARSYMMETRYBREAK}
This section discusses the classical Adams cobar \cite{Ada56} and Eilenberg-MacLane bar constructions \cite{EM53}.
These are {\em not anymore dual to each other}, i.e.\@ the Eilenberg-MacLane bar construction applied to $\mathcal{C}^{\op}$ does not give the Adams cobar construction!
We assume that $\otimes$ commutes with (countable) coproducts.
Let $X$ be a dg-coalgebra (not necessarily with counit). Then there is a sub-dg-algebra $T^{\oplus}(s^{-1}X) \subset T^{\prod}(s^{-1}X)$ whose underlying graded is the free algebra on $s^{-1}X$ and the differential is the restriction,   
giving a functor
\[ T^{\oplus}(s^{-1} -):  \Coalg^{\circ}(\Ch(\mathcal{A}), \tildeotimes) \to \Alg(\Ch(\mathcal{A}), \tildeotimes).  \]
Note that one can get back $T^{\prod}(s^{-1}X)$ as a completion of $T^{\oplus}(s^{-1}X)$.

However, observe that the functor {\em does not extend to $A_{\infty}$-coalgebras and their morphisms} --- their data, in contrast to the ones coming from usual dg-coalgebras, do not restrict to $T^{\oplus}(s^{-1}X)$!
\end{PAR}

Let $X \in \Ch_{\ge 0}(\mathcal{C})$ with $\alpha: X \to 1$ be an augmented dg-algebra. Denote $\overline{X} := \ker(\alpha)$ and
let $Y \in \Ch_{\ge 0}(\mathcal{C})$ with $\beta: 1 \to Y$ be an coaugmented dg-coalgebra. Denote $\overline{Y} = \coker(\beta)$. 
$\overline{X}$ and $\overline{Y}$ are again dg-(co)algebras (without (co)unit).

\begin{DEF}\label{DEFCOBARADAMS}
Assume that (countable) coproducts exist in $\mathcal{C}$.
\begin{enumerate}
\item 
The coaugmented dg-coalgebra\footnote{$T^{\coprod} = \colim \mlq T^{\coprod} \mrq$ always exists as graded object because it involves only finite coproducts in each degree.}
\[ \barconst^{\mathrm{EM}} := T^{\coprod}(s \overline{X}) \] 
with the differential from Lemma~\ref{LEMMAAINFTYCORRESPONDENCE} 
is called the {\bf Eilenberg-MacLane bar construction} of the augmented algebra $X$.
It comes equipped with a natural coaugmentation given by  $1 \to T^{\coprod}(s^{-1} \overline{X})$. 
\item 
 Asumme that $\otimes$ commutes with countable coproducts.
The augmented dg-algebra 
\[ \cobarconst^{\mathrm{Adams}} := T^{\oplus}(s^{-1} \overline{Y}) \] 
with the restriction of the differential from Lemma~\ref{LEMMAAINFTYCORRESPONDENCE} (cf.\@ \ref{PARSYMMETRYBREAK})
 is called the {\bf Adams cobar construction} of the coaugmented coalgebra $Y$. 
It comes equipped with a natural augmentation given by the projection $T^{\oplus}(s^{-1} \overline{Y}) \to 1$.
If $Y$ is connected, then it is again an object in $\Ch_{\ge 0}(\mathcal{C})$.
\end{enumerate}
\end{DEF}

For coalgebras, we have the functor ``connected cover'':
\[ P: \Coalg(\Ch(\mathcal{C}), \tildeotimes) \to \Coalg_{\mathrm{conn}}(\Ch(\mathcal{C})_{\ge 0}, \tildeotimes)  \]
which is left adjoint to the inclusion,  where connected means $X_i \cong 0$ for $i < 0$, and that the counit induces an isomorphism $X_0 \to 1$. 
It is defined by
\[ P(X)_i = \begin{cases} X_i & i >0 \\ 1 & i=0 \\ 0 & i < 0 \end{cases} \]
with obvious comultiplication, counit, and differential. 

Note that connected coalgebras are {\em canonically} coaugmented. We will only consider the Adams cobar construction on such connected objects where  $s^{-1} \overline{X}$ is just
 $\tau_{\ge 0} (s^{-1}X)$.

\subsection{Classical (co)bar  for Abelian categories}

Let $(\mathcal{C}, \otimes)$ be an Abelian tensor category. 
We will investigate the classical bar construction $\barconst_{(\mathcal{C}, \otimes) \to \OOO}$ and cobar construction $\cobarconst_{(\mathcal{C}, \otimes) \to \OOO}$ (cf.\@ \ref{DEFCOBARCLASS}).

\begin{LEMMA}\label{LEMMABARAB}
Let $(\mathcal{C}, \otimes)$ be an Abelian tensor category and $A \in  \mathcal{C}$ an algebra with augmentation $\alpha: A \to 1$. Denote $\overline{A}:= \ker \alpha$.
 Then we have an isomorphism of dg-coalgebras
\[ \xymatrix{ T^{\coprod}(s\overline{A}) \ar@<3pt>@{^{(}->}[rr] & &  \ar@<3pt>@{->}[ll]^-{a - \alpha(a) \cdot 1} (\rho^*)^{-1} \barconst(A)  }  \]
where $\barconst = \barconst_{(\mathcal{C}_{/1}, \otimes) \to \OOO}$ is the classical bar construction (Definition~\ref{DEFCOBARCLASS}) {\em in the category of augmented objects in $\mathcal{C}$} and
where $T^{\coprod}(s\overline{A})$ (with $\overline{A}$ considered as a complex concentrated in degree 0) 
is equipped with the differential $m: \overline{A}^{\otimes 2} \to \overline{A}$ extended to $T^{\coprod}(s\overline{A})$ by means of (the dual of) Lemma~\ref{LEMMATENSORCOALG}. 
\end{LEMMA}

\begin{proof}
The composition
\[ \xymatrix{  (\mathcal{C}_{/1}, \otimes)^{\OOO} \ar[rr]^-{\barconst = \widetilde{\pi_1^*}} && ((\mathcal{C}_{/1}, \otimes)^{\vee})^{\twop \OOO}  & \ar[l]_-{\rho^*}^-{\sim}  ((\mathcal{C}_{/1}, \otimes)^{\vee})^{(\Delta, \ast)^{\op}} = \Coalg(\mathcal{C}^{\Delta^{\op}}_{/1}, \tildeotimes) }  \]
can be described as follows. The functor $\barconst$ maps $A$ to object $B$ with
\[ B_{[i]} = A^{\otimes i} \]
with $A^{\otimes 0} = 1$ understood, with transition morphisms in $\Delta^{\op}_{\act}$ given by the corresponding algebra structure maps, 
with comultiplication
\[ B_{[i] \ast' [j]} \to B_{[i]} \otimes B_{[j]}  \]
being the obvious isomorphism, and with the counit being the isomorphism $B_{[0]} \to 1$. If 1 is final, as it is in the category of augmented objects, then by Corollary~\ref{KORRHO} this extends to an object 
\[ (\rho^*)^{-1} B \in ((\mathcal{C}, \otimes)^{\vee})^{(\Delta, \ast)^{\op}} = \Coalg(\mathcal{C}^{\Delta^{\op}}, \tildeotimes), \]
i.e.\@ to a dg-coalgebra. Inert morphisms go to the corresponding
\[ A^{\otimes i}  \to A^{\otimes j} \]
given by applying the augmentation to the appropriate slots and the comultiplication factors through the canonical degeneracy (this follows from Lemma~\ref{LEMMARHO}):
\[ B_{[i] \ast' [j]} \to B_{[i] \ast [j]} \to  B_{[i]} \otimes B_{[j]}  \]
yielding the structure of dg-coalgebra on $\rho^* B$.
Let $\overline{A}$ be the kernel of the augmentation (considered as graded object in degree 0). Then we  have an isomorphisms of dg-coalgebras
\[ \xymatrix{ T^{\coprod}(s \overline{A}) \ar@<3pt>@{^{(}->}[rr] & &  \ar@<3pt>@{->}[ll]^-{a - \alpha(a) \cdot 1} (\rho^*)^{-1} \barconst(A)  }  \]
for one composition is clearly the identity, whereas all tensors of the form $a_1 \otimes \cdots \otimes (\alpha(a_i) \cdot 1) \otimes \cdots \otimes a_n$ are degenerate. 
The identification of the differential is left to the reader. 
\end{proof}

 \begin{LEMMA}\label{LEMMACOBARABELIAN}
 Let $(\mathcal{C}, \otimes)$ be an Abelian tensor category with countable colimits such that $\otimes$ commutes with them. 
  Let $A \in  \Coalg(\Ch_{\ge 0}(\mathcal{C}), \tildeotimes)$ be a dg-coalgebra. Then for the classical cobar construction (Definition~\ref{DEFCOBARCLASS})
 $\cobarconst = \cobarconst_{(\mathcal{C}, \otimes) \to \OOO}$: 
 \[ \cobarconst \rho^* A  \cong T^{\oplus}(A_{1}) / I \] 
 where $I$ is the ideal generated by $A_{2}$ under $m_{1,1} + \mathrm{d}$, where $m_{1,1}: A_2 \to A_1 \otimes A_1$ is the 1,1-component of the comultiplication. The induced algebra structure is given by (the quotient of) the product structure in $T^{\oplus}(A_{1})$. If $A \in ((\mathcal{C}_{/1}, \otimes)^{\vee})^{(\Delta,*)^{\op}}$ is an augmented coalgebra then the same formula is true and the augmentation of the result coincides with the projection $T^{\oplus}(A_{1}) / I \to 1$.
 
 Note that the indices refer to the indices under Dold-Kan, i.e.\@ $A_1 = A_{[1]} / A_{[0]}$ and  $A_2 = A_{[2]} / A_{[2]}^{\mathrm{deg}}$.
 \end{LEMMA}
 \begin{proof}
By \ref{PARCOBAREXPLICIT}, we have
\begin{gather*} 
  \cobarconst \rho^* A \cong \\
\coker( \bigoplus A_{[1]} \otimes \cdots \otimes A_{[2]} \otimes \cdots \otimes A_{[1]}  \oplus  \bigoplus A_{[1]} \otimes \cdots \otimes A_{[0]} \otimes \cdots \otimes A_{[1]}  \rightarrow \\
  \bigoplus_{n=0}^{\infty} A_{[1]}^{\otimes n} = T^{\oplus}(A_{[1]}))
 \end{gather*}
with the morphism
\[ \id- \varepsilon: \bigoplus A_{[1]} \otimes \cdots \otimes A_{[0]} \otimes \cdots \otimes A_{[1]}  \to  T^{\oplus}(A_{[1]})  \]
with $A_{[0]} \subset A_{[1]}$ via the (only) degeneracy understood and $\varepsilon: A_{[1]} \to 1$ is the counit. 
We have an exact sequence (using that $\otimes$ commutes with quotients)
\[ \xymatrix{  \bigoplus A_{[1]} \otimes \cdots \otimes A_{[0]} \otimes \cdots \otimes A_{[1]} \ar[r]^{} & T^{\oplus}(A_{[1]})  \ar[r]^{} & T(A_1) \ar[r] & 0. } \]
We have the automorphism
\[ \Omega: T^{\oplus}(A_{[1]}) \to T^{\oplus}(A_{[1]})   \]
given on $A_{[1]}$ by $a \mapsto a - \varepsilon(a)$. This yields an exact sequence
\[ \xymatrix{  \bigoplus A_{[1]} \otimes \cdots \otimes A_{[0]} \otimes \cdots \otimes A_{[1]} \ar[r]^-{\id-\varepsilon} & T(A_{[1]})  \ar[r]^{\Omega^{-1}} & T^{\oplus}(A_1) \ar[r] & 0 } \]
and the ideal $\Omega(J)$ coincides with the image of $X$. The morphism
\[ \bigoplus A_{[1]} \otimes \cdots \otimes A_{[2]} \otimes \cdots \otimes A_{[1]}  \to  T^{\oplus}(A_{[1]})  \]
is induced by the map $A_{[2]} \to T^{\oplus}(A_{[1]})$ given by $\mu_{1,1}' - \delta_1$, where
 $\mu_{1,1}'$ is the composition $A_{[2]} \to A_{[3]} \to A_{[1]} \otimes A_{[1]}$ (composition of canonical degeneracy and structure map of the coalgebra, cf.\@ the discussion in \ref{PARCOALG}). 

We claim that the following is commutative
\[ \xymatrix{  A_{[2]} \ar[r]^-{\mu_{1,1}'-\delta_1} \ar@{->>}[d]  & T(A_{[1]})  \ar@{->>}[d]^{\Omega^{-1}}  \\
A_2 \ar[r]^-{m_{1,1} + \mathrm{d}} & T(A_1) } \]

The upper composition is given by, writing $\mu(a) = \sum a_i \otimes b_i$ where $a \in A_{[2]}$ and $a_i,b_i \in A_{[1]}$:
\[ a \mapsto \mu(a) + \sum \varepsilon(a_i)b_i + \sum \varepsilon(b_i)a_i + \sum \varepsilon(a_i)\varepsilon(b_j)  - \delta_1(a) - \varepsilon(\delta_1(a)).    \]
 We have $\sum \varepsilon(a_i)b_i = \delta_0(a)$ and $\sum a_i\varepsilon(b_i) = \delta_2(a)$ and $\sum \varepsilon(a_i)\varepsilon(b_j) = \varepsilon(a)$  because of the coalgebra structure.   Obviously $\varepsilon(\delta_1(a)) = \varepsilon(a)$. Hence we are left with
 \[a \mapsto \mu(a) + \mathrm{d}(a) \]
 where $\mathrm{d}$ is the alternating face map. In $T(A_1)$ indeed degenerate elements are mapped to 0 and, because of
the commutative diagram
\[ \xymatrix{  A_{[2]} \ar[r]^-{\mu_{1,1}'} \ar@{->>}[d]  & A_{[1]} \otimes A_{[1]}   \ar@{->>}[d]  \\
A_2 \ar[r]_-{m_{1,1}} & A_1 \otimes A_1 } \]
where  $m_{1,1}$ in the bottom line is the 1,1-component of the multiplication in $A$ considered as coalgebra w.r.t.\@ $\tildeotimes$, 
 the induced map on $A_2$ is given by $m_{1,1} + \mathrm{d}$. 
 \end{proof}

 \begin{SATZ}\label{SATZCOBARADAMS1}
Let $(\mathcal{C}, \otimes)$ be an Abelian tensor category with countable colimits such that $\otimes$ commutes with them. 

There is an isomorphism 
 \[ \boxed{ \cobarconst \circ \rho^* \cong H_0 \circ \cobarconst^{\mathrm{Adams}} \circ P } \]
of functors\footnote{the RHS is the category of augmented algebras in $(\Ch_{\ge 0}(\mathcal{C}), \tildeotimes)$ which is the same as algebra objects in augmented objects.}
\[ \Coalg(\Ch_{\ge 0}(\mathcal{C}), \tildeotimes)  \to \Alg(\mathcal{C}_{/1}, \otimes)   \]  
where  
$\cobarconst = \cobarconst_{(\mathcal{C}, \otimes) \to \OOO}$ is the classical cobar construction (Definition~\ref{DEFCOBARCLASS}), and where $P$ is the functor ``connected cover'' with its canonical coaugmentation.
\end{SATZ}

\begin{proof}
For $A \in \Coalg(\Ch_{\ge 0}, \tildeotimes)$ a dg-coalgebra, we have 
\[ (H_0 \circ \cobarconst^{\mathrm{Adams}} \circ P) (A) \cong T^{\oplus}(A_1)/I,  \]
where $I$ is the ideal generated by $\im(\dd + m_{1,1})$. Therefore this follows from Lemma~\ref{LEMMACOBARABELIAN}.
\end{proof}

By definition, $\cobarconst^{\mathrm{Adams}}(PA) = T^{\oplus}(s^{-1}\overline{PA})$, 
and $s^{-1}(\overline{PA})$ is the same as the truncation $\tau_{\ge 0} (s^{-1} A) = s^{-1} (\tau_{\ge 1} A)$. Furthermore $\tau_{\ge 1}$ is oplax monoidal, neglecting the unit, i.e.\@ a functor of cooperads
\[ \tau_{\ge 1}: ((\Ch(\mathcal{C}), \tildeotimes)^{\circ,\vee})  \to ((\Ch_{\ge 1}(\mathcal{C}), \tildeotimes)^{\circ, \vee}) \]
 hence induces a functor
\[ \tau_{\ge 1}: \Coalg^{\circ}(\Ch(\mathcal{C}), \tildeotimes) \to \Coalg^{\circ}(\Ch(\mathcal{C}), \tildeotimes). \]
For a morphism $A \to B$ of dg-coalgebras, we obtain a commutative diagram of dg-algebras
\[ \xymatrix{
\mlq T^{\prod} \mrq(s^{-1}A)  \ar[d] \ar[r] & \mlq T^{\prod} \mrq(s^{-1}B) \ar[d] \\
\mlq T^{\prod} \mrq (s^{-1}\overline{PA}) \ar[r] & \mlq T^{\prod}\mrq(s^{-1}\overline{PB})
}\]

More generally, any morphism of dg-algebras $\mlq T^{\prod} \mrq(s^{-1}A) \to \mlq T^{\prod} \mrq(s^{-1}B)$, i.e.\@ in particular an $A_{\infty}$-morphism of ($A_{\infty}$-)coalgebras $A \to B$, induces a unique morphism 
$\mlq T^{\prod} \mrq(s^{-1}\overline{PA}) \to \mlq T^{\prod} \mrq(s^{-1}\overline{PA})$, such that the analogous diagram commutes, but does not necessarily induce a morphism
$T^{\oplus}(s^{-1}\overline{PA}) \to T^{\oplus}(s^{-1}\overline{PA})$.

We define thus
\begin{DEF}\label{DEFCOBARHAT}
Assume that countable products exist in $\mathcal{C}$. Let $A$ be a  dg-coalgebra with coaugmention $1 \to A$.
\[ \widehat{\cobarconst}^{\mathrm{Adams}}(A) := T^{\prod}(s \overline{A}) \]
equipped with its natural algebra structure and with the differential from Lemma~\ref{LEMMAAINFTYCORRESPONDENCE}.
\begin{gather*}
 \widehat{\cobarconst}(A) := \coker\left ( 
  \substack{ \prod A_{[1]} \otimes \cdots \otimes A_{[2]} \otimes \cdots \otimes A_{[1]} \\ \oplus \\  \prod A_{[1]} \otimes \cdots \otimes A_{[0]} \otimes \cdots \otimes A_{[1]} } \rightarrow  \prod_{n=0}^{\infty} A_{[1]}^{\otimes n} \right).  
  \end{gather*}
\end{DEF}
We have then still
\begin{PROP}
 \[ \widehat{\cobarconst} \rho^* A  \cong T^{\prod}(A_{1}) / \widehat{I} \cong H_0 \circ \widehat{\cobarconst}^{\mathrm{Adams}} \circ P  \]
 where $\widehat{I}$ is the closure of the ideal $I$ of Lemma~\ref{LEMMACOBARABELIAN} and this is functorial in morphisms of $A_{\infty}$-coalgebras. 
\end{PROP}

\subsection{Classical (co)bar for complexes in Abelian categories}\label{SECTCOBARDG}

\begin{PAR}\label{PARADJCOBARAB}
Let $(\mathcal{C}, \otimes)$ be an Abelian tensor category. Apply the last section to $(\Ch_{\ge 0}(\mathcal{C}), \tildeotimes)$.
Composing with the adjunction \ref{PARDEC} we get an adjunction
\[  \xymatrix{  \Alg(\Ch_{\ge 0}(\mathcal{C})_{/1}, \tildeotimes) \ar@<3pt>[rrr]^-{\cobarconst \circ \rho^*  \circ \dec^*_{\tildeotimes}} & & &  \ar@<3pt>[lll]^-{ \dec_{\tildeotimes,*} \circ (\rho^*)^{-1} \circ \barconst} \Coalg(\Ch_{\ge 0}(\mathcal{C}), \tildeotimes)  } \]
where $\barconst$ and $\cobarconst$ are understood w.r.t.\@ the cofibration $(\Ch_{\ge 0}(\mathcal{C}), \tildeotimes) \to \OOO$.
Here $\dec^*_{\tildeotimes}$ is the decalage $\dec^*$ turned into a lax monoidal functor (=functor of cooperads) as described in Section~\ref{SECTDECTILDEOTIMES}.

The adjoint functors will be identified in this section in terms of the Eilenberg-MacLane bar and Adams cobar construction. 
\end{PAR}

\begin{SATZ}\label{SATZBAREM}
There is an isomorphism 
\[ \boxed{ \barconst^{EM} \cong \dec_{\tildeotimes, *} \circ (\rho^*)^{-1} \circ \barconst }   \] 
of functors\footnote{the LHS is the category of augmented algebras in $(\Ch_{\ge 0}(\mathcal{C}), \tildeotimes)$ which is the same as algebra objects in augmented objects.}
\[ \Alg(\Ch_{\ge 0}(\mathcal{C})_{/1}, \tildeotimes) \to \Coalg_{\mathrm{conn}}(\Ch_{\ge 0}(\mathcal{C}), \tildeotimes)  \]  
where $\barconst = \barconst_{(\Ch_{\ge 0}(\mathcal{C})_{/1}, \tildeotimes) \to \OOO}$ is the classical bar construction (Definition~\ref{DEFCOBARCLASS})\footnote{``$\barconst$'' commutes with the forgetful functor forgetting the augmentation (in the obvious sense), hence we didn't mention the augmentation in the index of ``$\barconst$'', but it is important to keep the augmentation along to be a able to interpret the result via $(\rho^*)^{-1}$ as a double complex. (Without the augmentation it is only a diagram of shape $\Delta^{\op} \times \Delta^{\op}_{\act}$).},
and where $\dec_* = \tot$ is the total complex functor considered as a monoidal functor as in \ref{PARDEC}.
Both sides are connected and hence canonically coaugmented coalgebras. 
\end{SATZ}
\begin{proof}
Let $A \in \mathcal{C}^{\Delta^{\op}}$ be an algebra w.r.t.\@ $\tildeotimes$, i.e.\@ a dg-algebra, with augmentation $A \to 1$.
By Lemma~\ref{LEMMABARAB}  applied to $(\mathcal{C}^{\Delta^{\op}}, \tildeotimes)$ there is an isomorphism
\[ (\rho^*)^{-1} \barconst(A) = T^{\coprod}_{(\mathcal{C}^{\Delta^{\op}}, \tildeotimes)}(s_r \overline{A}) \]
of double complexes (with rows $\overline{A} \tildeotimes \cdots \tildeotimes \overline{A}$), where $s_r$ is the shift w.r.t.\@ the second (column) index. 
We have to see that there is an isomorphism of dg-algebras: 
\[ \dec_* T^{\coprod}_{(\mathcal{C}^{\Delta^{\op}}, \tildeotimes)}(s_r \overline{A})  \cong  T^{\coprod}_{(\mathcal{C}, \otimes)}(s \overline{A}) \]
where $T^{\coprod}_{(\mathcal{C}, \otimes)}$ considers $s \overline{A}$ as a graded object and is equipped with the differential
\[ \dd_1 + \dd_2  \]
from Lemma~\ref{LEMMAAINFTYCORRESPONDENCE} and where $\dec_*$ is considered a monoidal functor as described in \ref{PARDEC}. 
Both sides are in a natural way direct sums of $A_{i_1} \otimes \cdots \otimes A_{i_m}$ in degree $n=m+\sum_k i_k$.
However, because of the sign in the coalgebra structure on the left hand side (cf.\@ \ref{PARDEC}), the identity would not be compatible with the coalgebra structures,
but we have an epimorphism of graded objects  
\[ p:  \dec_* T^{\coprod}_{(\mathcal{C}^{\Delta^{\op}}, \tildeotimes)}(s_l \overline{A})  \to s \overline{A} \]
by the obvious projection. Like in (the dual of) Lemma~\ref{LEMMATENSORCOALG} it extends to an isomorphism  
\[ p':  \dec_* T^{\coprod}_{(\mathcal{C}^{\Delta^{\op}}, \tildeotimes)}(s_l \overline{A})  \to T^{\coprod}(s \overline{A}) \]
of graded coalgebras, introducing the sign $(-1)^{\sum_k (m-k)\deg i_{m-k}}$ on the summand $A_{i_1} \otimes \cdots \otimes A_{i_m}$.
Finally observe that, for $A_n$ in complex bidegree $(n, 1)$ resp.\@ in degree $n+1$, the following commutes:
\[ \xymatrix{  A_{n} \ar@{=}[r] \ar[d]_{(-1)^{n} m_{i_1,i_2}} & A_{n} \ar[d]^{ \dd_{2,i_1-1,i_2-1}}  \\
A_{i_1} \otimes A_{i_2} \ar[r]_{(-1)^{i_2}} & A_{i_1} \otimes A_{i_2}
} \]
where the $(-1)^{n} m_{i_1,i_2}$ is the second summand of the differential on the total complex. 
\end{proof}

\begin{PAR}\label{PARPREPCOBARADAMS}
Let $B \in \Coalg(\Ch_{\ge 0}(\mathcal{C}), \tildeotimes)$. Apply Lemma~\ref{LEMMACOBARABELIAN} to the double complex $A:=\dec^* B$ considered as object
 in $\Coalg(\mathcal{D}^{\Delta^{\op}}, \tildeotimes)$ with $\mathcal{D} = \Ch_{\ge 0}(\mathcal{C})$.
 We have by Proposition~\ref{PROPEXPLICITAB}, 2.:
 \[ A_1 = s^{-2}B \oplus s^{-1}B \quad A_{2} = s^{-3}B \oplus s^{-2}B \]
(not a direct sum as complexes!), where $s^{-1}$ is considered as endomorphism of $\Ch_{\ge 0}(\mathcal{C})$ (i.e.\@ really $\tau_{\ge 0} s^{-1}$) and 
the differential (between the columns) is 
 \[  \xymatrix{ s^{-3}B \oplus s^{-2}B \ar[rr]^{\Mat{0 & 1 \\ 0 & 0}} &&  s^{-2}B \oplus s^{-1}B } \]
 and the 1,1-component of the comultiplication is 
 \[  \xymatrix{ s^{-3}B \oplus s^{-2}B  \ar[rr]^{} &&  (s^{-2}B \oplus s^{-1}B) \otimes (s^{-2}B \oplus s^{-1}B)  } \]
 whose components are by Lemma~\ref{LEMMADEC2}, 3.\@ given by: 
 \begin{align*}  m_{1,1}:  B_{i+j+3} \oplus B_{i+j+2} &\to  (B_{i+2}  \otimes B_{j+2})   \oplus (B_{i+2} \otimes B_{j+1})  
\oplus (B_{i+1} \otimes B_{j+2} ) \oplus ( B_{i+1} \otimes  B_{j+1})   \\
   (b_{n+3},b_{n+2}) &\mapsto  (0,\beta_{i+2,j+1}(b_{n+3}), (-1)^{j+1}  \beta_{i+1,j+2}(b_{n+3}), (-1)^{j}  \beta_{i+1,j+1}(b_{n+2})).
 \end{align*}
 \end{PAR}

\begin{SATZ}\label{SATZCOBARADAMS2}
Let $(\mathcal{C}, \otimes)$ be an Abelian tensor category with countable colimits such that $\otimes$ commutes with them. 
There is an isomorphism 
 \[ \boxed{ \cobarconst^{\mathrm{Adams}} \circ P  \cong \cobarconst \circ \rho^* \circ \dec^*_{\tildeotimes} } \]
of functors\footnote{the RHS is the category of augmented algebras in $(\Ch_{\ge 0}(\mathcal{C}), \tildeotimes)$ which is the same as algebra objects in augmented objects.}
\[ \Coalg(\Ch_{\ge 0}(\mathcal{C}), \tildeotimes)  \to \Alg(\Ch_{\ge 0}(\mathcal{C})_{/1}, \tildeotimes)   \]  
where $\cobarconst = \cobarconst_{(\mathcal{C}^{\Delta^{\op}}, \tildeotimes) \to \OOO}$ is the classical cobar construction (Definition~\ref{DEFCOBARCLASS}), and where
 $\dec^*_{\tildeotimes}$ is $\dec^*$ considered as a lax monoidal functor (= functor of cooperads) as in Section~\ref{SECTDECTILDEOTIMES}. 
\end{SATZ}
\begin{proof}
This follows from Proposition~\ref{PROPCOBARABDEC} below.
\end{proof}

\begin{PROP}\label{PROPCOBARABDEC}
Let $\mathcal{C}$ be an Abelian tensor category with countable colimits such that $\otimes$ commutes with them.
 In the situation of Lemma~\ref{LEMMACOBARABELIAN} for $\mathcal{C}^{\Delta^{\op}}$ with tensor product $\tildeotimes = \dec_* - \boxtimes -$ (i.e.\@ under Dold-Kan the usual tensor product on complexes) there is an isomorphism of dg-algebras
\begin{eqnarray*} \phi:  T^{\oplus}((\dec^*B)_1)/ I  &\rightarrow& T^{\oplus}(s^{-1}B) \\
 b_{n+2},b_{n+1} &\mapsto& (-1)^{n} s^{-1}(b_{n+1}) + \dd_2 s^{-1}(b_{n+2}) 
\end{eqnarray*} 
where $I$ is the ideal from Lemma~\ref{LEMMACOBARABELIAN} and where $T^{\oplus}(s^{-1}B)$ is equipped with the restriction of the differential $\dd = \dd_1 + \dd_2$ induced by the coalgebra structure on $B$ as described in Lemma~\ref{LEMMAAINFTYCORRESPONDENCE} and where $s^{-1}$ is the endomorphism of $\Ch(\mathcal{C}_{\ge 0})$ (i.e.\@ really $\tau_{\ge 0}s^{-1})$. 
\end{PROP}

\begin{proof}
Recall that $\dd_1 = -\dd$ and $\dd_2: s^{-1}B \to s^{-1}B \tildeotimes s^{-1}B$ is the comultiplication twisted with sign: $\dd_{2,i+1,j+1} = (-1)^i \beta_{i,j}$ (both extended to $T(s^{-1}B)$ using the graded Leibniz rule). 
We first check that $\phi$ is compatible with the differential and maps $I$ to zero. We omit the $s^{-1}$ for better readability. 

$I$ is  generated by elements of the form 
\[   (b_{n+2},0) + \sum_{i+j=n}(0,\beta_{i+2,j+1}(b_{n+3}), (-1)^{j+1}  \beta_{i+1,j+2}(b_{n+3}), (-1)^{j}  \beta_{i+1,j+1}(b_{n+2})) \]
which are mapped by $\phi$ to 
\begin{gather*}
  \sum_{i+j=n} \left( (-1)^{i+1}  \beta_{i+1,j+1}(b_{n+2})+ (-1)^{n+j}  \beta_{i+1,j+1}(b_{n+2})  \right)  \\
+ \sum_{k+l+m+2=n} (-1)^{m} \left(((-1)^{k+1} \beta_{k+1,l+1} \otimes 1)\beta_{k+l+2,m+1}(b_{n+3}) \right. \\
+ \left. (-1)^{k}((-1)^{l+1}  1 \otimes \beta_{l+1,m+1} ) (-1)^{l+m+1} \beta_{k+1,l+m+2}(b_{n+3}) \right) = 0. 
\end{gather*}

We have to check the commutativity of
\[ \xymatrix{
B_{n+2} \oplus B_{n+1}  \ar[r]^-{\phi} \ar[d]_{\Mat{\dd& (-1)^n \\ & \dd}} & T^{\oplus}(s^{-1}B)_{n+2} \ar[d]^{\dd_1+\dd_2}  \\
B_{n+1} \oplus B_{n}  \ar[r]^-{\phi} & T^{\oplus}(s^{-1}B)_{n+1}  \\
} \]
The upper composition maps  a pair $(b_{n+2},b_{n+1})$  to 
\[ -\dd (-1)^{n} b_{n+1} - \dd \dd_2 b_{n+2} + \dd_2 (-1)^{n} b_{n+1} +  (\dd_2)^2 b_{n+2}. \]
The lower composition maps it  to 
\[  (-1)^{n-1} \dd b_{n+1}  +  \dd_2 (\dd b_{n+2}+(-1)^n b_{n+1}). \]
This is the same using $(\dd_2)^2 = 0$ and $\dd \dd_2 = -\dd_2 \dd$.
Defining
\begin{eqnarray*} \psi: T^{\oplus}(s^{-1}B)   &\rightarrow& T^{\oplus}((\dec^*B)_1)/ I  \\
 b_{n+1} &\mapsto& (0,(-1)^n b_{n+1})
\end{eqnarray*} 
we have $\phi \psi = \id$ and 
\begin{gather*}
 (\psi \phi - \id)(b_{n+2},b_{n+1})  = (-b_{n+2},0) +  \psi \dd_2 b_{n+2} \\
=  (-b_{n+2},0) + (-1)^n\sum_{i+j=n}\left(0,0,0,(-1)^{i+1}\beta_{i+1,j+1}(b_{n+2})\right)  
\end{gather*}
with the same decomposition in tensor-degree two as above. This is obviously in $I$.
\end{proof}

\begin{KOR}
Let $(\mathcal{C}, \otimes)$ be an Abelian tensor category with countable colimits such that $\otimes$ commutes with them. 
Then there is an adjunction
\[ \xymatrix{ \Alg(\Ch_{\ge 0}(\mathcal{C})_{/1}, \tildeotimes) \ar@<3pt>[rr]^-{\barconst^{EM}} && \ar@<3pt>[ll]^-{\cobarconst^{\mathrm{Adams}}} \Coalg_{\mathrm{conn}}(\Ch_{\ge 0}(\mathcal{C}), \tildeotimes).   } \]
\end{KOR}
\begin{proof}
This is the adjunction~\ref{PARADJCOBARAB}.
\end{proof}

\subsection{Comparison of classical and derived (co)bar}\label{SECTCOMPLURIE}

As we have seen, the bar constructions $\overline{W}$ and $\barconst^{EM}$ are given by our abstract ``classical bar construction'' for $I = \mathcal{S} = \OOO$ followed by the (totalization) functor $\dec_*$.
By definition\footnote{Lurie defines them differently, but we have seen in Corollary~\ref{KORCOBAR} that the definitions agree} the derived bar construction is given by this ``classical bar construction''
followed by a relative left Kan extension which (in the augmented case) can be computed as colimit over $\Delta^{\op}$. Thus, simply put, $\overline{W}$ and $\barconst^{EM}$ give Lurie's bar construction because
$\dec_*$ computes the (derived) colimit over $\Delta^{\op}$ (this is true essentially because of the Eilenberg-Zilber theorem, cf.\@ Section~\ref{SECTGEOMREAL}). This section is devoted to an elaboration of this explanation. 
 
Denote $\mathrm{S} = \Set^{\Delta^{\op}}[\mathcal{W}^{-1}] \cong \Gpd_{\infty}$ (localizations always as $\infty$-categories).

\begin{PROP}\label{PROPCOMPLURIE1}
We have a commutative diagram
\[ \xymatrix{ 
\mathrm{Mon}^{\Delta^{\op}} \ar[r] \ar[d]_{(\rho^*)^{-1} \circ \barconst_{(\Set^{\Delta^{\op}}, \times) \to \OOO}} & \Alg(\mathrm{S}, \times) \ar[d]^{(\rho^*)^{-1} \circ \barconst_{(\mathrm{S}, \times) \to \OOO}} \\
 \Set^{\Delta^{\op} \times \Delta^{\op}} \ar[r] \ar[d]_{\dec_*} & \mathrm{S}^{\Delta^{\op}} \ar[d]^{\colim_{\Delta^{\op}}}  \\
 \Set^{\Delta^{\op}} \ar[r] & \mathcal{S}
}\]
\end{PROP}
\begin{proof}
The commutation of the upper diagram is clear, and the lower follows from the isomorphism $\dec_* \cong \delta^*$ up to coherent weak equivalences (non-Abelian Eilenberg-Zilber Theorem~\ref{SATZEZ}) and the fact that $\delta^*$ computes the colimit in the localization (i.e.\@ the homotopy colimit) by Proposition~\ref{PROPHOCOLIM}.
\end{proof}

\begin{KOR}
We have
\[ \boxed{ \barlurie \cong  \dec_* \circ (\rho^*)^{-1} \circ \barconst  \cong \overline{W} }  \]
as functors
\[ \mathrm{Mon}^{\Delta^{\op}}  \to \mathrm{S} \]
where $\barlurie = \barlurie_{(\mathrm{S}, \times) \to \OOO}$ is Lurie's Bar construction (Definition~\ref{DEFCOBARLURIE}). 
\end{KOR}
\begin{proof}
We have seen in \ref{PARCOBARALG} that
\[ \barlurie = \colim_{\Delta^{\op}} \circ (\rho^*)^{-1} \circ \barconst \]
(both bar constructions w.r.t.\@ the cofibration $(\mathrm{S}, \times) \to \OOO$). Therefore this follows from  Proposition~\ref{PROPCOMPLURIE1}. 
\end{proof}

\begin{PAR}
Let $(\mathcal{C}, \otimes)$ be an Abelian tensor category.
Denote by $\mathcal{W}$ the  quasi-isomorphisms in $\Ch_{+}(\mathcal{C})$ (bounded below complexes). Assume that $\tildeotimes$ has
a left derived functor $\tildeotimes^{L}$ which is again associative. 
Denote $\mathcal{D}_+(\mathcal{C}) := \Ch_{+}(\mathcal{C})[\mathcal{W}^{-1}]$ the localization as $\infty$-category.
We assume the very mild conditions of \cite[7.9.1]{Cis19} ensuring that localization and taking functor categories commute (i.e.\@ homotopy colimits become $\infty$-categorical colimits). Those conditions  are 
for example satisfied, if $(\Ch_{+}(\mathcal{C}), \mathcal{W})$ enhances to a model category structure. This is true under very mild assumptions on $\mathcal{C}$.

Assume that there is a full subcategory $\Ch_{+}(\mathcal{C})^{\mathrm{split}} \subset \Ch_{+}(\mathcal{C})$ (closed under biproducts, $\otimes$, and containing 1) such that the restriction of the inclusion gives rise to a
{\em monoidal} functor
\begin{equation}\label{eqderived}  
(\Ch_{+}(\mathcal{C})^{\mathrm{split}}, \tildeotimes) \to  (\mathcal{D}_+(\mathcal{C}), \tildeotimes^L).
\end{equation}
It thus induces a functor of cooperads
\[  (\Ch_{+}(\mathcal{C})^{\mathrm{split}}, \tildeotimes)^{\vee} \to  (\mathcal{D}_+(\mathcal{C}), \tildeotimes^L)^{\vee}  \]
and finally a functor
\[  \Coalg(\Ch_{+}(\mathcal{C})^{\mathrm{split}}, \tildeotimes) \to \Coalg(\mathcal{D}_+(\mathcal{C}), \tildeotimes^L).  \]
 \end{PAR}

\begin{PROP}\label{PROPCOMPLURIE2}
We have a commutative diagram
\[ \xymatrix{ 
\Alg(\Ch_{+}(\mathcal{C})^{\mathrm{split}}_{/1}, \tildeotimes) \ar[r] \ar[d]_{(\rho^*)^{-1} \circ \barconst_{(\Ch_{+}(\mathcal{C}), \tildeotimes) \to \OOO}} & \Alg(\mathcal{D}_+(\mathcal{C})_{/1}, \tildeotimes^L) \ar[d]^{(\rho^*)^{-1} \circ \barconst_{(\mathcal{D}_+(\mathcal{C}), \tildeotimes^L) \to \OOO}} \\
 \Coalg((\Ch_{+}(\mathcal{C})^{\mathrm{split}}_{/1})^{\Delta^{\op}}, \tildeotimes) \ar[r] \ar[d]_{\dec_{\tildeotimes,*}} & \Coalg(\mathcal{D}_+(\mathcal{C})^{\Delta^{\op}}_{/1}, \dec_* -\tildeboxtimes^L -) \ar[d]^{\colim_{\Delta^{\op}}}  \\
 \Coalg(\Ch_{+}(\mathcal{C})^{\mathrm{split}}_{/1}, \tildeotimes) \ar[r] & \Coalg(\mathcal{D}_+(\mathcal{C})_{/1}, \tildeotimes^L)
}\]
\end{PROP}
Recall from the proof of Theorem~\ref{THEOREMLURIE} that the  functor denoted $\colim_{\Delta^{\op}}$ is in fact a relative left Kan extension along the functor $(\Delta, \ast)^{\op} \to \OOO^{\op}$ of cooperads, which exists, and is computed fiber-wise (i.e.\@ here commutes with forgetting the coalgebra structure) because this is an exponential fibration of cooperads that is $\infty$-coCartesian. 
\begin{proof}
The upper diagram commutes because the restriction of the localization (\ref{eqderived}) is monoidal by assumption. 

Consider the diagram
\[ \xymatrix{ 
 \Coalg(\Ch_{+}(\mathcal{C})^{\mathrm{split}}_{/1}, \tildeotimes) \ar[r] \ar@<-3pt>[d]_{\dec_{\tildeotimes}^?} \ar@<3pt>[d]^{\pi^*} & \Coalg(\Ch_{\ge 0}(\mathcal{C})[\mathcal{W}^{-1}], \tildeotimes^L)  \ar[d]^{\pi^*}  \\
 \Coalg((\Ch_{+}(\mathcal{C})^{\mathrm{split}}_{/1})^{\Delta^{\op}}, \tildeotimes) \ar[r]  & \Coalg(\mathcal{D}_+(\mathcal{C})^{\Delta^{\op}}, \dec_* -\tildeboxtimes^L -) 
}\]
in which the square involving the $\pi^*$ is clearly commutative. Furthermore, also $\dec_{\tildeotimes}^?$ is a functor of cooperads (Lemma~\ref{LEMMADECQMTILDEOTIMES}) and there is a natural transformation
\[ \pi^* \Rightarrow \dec_{\tildeotimes}^?  \]
which is point-wise a quasi-isomorphism  (Lemma~\ref{LEMMADECQM}). Therefore there is also a commutative square involving $\dec_{\tildeotimes}^?$. Now consider the mate (passing to the left adjoints vertically):
\[ \xymatrix{ 
 \Coalg((\Ch_{+}(\mathcal{C})^{\mathrm{split}}_{/1})^{\Delta^{\op}}, \tildeotimes) \ar[r] \ar[d]_{\dec_{\tildeotimes,*}} \ar@{}[rd]|{\Swarrow} & \Coalg(\mathcal{D}_+(\mathcal{C})^{\Delta^{\op}}, \dec_* -\tildeboxtimes^L -)  \ar[d]^{\colim_{\Delta^{\op}}} \\
 \Coalg(\Ch_{+}(\mathcal{C})^{\mathrm{split}}_{/1}, \tildeotimes) \ar[r]    & \Coalg(\mathcal{D}_+(\mathcal{C}), \tildeotimes^L)  
}\]
We have to show that the natural transformation is an isomorphism. However, this can be shown forgetting the coalgebra structure, i.e.\@ we have to show that 
\[ \xymatrix{ 
 (\Ch_{+}(\mathcal{C})^{\mathrm{split}}_{/1})^{\Delta^{\op}} \ar[r] \ar[d]_{\dec_{*}} \ar@{}[rd]|{\Swarrow} & \mathcal{D}_+(\mathcal{C})^{\Delta^{\op}}  \ar[d]^{\colim_{\Delta^{\op}}} \\
 \Ch_{+}(\mathcal{C})^{\mathrm{split}}_{/1} \ar[r]   & \mathcal{D}_+(\mathcal{C})   
}\]
commutes. However this is the restriction of 
\[ \xymatrix{ 
 (\Ch_{+}(\mathcal{C})_{/1})^{\Delta^{\op}} \ar[r] \ar[d]_{\dec_{*}} \ar@{}[rd]|{\Swarrow} & \mathcal{D}_+(\mathcal{C})^{\Delta^{\op}}  \ar[d]^{\colim_{\Delta^{\op}}} \\
 \Ch_{+}(\mathcal{C})_{/1} \ar[r]   & \mathcal{D}_+(\mathcal{C})   
}\]
where the RHS is a localization of the LHS (the first line at the object-wise quasi-isomorphisms, see \cite[Proposition~7.9.1]{Cis19}) and in the adjunction $\dec_{*}, \dec^?$ both functors preserve (object-wise) quasi-isomorphisms.
Therefore it descends to an adjunction between the localizations and (the descended) $\dec^?$ is isomorphic to $\pi^*$ --- thus also the descended $\dec_*$ is isomorphic to $\colim_{\Delta^{\op}}$.
\end{proof}

\begin{KOR}
We have
\[ \boxed{ \barlurie \cong  \dec_* \circ (\rho^*)^{-1} \circ \barconst  \cong \barconst^{EM} }  \]
as functors
\[ \Alg(\Ch_{+}(\mathcal{C}^{\mathrm{split}})_{/1}, \tildeotimes) \to \Coalg(\mathcal{D}_{+}(\mathcal{C}), \tildeotimes^L) \]
where $\barlurie = \barlurie_{(\mathcal{D}_{+}(\mathcal{C})_{/1}, \tildeotimes^L) \to \OOO}$ is Lurie's Bar construction (Definition~\ref{DEFCOBARLURIE}). 
\end{KOR}
\begin{proof}
We have seen in \ref{PARCOBARALG} that
\[ \barlurie = \colim_{\Delta^{\op}} \circ (\rho^*)^{-1} \circ \barconst \]
(both bar constructions w.r.t.\@ the cofibration $(\mathcal{D}_{+}(\mathcal{C}_{/1}, \tildeotimes^L) \to \OOO$). Therefore this follows from Proposition~\ref{PROPCOMPLURIE2}. 
The  isomorphism with $\barconst^{EM}$ is a slight generalization of Theorem~\ref{SATZBAREM} to bounded below complexes. 
\end{proof}

\subsection{An intermediate cobar construction}

Let $(\mathcal{C}, \otimes)$ be an Abelian tensor category with countable colimits such that $\otimes$ commutes with them.
We may also compute the classical cobar construction w.r.t.\@ the cofibration of operads $(\mathcal{C}^{\Delta^{\op}}, \otimes) \to \OOO$ (not $\tildeotimes$ !).
$(\mathcal{C}^{\Delta^{\op}}, \otimes)$ is again an Abelian tensor category and thus Lemma~\ref{LEMMACOBARABELIAN} applies giving
 \[ \cobarconst \circ \rho^* \circ \dec^*_{\otimes} A = T^{\oplus}_{(\mathcal{C}^{\Delta^{\op}}, \otimes)}((\dec^* A)_{1}) / I, \]
 where $\dec_{\otimes} = \mathfrak{AW}_2 \dec^*$ is the lax monoidal extension of $\dec^*$  
\[ (\mathcal{C}^{\Delta^{\op}}, \otimes)^{\vee} \rightarrow (\mathcal{C}^{\Delta^{\op} \times \Delta^{\op}}, \utildeotimes)^{\vee}  \]
 discussed in \ref{SECTDECOTIMES}, where $\utildeotimes := \dec_* - \boxtimes -$, and $\boxtimes$ is $\otimes$ applied point-wise in the first variable. 
 We will not calculate this further in any explicit manner, because $\otimes$ is complicated in terms of complexes. It will turn up as an
 intermediate step in comparing the Adams cobar construction and the geometric cobar construction in section \ref{SECTFUNCT}. 
 Actually, it compares very easily to the geometric cobar construction: 
  
\begin{LEMMA}\label{LEMMAINTERMEDIATECOBAR}
The following commutes
\[ \xymatrix{ \Set^{\Delta^{\op}} \ar[d]_{\cobarconst \circ \rho^* \circ \dec^*} \ar[r]^-{\Z[-]} &  \Coalg(\Ab^{\Delta^{\op}}, \otimes) \ar[d]^{\cobarconst_{\otimes} \circ \rho^* \circ \dec^*_{\otimes}} \\
\mathrm{Mon}^{\Delta^{\op}} \ar[r]_-{\Z[-]} \ar[r] & \Alg(\Ab^{\Delta^{\op}}_{/\Z}, \otimes)
 }\]
 where on the left the cobar construction w.r.t.\@ $(\Set^{\Delta^{\op}}, \times) \to \OOO$ and on the right the cobar construction w.r.t.\@ $(\Ab^{\Delta^{\op}}_{/\Z}, \otimes) \to \OOO$ is used. 
\end{LEMMA}
\begin{proof}
The diagram in question is the outer rectangle in: 
\[ \xymatrix{ \Set^{\Delta^{\op}} \ar[d]_{  \dec^*} \ar[r]^-{\Z[-]} &  \Coalg(\Ab^{\Delta^{\op}}, \otimes) \ar[d]^{  \dec^* } \\
 \Coalg(\Set^{\Delta^{\op}\times \Delta^{\op}}, \times) \ar[d]_{\AWfrak_2}^{\sim} \ar[r]^-{\Z[-]} &  \Coalg(\Ab^{\Delta^{\op} \times \Delta^{\op}}, \otimes) \ar[d]^{\AWfrak_2} \\
 \Coalg(\Set^{\Delta^{\op}\times \Delta^{\op}}, \utildetimes) \ar[d]_{\cobarconst \circ \rho^*} \ar[r]^-{\Z[-]} &  \Coalg(\Ab^{\Delta^{\op} \times \Delta^{\op}}, \utildeotimes) \ar[d]^{\cobarconst_{\otimes} \circ \rho^*} \\
\mathrm{Mon}^{\Delta^{\op}} \ar[r]_-{\Z[-]} \ar[r] & \Alg(\Ab^{\Delta^{\op}}_{/\Z}, \otimes)
 }\]
The left morphism  $\AWfrak_2$ is an isomorphism because of the trivial Eilenberg-Zilber Theorem~\ref{SATZEZTRIVIAL}.
In the third horizontal morphism a coalgebra with structure map $\dec^* X \to  X  \boxtimes X$ is mapped to $\dec^* \Z[X]$ with structure map
$\dec^* \Z[X] = \Z[\dec^* X] \to \Z[X  \boxtimes X] = \Z[X] \boxtimes \Z[X]$.
The commutativity of the two upper squares is clear and the commutativity of the lower square stems from the
commutativity of $\Z[-]$ with colimits (being a left adjoint) and its monoidality $(\Set^{\Delta^{\op}}, \times) \to (\Ab^{\Delta^{\op}}, \otimes)$.
\end{proof}

\begin{BEM}
The Lemma may be expressed by saying that there is an isomorphism: 
\[ \boxed{ \Z[-] \circ M^{\mathrm{Kan}} \cong (\cobarconst_{\otimes} \circ \rho^* \circ \dec^*_{\otimes}) \circ \Z[-] } \]
i.e.\@ the complex of normalized singular chains of the geometric cobar construction of $X$ is the algebraic cobar construction of
the complex of normalized singular chains of $X$. With the caveat that this is the algebraic cobar construction w.r.t.\@ $\otimes$ and not
w.r.t.\@ $\tildeotimes$ (which would be the Adams cobar construction). Nevertheless, the latter two can be compared, which is --- perhaps surprisingly --- 
fairly intricate. We will do this in section~\ref{SECTFUNCT}.
\end{BEM}

\subsection{Coherent vs.\@ $A_\infty$-transformations}\label{SECTCOHERENTAINFTY}

\begin{PROP}\label{PROPCOHERENTAINFTY}
Consider the cooperad\footnote{Recall that $(-)^{\circ}$ means that we neglect the counits.} $(\mathcal{C}^{\Delta^{\op}}, \tildeotimes)^{\circ, \vee}$ with its simplicial enrichment
$F\Aw^* \underline{\Hom}^{\tildeotimes}$  (cf.\@ Lemma~\ref{LEMMAYONEDA}).
Then for any pair of (non-counital) dg-coalgebras $X, Y$
we have a natural morphism
\[ \boxed{ \uCoh(X, Y)_0 \to \Hom^{A_{\infty}}(X, Y) }  \]
which restricted to $\Hom(X, Y)$ is the usual embedding of coalgebra morphisms into $A_{\infty}$-coalgebra morphisms (Lemma~\ref{LEMMAAINFTYCORRESPONDENCE}). 
The functor is compatible with the composition defined in Section~\ref{SECTCOMPCOH} on the left hand side and
composition of morphisms of $A_{\infty}$-coalgebras on the right. 
\end{PROP}

\begin{BEM}
There is an algebra-version of the above as well, which will not be needed. We leave the details to the reader. In fact, the whole construction works for unbounded complexes and so 
the other argument is dual. 
\end{BEM}

\begin{proof}
Let $\mu \in \uCoh(F, G)_0$ be a coherent transformation. We have to define morphisms
\[ \alpha_n: s^{-1}F \to (s^{-1}G)^{\tildeotimes n} \qquad n \ge 1 \]
of degree 0 satisfying $\dd \alpha = \alpha \dd$ (for the extension of $\alpha:=\mlq \prod \mrq \alpha_n$ to $\mlq T^{\prod} \mrq (s^{-1} X) \to \mlq T^{\prod} \mrq (s^{-1} Y)$), or
more elementarily: 
\begin{equation} \label{eqalpha}  \sum_{i+j=k} (\alpha_{i} \otimes \alpha_{j}) \dd_2  + \alpha_{k+1} \dd_1 = \dd_2 \alpha_{k} + \dd_1 \alpha_{k+1}.     \end{equation}
Consider the set of factorizations of $p_{k+1}: [1] \to [k+1]$ in $\Delta^{\circ}_{\act} = \OOO^{\circ,\op}_{\act}$ of the form
\[ \xymatrix{  [1] \ar@{=}[r]  &  [1] \ar[r]^-{\delta_{i_{k-1}+1}} & \cdots   \ar[r]  & [k]  \ar[r]^-{\delta_{i_{0}+1}}  & [k+1] \ar@{=}[r] & [k+1]     } \]
in $\Delta_{\act}^{\circ, \op}$, i.e.\@ a sequence of injections with endpoints fixed, where all injections are non-identities, and thus necessarily of the form $\delta_i$. 
We have $0 \le i_n < k-n$. As in (\ref{PARREDUCT}) we see them either as sequences $\underline{i} = (i_0, \dots, i_{k-1})$ or as a leveled tree. 
In total, there are $k!$ such factorizations which constitute non-degenerate $k$-simplices of $N_{p_{k+1}}([1] \times_{/\Delta^{\circ}_{\act}} \Delta^{\circ}_{\act} \times_{/\Delta^{\circ}_{\act}} [k+1])_k$.
 
Each $\mu(\underline{i})$ defines an element in $\uHom_{k}^{\tildeotimes}(F, G^{\tildeotimes k+1})$ which we consider as a degree $k$ morphism
$F\to G^{\tildeotimes (k+1)}$. 
We can thus define a  degree 0 morphism:
\begin{eqnarray*}
 \alpha_{k+1}: s^{-1}F &\to& (s^{-1}G)^{\tildeotimes k+1} \\
s^{-1}x & \mapsto & (s^{-1})^{\tildeotimes (k+1)}  \sum_{\underline{i}}  \sgn(\underline{i}) \mu(\underline{i})(x)  
\end{eqnarray*}
where is sum is over all $\underline{i}=(i_0, \dots, i_{k-1})$ with $0 \le i_j < k-j$ and where we used the Koszul sign convention (\ref{PARKOSZUL}) and where
\[ \mathrm{sgn}(\underline{i}) :=  (-1)^{\sum_{j=0}^{k-1} i_j}. \]
$\alpha_1$ is just the degree 0 map $F \to G$ evaluation at $[1] \to [1] \to [1]$.

We a left to show (\ref{eqalpha}).
Define for each $m=1, \dots, k-1$ an involution $\sigma_m$ on these factorizations given by
\[ (\delta_i, \delta_j) \mapsto \begin{cases} (\delta_{j-1}, \delta_i) & i < j \\ (\delta_j, \delta_{i+1}) & i \ge j \end{cases} \] 
at positions $m, m+1$. 
Each involution $\sigma_m$ changes the sign: 
\begin{equation}\label{eqsgninv}
\mathrm{sgn}(\sigma_m (\underline{i})) = -\mathrm{sgn}(\underline{i}).
\end{equation}

Since the maps $\mu_k: N_{p_{k+1}}([1] \times_{/\Delta^{\circ}_{\act}} \Delta^{\circ}_{\act} \times_{/\Delta^{\circ}_{\act}} [k+1])_k \to \uHom^{\tildeotimes}_{k}(F, G^{\tildeotimes (k+1)})$ are simplicial, we have\footnote{ using $\dd (s^{-1})^{\otimes k+1} = (-1)^k (s^{-1})^{\otimes k+1} \dd$.  }: 
\begin{align*}   (-1)^k (\dd \alpha_{k+1}   -  \alpha_{k+1} \dd )&= \sum_{j=0}^{k} (s^{-1})^{\tildeotimes (k+1)}  \sum_{\underline{i}}  (-1)^{j}\sgn(\underline{i}) \mu(\delta_j \underline{i})s  \\
 &= \sum_{\underline{i}}  (s^{-1})^{\tildeotimes (k+1)} \sgn(\underline{i}) \mu(\delta_0 \underline{i}) s +  \sum_{\underline{i}}  (s^{-1})^{\tildeotimes (k+1)} (-1)^{k} \sgn(\underline{i}) \mu(\delta_k \underline{i})s
\end{align*} 
because $\mu(\delta_m \underline{i}) = \mu(\delta_m \sigma_m \underline{i})$ for $m=1, \dots, k-1$ and the sign property (\ref{eqsgninv}).

$\delta_0 \underline{i}$ is equal to 
\[ \xymatrix{ [1] \ar[r]^-{\delta_1} & [2] \ar[r] & \cdots  \ar[r] & [k] \ar[r]^-{\delta_{i_0+1}} &   [k+1]  \ar@{=}[r]  &  [k+1].  } \]
Write this as $\delta_0 \underline{i} =   (\underline{i}_1, \underline{i}_2) \circ \delta_1$ where $\underline{i}_1$ and $\underline{i}_2$ are compositions of $\delta_i$'s and $\id$'s.
We have that $\underline{i}_1 = \sigma(\underline{i}_1')$ and $\underline{i}_2 = \tau(\underline{i}_2')$ for uniquely determined shuffle $\sigma: [{k-1}] \twoheadrightarrow [a], \tau:  [{k-1}] \twoheadrightarrow [b]$
for $a+b = k-1$ and where $\underline{i}_1'$ and $\underline{i}_2'$ are of the form previously considered, i.e.\@ non-degenerate with extremal identities. 
Because of the relations in the coherent end, we have that
\[ \mu(\underline{i}_1) \otimes \mu(\underline{i}_2) \in \uHom^{\tildeotimes}(F,G^{\tildeotimes a+1})_{[k]} \otimes \uHom^{\tildeotimes}(F,G^{\tildeotimes b+1})_{[k]}  \]
map to $\mu(\underline{i})$ under the composition
\begin{gather*}
\xymatrix{
 \uHom_{[k]}(F,G^{\otimes a+1}) \otimes \uHom_{[k]}(F,G^{\otimes b+1})  \ar[r]^{\Awfrak} &  (\uHom(F,G^{\tildeotimes a+1}) \tildeotimes \uHom(F,G^{\tildeotimes b+1}))_{[k]} } \\
\xymatrix{ \ar[r]  & \uHom_{[k]}(F,G^{\tildeotimes k+1}) }
\end{gather*}
where the last map is the composition with the comultiplication $F \to F^{\tildeotimes 2}$. By Lemma~\ref{LEMMAYONEDA} thus
 to the morphism
\[ F \to G^{\tildeotimes k+1}  \]
of degree $k$, given by
\[ (\mu(i_1') \tildeotimes \mu(i_2')) F(m) \]
with Koszul sign convention, if $\sigma = s_l$ and $\tau = s_r$ and other-wise to something degenerate. 
For these we have $\sgn(\underline{i}_1') \sgn(\underline{i}_2') = (-1)^{(a+1)b} \sgn(\underline{i})$
in such a way that 
\begin{align*}
& \sum_{\underline{i}}  (s^{-1})^{\tildeotimes (k+1)} \sgn(\underline{i}) \mu(\delta_0 \underline{i})  \\
=& \sum_{a+b=k-1}  (s^{-1})^{\tildeotimes (k+1)} (-1)^{(a+1)b} \sgn(\underline{i}_1)\sgn(\underline{i}_2)  (\mu(\underline{i}_1') \otimes \mu(\underline{i}_2')) F(m)   \\
=& - \sum_{a+b=k-1}  (-1)^{(a+1)b} \sgn(\underline{i}_1)\sgn(\underline{i}_2) \cdot \\
 & ((s^{-1})^{\tildeotimes (a+1)} \tildeotimes (s^{-1})^{\tildeotimes (b+1)})  (\mu(\underline{i}_1') \tildeotimes \mu(\underline{i}_2')) (s \tildeotimes s) (s^{-1} \tildeotimes s^{-1}) F(m)    \\
=& - \sum_{a+b=k-1}  (-1)^{(a+1)b}   (-1)^{a(b+1)} (-1)  (\alpha_a \tildeotimes \alpha_b)  \widetilde{F(m)}    \\
=& \sum_{a+b=k-1} (-1)^{k-1} (\alpha_a \tildeotimes \alpha_b) \widetilde{F(m)} 
 \end{align*}

$\delta_k \underline{i}$ is equal to 
\[ \xymatrix{ [1] \ar@{=}[r] & [1] \ar[r]^{\delta_{i_{k-1}+1}} & [2] \ar[r] &  \cdots \ar[r] &  [k]  \ar[r]^{\delta_{i_{0}+1}} &   [k+1].    } \]
Thus $\delta_k \underline{i} = \delta_{i_{0}+1} \circ  \underline{i}$ and therefore
\[  (1^{\tildeotimes i_0} \tildeotimes G(m) \tildeotimes 1^{\tildeotimes k-i_0-1})  \mu(\underline{i}') = \mu(\delta_0 \underline{i}) \]
where $\underline{i}' = (i_1, \dots, i_{k-1})$. 
and therefore
\[ \sum_{\underline{i}}  (-1)^{k}\sgn(\underline{i})  (s^{-1})^{\tildeotimes (k+1)} \mu(\delta_k \underline{i}) = \sum_{\underline{i}}   (s^{-1})^{\tildeotimes (k+1)}   (-1)^{k}\sgn(\underline{i})  (1^{\tildeotimes i_0} \tildeotimes G(m) \tildeotimes 1^{\tildeotimes k-i_0-1})  \mu(\underline{i}') \]
\[ = \sum_{\underline{i}}  (-1)^{k}\sgn(\underline{i}) (-1)^{i_0} (1^{\tildeotimes i_0} \tildeotimes (s^{-1} \otimes s^{-1}) G(m) s \tildeotimes 1^{\tildeotimes k-i_0-1})   (s^{-1})^{\tildeotimes k}   \mu(\underline{i}')  \]
\[ = \sum_{\underline{i}'}  (-1)^{k}\sgn(\underline{i}')  \widetilde{G(m)}   (s^{-1})^{\tildeotimes k}  \mu(\underline{i}')  \]
\[ =   (-1)^{k} \widetilde{G(m)}  \alpha_k \]
where $\widetilde{G(m)}$ now has been extended --- as derivation --- to $s^{-1} G$.

Putting everything together, we arrive at
\[ (-1)^k( \dd_G \alpha_{k+1}  - \alpha_{k+1} \dd_F)    = (-1)^k \widetilde{G(m)} \alpha_{k} - (-1)^k \sum_{i+j=k} (\alpha_{i} \otimes \alpha_{j}) \widetilde{F(m)}    \]
i.e.\@
\[  \sum_{i+j=k} (\alpha_{i} \otimes \alpha_{j}) \widetilde{F(m)}  -   \alpha_{k+1} \dd_F  = \widetilde{G(m)} \alpha_{k} - \dd_G \alpha_{k+1}      \]
so indeed $(\alpha_k)$ is a morphism of $A_{\infty}$-coalgebras, observing that (on $F$ say) $\dd_1 = -\dd$ and $\dd_2 = \widetilde{F(m)}$.

The proof that this association is compatible with composition is omitted for the moment. 
\end{proof}

\subsection{Functoriality (Szczarba and Hess-Tonks morphisms)}\label{SECTFUNCT}

\begin{PAR}In this section, we relate the following two objects for a simplicial set $X \in \Set^{\Delta^{\op}}$:
 \begin{align*} 
\EZfrak^{\vee} \circ  \Z[M^{\mathrm{Kan}}(X)] \cong& \, \EZfrak^{\vee} \circ \Z[\cobarconst \circ \rho^* \circ \dec^* X],  \\
  \cobarconst^{\mathrm{Adams}} \circ P \circ  \AWfrak \circ \Z[X] \cong& \, \cobarconst_{\tildeotimes}  \circ \rho^* \circ \dec^*_{\tildeotimes} \circ \AWfrak\circ  \Z[X].  
 \end{align*}
 which, in more traditional language, would be called $C M^{\mathrm{Kan}}(X)$ and $\Omega C(X)$ 
 --- the translation of the (co)algebra structures via AW and EZ understood.
\end{PAR}

\begin{PAR}
 Recall from the Abelian Eilenberg-Zilber Theorem~\ref{KOREZAB} that we have functors (with mate) of (co)operads: 
 \[ \AWfrak: (\Ab^{\Delta^{\op}}, \otimes)^{\vee} \to  (\Ab^{\Delta^{\op}}, \tildeotimes)^{\vee} \qquad  \AWfrak^{\vee}: (\Ab^{\Delta^{\op}}, \tildeotimes) \to  (\Ab^{\Delta^{\op}}, \otimes)  \]
 and
 \[ \EZfrak: (\Ab^{\Delta^{\op}}, \tildeotimes)^{\vee} \to  (\Ab^{\Delta^{\op}}, \otimes)^{\vee} \qquad  \EZfrak^{\vee}: (\Ab^{\Delta^{\op}}, \otimes) \to  (\Ab^{\Delta^{\op}}, \tildeotimes)  \]
\end{PAR}
 
\begin{PAR}
 We have seen in Lemma~\ref{LEMMAINTERMEDIATECOBAR} that (quite obviously because $\Z[-]$ is cocontinuous and monoidal): 
 \[ \cobarconst_{\otimes}  \circ \rho^* \circ  \dec^*_{\otimes} \circ \Z[X] \cong  \Z[\cobarconst \circ \rho^* \circ \dec^* X].   \]
Thus, it remains to compare the following two dg-algebras
 \begin{gather*} 
    \EZfrak^{\vee} \circ  \cobarconst_{\otimes}  \circ \rho^* \circ  \dec^*_{\otimes} \circ \Z[X], \\
 \cobarconst_{\tildeotimes}  \circ \rho^* \circ \dec^*_{\tildeotimes} \circ \AWfrak\circ  \Z[X].   
 \end{gather*}
 \end{PAR}

 \begin{DEF}\label{DEFSZCZARBAHESSTONKS}
\begin{enumerate}
\item  The following composition of morphisms of dg-algebras is called the {\bf Szczarba morphism} $\phi$:
\[   \xymatrix{
 \cobarconst_{\tildeotimes}  \circ \rho^* \circ \dec^*_{\tildeotimes} \circ \AWfrak \circ  \Z[X] \ar[d]^{\EZfrak-\text{functoriality (Proposition~\ref{PROPFUNCTCOBAR})}}  \\
\EZfrak^{\vee} \circ \cobarconst_{\otimes}  \circ \rho^* \circ \EZfrak_1 \circ \dec^*_{\tildeotimes}  \circ \AWfrak \circ \Z[X] \ar[d]_{\sim}^{\text{Lemma~\ref{LEMMADECCOMM}}} \\
\EZfrak^{\vee} \circ \cobarconst_{\otimes}  \circ \rho^* \circ  \dec^*_{\otimes} \circ \EZfrak \circ \AWfrak \circ \Z[X] \ar@{-->}[d]_{\Sh(\Z[X])}^{\text{The Shih-operator in disguise, see \ref{PARSHIHSZCZARBA}}}  \\
\EZfrak^{\vee} \circ \cobarconst_{\otimes}  \circ \rho^* \circ  \dec^*_{\otimes} \circ \Z[X] \\ 
 } \] 
 where we wrote $\dec^*_{\otimes} := \AWfrak_2 \circ \dec^*$.
\item  The following composition of morphisms of dg-algebras is called the {\bf Hess-Tonks morphism} $\psi$: 
\[   \xymatrix{
\EZfrak^{\vee} \circ \cobarconst_{\otimes}  \circ \rho^* \circ \dec^*_{\otimes} \circ \Z[X]  \ar[d]^{\AWfrak-\text{functoriality (Proposition~\ref{PROPFUNCTCOBAR})}}  \\
 \cobarconst_{\tildeotimes}  \circ \rho^* \circ \AWfrak_1 \circ \dec^*_{\otimes} \circ \Z[X] \ar@{-->}[d]_{\Sh(\Z[X])^{-1}}^{\text{The Shih-operator in disguise, see \ref{PARSHIHSZCZARBA}}}  \\
\widehat{\cobarconst}_{\tildeotimes}  \circ \rho^* \circ  \AWfrak_1 \circ \dec^*_{\otimes} \circ \EZfrak \circ \AWfrak \circ \Z[X]\ar[d]_{\sim}^{\text{Lemma~\ref{LEMMADECCOMM}}} \\
 \widehat{\cobarconst}_{\tildeotimes}  \circ \rho^* \circ \dec^*_{\tildeotimes} \circ \AWfrak \circ  \Z[X]  
 } \] 
\end{enumerate}
Here $\widehat{\cobarconst}_{\tildeotimes}$ is a natural completion of $\cobarconst$ defined in Definition~\ref{DEFCOBARHAT}. 
\end{DEF}
We will see later in Corollary~\ref{KORSZCZARBA} that, the first morphism is the same as the one given by Szczarba in \cite{Szc61} (up to different indexing convention, see \ref{BEMINDEX}). That there is a relation between the Shih and Szczarba constructions has been observed before, cf.\@ \cite{Fra21}. It is not completely clear (to me), however, how the
construction of the deformation in \cite{Fra21} relates to the construction in \ref{PARSHIHSZCZARBA} below. 

 It seems likely that the second morphism, optimistically called Hess-Tonks morphism, gives the morphism defined in \cite[2.1]{HT10} (at least up to the same reindexing) but this remains to be checked in detail.

\begin{BEM}Notice that, in 2., there is no direct map of the form
\[   \xymatrix{
 \cobarconst_{\tildeotimes}  \circ \rho^* \circ \AWfrak_1 \circ \dec^*_{\otimes} \circ \Z[X] \ar@{-->}[d]^{\lightning}  \\
 \cobarconst_{\tildeotimes}  \circ \rho^* \circ \dec^*_{\tildeotimes} \circ \AWfrak \circ  \Z[X]  
 } \] 
 because the analogous diagram to the one in Lemma~\ref{LEMMADECCOMM} with the Alexander-Whitney morphisms does not commute!
\end{BEM}

\begin{PAR}\label{PARSHIHSZCZARBA}
{\bf The transport of the Shih operator: }

It remains to see how the deformation $\Xi^{\vee}$ (cf.\@ Theorem~\ref{SATZCOHEZ}, 3.) gives rise to the dashed morphisms in Definition~\ref{DEFSZCZARBAHESSTONKS}.

Recall from~\ref{PARHIGHERSHIH} the construction of the higher Shih operators which constitute a coherent transformation $\Ezfrak^* H =\Ezfrak^*  \exp(L_{(\mathcal{C},\otimes)^{\vee}}(\Xi^{\vee})) \in \uCoh_0(\id, \EZfrak \circ \AWfrak)$ where on the source the discrete enrichment and on the destination the weak enrichment $F \uHom^{\tildeotimes}_{(\Ab^{\Delta^{\op}}, \otimes)^{\vee}}$ (\ref{PARENRICHMENTS}) is chosen. 

For a dg-coalgebra $C \in \Coalg(\Ab^{\Delta^{\op}}, \otimes)$, it yields a coherent transformation $\Ezfrak^* H(C) \in  \uCoh_0(C, \EZfrak \circ \AWfrak \circ C  )$ for the discrete enrichment on $\OOO^{\op}$ and for the weak enrichment $F \uHom^{\tildeotimes}_{(\Ab^{\Delta^{\op}}, \otimes)^{\vee}}$ on the target. 

Forgetting the counit, we get a coherent transformation
\[ \Ezfrak^* H(C) \in  \uCoh_0(C, \EZfrak \circ \AWfrak \circ C  ) \]
for the truncated weak enrichment $F \uHom^{\tildeotimes, t}_{(\Ab^{\Delta^{\op}}, \otimes)^{\circ, \vee}}$ (cf.\@ Definition~\ref{DEFTRUNC}) on the target. 
{\em Reason:} We have obviously $\Z[N([1] \times_{/\Delta^{\circ}_{\act}} \Delta^{\circ}_{\act} \times_{/\Delta^{\circ}_{\act}} [k+1])]_n = 0$ for $n > k$ because every element of degree $> k$ in this nerve must be degenerate.  

Proposition~\ref{PROPSIMPLICIALLYENRICHEDDEC} states that $\tau \dec^*_{\otimes}$ is a weakly enriched functor
\[ (\mathcal{C}^{\Delta^{\op}}, \otimes)^{\circ, \vee} \rightarrow (\mathcal{C}^{\Delta^{\op} \times \Delta^{\op} }, \utildeotimes)^{\circ, \vee}   \]
for the weak $(\Ab^{\Delta^{\op}}, \otimes)$-enrichment $\Hom^{\tildeotimes, t}_{(\Ab^{\Delta^{\op}}, \otimes)^{\circ, \vee}}$ on the left and the $(\Ab^{\Delta^{\op}}, \otimes)$-enrichment  $\Aw^* \Hom^{\tildeotimes, t}_{(\Ab^{\Delta^{\op} \times \Delta^{\op}}, \utildeotimes)^{\circ, \vee}}$ (second variable) on the right. 

The coherent transformation $\Ezfrak^* H(C) $ yields thus a coherent transformation 
\[ \tau \dec^*_{\otimes} \Ezfrak^* H(C) \in \uCoh_0(\tau \dec^*_{\otimes} C, \tau \dec^*_{\otimes} \EZfrak \circ \AWfrak \circ C ).  \]

Applying Proposition~\ref{PROPCOHERENTAINFTY} produces a morphism of $A_{\infty}$-coalgebras, by definition a morphism of dg-algebras: 
\[ \Sh(C):  \mlq T^{\prod} \mrq  (s \tau \dec^*_{\otimes} C)  \to  \mlq T^{\prod} \mrq ( s \tau \dec^*_{\otimes} \EZfrak \circ \AWfrak \circ C ).  \]
Note that because the first component $\Sh(C)_1$ is the identity, this morphism of $A_{\infty}$-coalgebras is invertible. 
\end{PAR}

\begin{SATZ}\label{SATZSZCZARBA}
The morphism $\Sh(C)$ constructed above is bounded (that is, it maps $\cobarconst$ to $\cobarconst$ (without completion)) and we have a commutative diagram:
\[ \tiny \xymatrix{
H^0(T^{\oplus} (s \tau \dec^*_{\otimes} \circ \EZfrak \circ \AWfrak \circ \Z[X]) \ar[d]^{\Sh(\Z[X])}  \ar[r]^-\sim & \bigoplus_{k=0}^{\infty} (\tau_{\ge 1}\dec^* \Z[X])_1^{\otimes k} / I' & \ar[l]^-{s_{\can}} s \tau_{\ge 1} \Z[X]   \ar[d]^{\Sz(\Z[X])} \\
H^0(T^{\oplus} (s \tau \dec^*_{\otimes} \circ \Z[X]))  \ar[r]^-\sim & \bigoplus_{k=0}^{\infty} (\tau_{\ge 1} \dec^* \Z[X])_1^{\otimes k} / I  & \ar[l]^{1+\varepsilon}_{\sim} \bigoplus_{k=0}^{\infty} (\tau_{\ge 1} \dec^* \Z[X])_{[1]}^{\otimes k} / \widetilde{I}   }  \]
where the  map $\Sz$ on the right has components 
\[ \Sz_k: X_{[k]} \to \bigoplus_{k=0}^{\infty} (\tau_{\ge 1} \dec^* \Z[X])_{[1]}^{\otimes k} \] 
given by
\[ \Sz_k(s^{-1}x) =\begin{cases} 
-1 + x & k=1\\
 \sum_{\underline{i}}   \sgn(\underline{i}^{\vee}) s_{\can} (\Sz_{\underline{i}}^0 \ast' \id_{\Delta_1}) (s^{-1}x) \otimes \cdots \otimes (\Sz_{\underline{i}}^k \ast' \id_{\Delta_1})(s^{-1}x)  & k > 1
\end{cases}
\]
for the Szczarba-operators $\Sz_{\underline{i}}^j$ defined in Definition~\ref{DEFSZCZARBA}.
\end{SATZ}

\begin{KOR}\label{KORSZCZARBA}
The Szczarba morphism in Definition~\ref{DEFSZCZARBAHESSTONKS} is up a different indexing (cf.\@ \ref{BEMINDEX}) the  morphism given by Szczarba in \cite{Szc61}, see also \cite[Theorem~7]{HT10}. 
\end{KOR}

\begin{proof}[Proof of Theorem~\ref{SATZSZCZARBA}.]
The components 
\[ (\Ezfrak^* H)_{\underline{i}}  \in  \uHom_{k}(\Z[X], \Z[X^{k+1}]) \]
of $\Ezfrak^* H$, i.e.\@ the higher Shih operators, are 
by Proposition~\ref{PROPHIGHERSHIH} given by
\[  (\Ezfrak^* H)_{\underline{i}}^n = \sum_{\underline{b}} \sgn(\sigma_{\underline{b}}) H_{\underline{b}, \underline{i}}^n.   \]
and by Proposition~\ref{PROPHIGHERSHIHEXPL} and Proposition~\ref{PROPSHIHSZCZARBAGEN1} we have that (the projection of) $\sgn(\sigma_{\underline{b}}) H_{\underline{b}, \underline{i}}^n$ is given by
\[ \sgn(\sigma_{\underline{b}})   \mathcal{P} \mathcal{H}_{\underline{b}, \underline{i}}^n \equiv \begin{cases} \sum_{\substack{\widetilde{\underline{i}} \text{ such that } \\ \text{$\underline{i}$ is a $\underline{b}$-reduction of $\widetilde{\underline{i}}$}} }  \varepsilon (\Sz^{c_0}_{\widetilde{\underline{i}}} \ast' q_{0}) \otimes \cdots \otimes  (\Sz^{c_k}_{\widetilde{\underline{i}}} \ast' q_{k}) & \underline{b}_{0} < \widetilde{k}  \\
0 & \text{otherwise} \end{cases}    \]
modulo degenerates, where $q_{j} = s_{j,j+1}^{k+1}$ and
\[ \varepsilon =   (-1)^{(\widetilde{k}+1)k} \sgn(\underline{i}^{\vee}) \sgn(\widetilde{\underline{i}}^{\vee}).   \]
Thus $\mu(\underline{i})$ is represented by maps
\begin{align*}
 \Z[X]_{n-k} &\to (\Z[X]^{\otimes k+1})_{n} \\
 x &\mapsto  \sum_{\underline{b}} \sum_{\substack{\widetilde{\underline{i}} \text{ such that } \\ \text{$\underline{i}$ is a $\underline{b}$-reduction of $\widetilde{\underline{i}}$}} }  \varepsilon (\Sz^{c_0}_{\widetilde{\underline{i}}} \ast' q_{0}) \otimes \cdots \otimes  (\Sz^{c_k}_{\widetilde{\underline{i}}} \ast' q_{k})   
\end{align*}
modulo degenerates and the kernel of $\mathcal{P}$, where the $\underline{b}$ run over sequences $(b_0, \dots, b_{k})$ with $\widetilde{k}>b_0> \dots >b_{k} \ge 0$ only.

Under the functor $\tau \dec^*_{\otimes}$ of weakly enriched operads of Proposition~\ref{PROPSIMPLICIALLYENRICHEDDEC} this is mapped to a morphism: 
\[ \uHom_{k}(\tau_{\le 1}\dec^*\Z[X], (\tau_{\le 1} \dec^*\Z[X])^{\utildeotimes k+1}). \]
Contemplating the construction in \ref{PARDECOTIMESENRICHED} we can first map it to 
\begin{equation} \label{eqfirstlift}
\uHom_{k}(\tau_{\le 1}\dec^*\Z[X], (\tau_{\le 1} \dec^*\Z[X])^{\otimes k+1}) \end{equation}
by the morphism in Proposition~\ref{PROPEXPLICITDECENRICHED} and then apply the $\AWfrak_2$ morphism. 
The image in (\ref{eqfirstlift}) is thus represented by morphisms (we will only need the case in which the first component is zero)
\begin{align*}
 \Z[X]_{i+j+1} \oplus \Z[X]_{i+j}  &\to (\Z[X]^{\otimes k+1})_{i+j+k+1} \oplus (\Z[X]^{\otimes k+1})_{i+j+k}  \\
(0, x) &\mapsto (0, (-1)^{ik}\sum_{\underline{b}} \sum_{\substack{\widetilde{\underline{i}} \text{ such that } \\ \text{$\underline{i}$ is a $\underline{b}$-reduction of $\widetilde{\underline{i}}$}} }  \varepsilon (\Sz^{c_0}_{\widetilde{\underline{i}}} \ast' q_{0}) \otimes \cdots \otimes  (\Sz^{c_k}_{\widetilde{\underline{i}}} \ast' q_{k}) )  
\end{align*}
 where now $\widetilde{k}=i+j-1$. We compose this with the bottom line map in the commutative diagram (cf.\@ Proposition~\ref{PROPAWEZAB}): 
 \[ \tiny \xymatrix{ (\Z[X]^{\otimes k+1})_{[i+j+k]} \ar[d] \ar[r]^-{s_{\can}}  &  (\Z[X]^{\otimes k+1})_{[i+j+k+1]} \ar[r]^-{\delta^{(0)}, \dots, \delta^{(k)}} \ar[d] & \bigoplus_{j_0+\cdots+j_k=j+k} (\dec^*\Z[X])_{[i+j_0+1]} \otimes \cdots \otimes (\dec^*\Z[X])_{[i+j_k+1]} \ar[d]\\
 (\Z[X]^{\otimes k+1})_{i+j+k} \ar@{^{(}->}[r] &  (\dec^*\Z[X]^{\otimes k+1})_{i, j+k} \ar[r]^-{\AWfrak_2} & \bigoplus_{j_0+\cdots+j_k=j+k} ((\dec^*\Z[X])_{\bullet, j_0} \otimes \cdots \otimes (\dec^*\Z[X])_{\bullet, j_k})_i   \\
 } \]
 with 
 \[ \delta^{(e)} = \id_{[i]} \ast \delta^{j+k}_{\sum_{f<e}j_f,\sum_{f\le e}j_f} \]
  where, however, only summands with $j_0, \dots, j_k \ge 1$ occur because of the truncation. Now set $j=1$ (hence $i = \widetilde{k}$) because this is the column appearing in the cobar construction. 
 In this case, there is only one summand and we have
 \[ \delta^{(e)} s_{\can}  = \begin{cases} (\id_{[\widetilde{k}]} \ast \delta^{k+1}_{0,1} ) s_{\can}  & e=0, \\  \id_{[\widetilde{k}]} \ast \delta^{k+1}_{e-1,e} & e>1.  \end{cases} \]
 This shows that elements in the kernel of $\mathcal{P}$ are mapped to degenerate (in the second index of $\dec^*$) elements. Furthermore
 \begin{gather*}  (\id_{[\widetilde{k}]} \ast \delta^{k+1}_{e,e+1}) \circ s_{\can} \circ (\Sz^{c_0}_{\widetilde{\underline{i}}} \ast' q_{e}) \\
 =
  (\id_{[\widetilde{k}]} \ast \delta^{k+1}_{e,e+1}) \circ  (\Sz^{c_0}_{\widetilde{\underline{i}}} \ast s_{e,e+1}^{k+1}) \circ s_{\can}  \\
=    (\Sz^{c_0}_{\widetilde{\underline{i}}} \ast \id_{[1]}) \circ s_{\can}  = s_{\can} \circ (\Sz^{c_0}_{\widetilde{\underline{i}}} \ast' \id_{[1]}).  \end{gather*}
 The composition is equal to 
 \begin{align*}
 \Z[X]_{\widetilde{k}+1}  &\to \Z[X]_{[\widetilde{k}+2]} \otimes \cdots \otimes \Z[X]_{[\widetilde{k}+2]} \\
x &\mapsto (-1)^{\widetilde{k} k} \sum_{\underline{b}} \sum_{\substack{\widetilde{\underline{i}} \text{ such that } \\ \text{$\underline{i}$ is a $\underline{b}$-reduction of $\widetilde{\underline{i}}$}} }  \varepsilon s_{\can}   (\Sz^{c_0}_{\widetilde{\underline{i}}}\ast' \id_{[1]}) \otimes \cdots \otimes  (\Sz^{c_k}_{\widetilde{\underline{i}}}\ast' \id_{[1]}) 
\end{align*}
modulo degenerates.

Under the morphism in Proposition~\ref{PROPCOHERENTAINFTY} this yields a morphism of $A_{\infty}$-coalgebras with components (in degree 1 and $\widetilde{k}$)
\[ \alpha_{k+1} =  (s^{-1})^{k+1} (-1)^{k} \sum_{\underline{i}, \underline{b}} \sum_{\substack{\widetilde{\underline{i}} \text{ such that } \\ \text{$\underline{i}$ is a $\underline{b}$-reduction of $\widetilde{\underline{i}}$}} }
(-1)^{\frac{(k-1)k}{2}} \sgn(\widetilde{\underline{i}}^{\vee}) s_{\can} (\Sz^{c_0}_{\widetilde{\underline{i}}}\ast' \id_{[1]}) \otimes \cdots \otimes  (\Sz^{c_k}_{\widetilde{\underline{i}}}\ast' \id_{[1]}))    \]
(Notice that $\sgn(\underline{i}^{\vee})=(-1)^{\frac{(k-1)k}{2}} \sgn(\underline{i})$).
The $(s^{-1})^{k+1}$ contributes (Koszul convention) a sign of $(-1)^{\frac{k(k+1)}{2} }$,
so 
\[ \alpha_{k+1} = \sum_{\underline{i}, \underline{b}} \sum_{\substack{\widetilde{\underline{i}} \text{ such that } \\ \text{$\underline{i}$ is a $\underline{b}$-reduction of $\widetilde{\underline{i}}$}} }
\sgn(\widetilde{\underline{i}}^{\vee}) s_{\can} ((\Sz^{c_0}_{\widetilde{\underline{i}}}\ast' \id_{[1]})(s^{-1}x) \otimes \cdots \otimes  (\Sz^{c_k}_{\widetilde{\underline{i}}}\ast' \id_{[1]})(s^{-1}x)).    \]
If we take $x \in X_{\widetilde{k} + 1}$ we have:
\[ (1-\varepsilon) \sum_{n=0}^{\infty} \alpha_{n+1} (x) = \begin{cases}
-1 + x & \widetilde{k} =1  \\
\sum_{\underline{i}} \sgn(\underline{i}^{\vee})  (\Sz^{0}_{\underline{i}}(s^{-1}x) \otimes \cdots \otimes  \Sz^{k}_{\underline{i}}(s^{-1}x)) & \widetilde{k} >1  
\end{cases}    \]
by Proposition~\ref{PROPSZCZARBACANCELLATION}.
 The $-1$ is the constant term for k=1. For higher $k$ the constant term is degenerate and thus can be omitted. 
\end{proof}

\begin{PROP}
We have $\psi \phi = \id$ and there is a chain homotopy $\psi \phi \Rightarrow \id$.
\end{PROP}
\begin{proof}
Omitted for the moment.
\end{proof}

\appendix

\section{Exactness of some diagrams}

\begin{LEMMA}\label{LEMMAEXACT1}
The following natural morphism is an isomorphism of 1-profunctors (and also $\infty$-profunctors):
\[   {}^t\!\pi \cong  {}^t\! i\, \emptyset. \]
\end{LEMMA}
\begin{proof}
The diagram
\[ \xymatrix{  \Delta^{\op} \ar[d]_{\pi} \ar@{=}[r] &  \Delta^{\op} \ar[d]^{i} \\
\cdot   \ar[r]_-{\emptyset} \ar@{}[ru]|{\Nearrow} & \Delta^{\op}_{\emptyset}    }
 \]
is $\infty$-exact, because $\emptyset$ is final in $\Delta^{\op}_{\emptyset}$. 
\end{proof}

\begin{LEMMA}\label{LEMMAEXACT2}
The following natural morphism is an isomorphism of 1-profunctors (and also $\infty$-profunctors):
\[   {}^t\!\pr_{1}\, {}^t\!i  \cong {}^t\!(i, i)\,p_{1}.    \]
\end{LEMMA}
\begin{proof}
We have to see that the diagram
\[ \xymatrix{  \Delta^{\op} \times  \Delta^{\op}  \ar[r]^{(i, \pi)} \ar@{=}[d] &\Delta^{\op}_{\emptyset} \times \cdot \ar[d]^{p_1=(\id, \emptyset)} \\
\Delta^{\op} \times  \Delta^{\op} \ar[r]^{(i, i)} \ar@{}[ru]|{\Nearrow} & \Delta^{\op}_{\emptyset} \times  \Delta^{\op}_{\emptyset}   }
 \]
 is $\infty$-exact, which is the product of the one of Lemma~\ref{LEMMAEXACT1} with a trivially exact one. 
\end{proof}

\begin{LEMMA}\label{LEMMAEXACT3}
The following natural morphism is an isomorphism of 1-profunctors (and also $\infty$-profunctors):
\[ \dec\,  {}^t\!\pr_{i} \cong \id. \]
In particular, we also have $ \dec\,  {}^t\!\pi_{\Delta^{\op} \times \Delta^{\op}} \cong \,{}^t\!\pi_{\Delta^{\op}}$, in other words, $\dec$ is $\infty$-cofinal. 
\end{LEMMA}
\begin{proof}
We have to show that the diagram
\[ \xymatrix{  \Delta^{\op} \times \Delta^{\op} \ar[r]^-{\dec} \ar[d]_{\pr_i} & \Delta^{\op}   \ar@{=}[d] \\
\Delta^{\op}   \ar@{=}[r]^{} \ar@{}[ru]|{\Downarrow} & \Delta^{\op}   }
 \]
 is $\infty$-exact. 
Since $\pr_i$ is a cofibration, it suffices to see that 
\[ \xymatrix{  \Delta^{\op}  \ar[r]^{\dec_n} \ar[d]^{\pi} & \Delta^{\op} \ar@{=}[d] \\
\cdot   \ar[r]^{[n]} \ar@{}[ru]|{\Downarrow} & \Delta^{\op}   }
 \]
 is exact (where $\dec_n := ([n], \id) \circ \dec$) for all $n$. 
 Using criterion (Lemma~\ref{PROPEXACT}, 4.) it boils down to show that the following category is contractible: 
 Its objects are morphisms
 \[ [a] \leftarrow [b] \ast [n] \]
 such that the composition $[a] \leftarrow [n]$ is fixed, and morphisms are the morphisms $[b] \leftarrow [b']$ making
 \[ \xymatrix{ [a] \ar@{<-}[r] \ar@{<-}[rd] & [b] \ast [n] \ar@{<-}[d] \\
 & [b'] \ast [n]  } \]
 commute. 
 Now write $[a] = [a'] \ast' [a'']$ (where $\ast'$ identifies the endpoints) 
 such that $a''$ is minimal with a factorization of the given morphism: 
 \[ [a] =  [a'] \ast' [a''] \leftarrow [a''] \leftarrow [n]. \]
A morphism $[a] \leftarrow [b] \ast [n]$ is then completely determined by a morphism $[a'] \leftarrow [b]$ to the extent that the category in question is isomorphic
to 
\[ [a'] \times_{/\Delta^{\op}} \Delta^{\op} \] 
which is contractible, having an initial object.
\end{proof}

\begin{LEMMA}\label{LEMMAEXACT5}
$\delta$ is $\infty$-cofinal, i.e.\@ we have that the natural morphism
\[ \delta\,{}^t\!\pi  \cong {}^t\!\pi. \]
is an isomorphism of $\infty$-profunctors. 
\end{LEMMA}
\begin{proof}
We have to show that the diagram
\[ \xymatrix{  \Delta^{\op}  \ar[r]^-{\pi} \ar[d]_{\delta} & \cdot  \ar@{=}[d] \\
\Delta^{\op} \times \Delta^{\op}  \ar[r]^-{\pi} \ar@{}[ru]|{\Downarrow} & \cdot   }
 \]
 is $\infty$-exact. 
 
 Using criterion (Lemma~\ref{PROPEXACT}, 4.)  it boils down to show that the following category is contractible: 
 \[ [n],[m] \times_{/\Delta^{\op} \times \Delta^{\op},\delta} \Delta^{\op}  \]
This category is cofibered over $\Delta^{\op}$ with fibers being the set $\Hom_{\Delta}(-, [n] \times [m])$. But the
latter simplicial set is the nerve of the category $[n] \times [m]$ which is obviously contractible, having an initial object (or here equivalently a final object). 
Moreover, for any category there
is a natural transformation $\int_{\Delta^{\op}} N(I) \to I$ which is a weak equivalence (i.e.\@ the nerve applied to it is a weak equivalence). 
 \end{proof}

\begin{LEMMA}\label{LEMMAEXACT4}
The following natural morphism is an isomorphism of 1-profunctors (and also $\infty$-profunctors):
\[   \dec\,{}^t\!(i, i)  \cong {}^{t}\!i\, \dec_{\emptyset}.    \]
\end{LEMMA}
\begin{proof}
Look at the compositions: 
\[ \xymatrix{ \Delta^{\op} \times  \Delta^{\op} \ar@{=}[r] \ar@{=}[d] &  \Delta^{\op} \times  \Delta^{\op} \ar[r]^-{\dec} \ar[d]_{(i, i)} \ar@{}[rd]|{=} & \Delta^{\op} \ar[d]^{i}  \\
 \Delta^{\op} \times  \Delta^{\op} \ar[r]_{(i,i)} & \Delta^{\op}_{\emptyset} \times  \Delta^{\op}_{\emptyset} \ar[r]_-{\dec_{\emptyset}}  & \Delta^{\op}_{\emptyset} }  \]
 and 
\[ \xymatrix{ \Delta^{\op}  \times \Delta^{\op}   \ar@{=}[r] \ar[d]^{i \pr_i} &  \Delta^{\op} \times  \Delta^{\op} \ar[r]^-{\dec} \ar[d]_{(i, i)} \ar@{}[rd]|{=} & \Delta^{\op} \ar[d]^{i}  \\
 \Delta^{\op}_{\emptyset}  \ar[r]_-{p_i} & \Delta^{\op}_{\emptyset} \times  \Delta^{\op}_{\emptyset} \ar[r]_-{\dec_{\emptyset}}  & \Delta^{\op}_{\emptyset} }  \]
 We have to show that the common right square is $\infty$-exact.  
 Since $(i,i)^*$, $p_1^*$, and $p_2^*$, are jointly conservative, it suffices to see that the left and composite squares are $\infty$-exact in both cases. This is either because $i$ and $(i, i)$ are fully-faithful, or a consequence of  Lemmas~\ref{LEMMAEXACT2}--\ref{LEMMAEXACT3}. 
 \end{proof}

\begin{LEMMA}\label{LEMMAIOTAFINAL}
$\iota: \Delta^{\op} \to \FinSet^{\op}$ is 1-cofinal, i.e.
\[ \xymatrix{\Delta^{\op} \ar[r] \ar[d] & \FinSet^{\op} \ar[d] \\ \cdot \ar@{=}[r] & \cdot }\]
is 1-exact.   
\end{LEMMA}
\begin{proof}
To show that the categories 
\[ \Delta \times_{/\FinSet} [n]  \]
are  connected. 
Let $[m] \to [n]$ and $[k] \to [n]$ be morphisms in $\FinSet$. They give rise to a morphism
\[ [m] \ast [k] \to [n] \]
in $\FinSet$ ($\ast$ is the coproduct in $\FinSet$). We get a commutative diagram
\[ \xymatrix{[m] \ar[rd] \ar@{^{(}->}[r] & [m] \ast [k] \ar[d] \ar@{<-^{)}}[r] & [k] \ar[ld]  \\
&  [n]  }\]
where the inclusions are morphisms in $\Delta$. The category, being non-empty, is thus connected. 
\end{proof}

\section{Leveled trees, Shih, and Szczarba}\label{APXSHIHSZCZARBA}

This section is devoted to the proof of Proposition~\ref{PROPSHIHSZCZARBAGEN1}.
Actually a slight generalization, Proposition~\ref{PROPSHIHSZCZARBAGEN2} will be stated and proven, in which the $e$ and $\sigma$ are needed 
to make the induction work.

\begin{PAR}\label{PARPATTERN}We need the following generalization of (\ref{PARP}).
Let $k, e, m$ be non-negative integers. 
For $\sigma: [k+e] \twoheadrightarrow [k]$, we define the operator
\[ \mathcal{P}_{\sigma}: \Z[\Hom([n], [m])]^{\otimes (k+1)} \to \Z[\Hom([n], [m])]^{\otimes (k+1)} \]
by
\[ \mathcal{P}_{\sigma} = \prod_{i=0}^{k+e} 1^{\otimes \sigma(i)} \otimes \mathcal{P}_{n-k-e-1+i} \otimes 1^{\otimes k-\sigma(i)}  \]
i.e.\@ it throws away all summands that are $n-k-e-1+i$ degenerate at $\sigma(i)$ for $i=0, \dots, k+e$.

We also let $\delta:=\Hom(\sigma, [1]): [k+1] \hookrightarrow  [k+e+1]$ the dual active face map, here with
\[  \delta(j+1) = \delta(j) + \#\sigma^{-1}(j). \]
For $\sigma = \id$, $e=0$ we also write $\mathcal{P}:=\mathcal{P}_{\sigma}$.
We denote the jumps 
\[ \alpha_j :=  \delta(j+1) - \delta(j) = \#\sigma^{-1}(j).  \]

We have
\[  \sigma = s^{\alpha_0} \ast' \id_{[1]} \ast' s^{\alpha_1} \ast' \cdots \ast' \id_{[1]} \ast' s^{\alpha_{k}},   \]
where $\alpha_i$ is the number of preimages, denoting $s^{\alpha_i}: [\alpha_i-1] \twoheadrightarrow [0]$.

By formula (\ref{eqsignshuffle}) for the signum, we have thus
\[ \sgn(\sigma) = \prod_{i} (-1)^{(k-i)(\alpha_i-1)}.   \]
\end{PAR}

\begin{PAR}\label{PARPATTERNDERIVE}
In a situation as in \ref{PARREDUCT}, given a vector $\underline{b} = (b_0, \dots, b_{k-1})$ with $\widetilde{k} > b_0> \dots > b_{k-1} \ge 0$ --- 
which is also seen as a shuffle $\sigma_{\underline{b}}: [n] \to [k]$, $\tau_{\underline{b}}: [n] \to [n-k]$ (notice the $n$ instead of $\widetilde{k}$) where $\tau_{\underline{b}}$ is degenerate precisely at the intervals $b_i+1$ ---
 and $\widetilde{\underline{i}}$ of length $\widetilde{k} \ge k$, such that 
$\underline{i}$ is a $\underline{b}$-reduction of $\widetilde{\underline{i}}$, then given $\sigma: [k+e]  \to [k]$, we get a unique extension
$\widetilde{\sigma}:   [\widetilde{k}+e] \to  [\widetilde{k}]$  
satisfying
\[ \# (\widetilde{\sigma})^{-1}(i) = \begin{cases} \# \sigma^{-1}(j) &  i=c_j \text{ for some $j$ } \\ 1 & \text{otherwise}   \end{cases}\]
Denoting
\[ \beta_j :=  c_{j} - c_{j-1} -1  \]
(setting $c_{-1} = -1$), we have
\[  \widetilde{\sigma} = \id_{[\beta_0]} \ast' s^{\alpha_0} \ast' \id_{[\beta_1+1]} \ast' s^{\alpha_1} \ast' \cdots \ast' \id_{[\beta_{k}+1]} \ast' s^{\alpha_{k}}   \]
and get by formula (\ref{eqsignshuffle}) for the signum
\[ \sgn(\widetilde{\sigma}) = \prod_{i<j} (-1)^{(\beta_j+1)(\alpha_i-1)}  \]
which specializes to the formula before, if all $\beta_j=0$. In particular, 
\begin{equation}\label{signsigmatildesigma}
 \sgn(\sigma) \sgn(\widetilde{\sigma}) = \prod_{i<j} (-1)^{\beta_j(\alpha_i-1)}.  
\end{equation}
\end{PAR}

\begin{DEF}
For a vector (i.e.\@ leveled tree) $\underline{i} = (i_0, \dots, i_k)$ and face $\delta: [k+1] \to [m]$ (arbitrary $m \ge k+1$) 
define
\[ \Sz_{\underline{i},\delta}^j = \Sz_{\underline{i}}^j \delta|_{[j]}.  \]
\end{DEF}
Directly from (\ref{eqszczarba2}) follows: 
\begin{LEMMA}\label{LEMMASZCZARBAMOD}
\begin{equation}\label{eqszrecdelta} 
\Sz_{\underline{i},\delta}^j = \begin{cases} 
\Sz_{\underline{i}',\delta'}^j \ast \delta_0^{\alpha_{i_0} -1} & j = i_0+1 \\
s_{i_0} \Sz_{\underline{i}', \delta'}^{j'} & \text{otherwise} 
\end{cases} 
\end{equation}
with $\delta' = \delta_{i_0+1} \delta$ and $j' = s_{i_0}(j)$ and $\alpha_{i_0} = \delta(i_0+1)-\delta(i_0)$.
\end{LEMMA}

The following gives Proposition~\ref{PROPSHIHSZCZARBAGEN1} in the special case $e=0$ and $\sigma = \id$. 

\begin{PROP}\label{PROPSHIHSZCZARBAGEN2} With the notation from $\ref{PARPATTERN}$, we have for arbitrary $n$, setting $\widetilde{k}:=n-k-1-e$:
\[  \mathcal{P}_{\sigma} \mathcal{H}_{\underline{b}, \underline{i}}^n \equiv \begin{cases} \sum_{\substack{\widetilde{\underline{i}} \text{ such that } \\ \text{$\underline{i}$ is a $\underline{b}$-reduction of $\widetilde{\underline{i}}$}} }  \varepsilon (\Sz^{c_0}_{\widetilde{\underline{i}}, \widetilde{\delta}} \ast' q_{\sigma,0}) \otimes \cdots \otimes  (\Sz^{c_k}_{\widetilde{\underline{i}}, \widetilde{\delta}} \ast' q_{\sigma,k}) & \underline{b}_{0} < \widetilde{k}  \\
0 & \text{otherwise} \end{cases}    \]
modulo degenerates, where $q_{\sigma,j} = s_{\delta(j),\delta(j+1)}^{k+1+e}$. In particular the expession is zero unless $\widetilde{k} \ge k$ (i.e.\@ $n \ge 2k+e+1$).

Here the $\widetilde{\underline{i}}$ are vectors of length $\widetilde{k}$ and $c_0, \dots, c_k$ are determined by $\widetilde{\underline{i}}$ as in \ref{PARREDUCT} and 
\[ \varepsilon = (-1)^{(\widetilde{k}+1)k}\sgn(\underline{i}^{\vee}) \sgn(\widetilde{\underline{i}}^{\vee}) \sgn(\sigma) \sgn(\widetilde{\sigma}) \sgn(\sigma_{\underline{b}})  \]

The $\Sz_{\overline{i},\delta}^j \ast' q_{\sigma j}$ have values in $[\widetilde{\delta}(c_j)] \ast' [\#\sigma^{-1}(j)] = [\widetilde{\delta}(c_j+1)]$ and are considered as $[n-k]$-valued via
the inclusion $[\widetilde{\delta}(c_j+1)] \subset [n-k]$. $\widetilde{\delta}$ is the (active) morphism $[\widetilde{k}+1] \hookrightarrow [\widetilde{k}+e+1]$ dual to $\widetilde{\sigma}$ (\ref{PARPATTERN}).
\end{PROP}

Before giving the proof, we need a couple of Lemmas: 
\begin{LEMMA}\label{LEMMASHIH0} Let $\tau \in \Hom([n-1], [m])$. Then  
$\mathcal{H}_b^n ( \tau )$ is a sum of terms of the form $\tau' \otimes \tau''$ that are
\begin{itemize}
\item for $i<b$, $i$-degenerate at the left and right if $\tau$ is degenerate at $i$ and not $i$-degenerate neither on the left nor right, otherwise.
\item for $i=b$, $i$-degenerate at the left (at least) 
\item for $i>b$ either $i$-degenerate at the right or at the left (not both) if $\tau$ is not-$i-1$-degenerate and fully $i$-degenerate, if $\tau$ is $i-1$-degenerate.
\end{itemize}
Furthermore, for $i \ge b$ each term is at least $i$-degenerate at the right for $i=b$ or such that $\tau$ is not $i-1$-degenerate.  
\end{LEMMA}
The proof is left to the reader. 
We can also say that $\mathcal{H}_b^n ( \tau )$ is a sum of terms of the form $s_b \tau' \otimes s_i \tau''$ for $i \ge b$ and such that either $i=b$ (thus fully degenerate) or $\tau$ is not $i-1$-degenerate.

\begin{LEMMA}\label{LEMMASHIH1}
Let $\tau_0 \otimes \cdots \otimes \tau_{k}$ be a summand of $\mathcal{H}_{\underline{b}, \underline{i}}^n$.
Then for $i>b_0$ at most one $\tau_j$ is not $i$-degenerate.
\end{LEMMA}
\begin{proof}
Induction over Lemma~\ref{LEMMASHIH0} (because of $i > b_0 > \dots > b_{k-1}$, in each step, the third case is relevant).
\end{proof}

\begin{LEMMA}\label{LEMMASHIH2}
If $\mathcal{P}_{\sigma} \mathcal{H}_{\underline{b}, \underline{i}}^n \not\equiv 0$ modulo degenerates then $b_{0} < n-k-e-1$.
\end{LEMMA}
\begin{proof}
We have
\[ \mathcal{H}_{\underline{b}, \underline{i}}^n (\tau) = (s_{b_{0}}^{\otimes i_{0}} \otimes \mathcal{H}_{b_{0}}^n \otimes s_{b_{0}}^{\otimes k-1-i_{0}}) \mathcal{H}_{\underline{b}', \underline{i}'}^{n-1}. \]
Let $\tau_0 \otimes \cdots \otimes \tau_{k-1}$ be a term of $\mathcal{H}_{\underline{b}', \underline{i}'}^{n-1}$.
The operator $\mathcal{H}_{b_{0}}^n (\tau_{i_{0}})$ is (modulo degenerates) a sum over terms of the form $s_{b_{0}} \tau' \otimes s_i \tau''$ for $i > b_{0}$ such that $\tau_{i_{0}}$ is not $i-1$-degenerate (Lemma~\ref{LEMMASHIH0}, and the comment after) and thus $\mathcal{H}_{\underline{b}, \underline{i}}^n (\tau)$ is a sum over terms of the form
\[ x := s_{b_{0}}\tau_0 \otimes \cdots \otimes (\underbrace{s_{b_{0}}\tau'}_{\text{ at } i_0} \otimes \underbrace{s_i \tau''}_{\text{ at } i_0+1}) \otimes \cdots \otimes s_{b_{0}}\tau_{k-1}   \]

However, since $i-1 > b_{1}$, $\tau_0 \otimes \cdots \otimes \tau_{k-1}$ is (either degenerate or) $i-1$-degenerate at all except one slot (Lemma~\ref{LEMMASHIH1}).
Thus all $\tau_j$ for $j \not= i_{0}$ are $i-1$-degenerate and thus  $x$ is $i$-degenerate except at $i_{0}$ and $i_{0}+1$.
If we assume $\mathcal{P}_{\sigma} x \not\equiv 0$,
 we must have $\sigma(i) = i_{0}$ or  $\sigma(i) = i_{0}+1$,
and also $\sigma(b_{0}) = i_{0}+1$ because $x$ is obviously $b_0$-degenerate at all other slots. 
Hence  $\sigma(i) = i_{0}+1$ (because $i>b_{0}$) and hence $\mathcal{P}_{\sigma} x \equiv 0$ anyway.
\end{proof}

Recall from Lemma~\ref{LEMMAH} the definition of the operators $B_{\underline{z}}^j$.
\begin{LEMMA}\label{LEMMAHALT}
\item Let $n, \alpha_0, \alpha_1$ be positive integers. 
\[ \mathcal{P}_{\alpha_0,\alpha_1} \mathcal{H}^{n} = \sum_{\underline{z}} (-1)^{\sum_{i<j} z_j(z_i+1) + \alpha_0 \sum z_i } (B_{\underline{z}}^0 \ast' s_{\alpha_0,\alpha_0+\alpha_1}^{\alpha_0+\alpha_1}  \otimes B_{\underline{z}}^1 \delta_0^{\alpha_0} \ast' s_{0,\alpha_0}^{\alpha_0+\alpha_1} ) \]
in which $\underline{z} = (z_0, \dots, z_{n-\alpha_0-\alpha_1-1})$ runs over all vectors with $z_i \in \{0, 1\}$.
Implicitly, in the formula the $B_{\underline{z}}^0 \ast' s_{\alpha_0,\alpha_0+\alpha_1}^{\alpha_0+\alpha_1}$ is composed with the inclusion $[m+\alpha_1] \hookrightarrow [n+1]$. 
\end{LEMMA}
\begin{proof}
Follows from Lemma~\ref{LEMMAH} identifying the $\underline{z}$-summand above with the  
\[ (\underbrace{0, \dots, 0}_{\alpha_0 \text{ times}}, \underbrace{1, \dots, 1}_{\alpha_1 \text{ times}}, z_0, \dots, z_{n-\alpha_0-\alpha_1-1})-\text{summand}. \qedhere \] 
\end{proof}

\begin{LEMMA}\label{LEMMABIJECTION}
Given $\underline{i} = (i_0, \dots, i_{k-1})$, $\widetilde{k} \ge k$, and $\underline{b}=(b_0, \dots, b_{k-1})$ with $\widetilde{k} > b_0 > \dots > b_{k-1} \ge 0$, and denoting $\underline{i}' = (i_1, \dots, i_{k-1})$
and $\underline{b}' = (b_1, \dots, b_{k-1})$,
we have a bijection
\[ \left\{ \substack{\widetilde{\underline{i}'} \\ \text{$\underline{i}'$ is a $\underline{b}'$-reduction of $\widetilde{\underline{i}'}$}  \\ 
\overline{z} \in \{0,1\}^{x} } \right\}  \cong \left\{ \substack{\widetilde{\underline{i}} \\ \text{$\underline{i}$ is a $\underline{b}$-reduction of $\widetilde{\underline{i}}$} } \right\}   \]
 where $\widetilde{\underline{i}}$ has length $\widetilde{k}$ (resp.\@ $\widetilde{\underline{i}'}$ has length $\widetilde{k}-1$), and $x$ is the cardinality of 
 \begin{equation} \label{eqx}
 \{j_0, \dots, j_{x-1}\} = \{0 \le j < \widetilde{k}-b_0-1 \text{ s.t.\@ } y_j' = i_0 \}
 \end{equation}  
 (cf.\@ \ref{PARREDUCT} for the definition of the $y_j'$). Then we have
 \begin{equation} \label{eqszs}
  \Sz^{c_{i_0}'}_{\widetilde{\underline{i}'},\widetilde{\delta}'}  = \Sz^{c_{i_0}'}_{\widetilde{\underline{i}'},\widetilde{\delta}'}  |_{[b_0]} \ast' s  
  \end{equation}
 where $\Sz^{c_{i_0}'}_{\widetilde{\underline{i}'},\widetilde{\delta}'}  |_{\Delta_{b_0}} : [b_0] \to [c_{i_0}'-x]$
 for a degeneracy $s: [\widetilde{k}-b_0-1] \to [x]$.

 Given $\sigma$ and $\sigma' = \sigma s_{i_0}$ as in \ref{PARPATTERN} with dual faces $\delta$ and $\delta' = \delta_{i_0+1}\delta$. 
 Recall their  extensions $\widetilde{\delta}: [\widetilde{k}+1] \to [n+1]$ and $\widetilde{\delta}': [\widetilde{k}] \to [n]$ defined in \ref{PARPATTERNDERIVE}.  
 Then we have: 
\begin{equation}\label{eqsz1} 
\Sz^{c_j}_{\widetilde{\underline{i}}, \widetilde{\delta}}  = \begin{cases} 
s_{b_0} \Sz^{c_{i_0}'}_{\widetilde{\underline{i}'}, \widetilde{\delta}'}|_{[b_0]}  \ast' (s B_{\underline{z}}^0)  & j = i_0, \\
\Sz^{c_{i_0}'}_{\widetilde{\underline{i}'}, \widetilde{\delta}'}|_{[b_0]}   \ast (s B_{\underline{z}}^1 \delta_0^{\alpha_{i_0}-1})   & j = i_0 + 1,   \\
s_{b_0} \Sz^{c_{s_{i_0}(j)}'}_{\widetilde{\underline{i}'}, \widetilde{\delta}'}  & \text{otherwise}.
\end{cases} 
\end{equation}

Furthermore, we have
\begin{gather} 
 \sgn(\underline{i}^{\vee})\sgn(\widetilde{\underline{i}}^{\vee}) \sgn(\sigma) \sgn(\widetilde{\sigma})  \nonumber \\
 = (-1)^{b_0-\widetilde{k}+1+\sum_{i<j} z_j (z_i+1) + \sum_i \alpha_{i_0} z_i} \sgn((\underline{i}')^{\vee})\sgn((\widetilde{\underline{i}'})^{\vee}) \sgn(\sigma') \sgn(\widetilde{\sigma}').  \label{eqsz1sgn}
\end{gather}
\end{LEMMA}
\begin{proof}
The bijection is given, identifying the extensions $\widetilde{\underline{i}'}$, and $\widetilde{\underline{i}}$, respectively, with collections $(y_0', \dots, y_{\widetilde{k}-k-1}')$, and  $(y_0, \dots, y_{\widetilde{k}-k-1})$, respectively, as in \ref{PARREDUCT}, by
\[ y_j := \begin{cases} y_j' & j \ge \widetilde{k}-b_0-1,  \\
 y_j'  & j < \widetilde{k}-b_0-1 \text{ and } y_j' < i_0, \\
 y_j'+z_{i}  & j < \widetilde{k}-b_0-1 \text{ and } y_j' = i_0, \text{ i.e.\@ } j = j_i \text{ for some $i$; cf.\@ (\ref{eqx}) },  \\
 y_j'+1 & j < \widetilde{k}-b_0-1 \text{ and } y_j' > i_0.
 \end{cases}. \]
 
The representation (\ref{eqszs}) follows directly from the recursive property of $\Sz$ and the construction of $\widetilde{\delta}'$.
We show (\ref{eqsz1}) and (\ref{eqsz1sgn}) by induction on the length $\widetilde{k}$ of $\widetilde{\underline{i}}$, starting from $b_0+1$. 

The induction base is $\widetilde{k} = b_0 + 1$, whence $c_{i_0}' = c_{i_0} = c_{i_0+1}-1$, $\widetilde{\underline{i}} = (c_{i_0}, \widetilde{i}'_0, \dots, \widetilde{i}'_{\widetilde{k}-1})$, $\widetilde{\delta}' = \delta_{c_{i}+1}\widetilde{\delta}$,
$\underline{z} = ()$, $s=\id_{[0]}$, and (\ref{eqszrecdelta}) yields: 

\begin{equation*} 
\Sz^{c_j}_{\widetilde{\underline{i}}, \widetilde{\delta}} = \begin{cases} 
 \Sz^{c_{i_0}'}_{\widetilde{\underline{i}'}, \widetilde{\delta}'} \ast \delta_0^{\alpha_{i_0}-1} & j = i_0+1, \\
s_{b_0} \Sz^{c_{s_{i_0}(j)}'}_{\widetilde{\underline{i}'}, \widetilde{\delta}'}   & \text{otherwise}.
\end{cases} 
\end{equation*}
These yield (\ref{eqsz1}). For the sign, notice that we have $\alpha_j' = \alpha_{\delta_{i_0+1}(j)}$ for $j \not= i_0$ and $\alpha_{i_0}' = \alpha_{i_0}+\alpha_{i_0+1}$ and $\beta_j' = \beta_{\delta_{i_0+1}(j)}$ for $j \not= i_0+1$ and $\beta_{i_0+1} = 0$. We have to show:
\[ c_{i_0}+i_0 \overset{!}{\equiv} \widetilde{k}-k+ \sum_{i<j}(\alpha_i-1)\beta_j + \sum_{i<j}(\alpha'_{i}-1)\beta'_j  \qquad \text{modulo 2}  \]
\[ \overset{!}{\equiv} \widetilde{k}-k - \sum_{i_0<j}\beta_j + \alpha_{i_0}\beta_{i_0+1}  \qquad \text{modulo 2}  \]
However, we have $c_{i_0} = i_0 + \sum_{i \le i_0} \beta_i$ and $\sum b_i = \widetilde{k} - k$, so the statement holds.

 We will now turn to 
the  case  $\widetilde{k} > b_0 + 1$ and and assume that the Lemma has been proven for length $\widetilde{k}-1$.

Denote 
$\widetilde{\underline{i}} = (\widetilde{i}_0, \widetilde{i}''_0, \dots, \widetilde{i}''_{\widetilde{k}-2})$, $\widetilde{\underline{i}'} = (\widetilde{i}_0', \widetilde{i'}''_0, \dots, \widetilde{i'}''_{\widetilde{k}-3})$, with relevant derived quantities $c_j''$ and $c_j'''$, 
$\underline{y} = (y_0, y_0'', \dots, y_{\widetilde{k}-k-2}'')$, $\underline{y}' = (y_0', y_0''', \dots, y_{\widetilde{k}-k-2}''')$, etc. 

Define 
\[ \delta'' :=  \delta \delta_{\delta(y_0)+1} \qquad \delta''' := \delta' \delta'_{\delta'(y_0')+1}. \]
Thus the corresponding $\alpha_{y_0}$ increases by 1 while  $\beta_{y_0}$ decreases by 1 and
we have: 
\[ \widetilde{\delta}'' :=  \delta_{c_{y_0}} \widetilde{\delta} \qquad \widetilde{\delta}''' := \delta'_{c_{y_0}'} \widetilde{\delta}'.  \]

By construction, we have $c_{y_0} = \widetilde{i}_0+1$, $y_0' = s_{i_0}(y_0)$ and $c_{y_0'}' = \widetilde{i}_0' + 1$ 
and one of the three cases:
\begin{enumerate}
\item[I] $y_0 \notin \{ i_0, i_0+1 \}$,  $\overline{z} = \overline{z}''$, $\alpha_{i_0}'' = \alpha_{i_0}$;
\item[II] $y_0 = i_0$, $\overline{z} = (0, z''_0, \dots, z''_{x-2})$, $\alpha_{i_0}'' = \alpha_{i_0}+1$;
\item[III] $y_0 = i_0+1$, $\overline{z} = (1, z''_0, \dots, z''_{x-2})$, $\alpha_{i_0}'' = \alpha_{i_0}$.
\end{enumerate}

We distinguish also the two cases: 
\begin{enumerate}
\item[A] $j \not= y_0$;
\item[B] $j = y_0$.
\end{enumerate}

We proceed to insert the recursive property of the $\Sz$ (\ref{eqszrecdelta}), respecively of the $B$ (\ref{eqb}), in all 6 cases: 

\begin{itemize}

\item[IA)] $c_j \not= \widetilde{i}_0+1$, $c_{j'} \not= \widetilde{i'}_0+1$
\begin{equation*} 
s_{\widetilde{k}-1} \Sz^{c_j''}_{\widetilde{\underline{i}}'', \widetilde{\delta}''}   \overset{!}{=} \begin{cases} 
s_{b_0}  \Sz^{c_{i_0}'''}_{\widetilde{\underline{i}'}'', \widetilde{\delta}'''}|_{[b_0]}  \ast' (s_{\widetilde{k}-2-b_0} s'' B_{\underline{z}''}^0)  & j = i_0, \\
\Sz^{c_{i_0}'''}_{\widetilde{\underline{i}'}'', \widetilde{\delta}'''}|_{[{b_0}]}   \ast (s_{\widetilde{k}-2-b_0}  s'' B_{\underline{z}''}^1 \delta_0^{\alpha_{i_0}-1})   & j = i_0 + 1,   \\
s_{b_0} s_{\widetilde{k}-2} \Sz^{c_{s_{i_0}(j)}'''}_{\widetilde{\underline{i}'}'', \widetilde{\delta}'''}  & \text{otherwise.} 
\end{cases} 
\end{equation*}

\item[IB)] $c_j = \widetilde{i}_0+1$, $c_{j'} = \widetilde{i}_0'+1$

\begin{equation*} 
(\Sz^{c_j''}_{\widetilde{\underline{i}}'', \widetilde{\delta}''}  \ast \id_{[0]})   \overset{!}{=} \begin{cases} 
\text{ --- does not occur in this case ---}  & j = i_0, \\
\text{ --- does not occur in this case ---}  & j = i_0 + 1,   \\
s_{b_0} (\Sz^{c_{s_{i_0}(j)}'''}_{\widetilde{\underline{i}'}'', \widetilde{\delta}'''} \ast \id_{[0]})  & \text{otherwise.} 
\end{cases} 
\end{equation*}

\item[IIA)] $c_j \not= \widetilde{i}_0+1$, $c_{j}' = \widetilde{i}_0'+1$ if $j=i_0+1$ and $c_{j'} \not= \widetilde{i}_0'+1$ otherwise.
\begin{equation*} 
s_{\widetilde{k}-1} \Sz^{c_j''}_{\widetilde{\underline{i}}'', \widetilde{\delta}''}   \overset{!}{=} \begin{cases} 
\text{ --- does not occur in this case ---} & j = i_0, \\
\Sz^{c_{i_0}'''}_{\widetilde{\underline{i}'}, \widetilde{\delta}'}|_{[b_0]}   \ast (  (s'' \ast  \id_{[0]})  s_{x-1} B_{\underline{z}''}^1 \delta_0  \delta_0^{\alpha_{i_0}-1})   & j = i_0 + 1,  \\
s_{b_0} s_{\widetilde{k}-2} \Sz^{c_{s_{i_0}(j)}'''}_{\widetilde{\underline{i}'}'', \widetilde{\delta}'''}  & \text{otherwise.} 
\end{cases} 
\end{equation*}

\item[IIB)] $c_j = \widetilde{i}_0+1$, $c_{j}' = \widetilde{i}_0'+1$
\begin{equation*} 
(\Sz^{c_j''}_{\widetilde{\underline{i}}'', \widetilde{\delta}''} \ast \id_{[0]})  \overset{!}{=} \begin{cases} 
s_{b_0}  \Sz^{c_{i_0}'''}_{\widetilde{\underline{i}'}'', \widetilde{\delta}'''}|_{[{b_0}]}  \ast'  ( (s'' \ast  \id_{[0]})  (B_{\underline{z}''}^0 \ast \id_{[0]}))  & j = i_0, \\
\text{ --- does not occur in this case ---} & j = i_0 + 1,   \\
\text{ --- does not occur in this case ---}  & \text{otherwise.} 
\end{cases} 
\end{equation*}

\item[IIIA)] $c_j \not= \widetilde{i}_0+1$, $c_{j}' = \widetilde{i}_0'+1$ if $j=i_0$ and $c_{j}' \not= \widetilde{i}_0'+1$ otherwise.
\begin{equation*} 
s_{\widetilde{k}-1} \Sz^{c_j''}_{\widetilde{\underline{i}}'', \widetilde{\delta}''}   \overset{!}{=} \begin{cases} 
s_{b_0}  \Sz^{c_{i_0}'''}_{\widetilde{\underline{i}'}'', \widetilde{\delta}'''}|_{[{b_0}]}  \ast' (  (s'' \ast  \id_{[0]}) s_{x-1} B_{\underline{z}''}^0)  & j = i_0, \\
\text{ --- does not occur in this case ---}   & j = i_0 + 1,   \\
s_{b_0} s_{\widetilde{k}-2} \Sz^{c_{s_{i_0}(j)}'''}_{\widetilde{\underline{i}'}'', \widetilde{\delta}'''}  & \text{otherwise.} 
\end{cases} 
\end{equation*}

\item[IIIB)] $c_j = \widetilde{i}_0+1$, $c_{j'} = \widetilde{i}_0'+1$
\begin{equation*} 
(\Sz^{c_j''}_{\widetilde{\underline{i}}'', \widetilde{\delta}''} \ast \id_{[0]})  \overset{!}{=} \begin{cases} 
\text{ --- does not occur in this case ---} & j = i_0, \\
\Sz^{c_{i_0}'''}_{\widetilde{\underline{i}'}, \widetilde{\delta}'}|_{[{b_0}]}   \ast (  (s'' \ast  \id_{[0]}) (B_{\underline{z}''}^1 \ast \id_{[0]}) \delta_0^{\alpha_{i_0}-1})   & j = i_0 + 1,   \\
\text{ --- does not occur in this case ---} & \text{otherwise.} 
\end{cases} 
\end{equation*}
\end{itemize}

Note that e.g.\@ $s_{\widetilde{k}-2-b_0} s'' = (s'' \ast  \id_{[0]})s_{x-1}$. 
In each case the induction hypothesis thus gives equality. 
For the sign equation (\ref{eqsz1sgn}), note that passing from $\delta$ to $\delta''$ and $\delta'$ to $\delta'''$, the quantity $\alpha_{y_0}$ (resp.\@ $\alpha_{y_0'}'$)  increases by 1 and $\beta_{y_0}$ (resp.\@ $\beta_{y_0'}'$) decreases by 1, hence 
\begin{eqnarray*}
 \sgn(\widetilde{\underline{i}}^{\vee})  &=& (-1)^{\widetilde{k}-c_{y_0}+1} \sgn((\widetilde{\underline{i}}'')^{\vee}), \\
 \sgn((\widetilde{\underline{i}'})^{\vee})  &=& (-1)^{\widetilde{k}-1-c_{y_0'}'+1}  \sgn((\widetilde{\underline{i}'}'')^{\vee}), \\
 \sgn(\sigma) \sgn(\widetilde{\sigma})  &=& \sgn(\sigma'') \sgn(\widetilde{\sigma}'')  (-1)^{\sum_{y_0<j} \beta_j + \sum_{i<y_0}(\alpha_i-1)}, \\
\sgn(\sigma') \sgn(\widetilde{\sigma}')  &=& \sgn(\sigma''') \sgn(\widetilde{\sigma}''')   (-1)^{\sum_{y_0'<j} \beta_j' + \sum_{i<y_0'}(\alpha_i'-1)}.
\end{eqnarray*}
Inserting $c_{y_0} = y_0 + \sum_{j \le y_0} \beta_{j}$ and $c_{y_0'}' = y_0' + \sum_{j \le y_0'} \beta_{j}'$
we see that the LHS of (\ref{eqsz1sgn}) picks up a sign of
 \[ (-1)^{y_0-1-k+\sum_{i<y_0}(\alpha_i-1)}, \]
 and the RHS of (\ref{eqsz1sgn}) picks up a sign (taking into account that $\widetilde{k}$ decreases by 1) of
 \[ (-1)^{y_0'-1-k+\sum_{i<y_0'}(\alpha_i'-1)}. \]
 In case I and II their product is obviously $1$.  In case III their product is $(-1)^{\alpha_{i_0}}$.
 In turn, in case II, the expression $(-1)^{\sum_{i<j} z_j (z_i+1) + \sum_i \alpha_{i_0} z_i}$ does not change (notice that $\alpha_{i_0}$ also increases by 1).
 In case III the expression $(-1)^{\sum_{i<j} z_j (z_i+1) + \sum_i \alpha_{i_0} z_i}$ changes by a $(-1)^{\alpha_{i_0}}$.
\end{proof}

\begin{proof}[Proof of Proposition~\ref{PROPSHIHSZCZARBA}]
In order to be non-degenerate, we must have $b_{0} < \widetilde{k}$ by Lemma~\ref{LEMMASHIH2}. By Lemma~\ref{LEMMASHIH0} we can write
\[ \mathcal{P}_{\sigma} \mathcal{H}_{\underline{b}, \underline{i}}^n = \mathcal{P}_{\sigma} (s_{b_{0}}^{\otimes i_{0}} \otimes \mathcal{H}_{b_{0}}^n \otimes s_{b_{0}}^{\otimes k-1-i_{0}}) \mathcal{P}_{\sigma'} \mathcal{H}_{\underline{b}', \underline{i}'}^{n-1} \]
where $\sigma' = \sigma s_{i_{0}}$ and thus 
$\delta' = \delta_{i_{0}+1}  \delta$.
We will show the statement for $k=0$, $\underline{i} = ()$, where the LHS is $\id_{[n]}$ by definition. Then $\widetilde{\underline{i}} = (0,1,\dots,\widetilde{k}-1)$ and thus
\[ \sgn(\underline{i}^{\vee}) = 1, \qquad \sgn(\widetilde{\underline{i}}^{\vee}) = 1,  \]
$\sigma: [e] \to [0]$ is the contraction, and $\widetilde{\sigma}: [\widetilde{k}+e] \to [\widetilde{k}]$
is $\id_{[{k}]} \ast' \sigma$, hence, $\sgn(\sigma)=\sgn(\widetilde{\sigma})=1$. Furthermore
\[ \Sz_{\widetilde{\underline{i}}, \widetilde{\delta}}^{c_0} = \id_{[\widetilde{k}]} \qquad s_{\delta(0), \delta(1)}^{e+1} = \id_{[{e+1}]}. \]

By induction on $k$, we may assume that the statement is true for $k-1, n-1, e+1, \overline{i}'=(i_1, \dots, i_{k-1})$, with $\widetilde{k}' = \widetilde{k}-1$.  hence
\[ \mathcal{P}_{\sigma'} \mathcal{H}_{\underline{b}', \underline{i}'}^{n-1} \equiv \sum_{\substack{\widetilde{\underline{i}'} \\ \text{$\underline{i}'$ is a $\underline{b}'$-reduction of $\widetilde{\underline{i}'}$}} }  \varepsilon (\Sz^{c_0'}_{\widetilde{\underline{i}'},\widetilde{\delta}'}  \ast' s_{\delta'(0),\delta'(1)}^{k+1+e}) \otimes \cdots \otimes  (\Sz^{c_{k-1}'}_{\widetilde{\underline{i}'}, \widetilde{\delta}'}   \ast' s_{\delta'(k-1),\delta'(k)}^{k+1+e}) \]
therefore 
\begin{gather*}  \mathcal{P}_{\sigma} \mathcal{H}_{\underline{b}, \underline{i}}^n 
 \equiv \sum_{\widetilde{\underline{i}'}} \varepsilon' (s_{b_0} \Sz^{c_0'}_{\widetilde{\underline{i}'}, \widetilde{\delta}'}   \ast' s_{\delta(0),\delta(1)}^{k+1+e}) \otimes \cdots \otimes \\
  \left( \mathcal{P}^{\widetilde{k}+\delta(i_0)}_{\alpha_{i_0},\alpha_{i_0+1}}   \mathcal{H}_{b_0}^n(\Sz^{c_{i_{0}}'}_{\widetilde{\underline{i}'}, \widetilde{\delta}'}  \ast' s_{\delta(i_0), \delta(i_0+2)}^{k+1+e}) \right)     \\
\otimes \cdots   \otimes (s_{b_0}  \Sz^{c_{k-1}'}_{\widetilde{\underline{i}'}, \widetilde{\delta}'}    \ast' s_{\delta(k),\delta(k+1)}^{k+1+e}).  
 \end{gather*}
where 
\[ \mathcal{P}_{a,b}^c = \prod_{i=0}^{a-1}(\mathcal{P}_{c+i} \otimes 1) \prod_{i=0}^{b-1}(1 \otimes \mathcal{P}_{c+a+i})  \]
(and recall $\alpha_j = \delta(j+1)-\delta(j)$).

We will show 
\begin{gather}
\left( \mathcal{P}^{\widetilde{k}+\delta(i_0)}_{\alpha_{i_0},\alpha_{i_0+1}} \mathcal{H}_{b_0}^n(\Sz^{c_{i_{0}}'}_{\widetilde{\underline{i}'}, \widetilde{\delta}'}  \ast' s_{\delta(i_0), \delta(i_0+2)}^{k+e+1}) \right) \nonumber \\
 = \sum_{\underline{z} }  \varepsilon_{\underline{z}}
(\Sz^{c_{i_{0}}}_{\widetilde{\underline{i}}, \widetilde{\delta}}  \ast' s_{\delta(i_0), \delta(i_0+1)}^{k+e+1}) \otimes (\Sz^{c_{i_{0}+1}}_{\widetilde{\underline{i}}, \widetilde{\delta}}  \ast' s_{\delta(i_0+1), \delta(i_0+2)}^{k+e+1})    \label{eqsz0} 
\end{gather}
with
\[ \varepsilon_{\underline{z}} = (-1)^{\sum_{i<j} z_j (z_i+1) + \sum_i \alpha_{i_0} z_i} \]
where, in each summand, the extension $\widetilde{\underline{i}}$ is determined by $\widetilde{\underline{i}'}$ and 
$\underline{z}$ as in Lemma~\ref{LEMMABIJECTION}.
Equation (\ref{eqsz0}) together with (\ref{eqsz1sgn})
and case 3.\@ of (\ref{eqsz1}) in Lemma~\ref{LEMMABIJECTION} achieve the induction step and produce a global sign of
\[ \prod_{i=0}^{k-1} (-1)^{b_i-\widetilde{k}-i+1}. \]
On the other hand by (\ref{eqsignshuffle2})
\[ \sgn(\sigma_{\underline{b}}) = \prod_{i=0}^{k-1} (-1)^{b_i-(k-i-1)} \]
where $\sigma_{\underline{b}}: [n] \twoheadrightarrow [n-k]$. This yields a product of $(-1)^{(\widetilde{k}+k)k}=(-1)^{(\widetilde{k}+1)k}$.

It remains to show equation (\ref{eqsz0}).
By  Lemma~\ref{LEMMABIJECTION}, we can write $\Sz^{c_{i_0}'}_{\widetilde{\underline{i}'}, \widetilde{\delta}'}  = \Sz^{c_{i_0}'}_{\widetilde{\underline{i}'}, \widetilde{\delta}'}|_{[{b_0}]} \ast' s$ for a suitable degeneracy $s: [{\widetilde{k}-b_0-1}] \twoheadrightarrow [{x}]$. Then
\begin{eqnarray*}
 & &  \mathcal{P}^{\widetilde{k}+\delta(i_0)}_{\alpha_{i_0},\alpha_{i_0+1}}   \mathcal{H}_{b_0}^n( \Sz^{c_{i_0}'}_{\widetilde{\underline{i}'}, \widetilde{\delta}'}  \ast' s_{\delta(i_0), \delta(i_0+2)}^{k+1+e})  \\
 &=&  (\id_{[{b_0}]} \ast \mathcal{P}^{\widetilde{k}-b_0-1+\delta(i_0)}_{\alpha_{i_0},\alpha_{i_0+1}}\mathcal{H}^{n-b_0-1}) s_{b_0} ( \Sz^{c_{i_0}'}_{\widetilde{\underline{i}'}}|_{[{b_0}]}  \ast' s  \ast' s_{\delta(i_0), \delta(i_0+2)}^{k+1+e})    \\
 &=&  (\id_{[{b_0}]} \ast \mathcal{P}^{\widetilde{k}-b_0-1+\delta(i_0)}_{\alpha_{i_0},\alpha_{i_0+1}} \mathcal{H}^{n-b_0-1})  ( \Sz^{c_{i_0}'}_{\widetilde{\underline{i}'}, \widetilde{\delta}'}|_{[{b_0}]} \ast s  \ast' s_{\delta(i_0), \delta(i_0+2)}^{k+1+e})  s_{\widetilde{\delta}'(c_{i_0}')-x}   \\
 &=&  ( \id_{[{b_0}]}  \ast s \ast' s_{\delta(i_0), \delta(i_0+2)}^{k+1+e})  (\Sz^{c_{i_0}'}_{\widetilde{\underline{i}'}, \widetilde{\delta}'}|_{[{b_0}]} \ast \mathcal{P}^{x}_{\alpha_{i_0}, \alpha_{i_0+1}} \mathcal{H}^{x+\alpha_{i_0}+\alpha_{i_0+1}}   )  s_{\widetilde{\delta}'(c_{i_0}')-x}     \\
 &=& \sum_{\underline{z}} (-1)^{\sum_{i<j} z_i (z_j+1) + a \sum z_i }  \\
 & \cdot &  ( \Sz^{c_{i_0}'}_{\widetilde{\underline{i}'}, \widetilde{\delta}'}|_{[{b_0}]} \ast' s B_{\underline{z}}^0    \ast' \underbrace{s_{\delta(i_0), \delta(i_0+2)}^{k+1+e}  s_{\alpha_{i_0},\alpha_{i_0}+\alpha_{i_0+1}}^{\alpha_{i_0}+\alpha_{i_0+1}}}_{s_{\delta(i_0),\delta(i_0+1)}^{k+1+e}} )     \\  
 & \otimes & (\Sz^{c_{i_0}'}_{\widetilde{\underline{i}'},\widetilde{\delta}'}|_{[{b_0}]} \ast  s B_{\underline{z}}^1 \delta_0^{\alpha_{i_0}-1}  \ast'  \underbrace{s_{\delta(i_0), \delta(i_0+2)}^{k+1+e} s_{0,\alpha_{i_0}}^{\alpha_{i_0}+\alpha_{i_0+1}}}_{s_{\delta(i_0+1),\delta(i_0+2)}^{k+1+e}} )     \\
  \end{eqnarray*}
using Lemma~\ref{LEMMAH}, 2.
Thus equation (\ref{eqsz0}) follows from cases 1.\@ and 2.\@ of (\ref{eqsz1}) in Lemma~\ref{LEMMABIJECTION}.
\end{proof}

\section{Decalage}

This appendix contains a rather technical discussion of the compatibility of $\dec^*$ (total decalage), in the Abelian case, with $\otimes$, $\tildeotimes$, and some (weak)
simplicial enrichments used throughout the text. 
It should be consulted only if needed. 

Recall the explicit formul\ae{} for $\dec^*$ and $\dec_*$ in the Abelian case (Proposition~\ref{PROPEXPLICITAB}).

\subsection{Symmetry }

\begin{PAR}
Let $\sigma: \Delta^{\op} \times \Delta^{\op}$ be the morphism switching the factors. On the level of double complexes $\sigma^*$ has the effect of switching the indices. Let $X$ be a double complex with isomorphism
\[ \alpha: \sigma^* X \cong X  \]
We do not have an isomorphism $\dec \sigma \cong \dec$ but certainly an isomorphism of $\Ab$-enriched pro-functors 
\[ \widetilde{\sigma}: \dec\, \sigma \cong C\, \dec\, \sigma\, \iota \cong C \dec \iota \cong \dec.   \]
\end{PAR}

\begin{LEMMA}\label{LEMMADEC1}
Let $X$ be a symmetric double complex. 
The composition of $\dec_* \alpha$ with the mate of the morphism $L(\widetilde{\sigma})$
\[  \dec_* X \cong \dec_* \sigma^* X \cong \dec_* X \]
is given under the identification $r$ (i.e.\@ the identification described by Proposition~\ref{PROPEXPLICITAB}) as the morphism
\begin{eqnarray*} \bigoplus_{i+j=n} A_{i,j} &\to& \bigoplus_{i+j=n} A_{i,j}  \\
 a_{i,j} &\mapsto& (-1)^{ij}\alpha(a_{i,j})  
 \end{eqnarray*}
\end{LEMMA}
\begin{proof}
It suffies to show that 
\[ \xymatrix{ 
 D_i^{} \boxtimes D_j^{} \ar[d]^{s_{\can}} \ar[rr]^-{(-1)^{ij}} & &  \sigma^*(D_j^{} \boxtimes D_i^{}) \ar[d]^{s_{\can}} \\
 \dec^*(D_n^{}) \ar[rr]^{L(\widetilde{\sigma})} &  &  \dec^* \sigma^*(D_n^{})  }\]
commutes. Like in the proof of Proposition~\ref{PROPAWEZAB}, 2.\@ it suffices to see that in degree $n$, we have
\[ (-1)^{ij} c \dd_r s_{\can} = c  L(\widetilde{\sigma}) \dd_r s_{\can}.   \]
Since $\dd_r s_{\can} = \id$, the right hand side maps $[\{0,\dots, i\}] \otimes [\{0,\dots, j\}]$ to 
$[\{0,\dots, i+j+1\}]$
and then applies the isomorphism $[i] \ast [j] \cong [j] \ast [i]$ in $\FinSet$ switching the factors. The left hand side morphism
maps it to 
$(-1)^{ij}[\{0,\dots, i+j+1\}]$, which is the same (cf.\@ Lemma~\ref{LEMMAFINSET} for a description of the action).
\end{proof}

\begin{LEMMA}\label{LEMMADEC2}
\begin{enumerate}
\item The mate
\[ \dec^* \dec_* \to \dec_{12,3,*} \dec_{1,23}^*  \]
induced by ($L$ applied to) the obvious isomorphism $\dec\, \dec_{12,3} \cong \dec\, \dec_{1,23}$ is given by morphisms
\[ \alpha_{n,m}: \prod_{i+j'=n+m+1} A_{i,j'} \oplus \prod_{i+j'=n+m} A_{i,j'}  \to \prod_{i+j=n} \left( A_{i,j+m+1} \oplus A_{i,j+m} \right)  \]
product of the identities 
\[ A_{i,j+m+1} \oplus A_{i,j+m} \to A_{i,j+m+1} \oplus A_{i,j+m}  \]
over $i+j=n$.
\item The mate
\[ \dec^* \dec_* \to \dec_{2,13,*} \dec_{1,23}^*  \]
induced by ($L$ applied to) the obvious isomorphism $C\,  \dec\, \dec_{2,13}\, \iota  \cong C\, \dec\, \dec_{1,23}\, \iota$ is given by  morphisms
\[  \alpha_{n,m}: \prod_{i+j'=n+m+1} A_{i,j'} \oplus \prod_{i+j'=n+m} A_{i,j'}  \to \prod_{i+j=m} \left( A_{i,n+j+1} \oplus A_{i,n+j} \right)  \]
product of
\[ A_{i,j+n+1} \oplus A_{i,j+n} \to A_{i,n+j+1} \oplus A_{i,n+j}  \]
\[  \Mat{(-1)^{i(n+1)} & 0 \\ (-1)^{(i+1)n} \mathrm{d}_l & (-1)^{in} }  \]
over $i+j=m$.
\item The mate
\[ \dec^* \dec_* \to \dec_{13,24,*} \dec_{12,34}^*  \]
induced by ($L$ applied to) the obvious isomorphism $C\, \dec\, \dec_{13,24}\, \iota \cong C\, \dec\, \dec_{12,34}\, \iota$ is given by morphisms
\begin{gather*} \alpha_{n,m}:  \prod_{i'+j'=n+m+1} A_{i',j'} \oplus \prod_{i'+j'=n+m} A_{i',j'}  \\
\to \prod_{i+j=n} \prod_{k+l=m}  \left( A_{i+k+1,j+l+1} \oplus A_{i+k+1,j+l}\oplus A_{i+k,j+l+1}\oplus A_{i+k,j+l} \right)  
\end{gather*}
product of
\[ \left(A_{i+k+1,j+l} \oplus A_{i+k,j+l+1}\right) \oplus A_{i+k,j+l} \to A_{i+k+1,j+l+1} \oplus A_{i+k+1,j+l}\oplus A_{i+k,j+l+1}\oplus A_{i+k,j+l} \]
 \[ \Mat{ 0 & 0 & 0 \\ 
 (-1)^{j(k+1)} & 0 & 0 \\
 0& (-1)^{(j+1)k} & 0 \\
 0 & 0 & (-1)^{jk} 
 } \]
(with the left two factors from the first product) over $i+j=n, k+l=m$
\end{enumerate}
\end{LEMMA}
\begin{proof}
The morphisms are in each case induced by a collection for $n+m=x$
\[ \alpha_{n,m}:  \prod_{i+j'=n+m} A_{i,j'}  \to \prod_{i+j=n} \left( A_{i,j+m+1} \oplus A_{i,j+m} \right)  \]
maps $A_{i,j'}$ to the factor $A_{i,j+m}$ on the right with $j':=j+m$.
\[  \alpha_{n,m}: \prod_{i+j'=n+m} A_{i,j'}  \to \prod_{i+j=m} \left( A_{i,n+j+1} \oplus A_{i,n+j} \right)  \]
maps $A_{i,j'}$ to the factor $A_{i,j+n}$ on the right with $j':=j+n$ with sign $(-1)^{in}$.
\[ \alpha_{n,m}: \prod_{i'+j'=n+m} A_{i',j'}  \to \prod_{i+j=n, k+l=m}  \left( A_{i+k+1,j+l+1} \oplus A_{i+k,j+l+1}\oplus A_{i+k+1,j+l}\oplus A_{i+k,j+l} \right)  \]
maps $A_{i',j'}$ to the factor $A_{i+k,j+l}$ on the right with $i' = i+k$ and $j':=j+l$ with sign $(-1)^{jk}$.

In each case, we get on 
\[  \prod_{i'+j'=n+m+1} A_{i',j'} \oplus \prod_{i'+j'=n+m} A_{i',j'}   \]
the map 
\[ (\mathrm{d}_m \alpha_{n,m+1}, \alpha_{n,m}) \]
where $\mathrm{d}_m$ is the differential in the $m$-direction (cf.\@ proof of Proposition~\ref{PROPEXPLICITAB}, 1.). 

In case 1 the map is
\[ \Mat{0 & 1 \\ 0 & 0 } \Mat{0 & 0 \\ \id & 0} + \Mat{0 & 0 \\ 0 & \id}  =  \Mat{\id & 0 \\ 0 & \id}  \]
In case 2 the map is
\begin{gather*}  
 \Mat{\mathrm{d}_l & 0 \\ 0 & \mathrm{d}_l } \Mat{0 & 0 \\ (-1)^{(i+1)n}  & 0} + (-1)^{i} \Mat{0 & 1 \\ 0 & 0 } \Mat{0 & 0 \\ (-1)^{in}  & 0}  + \Mat{0 & 0 \\ 0 & (-1)^{in}}   \\
  =  \Mat{(-1)^{i(n+1)} & 0 \\ (-1)^{(i+1)n} \mathrm{d}_l & (-1)^{in} }  
\end{gather*}  
In case 3 the map is
\begin{gather*}  
\footnotesize \Mat{& & 1 &  \\
 &  &  & 1 \\
 &  &  &  \\
 &  &  & 
}  \Mat{0&0&0 \\ 0&0&0\\ 0&0 &0\\ (-1)^{j(k+1)}  &0&0}   
+(-1)^k \Mat{ & 1 &  &  \\
 &  &  &  \\
 &  &  & 1 \\
 &  &  &  
}  \Mat{0& 0&0 \\ 0&0&0\\ 0&0&0 \\ 0& (-1)^{jk}  &0}   
 + \Mat{0&0 & 0 \\ 0&0 &0\\ 0&0&0 \\0&0& (-1)^{jk} }   \\
 \raisebox{\dimexpr\depth}{ \xymatrix{= \Mat{ 0 & 0 &0 \\ 
 (-1)^{j(k+1)}&0 & 0 \\
0& (-1)^{(j+1)k} & 0 \\
0& 0 & (-1)^{jk} 
} }} \qedhere
\end{gather*}
\end{proof}

 \subsection{$\widetilde{\dec}^*$ is simplicially enriched}
 
\begin{PAR}
The functor $\dec^*: \Ch_{\ge 0}(\mathcal{C}) \to \Ch_{\ge 0}(\Ch_{\ge 0}(\mathcal{C}))$ extends naturally to an {\em unbounded} functor (given by the same formula as in Proposition~\ref{PROPEXPLICITAB}, 2.)
\[ \widetilde{\dec}^*:   \Ch(\mathcal{C}) \to  \Ch(\Ch(\mathcal{C})) \]
and is left adjoint to $\tot^{\prod}$. For this functor, we have (by inspecting Lemma~\ref{LEMMADEC2}, 2.) a canonical isomorphism
\[ \widetilde{\dec}^* (A \tildeotimes B) \cong A \tildeotimes \widetilde{\dec}^* B  \]
where $A \in \Ch_b(\mathcal{C})$ or $A \in \Ch_b(\Ab)$. On the right hand side the tensor is w.r.t.\@ the second variable.
 In particular 
\[ \widetilde{\dec}^*:   \Ch(\mathcal{C}) \to  \Ch(\Ch(\mathcal{C})) \]
is simplicially enriched (w.r.t.\@ the simplicial enrichment $\uHom^{\tildeotimes}$) and so is the restriction\footnote{but {\em not} the restriction $\Ch_{\ge 0}(\mathcal{C}) \to  \Ch_{\ge 0}(\Ch_{\ge 0}(\mathcal{C}))$!}:
\[ \widetilde{\dec}_{\ge 0}^*:   \Ch_{\ge 0}(\mathcal{C}) \to  \Ch(\Ch_{\ge 0}(\mathcal{C})). \]
\end{PAR}

\begin{PROP}\label{PROPEXPLICITDECENRICHED}
On $\Hom$-complexes the enrichment of $\widetilde{\dec}^*$ is explicitly given by: 
\begin{eqnarray*}  \uHom^{\tildeotimes}(X, Y)_q & \to & \uHom^{\tildeotimes}(\widetilde{\dec}^* X, \widetilde{\dec}^* Y)_q  \\
 f& \mapsto & \begin{cases} \Mat{f & \\ (-1)^{i+1} \partial f & f} & \text{$q$ even} \\
 \Mat{(-1)^{i+1}  f & \\ \partial f&  (-1)^{i} f} & \text{$q$ odd} \end{cases}
\end{eqnarray*}
where $\partial$ is the differential in the $\Hom$-complex. 
\end{PROP}

\begin{proof}
In Lemma~\ref{LEMMADEC2}, 2., working with unbounded complexes (in one direction) we have the morphism
\[ \dec^* \dec_* \to \dec_{2,13,*} \dec_{1,23}^*  \]
which is explicitly given by
\[  \alpha_{n,m}: \prod_{i+j'=n+m+1} A_{i,j'} \oplus \prod_{i+j'=n+m} A_{i,j'}  \to \prod_{i+j=m} \left( A_{i,n+j+1} \oplus A_{i,n+j} \right)  \]
whose value at the pair at $i,j$ is the value at the pair at $i, j' = j+n$ transformed by
\[  \Mat{(-1)^{i(n+1)} & 0 \\ (-1)^{(i+1)n} \mathrm{d}_l & (-1)^{in} }  \]
It extends to unbounded complexes in the $j$-direction and has an inverse there, 
 given by 
\[  \Mat{(-1)^{i(n+1)} & 0 \\ (-1)^{(i+1)(n+1)} \mathrm{d}_l & (-1)^{in} }   \]
and letting $j = j'-n$.

 Now let
 \[ \kappa: \Delta_q^{\circ} \tildeotimes B \to C  \]
 be a morphism determined up to degeneracies by $\kappa(\{0, \dots, q\} \otimes b_{i}) = \widetilde{\kappa}(b_{i})$.
 We have $\kappa \dd_l (\{0, \dots, q\} \otimes b_{i}) = \partial(\widetilde{\kappa})(b_{i})$ where $\partial$ is the differential in the $\uHom$-complex. 
 $\dec^* \kappa$ is the diagonal 
 \[ \prod_{i+j'=n+m+1} (\Delta_q^{\circ})_i \otimes B_{j'} \oplus \prod_{i+j'=n+m} (\Delta_q^{\circ})_i \otimes B_{j'}  \to C_{n+m+1} \oplus C_{n+m}   \]
 and the composition
 \[ \prod_{i+j=m} (\Delta_q^{\circ})_i \otimes (B_{n+j+1} \oplus  B_{n+j}) \to 
 \prod_{i+j'=n+m+1} (\Delta_q^{\circ})_i \otimes B_{j'} \oplus \prod_{i+j'=n+m} (\Delta_q^{\circ})_i \otimes B_{j'}  \to C_{n+m+1} \oplus C_{n+m}   \]
 maps $\{0, \dots, q\} \otimes (b_{m+j+1}, b_{m+j} )$ for $i=q$ to 

 \[  \Mat{(-1)^{q(n+1)} \widetilde{\kappa} & 0 \\ (-1)^{(q+1)(n+1)} \partial(\widetilde{\kappa})  & (-1)^{qn} \widetilde{\kappa}}  \Mat{b_{m+j+1} \\ b_{m+j}  }   \]
This gives the formula stated in the Proposition. 
\end{proof}

Notice that this enrichment does not work properly in the bounded case. However, see (\ref{PARDECOTIMESENRICHED}) for a workaround.

 \subsection{Compatibility with $\otimes$}\label{SECTDECOTIMES}
 
 $\dec^*$ is canonically monoidal for $\otimes$ because the latter is computed point-wise interpreting source and destination as (bi)simplicial diagrams. 
 We denote the corresponding functor as
 \[ \dec^*: (\mathcal{C}^{\Delta^{\op}}, \otimes) \to (\mathcal{C}^{\Delta^{\op} \times \Delta^{\op}}, \otimes)  \]
(the notation $\dec^*_{\otimes}$  being reserved for its composition with the Alexander-Whitney map in the second variable). 
 
Since this functor cannot be extended to the unbounded case, we cannot easily extend it to a simplicially enriched functor. However, we can do the following: 

\begin{PAR}\label{PARDECOTIMESENRICHED}
Let's call $\tau_{\ge 1}$ the truncation with $0$ in degree $0$ and the grading of $\dec^* X$ is --- as always --- w.r.t.\@ the second variable. 

Denote $\tau\dec^*_{\otimes} := \tau_{\ge 1} \Awfrak_2 \dec^*$. We will enhance this functor to a functor of weakly simplicially enriched operads: 
We define a morphism 
\begin{equation}\label{eqdecotimes}
   \tau\dec^*_{\otimes}: \uHom^{\tildeotimes, t}_{(\mathcal{C}^{\Delta^{\op}}, \otimes)^{\circ, \vee}}(X;  Y_1, \dots, Y_k) \to \uHom^{\tildeotimes, t}_{(\mathcal{C}^{\Delta^{\op}\times \Delta^{\op}}, \utildeotimes)^{\circ, \vee}}(\tau_{\ge 1} \dec^*X; \tau_{\ge 1} \dec^* Y_1, \dots, \tau_{\ge 1} \dec^* Y_k)  
\end{equation}
(cf.\ Definition~\ref{DEFTRUNC}) as follows:

The morphism
\[ \dec^* \dec_* \to \dec_{2,13,*} \dec_{1,23}^*  \]
gives a comparison map
\[ \Xi: \dec^* (\Delta_{n} \tildeotimes X) \to \Delta_{n} \tildeotimes (\dec^* X)  \]
which is an isomorphism {\em in degree $\ge n$} as follows from Lemma~\ref{LEMMADEC2}, 2.\@

Given a morphism
\[ \alpha: \Delta_n \tildeotimes X \to Y_1 \otimes \cdots \otimes Y_k  \]
Apply $\dec^*$, $\Awfrak_2$ and  $\tau_{\ge 1}$  to get
\[ \xymatrix{ \Delta_n \tildeotimes \dec^* X &  \ar[l]_{\Xi}  \dec^* (\Delta_n \tildeotimes X) \ar[d]^{\dec^* \alpha} \\
&  \dec^*Y_1 \otimes \cdots \otimes \dec^* Y_k \ar[d]^{\Awfrak_2} \\
 & \dec^*Y_1 \utildeotimes \dots \utildeotimes \dec^*Y_k  \ar[d] \\
& \tau_{\ge 1} \dec^*Y_1 \utildeotimes \dots \utildeotimes \tau_{\ge 1} \dec^*Y_k }   \]
Since the comparison map $c$ is an isomorphism in degree $\ge n$, for $n < k$, we get a unique extension: 
\[ \xymatrix{ \Delta_n \tildeotimes \dec^* X \ar@{-->}[ddd] &  \ar[l]  \dec^* (\Delta_n \tildeotimes X) \ar[d] \\
&  \dec^*Y_1 \otimes \cdots \otimes \dec^* Y_k \ar[d] \\
 & \dec^*Y_1 \utildeotimes \dots \utildeotimes \dec^*Y_k  \ar[d] \\
\Delta_n \tildeotimes (\tau_{\ge 1} \dec^* X) \ar@{-->}[r] & \tau_{\ge 1} \dec^*Y_1 \utildeotimes \dots \utildeotimes \tau_{\ge 1} \dec^*Y_k }   \]
\end{PAR}

\begin{PROP}\label{PROPSIMPLICIALLYENRICHEDDEC}
The morphisms (\ref{eqdecotimes}) 
\[ \tau \dec^*_{\otimes}:  \uHom^{\tildeotimes, t}_{(\mathcal{C}^{\Delta^{\op}}, \otimes)^{\circ, \vee}}(X; Y_1, \dots, Y_k) \to  \Aw^* \uHom^{\tildeotimes, t}_{(\mathcal{C}^{\Delta^{\op} \times \Delta^{\op}}, \utildeotimes)^{\circ, \vee}}(\tau_{\ge 1} \dec^*X; \tau_{\ge 1} \dec^* Y_1, \dots, \tau_{\ge 1} \dec^* Y_k)  \]
assemble to a functor between weakly $(\Ab^{\Delta^{\op}}, \otimes)$-enriched cooperads.
\end{PROP}
\begin{proof}We will show only the essential non-trivial case.
Given morphisms $X \to Y_1 \otimes Y_2$, $\Delta_n \tildeotimes Y_1 \to Z_1^{\otimes k_1}$, and $\Delta_n \tildeotimes Y_2 \to Z_2^{\otimes k_2}$,
recall that their composition is defined as
\[  \Delta_n \tildeotimes X  \to (\Delta_n \otimes \Delta_n) \tildeotimes (Y_1 \otimes Y_2) \to (\Delta_n \tildeotimes Y_1) \otimes (\Delta_n \tildeotimes Y_2) \to Z_1^{\otimes k_1} \otimes Z_2^{\otimes k_2}   \]
where the second map is the switch map (Definition~\ref{DEFSWITCH}). 
In the following diagram 
\[ \tiny \xymatrix{   \dec^* (\Delta_n \tildeotimes X) \ar[r] \ar[d] & \pr_2^* \Delta_n \utildeotimes \dec^* X \ar[d] \ar@{=}[r] &  \pr_2^* \Delta_n \utildeotimes \dec^* X \ar[d]   \\
\dec^* ((\Delta_n \otimes \Delta_n) \tildeotimes (Y_1 \otimes Y_2)) \ar[r] \ar[d]^{\dec^* \switch} & \pr_2^* (\Delta_n \otimes \Delta_n) \utildeotimes \dec^* (Y_1 \otimes Y_2) \ar[d]^{\switch} \ar[r] & 
\pr_2^* \Delta_n \utildeotimes \pr_2^*\Delta_n \utildeotimes \dec^* Y_1 \utildeotimes \dec^* Y_2 \ar[d] \\
\dec^*(\Delta_n \tildeotimes Y_1) \otimes  \dec^*(\Delta_n \tildeotimes Y_2) \ar[r] \ar[d] &  (\pr_2^* \Delta_n \utildeotimes  \dec^* Y_1) \otimes (\pr_2^* \Delta_n  \utildeotimes \dec^* Y_2) \ar[d] \ar[r] & 
\pr_2^* \Delta_n \utildeotimes \dec^* Y_1 \utildeotimes  \pr_2^*\Delta_n \utildeotimes \dec^* Y_2 \ar[d] \\ 
\dec^* Z_1^{\otimes k_1} \otimes  \dec^* Z_2^{\otimes k_1} \ar[r]  &  (\tau_{\ge 1} \dec^* Z_1) ^{\utildeotimes k_1} \otimes  (\tau_{\ge 1} \dec^* Z_1) ^{\utildeotimes k_1} \ar[r]  &
(\tau_{\ge 1} \dec^* Z_1) ^{\utildeotimes k_1} \utildeotimes  (\tau_{\ge 1} \dec^* Z_1) ^{\utildeotimes k_1}    &  \\
  } \]
the middle left square commutes by Lemma~\ref{LEMMADECSWITCH} below and the others in the left column obviously. The right column is essentially the same as in Proposition~\ref{PROPSIMPLICIALLYENRICHED}, 2. 
The statement follows (suitably applying $\tau_{\ge 1}$). 
\end{proof}

\begin{LEMMA}\label{LEMMADECSWITCH}
We have a commutative diagram in which the switch map is defined in Definition~\ref{DEFSWITCH}:
\[ \xymatrix{ 
  \dec^* (Y_1 \otimes Y_2) \utildeotimes \pr_2^* (\Delta_n \otimes \Delta_n)  \ar@{<-}[d]_\Xi \ar[r]^-{\mathrm{switch}} &    ((\dec^* Y_1) \utildeotimes (\pr_2^* \Delta_n)) \otimes ((\dec^*Y_2)) \utildeotimes (\pr_2^* \Delta_n)) \ar@{<-}[d]^{\Xi \otimes \Xi} \\   
\dec^*((Y_1 \otimes Y_2) \tildeotimes (\Delta_n \otimes \Delta_n)) \ar[r]_-{\dec^*(\mathrm{switch})} & \dec^*(Y_1 \tildeotimes \Delta_n) \otimes \dec^*(Y_2 \tildeotimes \Delta_n)
}\]
\end{LEMMA}
\begin{proof}
By definition of the switch morphism this boils down to see the commutativity of the three squares: 
1.
 \[ \xymatrix{ 
\dec^*(A \tildeotimes B) \ar[r]^{\dec^*(\Ezfrak)} \ar[d]_{\Xi} & \dec^*(A \otimes B) \ar[d]^{\pr} \\ 
  \dec^*B \utildeotimes \pr_2^*(A)   \ar[r]_{\Ezfrak} &   \dec^*(B) \otimes \pr_2^*(A)  } \]
  This would follow from the commutativity of: 
 \[ \xymatrix{ 
\dec^* \dec_* \ar[rr]^{\dec^*\Ez} \ar[d]_{\Xi} & & \dec^* \delta^* \ar[d]^{\pr} \\ 
\dec_{23,*} \dec_{12}^*    \ar[rr]_{\Ez \dec_{12}^*} & & \delta_{23}^* \dec_{12}^*   } \]
which is adjoint to
 \[ \xymatrix{ 
 \dec_* \ar[r]^-{\Ez} \ar[d]_{\Xi} &  \delta^* \ar[d]^{\pr} \\ 
\dec_* \dec_{23,*} \dec_{12}^*    \ar[r]_-{\Ez} & \dec_*\delta_{23}^* \dec_{12}^*   } \]
This is in standard form (\ref{PARSTANDARD}) such that it suffices to show that 
\[ \xymatrix{  \dec^* \ar[r]^-{\mlq c^{\op} \mrq} \ar[d]_-{\mlq c^{\op} \mrq} &  \dec^* \delta^* \dec^*  \ar[d]^{\pr} \\
   \delta^*_{12} \dec^*_{12}  \dec^*    \ar[r]_-{\Xi} &   \delta^*_{12} \dec^*_{23} \dec^*  } \]
   is commutative. Since $C$ commutes with $\dec^*$ (cf.\@ Lemma~\ref{LEMMACCOH}) it suffices to see that the  diagram of functors $\FinSet^2 \to \FinSet$
\[ \xymatrix{  \dec_s \ar@{<-}[r]^-{c} \ar@{<-}[d]_{c} &  \dec_s \delta_s \dec_s  \ar@{<-}[d]^{\pr} \\
    \dec_s \dec_{s,12} \delta_{s,12}    \ar@{<-}[r]_-{\Xi} &    \dec_s \dec_{s,23} \delta_{s,12}   } \]
   commutes. This is a straightforward check.

2. The commutativity of
\[ \footnotesize \xymatrix{ 
\dec^*(\Delta_n \otimes \Delta_n \otimes X \otimes Y) \ar[r]^{\dec^*(\sigma)} \ar[d]_{\pr} & \dec^*(\Delta_n \otimes X \otimes \Delta_n \otimes Y) \ar[d]^{\pr} \\ 
\pr_2^*(\Delta_n) \otimes \pr_2^*(\Delta_n) \otimes \dec^*(X) \otimes \dec^*(Y)   \ar[r]_{\sigma} & \pr_2^*(\Delta_n) \otimes \dec^*(X) \otimes \pr_2^*(\Delta_n) \otimes \dec^*(Y)   } \]
 is trivial.

3.\@ The commutativity of 
\[ \xymatrix{ 
\dec^*( X \otimes \Delta_n)\ar[r]^{\dec^*(\Aw)} \ar[d]_{\pr} & \dec^*(  X \tildeotimes \Delta_n) \ar[d]^{\Xi} \\ 
  \dec^*(X) \otimes \pr_2^*(\Delta_n)   \ar[r]_{\Aw} &   \dec^*(X) \utildeotimes  \pr_2^*(\Delta_n)  } \]
 would follow from the commutativity of
 \[ \xymatrix{ 
\dec^* \dec_* \ar@{<-}[r]^-{\Aw} \ar[d]_-{\Xi} & \dec^* \delta^* \ar[d]^{\pr} \\ 
\dec_{23,*} \dec_{12}^*    \ar@{<-}[r]_-{\Aw} & \delta_{23}^* \dec_{12}^*   } \]
which is adjoint to
 \[ \xymatrix{ 
 \dec_* \ar@{<-}[r]^-{\Aw} \ar[d]_{\Xi} &  \delta^* \ar[d]^{\pr} \\ 
\dec_* \dec_{23,*} \dec_{12}^*  \ar@{<-}[r]^-{\Aw} & \dec_*\delta_{23}^* \dec_{12}^*   } \]
This is in standard form (\ref{PARSTANDARD}) such that it suffices to show that 
\[ \xymatrix{  \dec_{23}^* \dec^*  \delta^*  \ar[r]^-{\Xi} \ar[d]_{\pr} & \dec_{12}^* \dec^* \delta^*  \ar[d]^{\dec_{12}^* u^{\op}} \\
  \dec_{23}^* \delta_{23}^* \dec_{12}^*    \ar[r]_-{u^{\op} \dec_{12}^* } &  \dec_{12}^*   } \]
   is commutative, which follows from the commutativity (straighforward check) of the diagram of functors $\Delta^3 \to \Delta^2$ (no $\FinSet$-symmetry involved):
   \[ \raisebox{\dimexpr\depth-3\fboxsep}{ \xymatrix{   \delta \dec \dec_{23}  \ar@{<-}[rr]^{\Xi} \ar@{<-}[d]_{\pr} & & \delta \dec  \dec_{12}  \ar[d]^{u \dec_{12}} \\
 \dec_{12}  \delta_{23}  \dec_{23}  \ar@{<-}[rr]_{\dec_{12} u} & &  \dec_{12}   }}  \qedhere \]
\end{proof}

 \subsection{Compatibility with $\tildeotimes$} \label{SECTDECTILDEOTIMES} \label{PARDEC}

Recall the adjunction
\[ \xymatrix{ \mathcal{C}^{\Delta^{\op}} \ar@<3pt>[r]^-{\dec^*}  & \ar@<3pt>[l]^-{\dec_*} \mathcal{C}^{\Delta^{\op} \times \Delta^{\op}} }  \]
where $\dec_* \cong \tot$ interpreting the objects as complexes, and double complexes, respectively (cf.\@ Proposition~\ref{PROPEXPLICITAB}, 1.). 
It actually induces an adjunction between cooperads
\[ \xymatrix{ (\mathcal{C}^{\Delta^{\op}}, \tildeotimes)^{\vee}  \ar@<3pt>[r]^-{\dec^*_{\tildeotimes}}  & \ar@<3pt>[l]^-{\dec_{\tildeotimes,*}} (\mathcal{C}^{\Delta^{\op} \times \Delta^{\op}}, \tildeotimes)^{\vee} }  \]
as follows, where the monoidal product $\tildeotimes$ on the right hand side is given by
\[ \dec_{13,24,*} - \boxtimes - \]
which can also be seen with the interpretation $\mathcal{C}^{\Delta^{\op} \times \Delta^{\op}} = (\mathcal{C}^{\Delta^{\op}})^{\Delta^{\op}}$ as the usual $\tildeotimes$-product
\[ A \tildeotimes B = \dec_* A \boxtimes B\]
where $\boxtimes$ is the exterior product $(\mathcal{C}^{\Delta^{\op}})^{\Delta^{\op}} \times (\mathcal{C}^{\Delta^{\op}})^{\Delta^{\op}} \to (\mathcal{C}^{\Delta^{\op}})^{\Delta^{\op} \times \Delta^{\op}}$ applying (the internal) 
$\tildeotimes$ point-wise. 

The functor $\dec_*$ is actually monoidal by means of the  isomorphism: 
\[  \dec_*(A \tildeotimes B) = \dec_* \dec_{13,24,*} A \boxtimes B  \ \cong\  \dec_* \dec_{12,34,*} A \boxtimes B  =  (\dec_*A) \tildeotimes (\dec_*B)  \]
where the middle isomorphism is induced by the isomorphism
\[ \dec \dec_{13,24} \cong C \dec \dec_{13,24} \iota \cong C \dec \dec_{12,34} \iota \cong \dec \dec_{12,34} \]
of $\Ab$-enriched pro-functors. It follows from Lemma~\ref{LEMMADEC1}, that explicitly, for a double complexes $A$, $B$:
\begin{eqnarray*}
 \tot(A \tildeotimes B) = \bigoplus_{i,j,k,l} A_{i,j} \otimes B_{k, l}  &\cong& (\tot A) \tildeotimes (\tot B) =  \bigoplus_{i,j,k,l} A_{i,j} \otimes B_{k, l}  \\
 a_{i,j} \otimes b_{k,l} &\mapsto& (-1)^{jk} a_{i,j} \otimes b_{k,l}  
\end{eqnarray*}
The left adjoint $\dec^*$ is automatically lax monoidal (hence a functor of cooperads) via the mate 
\[  \dec^*(A \tildeotimes B) = \dec^* \dec_* A \boxtimes B  \ \to\  \dec_{13,24,*} \dec_{12,34}^* A \boxtimes B  =  \dec^*A \tildeotimes \dec^*B  \]
and induces an adjunction between the categories of coalgebras. An explicit formula of the mate is given in Lemma~\ref{LEMMADEC2}, 3.
We denote the corresponding functors of cooperads (and of coalgebras) by $\dec_{\tildeotimes}^*$, and $\dec_{\tildeotimes,*}$, respectively.

Finally, in the corresponding adjunction
\[ \xymatrix{ \mathcal{C}^{\Delta^{\op}}_{\mathrm{/1}}  \ar@<3pt>[r]^-{\dec^*}  & \ar@<3pt>[l]^-{\dec_*} \mathcal{C}^{\Delta^{\op} \times \Delta^{\op}}_{\mathrm{/1}} }  \]
between augmented objects, both functors commute with the forgetful functors forgetting the augmentation. For $\dec^*$ this is clear, and for $\dec_*$ it follows because the latter is computed by a connected finite limit (Proposition~\ref{PROPEXPLICIT}).
Hence we will use them on augmented objects without further distinction.

\subsection{More adjoints}\label{SECTMOREADJOINTS}

We also have the formula
\[ \dec_* = \int^n \Z[\Delta_n] \tildeotimes X_{n,\bullet}  \]
which comes from the tautological formula (Yoneda):
\[ X = \int^n \Delta_n \boxtimes X_{n,\bullet}  \]
and the fact that $\dec_*$ obviously commutes with colimits (as one can see from the explicit formula $\dec_* \cong \tot$, cf.\@ Proposition~\ref{PROPEXPLICITAB}). 
This implies that $\dec_*$ has a further right adjoint given by
\[ (\dec^? X)_{[n], \bullet} =  \mathcal{HOM}(\Z[\Delta_n], X)_{\bullet},  \]
or also (modulo degeneracies) 
\[ (\dec^? X)_{n, \bullet} =  \mathcal{HOM}(D_n, X)_{\bullet},  \]
which is similar to $\dec^*$ except in low degrees: 

\[ \xymatrix{ 
\ddots & \vdots \ar[d]& \vdots \ar[d]& \vdots \ar[d] \\
\cdots \ar[r] & X_4 \oplus X_3 \ar[d]   \ar[r] & X_3 \oplus X_2 \ar[r] \ar[d]  & X_2 \ar[d] \\
\cdots \ar[r] & X_3 \oplus X_2 \ar[d]  \ar[r] & X_2 \oplus X_1 \ar[r] \ar[d] & X_1 \ar[d] \\
\cdots \ar[r] & X_2 \ar[r] & X_1 \ar[r] & X_0 
} \]

\begin{LEMMA}\label{LEMMADECQM}
There is a natural transformation 
\[ \pr_2^* \Rightarrow \dec^?  \]
(adjoint of $\dec_* \pr_2^* \cong \id$ whose inverse is induced by $\dec^* \Rightarrow \pr_2^*$. ) 
which is a row-wise quasi-isomorphism.
\end{LEMMA}
\begin{proof}
This can be seen directly by inspection or by the argument of Lemma~\ref{LEMMAW1} using that $\Delta_1 \tildeotimes X \to X$ is a quasi-isomorphism and hence by associativity $\mathcal{HOM}^{\tildeotimes}$ maps the
homotopy equivalence (w.r.t.\@ $\otimes$ and hence also w.r.t.\@ $\tildeotimes$) $\Z[\Delta_n] \to \Z[\Delta_0]$ to a quasi-isomorphism. 
\end{proof}

We have a functor of cooperads
\[ \pr^*_{2,\tildeotimes}: (\mathcal{C}^{\Delta^{\op}}, \tildeotimes)^{\vee} \to ((\mathcal{C}^{\Delta^{\op}}, \tildeotimes)^{\vee})^{(\Delta, \ast)^{\op}} \cong  (\mathcal{C}^{\Delta^{\op} \times \Delta^{\op}}, \tildeotimes)^{\vee}  \]
which comes from the functoriality of the Day-convolution w.r.t.\@ the projection $(\Delta, \ast)^{\op} \to \OOO^{\op}$. Unraveling the definition it boils down to the obvious morphism
\[ \pr_2^{*}(X \tildeotimes Y)\to \pr_2^{*}X \tildeotimes \pr_2^{*}Y  \]
noticing that, in the complexes viewpoint, $\pr_2^* X$ places $X$ in horizontal degree 0 and is zero other-wise. Since this is an isomorphism, we also have a morphism 
between the corresponding operads: 
\[ \pr^*_{2,\tildeotimes}: (\mathcal{C}^{\Delta^{\op}}, \tildeotimes) \to  (\mathcal{C}^{\Delta^{\op} \times \Delta^{\op}}, \tildeotimes).  \]

On the other hand, in \ref{PARDEC}, we have enhanced $\dec_*$ to a morphism of cooperads: 
\[ \dec_{\tildeotimes,*}:    (\mathcal{C}^{\Delta^{\op} \times \Delta^{\op}}, \tildeotimes)^{\vee}  \to (\mathcal{C}^{\Delta^{\op}}, \tildeotimes)^{\vee} \]
Being monoidal, it is actually also a morphism of the corresponding operads, and induces a right adjoint
\[ \dec_{\tildeotimes}^?:   (\mathcal{C}^{\Delta^{\op}}, \tildeotimes)  \to   (\mathcal{C}^{\Delta^{\op} \times \Delta^{\op}}, \tildeotimes)   \]

\begin{LEMMA}\label{LEMMADECQMTILDEOTIMES}
The natural transformation from Lemma~\ref{LEMMADECQM} is actually a natural transformation between functors of operads: 
\[  \pr^*_{2,\tildeotimes} \Rightarrow \dec_{\tildeotimes}^?  \]
\end{LEMMA}
\begin{proof}
This boils down to the commutativity of
\[ \xymatrix{ \pr_2^{*}(X \tildeotimes Y) \ar@{<-}[r] \ar[d]  & \pr_2^{*}X \tildeotimes \pr_2^{*}Y \ar[d]  \\
\dec^? (X \tildeotimes Y) \ar@{<-}[r] & \dec^? X \tildeotimes \dec^? Y }\]
or equivalently to the commutativity of
\[ \xymatrix{ \dec_* \pr_2^{*}(X \tildeotimes Y) \ar@{<-}[r] \ar@{<-}[d]^-{\sim}  & \dec_*(\pr_2^{*}X \tildeotimes \pr_2^{*}Y) \ar@{<-}[d]^{\widetilde{\sigma}}  \\
X \tildeotimes Y \ar[r]^-{\sim} & \dec_* \pr_2^* X \tildeotimes \dec_* \pr_2^* Y }\]
which we leave to the reader to verify. 
\end{proof}

\subsection{Further compatibilities}

Recall the enhancements of $\dec^*$ to a functor of cooperads (which is monoidal, cf.\@ also \ref{PROPSIMPLICIALOPERAD})
\[ \dec^*: (\mathcal{C}^{\Delta^{\op}}, \otimes)^{\vee} \rightarrow (\mathcal{C}^{\Delta^{\op} \times \Delta^{\op}}, \otimes)^{\vee}  \]
and (cf.\@ \ref{PARDEC})
\[ \dec^*_{\tildeotimes}: (\mathcal{C}^{\Delta^{\op}}, \tildeotimes)^{\vee} \rightarrow (\mathcal{C}^{\Delta^{\op} \times \Delta^{\op}}, \tildeotimes)^{\vee}  \]

\begin{LEMMA}\label{LEMMADECCOMM}
The following is commutative: 
\[ \xymatrix{
 (\mathcal{C}^{\Delta^{\op}}, \tildeotimes)^{\vee}  \ar[r]^-{\dec^*_{\tildeotimes}}  \ar[d]_{\EZfrak}& (\mathcal{C}^{\Delta^{\op} \times \Delta^{\op}}, \tildeotimes)^{\vee} \ar[d]^{(\EZfrak, \EZfrak)}  \\
 (\mathcal{C}^{\Delta^{\op}}, \otimes)^{\vee}  \ar[r]_-{\dec^*} & (\mathcal{C}^{\Delta^{\op} \times \Delta^{\op}}, \otimes)^{\vee} \\
}\]
\end{LEMMA}

\begin{BEM}\label{BEMDECNOTCOMM}
The corresponding square with the Alexander-Whitney-maps is {\em not} commutative!
\end{BEM}

\begin{proof}
Need to show the commutativity of
\[ \xymatrix{
\dec^* \dec_* \ar[d]_{\Ez} \ar[r] & \dec_{13,24,*}  \dec_{12,34}^*  \ar[d]^{(\Ez,\Ez)}  \\
\dec^* \delta^* \ar[r] &  \delta_{13,24}^* \dec^*_{12,34}
} \]
which is adjoint to
\[ \xymatrix{
 \dec_* \ar[d]_{\mathrm{unit} \circ \Ez} \ar[r] &  \dec_* \dec_{13,24,*}  \dec_{12,34}^*  \ar[d]^{(\Ez,\Ez)}  \\
\dec_* \dec^* \delta^* \ar[r] &  \dec_* \delta_{13,24}^* \dec^*_{12,34}
} \]
which is in standard form (\ref{PARSTANDARD}) 
and thus it suffices to see that
\[  \xymatrix{
\dec^* \ar[rr]^-{ (\mlq c^{\op} \mrq,\mlq c^{\op} \mrq) \dec^* }  \ar[d]_-{\dec^* \mlq c^{op} \mrq } &  & \delta_{13,24}^* \dec_{13,24}^* \dec^*  \ar[d]^{\sigma} \\
\dec^* \delta^* \dec^* \ar@{=}[rr] & &  \delta_{13,24}^*   \dec_{12,34}^* \dec^*  } \]
commutes.
Using the commutativity of $C$ with $\dec^*$ (Lemma~\ref{LEMMACCOH}), this boils down to the (trivial) commutativity of the following diagram of functors
$\FinSet^2 \to \FinSet$: 
\[ \raisebox{\dimexpr\depth-3\fboxsep}{ \xymatrix{
\dec_s \ar@{<-}[r]^-{\dec_s (c,c)}  \ar@{<-}[d]_-{c \dec_s}  & \dec_s \dec_{s,13,24} \delta_{s,13,24} \ar@{<-}[d]^{\sigma} \\
\dec_s \delta_s \dec_s \ar@{=}[r] &\dec_s  \dec_{s,12,34}  \delta_{s,13,24}   }} \qedhere \]
\end{proof}

\newpage

\bibliographystyle{abbrvnat}
\bibliography{cobar}

\end{document}